\documentclass[a4paper,11pt,reqno]{amsart}
\usepackage{amsthm}
\usepackage{amsmath}
\usepackage{enumerate}
\usepackage{amsfonts}
\usepackage{amssymb}
\usepackage{fullpage}
\usepackage{amsmath,amscd}
\usepackage{graphicx}
\usepackage{array}
\usepackage{amsmath,amssymb,amsfonts}
\usepackage[utf8]{inputenc} 
\usepackage[T1]{fontenc}
\usepackage[english,french]{babel}

 \usepackage{xspace}
\usepackage{xparse}
\usepackage{graphicx}
\usepackage{float}
\usepackage{pdfpages}
\usepackage[a4paper, margin = 3cm, bottom = 3cm]{geometry}
\usepackage{ifpdf}
\usepackage{marginnote}
\usepackage{hyperref}
\usepackage{tikz}
\usepackage{csquotes}
\usetikzlibrary{calc,matrix,arrows,shapes,decorations.pathmorphing,decorations.markings,decorations.pathreplacing}
\hypersetup{pdfborder=0 0 0, 
	    colorlinks=true,
	    citecolor=black,
	    linkcolor=blue,
	    urlcolor=black,
	    pdfauthor={Guillaume Tahar,Quentin Gendron}
	   }

\usepackage[
	    style=alphabetic,
            isbn=false,
            doi=false,
            url=false,
            backend=bibtex, 
            maxnames = 5,
            sorting=nty,
           ]{biblatex}
\defbibheading{bibliography}[\bibname]{%
  \section*{Références}%
}

\DeclareFieldFormat[article]{title}{{\it #1}}
\DeclareFieldFormat{journaltitle}{{\rm #1}}
\renewbibmacro{in:}{%
  \ifentrytype{article}{}{\printtext{\bibstring{in}\intitlepunct}}}
  \bibliography{res_biblio}

\newcommand{\e}{\mathrm{e}}

\newcommand{\RR}{\mathbf{R}}

\newcommand{\CC}{\mathbf{C}}
\newcommand{\NN}{\mathbf{N}}
\newcommand{\ZZ}{\mathbf{Z}}
\newcommand{\PP}{{\mathbb{P}}}
\newcommand{\calG}{{\mathcal{G}}}

\newcommand{\pgcd}{{\rm pgcd}}

\newcommand{\ord}{\operatorname{ord}\nolimits}

\renewcommand{\tilde}{\widetilde}

\newcommand{\rot}{{\rm rot}} 
\newcommand{\ind}{{\rm Ind}}

\newcommand{\wh}[1]{{\widehat{#1}}}
\newcommand\Res{\operatorname{Res}}
\newcommand{\Resk}[1][k]{\Res^{#1}}

\newcommand{\appres}[1][g]{\mathfrak{R}_{#1}}
\DeclareDocumentCommand{\appresk}{ O{g} O{k}}{\mathfrak{R}_{#1}^{#2}}
\newcommand{\espres}[1][g]{\mathcal{R}_{#1}}
\DeclareDocumentCommand{\espresk}{O{g} O{k}}{\mathcal{R}_{#1}^{#2}}

\newcommand{\barmoduli}[1][g]{{\overline{\mathcal M}}_{#1}}
\newcommand{\moduli}[1][g]{{\mathcal M}_{#1}}
\newcommand{\komoduli}[1][g]{{\Omega^k\mathcal M}_{#1}}
\newcommand{\omoduli}[1][g]{{\Omega\mathcal M}_{#1}}

\newcommand{\whX}{\widehat{X}}
\newcommand{\whomega}{\widehat{\omega}}

\newcommand{\Weierstrass}{Weierstra\ss{}}

\newcommand{\coeur}{c\oe{}ur\xspace}

\def\={\;=\;}

\newcommand{\Quentin}[1]{{\color{purple}{Quentin: #1}}}
\newtheorem{thm}{Théorème}[section]
\newtheorem{cor}[thm]{Corollaire}
\newtheorem{prop}[thm]{Proposition}
\newtheorem{lem}[thm]{Lemme}

\theoremstyle{definition}
\newtheorem{defn}[thm]{Définition}
\theoremstyle{remark}
\newtheorem{rem}[thm]{Remarque}
\theoremstyle{definition}

\theoremstyle{definition}
\newtheorem{ex}[thm]{Exemple}
\theoremstyle{definition}

\theoremstyle{definition}

\numberwithin{equation}{section} 
\title{différentielles à singularités prescrites}
\author{Quentin Gendron}
\address[Quentin Gendron]{Institut f\"ur algebraische Geometrie, Leibniz Universit\"at Hannover, Welfengarten 1,
30167 Hannover, Germany}
\curraddr{Centro de Ciencias Matem\`aticas-UNAM, Antigua Car. a P\`atzcuaro 8701,
Col. Ex Hacienda San Jos\'e de la Huerta,
Morelia, Mich., M\'exico}
\email{gendron@matmor.unam.mx}
\author{Guillaume Tahar}
\address[Guillaume Tahar]{Institut de Math{\'e}matiques de Jussieu - UMR CNRS 7586}
\curraddr{Faculty of Mathematics and Computer Science, Weizmann Institute of Science,
Rehovot, 7610001, Israel}
\email{tahar.guillaume@weizmann.ac.il}

\date{\today}
\keywords{Abelian differential, Flat surface, Residue, $k$-differentials}
\begin{document}

\selectlanguage{english}

\begin{abstract}
We study the local invariants that a meromorphic $k$-differential on a Riemann surface of genus $g\geq0$ can have. These local invariants are the orders of zeros and poles, and the $k$-residues at the poles. We show that for a given pattern of orders of zeroes, there exists, up to a few exceptions, a primitive $k$-differential having these orders of zeros (see Theorem~\ref{thm:strateshol}). The same is true for meromorphic $k$-differentials and in this case, we describe the tuples of complex numbers that can appear as $k$-residues at their poles. For genus $g\geq2$, and even $g=1$ if $k\neq2$, it turns out that every expected tuple appears as $k$-residues (see Theorem~\ref{thm:ggeq2}). On the other hand, some expected tuples are not the $k$-residus of a $k$-differential in some remaining strata. This happens for $k=1$ or $k\geq3$ in genus zero for a finite number (up to simultaneous scaling) of them. The abelian case is particularly interesting since the missing possibilities are either the zero tuple or special collinear numbers (see Theorem~\ref{thm:geq0keq1}). In the quadratic differentials cases, we show that one tuple can not be obtained in four families of strata in genus one (see Theorem~\ref{thm:geq1}). In genus zero, up to a $2$-dimensional locus can be impossible to realise as residue of quadratic differentials in some strata (see in particular Theorem~\ref{thm:geq0quad2}). We also give consequences of these results in algebraic and flat geometry.
\end{abstract}

\selectlanguage{french}

\maketitle
\setcounter{tocdepth}{1}
\tableofcontents

\section{Introduction}

Soient $X$ une surface de Riemann de genre $g$ et $K_{X}$ son fibré en droites canonique. Les sections méromorphes de $K_{X}$ sont les {\em différentielles abéliennes} de $X$ et les sections du produit tensoriel $K_{X}^{\otimes k}$ sont les {\em pluridifférentielles} ou {\em $k$-différentielles} de $X$. Localement, une $k$-différentielle s'écrit $f(z)(dz)^{k}$, où $f$ est une fonction méromorphe.

Il est bien connu (voir par exemple \cite[Encadré~III.2]{dSG}) que les invariants en un point~$P$ d'une différentielle abélienne $\omega$ sont l'{\em ordre} de la différentielle en~$P$ et le {\em résidu $\Res_{P}(\omega)$}  de celle-ci dans le cas où $P$ est un pôle de $\omega$. Ce résultat a été étendu au cas des différentielles quadratiques dans \cite{strebel} et à toutes les pluridifférentielles dans \cite{BCGGM3}.
Plus précisément, les invariants en $P$ d'une pluridifférentielle $\xi$ sont  l'{\em ordre} de la différentielle en $P$ et le {\em $k$-résidu}  $\Resk_{P}(\xi)$, si $P$ est un pôle de $\xi$.

Ces invariants ne peuvent pas être fixés arbitrairement mais vérifient certaines relations. Tout d'abord le $k$-résidu d'un pôle d'ordre $-k$ est toujours non nul alors que le $k$-résidu d'un pôle dont l'ordre n'est pas divisible par $k$ est toujours nul. Ensuite, la somme des ordres des zéros et des pôles d'une $k$-différentielle est égale à $k(2g-2)$. Enfin, dans le cas des différentielles abéliennes, la somme des résidus s'annule.

Il existe deux théorèmes célèbres donnant l'existence de pluridifférentielles méromorphes avec des conditions locales prescrites. 
\par
Le premier est le théorème de Riemann-Roch (voir le théorème 9.3 de \cite{reyssat}) qui permet d'obtenir certaines pluridifférentielles avec des ordres fixés. Toutefois, ce théorème ne donne aucune information sur les résidus de ces pluridifférentielles. De plus, il est souvent délicat d'en déduire des résultats pour tous les ordres que nous pouvons considérer. Par exemple, les résultats d'existence les plus fins que nous ayons (voir \cite{masm} et \cite{Diaz-Marin} dans le cas abélien et quadratique respectivement dans le cas holomorphe et méromorphe et \cite{troyanov} pour les $k$-différentielles holomorphes) ne font que peu appel à des résultats algébriques.
\par
 Le second théorème est le théorème de Mittag-Leffler (voir la proposition 9.3 de \cite{reyssat}) qui démontre l'existence de différentielles avec l'ordre des pôles et les résidus correspondants imposés. Toutefois, ce théorème ne donne aucune information sur les zéros de ces différentielles.

Dans cet article, nous nous proposons de répondre à la question suivante.

\begin{center}
{\em \'Etant donnés les ordres des zéros et pôles ainsi que les résidus aux pôles, existe-il une pluridifférentielle (primitive) ayant ces invariants locaux?}
\end{center}

\smallskip
\par
\subsection{Définitions.}
Afin de préciser la question centrale, nous introduisons un certain nombre de notions. Nous dénoterons par 
$$\mu:=(a_{1},\dots,a_{n};-b_{1},\dots,-b_{p};(-1^{s}))$$ une partition de $2g-2$  qui contient~$s$ fois $-1$, où les $a_{i}$ sont des nombres strictement positifs et les $b_{i}$ sont supérieurs ou égaux à $2$. 
De même, les partitions de $k(2g-2)$ seront notées
$$\mu:=(a_{1},\dots,a_{n};-b_{1},\dots,-b_{p};-c_{1},\dots,-c_{r};(-k^{s})),$$
où les $a_{i}$ sont supérieurs ou égaux à $-k+1$, les $b_{i}:= k\ell_{i}$ sont supérieurs ou égaux à $2k$ et divisibles par~$k$, les~$c_{i}$ sont supérieurs ou égaux à $k$ et non divisibles par $k$  et qui contient~$s$ fois~$-k$.
 La {\em strate primitive} $\komoduli(\mu)$ paramètre les $k$-différentielles {\em primitive} de type~$\mu$. On appelle {\em primitives} les $k$-différentielles qui ne sont pas des puissances de $k'$-différentielles pour $1\leq k' < k$. Les strates (non vides) de différentielles sont des orbifoldes de dimension $2g-1+n$ dans le cas holomorphe et $2g-2+n+p+s$ dans le cas méromorphe.  Les strates primitives de $k$-différentielles sont de dimension $2g-2+n+p+r+s$.
 Dans le cas abélien, le théorème des résidus implique que les strates $\omoduli(a_{1},\dots,a_{n};-1)$ sont vides. Dans la suite nous ne considérerons que des partitions de $2g-2$ distinctes de celles-ci.
 
 Il est bien connu que le résidu d'une différentielle $\omega$ à un pôle simple est non nul et que la somme des résidus de $\omega$ est nulle. Donc si $\omoduli(\mu)$ est une strate abélienne, on définit l'{\em espace résiduel de type $\mu$} par 
 \begin{equation}
\espres(\mu) := \left\{ (r_{1},\dots,r_{p+s})\in \CC^{p}\times(\CC^{\ast})^{s}:\ \sum_{i=1}^{p+s} r_{i}=0 \right\}.
\end{equation}
Cet espace paramètre l'ensemble des configurations de résidus que peut prendre une différentielle de $\omoduli(\mu)$.

Dans le cas des $k$-différentielles pour $k\geq2$, la description change quelque peu. Nous rappelons tout d'abord la notion de $k$-résidu. Plus de détails sont donnés dans \cite[Section~3]{BCGGM3}. Pour une $k$-différentielle $\xi$, au voisinage de chaque point $P$ de $X$, il existe une coordonnée $z$ telle que $\xi$ est de la forme
\begin{equation}\label{eq:standard_coordinates}
    \begin{cases}
      z^m\, (dz)^{k} &\text{si $m> -k$ ou  $k\nmid m$,}\\
      \left(\frac{r}{z}\right)^{k}(dz)^{k} &\text{si $m = -k$,}\\
        \left(z^{m/k} + \frac{t}{z}\right)^{k}(dz)^{k} &\text{si $m < -k$ et $k\mid m$,}
    \end{cases}
\end{equation}
où $t \in \CC$ et $r\in \CC^{\ast}$. Les nombres $r$ et $t$ sont définis à une racine de l’unité près. Le {\em $k$-résidu} de $\xi$ en $P$ est la puissance $k$-ième de $r$ dans le second cas, la puissance $k$-ième de $t$ dans le troisième cas et zéro sinon. Ainsi, le $k$-résidu est non nul dans le cas des pôles d'ordre $-k$ et  peut ne pas être nul uniquement dans le cas des pôles d'ordre divisible par $k$. Toutefois il n'existe pas de théorème des résidus pour les pluridifférentielles. Ainsi étant donné une strate $\komoduli(\mu)$, nous définissons  l'{\em espace $k$-résiduel de type $\mu$}  par\begin{equation}
\espresk(\mu) := \CC^{p}\times(\CC^{\ast})^{s}.
\end{equation}
Cet espace paramètre les configurations de résidus que peut prendre une pluridifférentielle de $\komoduli(\mu)$.

Dans tous les cas, l'{\em application résiduelle} est donnée par
\begin{equation}
\appresk(\mu) : \komoduli(\mu) \to \espresk(\mu):\ (X,\xi) \mapsto (\Resk_{P_{i}}(\xi)),
\end{equation}
où les $P_{i}$ sont les pôles de $\xi$ d'ordre divisible par $k$. Insistons sur le fait que par définition, les pluridifférentielles de $\komoduli(\mu)$ sont {\em primitives}.
Le but de cet article est de déterminer l'image de cette application pour chaque strate.
\par
\smallskip
\subsection{Genre supérieur ou égal à un.}

Nous sommes maintenant en mesure d'énoncer les résultats centraux de cet article. Rappelons que $\komoduli(\mu)$ paramètre les $k$-différentielles {\em primitive} de type $\mu$. Nous donnons d'abord le cas du genre supérieur ou égal à deux.

\begin{thm}\label{thm:ggeq2}
Si $g\geq2$ et $\mu$ contient un élément inférieur ou égal à $-k$, alors l'application résiduelle $\appresk(\mu) : \komoduli(\mu) \to \espresk(\mu)$ est surjective.
\end{thm}

Le cas des strates de genre un est légèrement plus complexe. En effet, il existe quatre familles de strates quadratiques exceptionnelles.

\begin{thm}\label{thm:geq1}
Soit $\komoduli[1](\mu)$ une strate de genre un et $\mu$ une partition de $k(2g-2)$ contenant  un élément inférieur ou égal à $-k$.
\begin{itemize}
\item[i)] Si $k=2$ et $\mu=(4a;(-4^{a}))$ ou $\mu=(2a-1,2a+1;(-4^{a}))$ pour $a\in\NN^{\ast}$, alors l'image de $\appresk[1][2](\mu)$ est égale à $\espresk[1][2](\mu)\setminus\left\{(0,\dots,0)\right\}$.
\item[ii)] Si $k=2$ et $\mu=(2s;(-2^{s}))$ ou $\mu=(s-1,s+1;(-2^{s}))$ avec $s$ un entier pair non nul, alors l'image de $\appresk[1][2](\mu)$ est égale à $\espresk[1][2](\mu)\setminus \CC^{\ast}\cdot(1,\dots,1)$.
\item[iii)]Dans tout autre cas, l'application résiduelle $\appresk[1](\mu)$ est surjective. 
\end{itemize} 
\end{thm}

En général, les strates de $k$-différentielles ne sont pas connexes. Les composantes connexes dans le cas abélien ont été classifiées par \cite{Bo}. Nous étendons les résultats précédents à chaque composante connexe des strates abéliennes.
\begin{prop}\label{prop:CC}
 \'Etant donné une composante connexe $S$ d'une strate de différentielles abéliennes $\omoduli(\mu)$ avec $g\geq1$. La restriction à $S$ de l'application résiduelle de $\omoduli(\mu)$ est surjective.
\end{prop}

Enfin nous traitons le cas des strates de différentielles n'ayant que des  singularités d'ordre strictement supérieur à $-k$. Ces différentielles correspondent à des surfaces plates à singularités coniques d'aire finie. Notre théorème~\ref{thm:strateshol} généralise le résultat de \cite{masm} dont les auteurs n'avaient considéré que les cas abélien et quadratique. Il permet également de préciser \cite{troyanov} qui ne se souciait pas de la primitivité des $k$-différentielles. 

\begin{thm}\label{thm:strateshol}
Soit $\mu=(a_{1},\dots,a_{n})$ une partition de $k(2g-2)$ telle que les $a_{i}$ soient strictement supérieurs à $-k$. La strate $\komoduli(\mu)$ paramétrant les diff\'erentielles primitives de profil $\mu$ est vide si et seulement si 
\begin{itemize}
\item[i)]$g=1$ et $\mu=(1,-1)$,
\item[ii)] $g=1$ et $\mu=\emptyset$ et $k\geq2$,
\item[iii)]$g=2$, $k=2$ et $\mu=(4)$ ou $\mu=(3,1)$.
\end{itemize}
\end{thm}

 \smallskip
\par
\subsection{Différentielles abéliennes en genre zéro.}
 
Pour de nombreuses strates de genre zéro, l'application résiduelle n'est pas surjective. Nous discutons d'abord le cas des différentielles abéliennes.

\begin{thm}\label{thm:geq0keq1}
Soient $g=0$, $k=1$ et $\mu=(a_{1},\dots,a_{n};-b_{1},\dots,-b_{p};(-1^{s}))$. L'un des cas suivants est vérifié.
\begin{itemize}
\item[i)]  Si $s=0$ et qu'il existe un indice $i$ tel que
\begin{equation}\label{eq:genrezeroresiduzerofr}
 a_{i}>\sum_{j=1}^{p}b_{j}-(p+1),
\end{equation}
alors l'image de $\appres[0](\mu)$ est $\espres[0](\mu)\setminus\left\{0\right\}$.
\item[ii)] S'il n'y a que des pôles simples (i.e. $p=0$), alors l'image de $\appres[0](\mu)$ est décrite dans les propositions~\ref{prop:gzeropolesimples} et \ref{prop:g0p-1plusieurszero}. 
\item[iii)] Dans les autres cas, l'application résiduelle $\appres[0](\mu)$ est surjective.
\end{itemize}
\end{thm}

Nous décrivons maintenant l'application résiduelle des strates $\omoduli[0](s-2;(-1^{s}))$. Pour cela nous définissons le graphe suivant.
 Un {\em graphe de connexions} est un arbre biparti connexe $\Gamma$ possédant $A$ arêtes, dont les sommets sont partitionnés en $\Gamma_{-}\cup\Gamma_{+}$ et auxquels sont attribués des poids réels strictement positifs, tels que:
\begin{enumerate}[i)]
 \item la somme des poids des sommets de $\Gamma_{+}$ est égale à celle des poids des sommets de~$\Gamma_{-}$;
 \item considérant l'opération qui consiste à retirer une feuille et soustraire le poids de ce sommet à celui qui lui est relié, appliquant cette opération entre une et $A-1$ fois à $\Gamma$, on obtient alors des graphes dont les poids sont strictement positifs.
\end{enumerate}
Si les nombres $r_{1},\dots,r_{s}$ sont $\RR$-colinéaires de somme nulle, alors un {\em graphe associé aux~$r_{i}$} est un arbre biparti connexe vérifiant les propriétés suivantes.  \'Etant donné $\alpha\in\CC^{\ast}$ tel que $r_{i}':=\alpha r_{i}\in\RR^{\ast}$ pour tout $i\leq s$. Les sommets de $\Gamma_{+}$ (resp. $\Gamma_{-}$) sont en bijection avec les $r_{i}'$ positifs (resp. négatifs). Le poids du sommet correspondant à $r_{i}$ est $|r_{i}'|$.

\begin{prop}\label{prop:gzeropolesimples}
 Soit $\omoduli[0](s-2;(-1^{s}))$ une strate de genre zéro avec $s$ pôles simples et un unique zéro d'ordre $s-2$. Les résidus $(r_{1},\ldots,r_{s})$ sont dans l'image de l'application résiduelle si et seulement si 
 l'une des propriétés est satisfaite.
 \begin{enumerate}
  \item Les résidus ne sont pas colinéaires.
  \item Il existe un graphe associé aux $r_{i}$ qui est un graphe de connexion.
 \end{enumerate}
\end{prop}

Nous montrerons  dans la proposition~\ref{prop:finietcommens} que les $s$-uplets $(r_{1},\dots,r_{s})$ qui ne sont pas dans l'image de $\appres[0](s-2;(-1^{s}))$ sont en nombre fini (modulo multiplication par une constante) pour tout $s\geq2$. De plus, dans ce cas, les $r_{i}$ sont commensurables entre eux.

Le cas des strates ayant plus de zéros est un peu moins explicite. Afin de décrire l'image de l'application résiduelle, nous faisons appel à la notion de différentielle stable rappelée dans la section~\ref{sec:bao}.

\begin{prop}\label{prop:g0p-1plusieurszero}
 Soit $\omoduli[0](a_{1},\dots,a_{n};(-1^{s}))$ une strate de genre zéro avec $s$ pôles simples et $n$ zéros. Le $s$-uplet $(r_{1},\ldots,r_{s})$ est dans l'image de l'application résiduelle si et seulement s'il existe une différentielle stable de genre zéro $(X,\omega)$ avec un pôle simple à chaque nœud, dont les autres singularités sont d'ordres  $(a_{1},\dots,a_{n};(-1^{s}))$, les résidus à ces pôles sont $(r_{1},\dots,r_{s})$ et qui possède un unique zéro sur chaque composante irréductible de $X$.
\end{prop}

Cette proposition permet donc d'étudier l'image de l'application résiduelle en se ramenant au cas d'un zéro. De plus, cette condition peut se reformuler en termes de graphes munis de poids. Nous le ferons dans la section~\ref{sec:MDPR} et donnerons des exemples.

\smallskip
\par
\subsection{ Pluridifférentielles en genre zéro.}

Le cas des pluridifférentielles en genre zéro présente de nombreuses difficultés. En particulier, il existe des strates pour lesquelles le complémentaire de l'image de l'application résiduelle est de dimension $1$ ou $2$ (modulo dilatations). 

Rappelons qu'étant donné une partition
$$\mu:=(a_{1},\dots,a_{n};-b_{1},\dots,-b_{p};-c_{1},\dots,-c_{r};(-k^{s}))$$
de $-2k$, l'espace $\komoduli[0](\mu)$ paramètre les $k$-différentielles primitives de type $\mu$. 
On commence par remarquer (cf lemme~\ref{lem:puissk}) que ces strates sont non vides si et seulement si $\pgcd(\mu,k)=1$. Dans la suite, cette condition sera toujours implicitement satisfaite.

Dans le cas où $r\neq0$, on a le résultat suivant.
\begin{thm}\label{thm:g=0gen1}
 Soit $\komoduli[0](a_{1},\dots,a_{n};-b_{1},\dots,-b_{p};-c_{1},\dots,-c_{r};(-k^{s}))$ une strate de genre zéro telle que $r\neq0$. L'image de l'application résiduelle est
 \begin{itemize}
 \item[i)]  $\espresk[0](\mu)\setminus\left\{0\right\}$ si $r=1$ et $s=0$ et il existe au plus un zéro d'ordre non divisible par $k$ et la somme des ordres des zéros d'ordres divisibles par $k$ est strictement inférieur à~$kp$,
 \item[ii)] $\espresk[0](\mu)$ sinon.
\end{itemize}  
\end{thm}

Nous donnons maintenant la description des strates pour $r=0$, $s\neq0$ et $p\neq0$. S'il y a des pôles d'ordre $-k$, on a la description suivante.
\begin{thm}\label{thm:r=0sneq0}
L'application résiduelle des strates $\komoduli[0](a_{1},\dots,a_{n};-b_{1},\dots,-b_{p};(-k^{s}))$ avec $s\neq0$ est surjective sauf dans les cas exceptionnels suivants.
\begin{enumerate}[i)]
\item L'image de $\appresk[0][2](2s'-1;2s'+1;-4;(-2^{2s'}))$ avec $s'\geq1$  est $\espresk[0][2](\mu)\setminus\CC^{\ast}\cdot(0;1,\dots,1)$.

\item L'image de $\appresk[0][2](2a-1;2a+1;(-4^{a});(-2^{2}))$ avec $a\geq 0$ est $\espresk[0][2](\mu)\setminus\CC^{\ast}\cdot(0,\dots,0;1,1)$.

\item L'image de $\appresk[0][2](2s'+1;2s'+1;-4;(-2^{2s'+1}))$ avec $s'\geq0$  est $\espresk[0][2](\mu)\setminus\CC^{\ast}\cdot(1;1,\dots,1)$.

\item L'image de $\appresk[0][2](2a-1,2a-1;(-4^a);-2)$ avec   $s\geq0$ est $\espresk[0][2](\mu)\setminus\CC^{\ast}\cdot(1,0,\dots,0;1) $.
\end{enumerate}
\end{thm}
Dans le cas où $r=s=0$ et $p\neq0$,  l'image de l'application $k$-résiduelle est donné par le résultat suivant.
%
%
%
%

\begin{thm}\label{thm:r=0s=0}
Soit $\komoduli[0](a_{1},\ldots,a_{n};-b_{1},\dots,-b_{p})$ une strate non vide de genre zéro.

L'image de l'application $k$-résiduelle de cette strate contient $\espres[0](\mu)\setminus\lbrace(0,\dots,0)\rbrace$ à l'exception des  deux familles de strates quadratiques $\Omega^{2}\moduli[0](2p+b-5,2p+b-5;-b,-b-2,(-4^{p-2}))$ et $\Omega^{2}\moduli[0](2p+b-7,2p+b-5;-b,-b,(-4^{p-2}))$ avec $p\geq2$ et $b\geq4$ pair.
Dans le cas de ces strates, l'image de l'application $2$-résiduelle est le complémentaire de la droite engendrée par le vecteur $(1,1,(0^{p-2}))$.

De plus, l'origine est dans l'image de l'application $k$-résiduelle si et seulement si 
\begin{enumerate}[i)]
 \item $p=1$ et $n\geq3$, ou
 \item $p\geq 2$, $n\geq 3$ et il existe au moins trois zéros d'ordres non divisibles par $k$, ou
 \item  $p\geq 2$, $n\geq 3$, il existe deux zéros d'ordres non divisibles par $k$ et la somme des ordres des zéros divisibles par $k$ est  supérieure ou égale à~$kp$.
 
\end{enumerate}
En particulier, si $n=2$, alors $(0,\dots,0)$ n'est pas dans l'image de l'application $k$-résiduelle. 
\end{thm}

Enfin nous donnons la description de l'image par l'application résiduelle des strates telles que $p=r=0$. Le cas des différentielles quadratiques est très différent du cas des $k$-différentielles avec $k\geq3$. Obtenir une caractérisation précise du cas quadratique nécessiterait un autre travail. Nous donnons toutefois un certain nombre de résultats partiels intéressants.
Dans les cas où il y a plus de trois zéros, nous pouvons énoncer le résultat suivant.
\begin{prop}\label{prop:quadsurjbcpimp}
L'application résiduelle des strates $\Omega^{2}\moduli[0](a_{1},\dots,a_{n};(-2^{s}))$ est surjective dans les deux cas suivants.
\begin{enumerate}[1)]
\item $n\geq4$ et au moins quatre $a_{i}$ sont impairs. 
\item $n=3$ et $a_{1}+a_{2}< a_{3}$ avec $a_{3}$ pair.
\end{enumerate}
\end{prop}

Dans le cas des strates avec deux zéros, la description est très complexe. Si les résidus quadratiques sont sur le même rayon réel issu de l'origine, on peut utiliser une généralisation des graphes de connexion. Toutefois, la combinatoire devient très complexe et il semble difficile de tirer des informations d'une telle description. On pourra toutefois consulter la proposition~\ref{prop:quadmoinsun} dans le cas où $a_{1}=-1$. Dans ces strates, le concept suivant est important pour comprendre l'application résiduelle.
\begin{defn}\label{def:triangulaire}
 Des nombres $R_{1},R_{2},R_{3}$ sont {\em triangulaires} s'il existe des racines carrées $r_{1},r_{2},r_{3}$ de ces nombres telles que $r_{1}+r_{2}+r_{3}=0$.
\end{defn} 

On a alors le résultat surprenant suivant.
\begin{thm}\label{thm:geq0quad2}
L'application résiduelle de $\Omega^{2}\moduli[0](a_{1},\dots,a_{n};(-2^{s}))$ contient tous les $s$-uplets qui n'appartiennent pas à un même rayon issu de l'origine sauf dans les deux cas suivants.
\begin{enumerate}[i)]
\item L'image par l'application résiduelle des strates 
$\Omega^{2}\moduli[0](2s-1,2s+1;(-2^{2s'+2}))$  contient les $2$-résidus n'appartenant pas à un même rayon issu de l'origine sauf ceux proportionnels à  $(1,\dots,1,R,R)$ pour tout $R\in\CC^{\ast}$.

\item L'image par l'application résiduelle des strates $\Omega^{2}\moduli[0](2s'-1,2s'-1;(-2^{2s'+1}))$   contient les $2$-résidus n'appartenant pas à un même rayon issu de l'origine sauf ceux proportionnels à
$(R_{1},R_{2},R_{3},\dots,R_{3})$ où les $R_{i}$  sont triangulaires.
\end{enumerate}
\end{thm}

Maintenant, nous considérons les strates de $k$-différentielles pour $k\geq3$. Mis à part les strates de la forme $\komoduli[0](-1,1;-k,-k)$, toutes les exceptions proviennent du fait que les racines $k$-ième de l'unité engendrent un réseau de $\CC$ si et seulement si $k\in\left\{3,4,6\right\}$.
\begin{thm}\label{thm:geq0kspe}
L'application résiduelle des strates $\komoduli(a_{1},\dots,a_{n};(-k^{s}))$ est surjective pour $k\geq3$ sauf dans les cas suivants.
\begin{enumerate}
\item L'image de  $\appresk[0](-1,1;-k,-k)$ est $(\CC^{\ast})^{2}\setminus \CC^{\ast}\cdot(1,(-1)^{k})$.
\item L'image de   $\appresk[0][3](-1,4;(-3^{3}))$ est $(\CC^{\ast})^{3}\setminus \CC^{\ast}\cdot(1^{3})$. 
 \item L'image de  $\appresk[0][3](1,2;(-3^{3}))$ est $(\CC^{\ast})^{3}\setminus \CC^{\ast}\cdot(1^{3})$.
 \item  L'image de  $\appresk[0][3](2,4;(-3^{4}))$ est $(\CC^{\ast})^{4}\setminus \CC^{\ast}\cdot(1,1,-1,-1)$.
\item L'image de  $\appresk[0][3](2,7;(-3^{5}))$ est $(\CC^{\ast})^{5}\setminus \CC^{\ast}\cdot(1^{5}) \cup\CC^{\ast}\cdot((1^{4}),-1)$.
\item L'image de  $\appresk[0][3](2,10;(-3^{6}))$ est $(\CC^{\ast})^{6}\setminus \CC^{\ast}\cdot((1^{3}),(-1^{3}))\cup \CC^{\ast}\cdot(1^{6})$.
\item  L'image de $\appresk[0][3](7,5;(-3^{6}))$ est $(\CC^{\ast})^{6}\setminus \CC^{\ast}\cdot(1^{6})$.
\item L'image de  $\appresk[0][4](-1,5;(-4^{3}))$ est $(\CC^{\ast})^{3}\setminus \CC^{\ast}\cdot(1,1,-4)$.
\item L'image de $\appresk[0][4](3,5;(-4^{4}))$ est $(\CC^{\ast})^{4}\setminus \CC^{\ast}\cdot(1^{4})$.
\item L'image de $\appresk[0][4](-1,9;(-4^{4}))$ est $(\CC^{\ast})^{4}\setminus \CC^{\ast}\cdot(1^{4})$.
\item L'image de $\appresk[0][4](3,13;(-4^{6}))$ est $(\CC^{\ast})^{6}\setminus \CC^{\ast}\cdot(1^{6})$.
\item L'image de  $\appresk[0][6](-1,7;(-6^{3}))$ est $(\CC^{\ast})^{3}\setminus \CC^{\ast}\cdot(1^{3})$.
\end{enumerate}
\end{thm}

\smallskip
\par
\subsection{Applications.}
\label{sec:apliintro}

Nous donnons maintenant quelques applications de nos résultats. Ils fournissent des informations intéressantes  couplés avec les théorèmes de \cite{BCGGM} et \cite{BCGGM3}.
Ces travaux décrivent un bord des strates de $k$-différentielles via le concept de {\em différentielles entrelacées} (voir section~\ref{sec:bao}). Afin de savoir si une différentielles entrelacée est au bord d'une strate donnée, on a besoin de connaître l'existence de différentielles où l'on connaît les ordres des zéros et des pôles ainsi que des résidus de différentielles. Notre article peut donc être vu comme la dernière pierre dans cette description ensembliste du bord des strates.

\smallskip
\par
\paragraph{\bf Lieu de Weierstra\ss.}
La première conséquence est un résultat sur la géométrie de l'espace des modules des courbes pointées. Notons que la première des deux assertion a été prouvée par Eisenbud et Harris dans le théorème 3.1 de \cite{eiha} à l'aide des séries linéaires limites. 
\begin{prop}\label{prop:limWei}
Soit $\overline{\mathfrak{W}_{g}}$ l'adhérence dans $\barmoduli[g,1]$ du lieu paramétrant les points de Weierstra\ss{}~des courbes algébriques de genre $g$. 
\begin{itemize}
\item[i)] $\overline{\mathfrak{W}_{g}}$ ne rencontre pas le lieu des courbes stables où $g$ courbes elliptiques sont attachées à un $\PP^{1}$ contenant le point marqué.
\item[ii)] $\overline{\mathfrak{W}_{g}}$ intersecte le lieu où $g-1$ courbes elliptiques sont attachées à un $\PP^{1}$ contenant le point marqué et l'une de ces courbes elliptiques est attachée par deux points au $\PP^{1}$. 
\end{itemize}
\end{prop}

\smallskip
\par
\paragraph{\bf Éclatement de zéros}
Nous donnons une application dans l'esprit de la géométrie plate. Cette proposition est connue dans le cas abélien \cite{EMZ} et quadratique \cite{lanneauquad}.
\begin{prop}\label{prop:eclatintro}
Soient $\xi$  une $k$-différentielle et $z_{0}$ une singularité d'ordre $a_{0}>-k$ de~$\xi$.
Il est pas possible de scinder {\em localement} $z_{0}$ en $t$ singularités d'ordres $(\alpha_{1},\dots,\alpha_{t})$ avec $a_{0}=\sum \alpha_{i}$ et $\alpha_{i}>-k$ si et seulement si $k\geq2$, $t=2$, $k\mid a_{0}$ et  $k\nmid\pgcd(\alpha_{i})$.
\end{prop}

\smallskip
\par
\paragraph{\bf Cylindres sur une surface plate.}
Nous donnons une application à la géométrie des surfaces plates d'aire finie.  Naveh a montré dans \cite{Na} que dans une strate de différentielles abéliennes, le nombre maximal de cylindres disjoints pour une surface donnée est $g+n-1$ et que cette borne est toujours atteinte. Nous décrivons les périodes possibles des circonférences de ces cylindres.
\begin{prop}\label{prop:cylindres}
Soient $S:=\omoduli(a_{1},\dots,a_{n})$ une strate de différentielles abéliennes et $\lambda:=(\lambda_{1},\dots,\lambda_{t})\in (\mathbb{C}^{\ast})^{t}$.
Il existe  une différentielle dans $S$ avec $t$ cylindres disjoints dont les circonférences sont $\lambda_{1},\dots,\lambda_{t}$ si et seulement s'il existe une
différentielle stable $(X,\omega)$ avec un pôle simple aux $t$ nœuds, telle que les résidus à ces nœuds sont $\pm \lambda_{i}$ et dont les ordres des autres singularités sont donnés par  $(a_{1},\dots,a_{n})$
\end{prop}
Nous énoncerons ce résultat dans le langage des graphes dans la section~\ref{sec:appli}.
\smallskip
\par
\subsection{Organisation de cet article.}

Le schéma de la preuve de ces théorèmes est le suivant. Dans un premier temps nous utilisons la correspondance entre les (pluri)différentielles méromorphes et certaines classes de surfaces plates introduites par \cite{Bo} et étudiées dans \cite{tahar} dans le cas $k=1$ et par \cite{BCGGM3} dans le cas $k\geq2$. Cette correspondance nous permet de construire explicitement des pluridifférentielles ayant les propriétés souhaitées lors que le genre et le nombre de singularités sont petits. 

Dans un second temps, ce résultat nous déduisons les autres cas grâce à deux opérations introduites par \cite{kozo} pour $k=1$, \cite{lanneauquad} pour $k=2$ et généralisées par \cite{BCGGM3} pour $k\geq3$. Ces deux opérations sont l'{\em éclatement d'une singularité} et la {\em couture d'anse}. La première de ces opérations permet d'augmenter le nombre de singularités sans changer le genre d'une pluridifférentielle. La seconde préserve le nombre de singularités mais augmente le genre de la surface sous-jacente.

Enfin, dans les cas où l'application résiduelle n'est pas surjective, nous développons des méthodes ad hoc afin de montrer la non-existence de pluridifférentielles ayant certains invariants locaux. 

L'article s'organise comme suit. Pour terminer cette introduction, nous posons quelques conventions. Dans la section~\ref{sec:bao} nous faisons les rappels nécessaires sur les représentations plates des pluridifférentielles méromorphes et sur les deux opérations précédemment citées. De plus, nous introduisons dans cette section les briques élémentaires qui nous permettrons de construire les pluridifférentielles avec les propriétés souhaitées. La section~\ref{sec:MDPR} est dédiée aux différentielles abéliennes.   La section~\ref{sec:k-diff} est dédiée au cas des pluridifférentielles de genre~$0$. La section~\ref{sec:ggeq1} est dédiée aux cas des pluridifférentielles de genre supérieur ou égal à~$1$. Enfin les preuves des applications de la section~\ref{sec:apliintro} sont données dans la section~\ref{sec:appli}.

Nous conseillons au lecteur pressé de restreindre sa lecture aux sections \ref{sec:bao}, \ref{sec:MDPR} et \ref{sec:appli}.

\smallskip
\par
\subsection{Conventions.}
Dans cet article nous définirons le {\em résidu} d'une forme différentielle comme étant le résidu usuel multiplié par la constante $2i\pi$. En particulier, le théorème des résidus s'énonce 
\[\int_{\gamma}\omega=\sum_{i} \Res_{P_{i}}\omega,\]
où les $P_{i}$ sont les pôles de $\omega$ encerclés par $\gamma$.  De plus, pour une $k$-différentielle, le $k$-résidu sera $(2i\pi)^{k}$ fois le $k$-résidu défini par l'équation~\eqref{eq:standard_coordinates}. Remarquons que ces conventions n'ont aucune incidence sur l'énoncé des résultats, mais rend les preuves plus agréables.

Pour une $k$-différentielle $\xi$, nous appelons {\em zéro} une singularité de $\xi$ d'ordre strictement supérieur à $-k$ et {\em pôle} une singularité d'ordre inférieur ou égal à $-k$. Le $k$-résidu d'une $k$-différentielle sera noté avec une lettre majuscule $R$ tandis qu'une racine $k$-ième de $R$ sera notée par une lettre minuscule $r$. Cette convention de langage peut paraître curieuse, mais sera extrêmement commode.

Si une strate paramètre des différentielles avec $m$ singularités égales à $a$, alors nous noterons cela $(a^{m})$. Par exemple $\omoduli[3](3,3;-1,-1)$ pourra être notée $\omoduli[3]((3^{2});(-1^{2}))$. Plus généralement, si nous considérons une suite $(a,\dots,a)$ de $m$ nombres complexes tous identiques, nous noterons cette suite $(a^{m})$. Nous espérons que ces notations seront claires dans le contexte.

\smallskip
\par
\subsection{Remerciements.} Nous remercions Corentin Boissy pour des discussions enrichissantes liées à cet article. Une question de David Aulicino est à l'origine de la proposition~\ref{prop:cylindres}. Le logiciel {\em GeoGebra} a procuré une aide substantielle au premier auteur. De plus, celui-ci à été supporté par une bourse postdoctoral de la DGAPA, UNAM et par le projet CONACyT A1-S-9029 "Moduli de curvas y Curvatura en $A_{g}$" d'Abel Castorena lors de  la révision de cet article. Le deuxième auteur est financé par l’Israel Science Foundation (grant No. 1167/17) et le European Research Council (ERC) dans le cadre du European Union Horizon 2020 research and innovation programme (grant agreement No. 802107).
 Enfin, nous remercions chaleureusement le rapporteur anonyme pour sa relecture attentive et ses remarques précieuses qui ont grandement amélioré la qualité de cet article. Il nous a en particulier permis de corriger certains énoncés relatifs aux pluridifférentielles.

\section{Boîte à outils}
\label{sec:bao}
Dans cette section, nous introduisons les objets et les opérations de base pour nos constructions. Nous commençons par quelques rappels sur les pluridifférentielles dans la section~\ref{sec:pluridiffbao}. Ensuite nous introduisons dans la section~\ref{sec:briques} les briques élémentaires de nos surfaces plates. Nous poursuivons par un rappel sur les différentielles entrelacées et les opérations de scindage de zéro et de couture d'anse dans la section~\ref{sec:pluridiffentre}. Enfin, nous discuterons un cas spécial de surfaces plates dans la section~\ref{sec:coeur}.

\subsection{Pluridifférentielles méromorphes}
\label{sec:pluridiffbao}

Dans ce paragraphe, nous rappelons des résultats élémentaires sur les pluridifférentielles méromorphes et leur relation avec les surfaces plates. Plus de détails peuvent être trouvés dans \cite{Bo} pour les différentielles abéliennes et \cite{BCGGM3} pour les $k$-différentielles avec $k\geq1$.
\par
Soit $X$ une surface de Riemann de genre $g$ et $\omega$ une section méromorphe du fibré canonique $K_{X}$. On notera $Z$ les zéros  et $P$ les pôles de $\omega$. L'intégration de $\omega$ sur $X\setminus P$ induit une structure plate sur $X\setminus P$. Chaque zéro de $\omega$ d'ordre $a$ correspond à une singularité conique d'angle $2(a+1)\pi$ de la structure plate. Les pôles simples de $\omega$ correspondent à des demi-cylindres infini. Les pôles d'ordre $-b\leq -2$ correspondent à un revêtement de degré $b-1$ du plan dans lequel on a éventuellement fait une entaille correspondant au résidu. 
Inversement, une surface plate obtenue en attachant par translation un nombre fini de demi-cylindres infinis, des revêtements d'ordre $b-1$ du plan entaillé et des polygones correspond à une différentielle abélienne méromorphe sur une surface de Riemann.

Une théorie similaire a été développée dans le cas des sections méromorphes $\xi$ de la puissance tensorielle $k$-ième $K_{X}^{k}$ du fibré canonique de $X$. En effet, on peut passer au revêtement canonique $\pi:\whX\to X$ et choisir une racine $k$-ième $\whomega$ de $\pi^{\ast}\xi$ sur $\whX$. L'intégration de $\whomega$ le long d'un chemin de $\whX$ nous fournit une structure plate sur $\whX$. La surface plate ainsi obtenue possède une symétrie cyclique d'ordre $k$ provenant de la structure de revêtement. Le quotient de cette surface par ce groupe est une surface plate où l'on autorise les identifications par des translations et par des rotations d'angle multiple de  $\frac{2\pi}{k}$. Les pôles d'ordre~$-k$ correspondent à des demi-cylindres infinis et les pôles d'ordre $-b<-k$ à un revêtement d'ordre $b-k$ d'un domaine angulaire d'angle $\frac{2\pi}{k}$. Les pôles d'ordre $0>a>-k$ et les zéros d'ordre  $a\geq 0$ correspondent aux singularités coniques d'angle $(a+k)\frac{2\pi}{k}$.   Cette similitude explique la convention de langage que les singularités d'ordre $a\geq-k+1$ sont des {\em zéros} d'ordre $a$ de $\xi$.

Enfin, pour les $k$-différentielles non primitives, on a la propriété suivante. Si la $k$-différentielle $\xi$ est la puissance $d$-ième d'une $(k/d)$-différentielle $\eta$, alors pour tout pôle~$P$ nous avons
\begin{equation}\label{eq:multiplires}
 \Resk_{P}(\xi)=\left( \Resk[k/d]_{P}(\eta) \right)^{d}.
\end{equation}

\subsection{Briques élémentaires}
\label{sec:briques}

Dans ce paragraphe, nous introduisons des surfaces plates à bord qui nous serviront de briques pour construire les pluridifférentielles ayant les propriétés locales souhaitées.

Nous décrivons dans un premier temps les briques pour les différentielles abéliennes. 
Les éléments de base de nos constructions seront les {\em domaines basiques} introduit par \cite{Bo}.\footnote{Nous mettons en garde que notre définition des domaines basiques est plus générale que celle de \cite{Bo}.} Étant donnés des vecteurs $(v_{1},\dots,v_{l})$ dans $(\CC^{\ast})^{l}$. Nous considérons la ligne brisée $L$ dans $\CC$ donnée par la concaténation d'une demi-droite correspondant à $\RR_{-}$, des $v_{i}$ pour $i$ croissant et d'une demi-droite correspondant à $\RR_{+}$.
Nous supposerons que les $v_{i}$ sont tels que $L$ ne possède pas de points d'auto-intersection. 
 
Le {\em domaine basique positif (resp. négatif)} $D^{+}(v_{i})$ (resp. $D^{-}(v_{i})$) est l'adhérence de la composante connexe de $\CC\setminus L$ contenant les nombres complexes au dessus (resp. en dessous) de~$L$.
Étant donné un domaine positif $D^{+}(v_{i})$ et un négatif $D^{-}(w_{j})$, on construit le {\em domaine basique ouvert à gauche (resp. droite)} $D_{g}(v_{i};w_{j})$ (resp. $D_{d}(v_{i};w_{j})$) en collant par translation les deux demi-droites correspondant à $\RR_{+}$ (resp. $\RR_{-}$).

On se donne maintenant des vecteurs $(v_{1},\dots,v_{l})$ avec $l\geq1$ tels que la concaténation $V$ de ces vecteurs dans cet ordre n'a pas de points d'auto-intersection. De plus, on suppose qu'il existe deux demi-droites parallèles  $L_{D}$ et $L_{F}$ de vecteur directeur $\overrightarrow{l}$, issues respectivement du point de départ $D$ et final $F$ de $V$, ne rencontrant pas $V$ et telles que $(\overrightarrow{DF},\overrightarrow{l})$ est une base positive de $\RR^{2}$. On définit la {\em partie polaire $C(v_{i})$ d'ordre $1$ associé aux $v_{i}$}  comme le quotient du sous-ensemble de $\CC$ entre $V$ et les demi-droites $L_{D}$ et~$L_{F}$ par l'identification de $L_{D}$ à~$L_{F}$ par translation. Le résidu du pôle simple correspondant est donné par la somme $F-D$ des~$v_{i}$.

On se donne $b\geq2$ et $\tau\in\left\{1,\dots,b-1\right\}$.
Soient $(v_{1},\dots,v_{l};w_{1},\dots,w_{\ell})$ des vecteurs de~$\CC^{\ast}$ tels que la partie réelle de leurs sommes est positive et que l'argument (pris dans $\left]-\pi,\pi\right]$) des $v_{i}$ est décroissant, des $w_{j}$ est croissant.\footnote{Le lecteur prendra garde à la différence typographique entre $l$ et $\ell$ dans l'ensemble de l'article.}
La  {\em partie polaire d'ordre $b$ et de type $\tau$ associée à $(v_{1},\dots,v_{l};w_{1},\dots,w_{\ell})$} est la surface plate à bord obtenue de la façon suivante. Prenons l'union disjointe de  $\tau-1$ domaines basiques ouverts à gauche associé à la suite vide, $b-\tau-1$   domaines basiques ouverts à droite associé à la suite vide. Enfin prenons le domaine positif associé aux $v_{i}$ et le domaine négatif associé aux $w_{j}$. On colle alors par translation la demi-droite  inférieure du $i$-ième domaine polaire ouvert à gauche à la demi-droite supérieure du $(i+1)$-ième. La demi-droite inférieure du domaine $\tau-1$ est identifiée à la demi droite de gauche du domaine positif. La demi-droite de gauche du domaine négatif est identifiée à la positive du premier domaine ouvert à gauche. On procède de même à droite. La  figure~\ref{fig:ordreplusmoins} illustre cette construction.

\begin{figure}[htb]
\center
\begin{tikzpicture}

\begin{scope}[xshift=-5cm]
\fill[fill=black!10] (0,0)  circle (1.5cm);

\draw[] (0,0) coordinate (Q) -- (-1.5,0) coordinate[pos=.5](a);

\node[above] at (a) {$\tau$};
\node[below] at (a) {$1$};

\fill (Q)  circle (2pt);

\node at (2.5,0) {$\dots$};
\end{scope}

\begin{scope}[xshift=0cm]
\fill[fill=black!10] (0,0)  circle (1.5cm);
      \fill[color=white]
      (-.4,0.02) -- (.4,0.02) -- (.4,-0.02) -- (-.4,-0.02) --cycle;

\draw[] (-.4,0.02)  -- (.4,0.02) coordinate[pos=.5](a);
\draw[] (-.4,-0.02)  -- (.4,-0.02) coordinate[pos=.5](b);

\node[above] at (a) {$v$};
\node[below] at (b) {$v$};

\draw[] (-.4,0) coordinate (q1) -- (-1.5,0) coordinate[pos=.5](d);
\draw[] (.4,0) coordinate (q2) -- (1.5,0) coordinate[pos=.5](e);

\node[above] at (d) {$1$};
\node[below] at (d) {$2$};
\node[above] at (e) {$\tau+1$};
\node[below] at (e) {$\tau+2$};

\fill (q1) circle (2pt);
\fill[white] (q2) circle (2pt);
\draw (q2) circle (2pt);

\node at (2.5,0) {$\dots$};
\end{scope}

\begin{scope}[xshift=5cm]
\fill[fill=black!10] (0,0) coordinate (Q) circle (1.5cm);

\draw[] (0,0) coordinate (Q) -- (1.5,0) coordinate[pos=.5](a);

\node[above] at (a) {$b$};
\node[below] at (a) {$\tau+1$};

\fill[white] (Q)  circle (2pt);
\draw (Q)  circle (2pt);
\end{scope}

\end{tikzpicture}
\caption{Une partie polaire d'ordre $b$ de type $\tau$ associée à $(v;v)$.} \label{fig:ordreplusmoins}
\end{figure}

Si $\sum v_{i} =\sum w_{j}$ nous dirons que cette partie polaire est {\em triviale}. Dans le cas contraire, nous dirons que la partie polaire est {\em non triviale}. Sur la figure~\ref{fig:partiesnontrivial}, le dessin de gauche illustre une partie polaire non triviale et celui de droite une partie polaire d'ordre $1$.

\begin{figure}[htb]
\center
\begin{tikzpicture}

\begin{scope}[xshift=-7cm]
\fill[fill=black!10] (0,0) coordinate (Q) circle (1.5cm);

\coordinate (a) at (-.5,0);
\coordinate (b) at (.5,0);
\coordinate (c) at (0,.2);

\fill (a)  circle (2pt);
\fill[] (b) circle (2pt);
    \fill[white] (a) -- (c)coordinate[pos=.5](f) -- (b)coordinate[pos=.5](g) -- ++(0,-2) --++(-1,0) -- cycle;
 \draw  (a) -- (c) coordinate () -- (b);
 \draw (a) -- ++(0,-1.1) coordinate (d)coordinate[pos=.5] (h);
 \draw (b) -- ++(0,-1.1) coordinate (e)coordinate[pos=.5] (i);
 \draw[dotted] (d) -- ++(0,-.3);
 \draw[dotted] (e) -- ++(0,-.3);
\node[below] at (f) {$v_{1}$};
\node[below] at (g) {$v_{2}$};
\node[left] at (h) {$1$};
\node[right] at (i) {$1$};

\draw (b)-- ++ (1,0)coordinate[pos=.6] (j);
\node[below] at (j) {$2$};
\node[above] at (j) {$3$};
    \end{scope}

\begin{scope}[xshift=-3.5cm]
\fill[fill=black!10] (0,0) coordinate (Q) circle (1.5cm);

\draw[] (0,0) -- (1.5,0) coordinate[pos=.5](a);

\node[above] at (a) {$2$};
\node[below] at (a) {$3$};
\fill[] (Q) circle (2pt);
\end{scope}

\begin{scope}[xshift=3cm,yshift=-.5cm]
\coordinate (a) at (-.5,0);
\coordinate (b) at (.5,0);
\coordinate (c) at (0,.2);

    \fill[fill=black!10] (a) -- (c)coordinate[pos=.5](f) -- (b)coordinate[pos=.5](g) -- ++(0,1.5) --++(-1,0) -- cycle;
    \fill (a)  circle (2pt);
\fill[] (b) circle (2pt);
 \draw  (a) -- (c) coordinate () -- (b);
 \draw (a) -- ++(0,1.3) coordinate (d)coordinate[pos=.5](h);
 \draw (b) -- ++(0,1.3) coordinate (e)coordinate[pos=.5](i);
 \draw[dotted] (d) -- ++(0,.3);
 \draw[dotted] (e) -- ++(0,.3);
\node[below] at (f) {$v_{1}$};
\node[below] at (g) {$v_{2}$};
\node[left] at (h) {$3$};
\node[right] at (i) {$3$};

    \end{scope}
\end{tikzpicture}
\caption{Une partie polaire non triviale associée à $(v_{1},v_{2};\emptyset)$ d'ordre $3$ (de type $1$) à gauche et d'ordre $1$ à droite.} \label{fig:partiesnontrivial}
\end{figure}

\par
D'après le théorème des résidus, l'intégration de $\omega$ le long d'un lacet fermé $\gamma$ est égale à la somme des résidus aux pôles encerclés par $\gamma$. En effet, rappelons notre convention que notre résidu est égal à $2i\pi$ fois le résidu usuel.
Cela a la conséquence suivante qui, bien qu'élémentaire, est primordiale pour notre étude.
\begin{lem}\label{lm:residu}
Soient $(v_{i};w_{j})$ des nombres complexes, le pôle associé à la partie polaire d'ordre~$b$ et de type $\tau$ associée à $(v_{i};w_{j})$ est d'ordre $-b$ et de résidu $\sum v_{i}-\sum w_{j}$. 

Soit $(v_{1},\dots,v_{l})$ avec $l\geq1$, le pôle associé au domaine polaire d'ordre~$1$ associé à $v_{i}$ est d'ordre~$-1$ et possède un résidu égal à $\sum v_{i}$.
\end{lem}

\smallskip
\par
Nous introduisons maintenant les briques analogues pour les $k$-différentielles. Les constructions sont très similaires au cas abélien et nous ne donnerons pas les détails.

Étant donnés $b:=k\ell$ avec $\ell\geq 1$ et des vecteurs $(v_{i};w_{j})$ comme ci-dessus la {\em $k$-partie polaire d'ordre $k\ell$ associée à $(v_{i};w_{j})$} coïncide avec la partie polaire d'ordre $\ell$ associée à $(v_{i};w_{j})$. Cette $k$-partie polaire permettra de construire des pôles d'ordre $-k\ell$ ayant pour $k$-résidu $\left(\sum v_{i} -\sum w_{j}\right)^{k}$. 

Nous traitons maintenant le cas des pôles d'ordre non divisible par $k$. Soient $c=\ell k+r$ avec $0<r<k$ et $\ell\geq 1$. On se donne des vecteurs $v_{i}$ de $\CC^{\ast}$ de partie réelle positive. La {\em $k$-partie polaire d'ordre $c$ associée aux $(v_{1},\dots,v_{l};\emptyset)$} est donnée par la construction suivante. Nous concaténons les $v_{i}$ dans le plan et traçons deux demi-droites $L_{1}$ et $L_{2}$ issues respectivement des points finaux et  initiaux de la concaténation telles que l'angle entre $L_{1}$ et $L_{2}$ est $-r\frac{2\pi}{k}$.
Comme dans le cas des parties polaires d'ordre $b$, la ligne brisée ainsi formée doit être sans points d'intersection. Nous donnerons une condition suffisante pour que cela soit possible dans le lemme~\ref{lem:noninter}.
Nous considérons la surface au dessus de cette courbe brisée. Ensuite nous prenons $\ell-1$ domaines basiques ouvert dans la direction de $L_{1}$ associés à la suite vide. Puis nous identifions les demi-droites cycliquement par translation, à l'exception du dernier qui est identifié par translation et rotation d'angle $-r\frac{2\pi}{k}$ à la demi droite $L_{2}$. Cette construction est illustrée à gauche de la figure~\ref{fig:partiepolairekdiff}. 
\par
Pour simplifier certaines constructions, il est utile de considérer la $k$-partie polaire d'ordre~$c$ associée à $(\emptyset;v_{1},\dots,v_{l})$. Elle est définie de manière similaire à la partie polaire précédente en considérant la surface sous la courbe brisée suivante. Nous concaténons les $v_{i}$ dans le plan et traçons deux demi-droites $L_{1}$ et $L_{2}$ issues respectivement des points finaux et  initiaux de la concaténation telles que l'angle entre $L_{1}$ et $L_{2}$ est $r\frac{2\pi}{k}$. La fin de cette construction est similaire à la précédente. 
\par
Nous donnons maintenant une condition suffisante pour que la ligne brisée décrite aux paragraphes précédents soit sans points d'intersection.
\begin{lem}\label{lem:noninter}
Dans la construction des $k$-parties polaires d'ordre $c$, si les $v_{i}$ sont soit de partie réel strictement positive, soit de partie réelle nulle et de partie imaginaire strictement positive, alors, quitte à permuter l'ordre des $v_{i}$, il existe des demi-droites $L_{1}$ et $L_{2}$ telles que la ligne brisée formée des $v_{i}$ et des $L_{j}$ soit sans point d'intersection.
\end{lem}
\begin{proof}
 Quitte à permuter les $v_{i}$, on peut supposer que les arguments des vecteurs~$v_{i}$, appartenant à l'ensemble $\left] -\tfrac{\pi}{2};\tfrac{\pi}{2}\right]$ sont décroissants. Notons que pour tout $\theta\in \left] \tfrac{\pi}{2}; \tfrac{3\pi}{2}\right]$ la demi-droite de pente $\theta$ partant du point initial de la concaténation n'a pas d'autres points d'intersection avec celle-ci. On a le même résultat pour les demi-droites d'angle $\phi$ partant du point final pour tout  $\phi\in \left] \tfrac{-\pi}{2}; \tfrac{\pi}{2}\right]$. Cela implique que l'on peut trouver des droites $L_{1}$ et~$L_{2}$ sans point d'intersection avec le reste de la construction qui forment n'importe quel angle strictement  compris entre $0$ et $2\pi$.
\end{proof}

\par
Nous aurons besoin dans un cas d'une construction un peu plus générale que dans le cas précédent. On note comme précédemment $c=k\ell+r$ avec $-k<r<0$.  \'Etant donné $(v_{1},v_{2})\in\CC^{\ast}$ la {\em partie polaire d'ordre $c$ associée à $(\emptyset;v_{1},v_{2})$ de type $t$} avec $1\leq t \leq \ell$ est donnée par la construction suivante. Nous concaténons $v_{1}$ avec $v_{2}$ dans le plan et traçons deux demi-droites $L_{1}$ et $L_{2}$ issues respectivement des points finaux et  initiaux de la concaténation telles que l'angle entre $L_{1}$ et $L_{2}$ est $r\frac{2\pi}{k}$. Nous considérons la surface au dessus de cette courbe brisée. Ensuite nous prenons $t-1$ domaines basiques ouvert à droite associés à la suite vide. Puis nous identifions les demi-droites cycliquement par translation, à l'exception du dernier qui est identifié par translation et rotation d'angle $-r\frac{2\pi}{k}$ à la demi droite $L_{2}$. Puis nous coupons la surface en partant du point d'intersection entre $v_{1}$ et $v_{2}$ le long d'une demi-droite $L_{3}$. Puis nous collons $\ell-t$ domaines basiques de manière cyclique à cette demi droite. Cette construction est illustrée à droite de la figure~\ref{fig:partiepolairekdiff}.

\begin{figure}[htb]
\center
 \begin{tikzpicture}

\begin{scope}[xshift=-2cm]
 \fill[black!10] (-1,0)coordinate (a) -- (1.5,0)-- (a)+(1.5,0) arc (0:120:1.5)--(a)+(120:1.5) -- cycle;

   \draw (a)  -- node [below] {$v_{1}$} (0,0) coordinate (b);
 \draw (0,0) -- (1,0) coordinate[pos=.5] (c);
 \draw[dotted] (1,0) --coordinate (p1) (1.5,0);
 \fill (a)  circle (2pt);
\fill[] (b) circle (2pt);
\node[below] at (c) {$1$};

 \draw (a) -- node [above,rotate=120] {$2$} +(120:1) coordinate (d);
 \draw[dotted] (d) -- coordinate (p2) +(120:.5);
 
    \end{scope}

\begin{scope}[xshift=1cm,yshift=.5cm]
\fill[fill=black!10] (0,0) coordinate (Q) circle (1.2cm);

\draw[] (0,00) -- (1.2,0) coordinate[pos=.5](a);

\node[below] at (a) {$2$};
\node[above] at (a) {$1$};
\fill[] (Q) circle (2pt);
\end{scope}


\begin{scope}[xshift=6.5cm,yshift=1cm]
 \fill[black!10] (-1,0)coordinate (a) -- (-.7,.4)coordinate (b) -- (0,0) coordinate (c)-- (a)+(1.8,0) arc (0:-60:1.8)--(a)+(-60:1.8) -- cycle;

   \draw (a)  -- node [above left] {$v_{2}$} (b);
      \draw (b)  -- node [above] {$v_{3}$} (c);
 \draw (0,0) -- (1,0) coordinate[pos=.5] (d);
 \draw[dotted] (1,0) --coordinate (p1) (1.5,0);
 \fill (a)  circle (2pt);

\fill[] (c) circle (2pt);
\node[below] at (d) {$1$};
 \draw (a) -- node [below,rotate=-60] {$1$} +(-60:1.3) coordinate (d);
 \draw[dotted] (d) -- coordinate (p2) +(-60:.5);
 
 \draw (b)   -- ++(-60:1.9) coordinate[pos=.5](f);
\node[left] at (f) {$4$};
\node[right] at (f) {$3$};
 
 \fill[white] (b) circle (2pt);
  \draw[] (b) circle (2pt);

    \end{scope}
\begin{scope}[xshift=9cm,yshift=.5cm]
\fill[fill=black!10] (0,0) coordinate (Q) circle (1.2cm);
\draw[] (0,00) --++ (-60:1.2) coordinate[pos=.5](a);

\node[left] at (a) {$3$};
\node[right] at (a) {$4$};
 \fill[white] (0,0) circle (2pt);
  \draw[] (0,0) circle (2pt);
\end{scope}

\end{tikzpicture}
\caption{La $3$-partie polaire associée à $(v_{1};\emptyset)$ d'ordre~$7$ à gauche et la $6$-partie polaire associée à $(\emptyset;v_{2},v_{3})$ d'ordre~$13$ et de type~$1$ à droite.} \label{fig:partiepolairekdiff}
\end{figure}

Nous résumons maintenant les propriétés des constructions du paragraphe précédent.
\begin{lem}\label{lm:kresidu}
Soient $(v_{i};w_{j})$ des nombres complexes, le pôle obtenu à partir de la $k$-partie polaire d'ordre $b=k\ell$ et de type $\tau$ associée à $(v_{i};w_{j})$ est d'ordre $-b$ et possède un $k$-résidu égal à $\left(\sum v_{i}-\sum w_{j}\right)^{k}$. 

Soit $(v_{1},\dots,v_{l})$ avec $l\geq1$, le pôle associé au domaine basique d'ordre $k$ associé à $v_{i}$ est d'ordre $-k$ et possède un $k$-résidu égal à $\left(\sum v_{i}\right)^{k}$.
\end{lem}

\subsection{Pluridifférentielles entrelacées, éclatement de zéros et couture d'anses.}
\label{sec:pluridiffentre}

Dans ce paragraphe, nous rappelons certains cas particuliers des résultats obtenus dans \cite{BCGGM} et \cite{BCGGM3} au sujet des différentielles entrelacées. Cela nous permet de rappeler les constructions de l'{\em éclatement des zéros} et de la {\em couture d'anse}. 

Tout d'abord, nous rappelons la définition d'une différentielle entrelacée.
\'Etant donnée une partition~$\mu:=(m_{1},\dots,m_{t})$ telle que $\sum_{i=1}^t m_i = k(2g-2)$, une {\em $k$-différentielle entrelacée $\eta$ de type~$\mu$}
sur une courbe stable $n$-marquée $(X,z_1,\ldots,z_t)$
est une collection de
$k$-différentielles non nulles~$\eta_v$ sur les composantes irréductibles~$X_v$ de~$X$ satisfaisant aux conditions suivantes.
\begin{itemize}
\item[(0)] {\bf (Annulation comme prescrit)} Chaque $k$-différentielle $\eta_v$ est méromorphe et le support de son diviseur est inclus dans l'ensemble des nœuds et des points marqués de $X_v$. De plus, si un point marqué $z_i$ se trouve sur~$X_v$, alors $\ord_{z_i} \eta_v=m_i$.
\item[(1)] {\bf (Ordres assortis)} Pour chaque nœud de $X$ qui identifie $q_1 \in X_{v_1}$ à $q_2 \in X_{v_2}$,
$$\ord_{q_1} \eta_{v_1}+\ord_{q_2} \eta_{v_2} = -2k . $$
\item[(2)] {\bf (Résidus assortis aux pôles d'ordre $-k$)} Si a un nœud de $X$
qui identifie $q_1 \in X_{v_1}$ avec $q_2 \in X_{v_2}$ on a $\ord_{q_1}\eta_{v_1}=
\ord_{q_2} \eta_{v_2}=-k$, alors
$$\Resk_{q_1}\eta_{v_1} = (-1)^k\Resk_{q_2}\eta_{v_2}.$$
\end{itemize}

Ce n'est que pour des cas très particuliers que nous aurons besoin de savoir quand une $k$-différentielle entrelacée est lissable. Nous rappelons ici uniquement les cas qui nous intéressent.  Le premier cas est celui où l'ordre des $k$-différentielles à tous les nœuds est égal à $-k$.
\begin{lem}\label{lem:lisspolessimples}
Soit $\eta=\left\{\eta_{v}\right\}$ une $k$-différentielle entrelacée. Si l'ordre des $k$-différentielles~$\eta_v$ aux nœuds est $-k$, alors $\eta$ est lissable localement.
\end{lem}
 Notons que dans ce cas, la notion de différentielle entrelacée correspond à la notion classique de différentielle stable. Ainsi nous pourrons nous ramener à ce résultat pour prouver les propositions~\ref{prop:g0p-1plusieurszero} et \ref{prop:cylindres}.

Maintenant nous regardons le cas des pluridifférentielles entrelacées à deux composantes.
\begin{lem}\label{lem:lissdeuxcomp}
 Supposons que $X$ possède exactement deux composantes $X_{1}$ et $X_{2}$ reliées par un unique nœud  qui identifie $q_1 \in X_{1}$ à $q_2 \in X_{2}$. Si $\ord_{q_1} \eta_{1}>-k>\ord_{q_2} \eta_{2}$, alors la $k$-différentielle entrelacée est lissable si et seulement si l'une des deux conditions suivantes est vérifiée.
 \begin{enumerate}[i)]
  \item $\Resk_{q_2}\eta_{2}=0$
  \item $\eta_{1}$ n'est pas la puissance $k$-ième d'une différentielle abélienne holomorphe.
 \end{enumerate}
 De plus, le lissage peut se faire sans modifier les $k$-résidus de $\eta_{1}$ si et seulement si l'une des  deux conditions suivantes est vérifiée.
  \begin{enumerate}[i)]
  \item $\Resk_{q_2}\eta_{2}=0$
  \item $\eta_{1}$ n'est pas la puissance $k$-ième d'une différentielle abélienne méromorphe.
 \end{enumerate}
\end{lem}
Remarquons que la deuxième partie du lemme n'est pas explicitement prouvée dans \cite{BCGGM3}. Toutefois, cela peut se montrer sans problèmes en combinant  la preuve du théorème~1.5 et le  lemme~4.4 de \cite{BCGGM3}.

Maintenant, nous donnons deux applications cruciales du lemme~\ref{lem:lissdeuxcomp}.

\begin{prop}[éclatement d'un zéro]\label{prop:eclatZero}
Soient $(X,\xi)$ une $k$-différentielle de type $\mu$ et $z_{0}\in X$ un zéro d'ordre $a_{0}>-k$ de $\xi$. Soit $(\alpha_{1},\dots,\alpha_{t})$ un $t$-uplet d'entiers strictement supérieurs à $-k$ tel que $\sum_{i}\alpha_{i}=a_{0}$. 

Il existe une opération sur $(X,\xi)$ en $z_{0}$ qui fournit une  $k$-différentielle $(X',\xi')$ de type $(\alpha_{0},\dots,\alpha_{t},\mu\setminus\lbrace a_{0}\rbrace)$ qui ne modifie pas les $k$-résidus de $\xi$ si et seulement si  l'une des  deux conditions suivantes est vérifiée.
  \begin{enumerate}[i)]
  \item $\xi$ n'est pas la puissance $k$-ième d'une différentielle abélienne méromorphe.
  \item Il existe une $k$-différentielle de genre zéro et de type $(\alpha_{1},\dots,\alpha_{t};-a_{0}-2k)$ dont le $k$-résidu au pôle d'ordre $-a_{0}-2k$ est nul. 
 \end{enumerate} 
De plus, si $\xi=\omega^{k}$ avec $\omega$  une différentielle abélienne méromorphe, alors la $k$-différentielle~$\xi'$ est primitive si et seulement si  $\pgcd(\alpha_{i},d)=1$. 
\end{prop}
\begin{proof}
 Partons de $(X,\xi)$. On forme une différentielle entrelacée en attachant au point~$z_{0}$ une droite projective avec une différentielle ayant les ordres souhaités. Le lemme~\ref{lem:lissdeuxcomp} implique facilement la proposition~\ref{prop:eclatZero}.
\end{proof}

La seconde construction nous permettra en particulier de faire une récurrence sur le genre des surfaces de Riemann.

\begin{prop}[Couture d'anse]\label{prop:attachanse}
 Soient $(X,\xi)$ une $k$-différentielle (primitive) dans la strate $\komoduli(\mu)$ et $z_{0}\in X$ un zéro d'ordre $a_{0}$ de $\xi$. 
 Il existe une opération locale à $z_{0}$ qui produit une $k$-différentielle $(X',\xi')$ dans la strate $\komoduli[g+1](a_{0}+2k,\mu\setminus\left\{a_{0}\right\})$. 
\end{prop}
\begin{proof}
 Partons de $(X,\xi)$. On forme une différentielle entrelacée en attachant au point $z_{0}$ une courbe elliptique avec une différentielle de type $(a_{0}+2k;-a_{0}-2k)$. Le lemme~\ref{lem:lissdeuxcomp} permet de conclure.
\end{proof}

\subsection{Différentielle à cœur dégénéré.}
\label{sec:coeur}

Les différentielles à cœur dégénéré constitue une famille particulièrement simple d'exemples. En particulier, beaucoup de problèmes géométriques se simplifient en des problèmes combinatoires. Cela nous permettra de montrer  des résultats de non existence, en particulier dans la section~\ref{sec:ggeq1}.

Rappelons que le {\em cœur} d'une $k$-différentielle est l'enveloppe convexe des singularités coniques pour la métrique définie par la différentielle. On dit que le cœur est \textit{dégénéré} s'il est d'intérieur vide, c'est-à-dire s'il est réduit à l'union d'un nombre fini de liens selles.

Le complémentaire du cœur d'une surface plate admet autant de composantes connexes que de pôles. On appelle \textit{domaine polaire} la composante à laquelle un pôle appartient. Le bord d'un domaine polaire est toujours formé par un nombre fini de liens selles (voir le lemme~2.1 de \cite{tahar}).

L'intérêt de cette notion est donné par la proposition suivante que l'on montre aisément en utilisant le {\em flot contractant}.
\begin{prop}\label{prop:coeurdege}
Dans une strate donnée, le lieu des différentielles abéliennes ou quadratique  dont tous les résidus sont nuls est soit vide soit contient une différentielle à cœur dégénéré.
\end{prop}

\begin{proof}
Le lieu d'une strate où les résidus sont tous nuls est ${\rm GL}^{+}(2,\mathbb{R})$-invariant. Or, d'après le lemme 2.2 de \cite{tahar}, chaque ${\rm GL}^{+}(2,\mathbb{R})$-orbite contient une surface plate dont le cœur est dégénéré.
\end{proof}

On peut dans ce cas supposer que toutes les connexions de selles sont horizontales. De plus, il y a exactement $2g+n+\tilde{p}-2$ connexions de selles dans la surface $S$, où $n$ est le nombre de zéros et~$\tilde{p}$ le nombre de pôles dans la strate. Coupant le long de ces liens selles, nous obtenons~$\tilde{p}$ ($2$-)parties polaires. Le {\em graphe d'incidence d'une surface à cœur dégénéré} est le graphe dont les sommets sont les domaines polaires et deux sommets sont reliés par autant d'arêtes  qu'il y a de liens selles entre les deux domaines polaires.
De plus, le {\em graphe d'incidence simplifié} est  obtenu en enlevant tous les sommets de valence $2$ au graphe d'incidence. Les sommets du graphe d'incidence qui sont de valence supérieure ou égale à trois sont dits {\em spéciaux}.

\section{Différentielles abéliennes méromorphes}
\label{sec:MDPR}

Dans cette section, nous étudions le cas des différentielles abéliennes.  Dans la section~\ref{sec:casgen}, nous traitons le cas général. Puis nous considérons dans la section~\ref{sec:casreszero} le cas où tous les résidus sont nuls dans les strates de genre zéro. Ensuite, nous traitons le cas des résidus colinéaires dans les strates de genre zéro n'ayant que des pôles simples dans la section~\ref{sec:caslier}. Enfin, nous étudions l'application résiduelle sur chaque composante connexe des strates dans la section~\ref{sec:casCC}.

\subsection{Cas général}
\label{sec:casgen}

Soit $\mu:=(a_{1},\dots,a_{n};-b_{1},\dots,-b_{p};(-1^{s}))$ une partition  de $2g-2$. Dans ce paragraphe, nous prouvons la surjectivité de l'application résiduelle des strates $\omoduli(\mu)$ dans le cas $g\geq1$ et dans le cas où $g=0$ et $p,s\neq 0$. Le premier cas correspond aux théorèmes \ref{thm:ggeq2} et \ref{thm:geq1} dans le cas abélien. Le second cas est couvert par le début de la démonstration du théorème~\ref{thm:geq0keq1}. De plus, nous montrons que si $g=0$ et $s=0$, alors l'image de l'application résiduelle contient $\espres[0](\mu)\setminus\left\{(0,\dots,0)\right\}$. Nous montrons aussi que l'image de l'application résiduelle contient les résidus non colinéaires dans le cas $g=0$ et $p=0$.

La preuve suit le schéma suivant. Nous traitons tout d'abord dans le lemme~\ref{lem:gzerogen1} le cas des strates de genre zéro avec un unique zéro. Puis nous considérons le cas des strates de genre~$1$ avec un unique zéro et $p=0$ dans le lemme~\ref{lem:gunspe1} ou $s=0$  dans le lemme~\ref{lem:gunspe2}. Enfin nous concluons dans presque tout les cas grâce aux techniques d'éclatement du zéro et de couture d'anse dans le lemme~\ref{lem:abelgen}.

\begin{lem}\label{lem:gzerogen1}
 Soit $\mu:=(a;-b_{1},\dots,-b_{p};(-1^{s}))$ une partition  de $-2$ et $r=(r_{1} , . . . , r_{p+s})\in\espres[0](\mu)$
un élément non nul de l’espace résiduel. Si les $r_{ i }$ ne sont pas colinéaires
ou si $p\neq0$, alors $r$ est dans l’image de $\appres[0](\mu)$
\end{lem}

Avant de procéder à la preuve du lemme~\ref{lem:gzerogen1} nous introduisons un objet géométrique associé à un uplet $v:=(v_1,\ldots, v_{t})$ avec $v_{i}\neq 0$. Nous supposerons, quitte à permuter les indices, que l'argument des vecteurs $v_{1},\ldots,v_{t}$ est décroissant dans $\left]-\pi,\pi\right]$.
Le {\em polygone résiduel } $\mathfrak{P}(v)$ est le polygone obtenu en concaténant les vecteurs $v_{1}, \ldots,v_{t}$ dans cet ordre. Remarquons que~$\mathfrak{P}(v)$ est un polygone convexe, éventuellement dégénéré (i.e. d'intérieur vide).

\begin{proof}[Preuve du lemme~\ref{lem:gzerogen1}]
 Soient  $\omoduli[0](a;-b_{1},\dots,-b_{p};(-1^{s}))$ une strate de genre zéro et $r:=(r_1,\ldots, r_{p+s})$ dans $\espres[0](\mu)$. Nous notons par $r^{\ast}$ l'ensemble des résidus non nuls.  
 \smallskip
 \par
 Supposons que le polygone résiduel  $\mathfrak{P}(r^{\ast})$ est non dégénéré (i.e. que les $r_{i}$ ne sont pas colinéaires). Pour tous les pôles~$P_{i}$ d'ordre~$-b_{i}$ ayant un résidu non nul~$r_{i}$, on prend une partie polaire d'ordre~$b_{i}$ associée à $(r_{i};\emptyset)$ (et de type arbitraire). Pour tous les pôles $P_{j}$ d'ordre $-b_{j}$ ayant un résidu nul, on prend une partie triviale d'ordre $b_{j}$ associée à $(r_{i_{j}};r_{i_{j}})$, où $r_{i_{j}}\neq0$ est le résidu (non nul) au pôle $P_{i_{j}}$. Les collages de ces parties polaires avec le polygone résiduel se font de la façon suivante.

\'Etant donnée une partie polaire non triviale associée à un pôle $P_{i}$. S'il existe un pôle~$P_{j}$ qui a une partie polaire associée à $(r_{i};r_{i})$, alors nous collons le segment~$r_{i}$ de $P_{i}$ au segment~$r_{i}$ du domaine basique négatif de $P_{j}$. Nous continuons ces collages jusqu'à ce qu'il n'y ait plus de pôles sans résidu avec une partie polaire associée à $(r_{i};r_{i})$. Puis nous collons le dernier segment $r_{i}$ au segment $r_{i}$ du polygone résiduel. Nous faisons de même pour tous les pôles de résidus non nuls. Cette construction est illustrée par la figure~\ref{fig:casgeneralgenrezero}.

 \begin{figure}[htb]
 \center
\begin{tikzpicture}

\begin{scope}[yshift=-.2cm]
\coordinate (a) at (0,0);
\coordinate (b) at (-1,0);
\coordinate (c) at (-1,-1);

\filldraw[fill=black!10] (a) -- (b)coordinate[pos=.5](d) -- (c)coordinate[pos=.5](e) -- (a) coordinate[pos=.5](f);

\fill (a)  circle (2pt);
\fill (b) circle (2pt);
\fill (c) circle (2pt);

\node[above] at (d) {$1$};
\node[left] at (e) {$2$};
\node[below right] at (f) {$3$};
\end{scope}

\begin{scope}[yshift=-.5cm]
\coordinate (a) at (-1,1);
\coordinate (b) at (0,1);

    \fill[fill=black!10] (a)  -- (b)coordinate[pos=.5](f) -- ++(0,1) --++(-1,0) -- cycle;
    \fill (a)  circle (2pt);
\fill[] (b) circle (2pt);
 \draw  (a) -- (b);
 \draw (a) -- ++(0,.9) coordinate (d)coordinate[pos=.5](h);
 \draw (b) -- ++(0,.9) coordinate (e)coordinate[pos=.5](i);
 \draw[dotted] (d) -- ++(0,.2);
 \draw[dotted] (e) -- ++(0,.2);
\node[above] at (f) {$1$};
\node[left] at (h) {$4$};
\node[right] at (i) {$4$};
\end{scope}

\fill[fill=black!10] (-2.8,-.5) coordinate (Q) circle (1.1cm);
\coordinate (a) at (-2.8,0);
\coordinate (b) at (-2.8,-1);
\fill (a)  circle (2pt);
\fill (b) circle (2pt);
\draw (a) -- (b)coordinate[pos=.5](d);

\node[left] at (d) {$2$};
\node[right] at (d) {$5$};

\fill[fill=black!10] (-5.3,-.5) coordinate (Q) circle (1.1cm);
\coordinate (a) at (-5.4,0);
\coordinate (b) at (-5.4,-1);
\fill (a)  circle (2pt);
\fill (b) circle (2pt);
    \fill[white] (a) --  (b)coordinate[pos=.5](d) -- ++(1.3,0) --++(0,1) -- cycle;
 \draw  (a) -- (b)coordinate[pos=.5](g);
 \draw (a) -- ++(1,0) coordinate (e)coordinate[pos=.5] (h);
 \draw (b) -- ++(1,0) coordinate (f)coordinate[pos=.5] (i);
 \draw[dotted] (f) -- ++(.2,0);
 \draw[dotted] (e) -- ++(.2,0);
 
 \node[right] at (g) {$5$};
\node[above] at (h) {$6$};
\node[below] at (i) {$6$};

 \begin{scope}[yshift=.1cm]
\fill[fill=black!10] (2,-.5) coordinate (Q) circle (1.3cm);
\coordinate (a) at (2.4,.1);
\coordinate (b) at (1.4,-.9);
\fill (a)  circle (2pt);
\fill (b) circle (2pt);
    \fill[white] (a) --  (b)coordinate[pos=.5](d) -- ++(0,1.8) --++(1,0) -- cycle;
 \draw  (a) -- (b)coordinate[pos=.5](g);
 \draw (a) -- ++(0,.5) coordinate (e)coordinate[pos=.6] (h);
 \draw (b) -- ++(0,1.4) coordinate (f)coordinate[pos=.6] (i);
 \draw[dotted] (f) -- ++(0,.2);
 \draw[dotted] (e) -- ++(0,.1);
 
 \node[below right] at (g) {$3$};
\node[left] at (h) {$9$};
\node[right] at (i) {$9$};
 
 \draw (a) -- ++(.7,0) coordinate[pos=.5] (j);
\node[above] at (j) {$8$};
\node[below] at (j) {$7$};
\end{scope}

\fill[fill=black!10] (4.6,-.5) coordinate (Q) circle (1.1cm);

\draw[] (Q) -- ++(1.1,0) coordinate[pos=.5](a);

\node[above] at (a) {$7$};
\node[below] at (a) {$8$};
\fill[] (Q) circle (2pt);

\end{tikzpicture}
\caption{Une différentielle dans $\omoduli[0](4;-2,-2,-3;-1)$ avec résidus $(0,i,-1-i,1)$.} \label{fig:casgeneralgenrezero}
\end{figure}

La différentielle associée à cette surface plate possède clairement les invariants désirés aux pôles. Vérifions maintenant qu'elle est de genre zéro et possède un unique zéro. Remarquons que si l'on coupe cette surface le long d'un lien selle, on obtient deux surfaces connexes disjointes. En effet, les liens selles correspondent soit au bord des parties polaires, soit aux diagonales du polygone résiduel. Une telle propriété implique qu'il existe un unique zéro, car s'il y en avait deux, un lien selle entre les deux ne déconnecterait pas la surface. De manière similaire, on en déduit que le genre de la surface est nul. 

\smallskip
\par
Nous traitons maintenant le cas où le polygone résiduel est dégénéré mais pas réduit à un point dans les strates $\omoduli[0](a;-b_{1},\dots,-b_{p};(-1^{s}))$, avec $p\neq0$. Sans perte de généralité, nous supposerons que les résidus sont réels. De plus, nous ordonnons les indices de telle sorte que l'ensemble des résidus non nuls $r^{\ast}=\left\{ r_{j_{1}},\dots,r_{j_{t}}\right\}$ satisfait $r_{j_{i}}<0$ pour $i\leq u$ et $r_{j_{i}}$ pour $u<i\leq t$. Nous notons $J=\left\{ j_{1},\dots,j_{t}\right\}$.
Prenons le pôle $P_{1}$ d'ordre $-b_{1}$. Si $r_{1}=0$, nous associons à ce pôle une partie polaire d'ordre $b_{i}$ associée aux vecteurs $(-r_{j_{1}},\dots,-r_{j_{u}};r_{j_{u+1}},\dots,r_{j_{t}})$. Si $r_{1}\neq 0$, on suppose que $r_{1}=r_{j_{1}}$ et on associe à ce pôle une partie polaire d'ordre $b_{i}$ associée aux vecteurs $(-r_{j_{2}},\dots,-r_{j_{u}};r_{j_{u+1}},\dots,r_{j_{t}})$.
Nous associons aux pôles $P_{j_{i}}$, pour $1\leq i\leq u$ si $1\notin J$ et $2\leq i\leq u$ si $1\in J$, une partie polaire d'ordre $b_{j_{i}}$ associée à $(\emptyset;-r_{j_{i}})$. Pour les $u< i\leq t$ nous prenons une partie polaire d'ordre $b_{i}$ associée à $(r_{i};\emptyset)$. Pour tous les pôles $P_{k}$ d'ordre $-b_{k}$ ayant un résidu nul, on prend une partie triviale d'ordre $b_{k}$ associée à $(\pm r_{i_{k}};\pm r_{i_{k}})$ avec $i_{k}\in J \setminus 1$. On colle les parties polaires triviales comme expliqué au paragraphe précédent. On obtient une union de surfaces plates à bord.  Il reste à coller les bords restants aux segments de la partie polaire de $P_{1}$.
On vérifie comme au paragraphe précédent que la différentielle ainsi construite vérifie les propriétés désirées.
\end{proof}
\par
Nous traitons maintenant le cas des strates $\omoduli[1](a;-b_{1},\dots,-b_{p};(-1^{s}))$ de genre $1$ avec $p=0$ ou $s=0$.
\begin{lem}\label{lem:gunspe1}
 L'application résiduelle de la strate $\omoduli[1](s;(-1^{s}))$ est surjective pour $s>1$.
\end{lem}
\begin{proof}
Considérons la strate $\omoduli[1](s;(-1^{s}))$, avec $s>1$  et $r:=(r_{1},\dots,r_{s})$ dans $\espres[1](s;(-1^{s}))$. Prenons un tore plat $S_{1}$ tel que le lien selle le plus petit est strictement supérieur à $\sum|r_{i}|$. Nous enlevons de $S_{1}$ le polygone résiduel $\mathfrak{P}(-r)$. Cette opération est réalisable par notre hypothèse sur la longueur des liens selles de $S_{1}$.
Pour chacun des pôles~$P_{i}$, nous prenons une partie polaire d'ordre~$1$ associée à~$r_{i}$. Nous collons le bord de ces parties polaires au bord de $S_{1}\setminus \mathfrak{P}(r)$ par translation. La construction est représentée par le dessin de gauche de la figure~\ref{fig:casgeneralgenreun}. On vérifie sans problème que la surface ainsi obtenue est de genre $1$ et possède une unique singularité conique.

 \begin{figure}[htb]
 \center
\begin{tikzpicture}

\begin{scope}[xshift=-7.5cm,yshift=-.8cm]
\coordinate (a) at (-1,1);
\coordinate (b) at (0,1);

    \fill[fill=black!10] (a)  -- (b)coordinate[pos=.5](f) -- ++(0,1) --++(-1,0) -- cycle;
    \fill (a)  circle (2pt);
\fill[] (b) circle (2pt);
 \draw  (a) -- (b);
 \draw (a) -- ++(0,.9) coordinate (d)coordinate[pos=.5](h);
 \draw (b) -- ++(0,.9) coordinate (e)coordinate[pos=.5](i);
 \draw[dotted] (d) -- ++(0,.2);
 \draw[dotted] (e) -- ++(0,.2);
\node[above] at (f) {$1$};
\end{scope}

\begin{scope}[xshift=-2.5cm,yshift=-.8cm]
\coordinate (a) at (-1,1);
\coordinate (b) at (0,1);

    \fill[fill=black!10] (a)  -- (b)coordinate[pos=.5](f) -- ++(0,1) --++(-1,0) -- cycle;
    \fill (a)  circle (2pt);
\fill[] (b) circle (2pt);
 \draw  (a) -- (b);
 \draw (a) -- ++(0,.9) coordinate (d)coordinate[pos=.5](h);
 \draw (b) -- ++(0,.9) coordinate (e)coordinate[pos=.5](i);
 \draw[dotted] (d) -- ++(0,.2);
 \draw[dotted] (e) -- ++(0,.2);
\node[above] at (f) {$2$};
\end{scope}

\begin{scope}[xshift=-7.5cm,yshift=-1.2cm]
\coordinate (a) at (-1,1);
\coordinate (b) at (0,1);

    \fill[fill=black!10] (a)  -- (b)coordinate[pos=.5](f) -- ++(0,-1) --++(-1,0) -- cycle;
    \fill (a)  circle (2pt);
\fill[] (b) circle (2pt);
 \draw  (a) -- (b);
 \draw (a) -- ++(0,-.9) coordinate (d)coordinate[pos=.5](h);
 \draw (b) -- ++(0,-.9) coordinate (e)coordinate[pos=.5](i);
 \draw[dotted] (d) -- ++(0,-.2);
 \draw[dotted] (e) -- ++(0,-.2);
\node[below] at (f) {$3$};
\end{scope}

\begin{scope}[xshift=-2.5cm,yshift=-1.2cm]
\coordinate (a) at (-1,1);
\coordinate (b) at (0,1);

    \fill[fill=black!10] (a)  -- (b)coordinate[pos=.5](f) -- ++(0,-1) --++(-1,0) -- cycle;
    \fill (a)  circle (2pt);
\fill[] (b) circle (2pt);
 \draw  (a) -- (b);
 \draw (a) -- ++(0,-.9) coordinate (d)coordinate[pos=.5](h);
 \draw (b) -- ++(0,-.9) coordinate (e)coordinate[pos=.5](i);
 \draw[dotted] (d) -- ++(0,-.2);
 \draw[dotted] (e) -- ++(0,-.2);
\node[below] at (f) {$4$};
\end{scope}

\begin{scope}[xshift=-7cm,yshift=-.5cm]
\coordinate (a) at (0,0);
\coordinate (b) at (3,0);
\coordinate (c) at (3,1);
\coordinate (d) at (0,1);

    \fill[fill=black!10] (a)  -- (b) -- (c) -- (d) -- cycle;
 \draw  (a) -- (b) --(c)--(d)--(a);

\draw (.5,.5) coordinate (e) --(1.5,.5) coordinate (f)coordinate[pos=.5](h)  --(2.5,.5) coordinate (g)coordinate[pos=.5](i);
\fill (e)  circle (2pt);
\fill (f) circle (2pt);
\fill (g) circle (2pt);
\node[below] at (h) {$1$};
\node[below] at (i) {$2$};
\node[above] at (h) {$3$};
\node[above] at (i) {$4$};
\end{scope}


\begin{scope}[xshift=1cm,yshift=0cm]
\fill[fill=black!10] (0,0) coordinate (Q) circle (1.3cm);
\coordinate (a) at (-.5,-.5);
\coordinate (b) at (.5,-.5);
\coordinate (c) at (.5,.5);
\coordinate (d) at (-.5,.5);
\fill (a)  circle (2pt);
\fill (b) circle (2pt);
\fill (c)  circle (2pt);
\fill (d) circle (2pt);
\fill[fill=white] (a)  -- (b) -- (c) -- (d) -- cycle;
\draw (a)  -- (b)coordinate[pos=.5](e) -- (c)coordinate[pos=.5](f) -- (d)coordinate[pos=.5](g) -- (a)coordinate[pos=.5](h);

\node[below] at (e) {$2$};
\node[above] at (g) {$3$};
\node[right] at (f) {$1$};
\node[left] at (h) {$1$};
\end{scope}

\begin{scope}[xshift=3.5cm,yshift=0cm]
\fill[fill=black!10] (0,0) coordinate (Q) circle (1.1cm);
\coordinate (a) at (-.5,0);
\coordinate (b) at (.5,0);
\fill (a)  circle (2pt);
\fill (b) circle (2pt);
\draw (a) -- (b)coordinate[pos=.5](d);

\node[above] at (d) {$2$};
\node[below] at (d) {$3$};
\end{scope}

\end{tikzpicture}
\caption{Une différentielle de $\omoduli[1](4;(-1^{4}))$ avec résidus $(1,1,-1,-1)$ à gauche et de $\omoduli[1](4;(-2^{2}))$ avec résidus $(0,0)$ à droite.} \label{fig:casgeneralgenreun}
\end{figure}
\end{proof}

\begin{lem}\label{lem:gunspe2}
L'application résiduelle de $\omoduli[1](a;-b_{1},\dots,-b_{p})$ est surjective.
\end{lem}

\begin{proof}
Supposons que $r=(0,\dots,0)$.  Pour tout $P_{i}$ avec $i\geq2$, nous prenons une partie polaire triviale d'ordre $b_{i}$ associée à $(1;1)$.  Pour le pôle~$P_{1}$, nous prenons une partie polaire d'ordre~$b_{1}$ associée à $(i,1;1,i)$. Nous collons les deux segments de période $i$ par translation. Les autres segments sont collés cycliquement les uns aux autres. Plus précisément, nous collons le segment supérieur de $P_{i}$  au segment inférieur de $P_{i+1}$ (où ces indices sont pris modulo $p$). La construction est représentée par le dessin de droite de la figure~\ref{fig:casgeneralgenreun}. La surface plate obtenue a clairement une unique singularité conique. Le genre est~$1$ car nous pouvons couper les liens selles correspondant aux segments de période $1$ et $i$ de $P_{1}$ sans déconnecter la surface. En revanche couper un autre lien selle déconnecte la surface. Donc nous avons construit une différentielle dans $\omoduli[1](a;-b_{1},\dots,-b_{p})$ dont tous les résidus aux pôles sont nuls.

Supposons maintenant que $r\neq(0,\dots,0)$. D'après le lemme~\ref{lem:gzerogen1}, il existe une différentielle~$\omega$ dans $\omoduli[0](a-2;-b_{1},\dots,-b_{p})$ dont les résidus sont~$r$. La différentielle souhaitée est alors obtenue en cousant une anse au zéro de la différentielle~$\omega$. 
\end{proof}

Nous concluons maintenant à la surjectivité de l'application résiduelle pour la majorité des strates. Plus précisément, nous montrons le résultat suivant.
\begin{lem}\label{lem:abelgen}
 Soit $\omoduli(a_{1},\ldots,a_{n};-b_{1},\dots,-b_{p};(-1^{s}))$ une strate telle que, soit le genre~$g$ est supérieur ou égale à~$1$, soit $p\neq 0$ et $s\neq 0$. L'application résiduelle de cette strate est surjective.

 De plus, si $g=0$ et $s=0$, alors l'image de l'application résiduelle contient $\espres[0](\mu)\setminus \left\{0\right\}$.
\end{lem}

\begin{proof}
Commençons par les cas où $g=0$. On se donne une partition $\mu:=(a_{1},\ldots,a_{n};-b_{1},\dots,-b_{p};(-1^{s}))$ de $-2$ telle que,  soit $p\neq0$ et $s\neq0$, soit $s=0$. D'après le lemme~\ref{lem:gzerogen1}, il existe une différentielle $(X,\omega)$ dans  $\omoduli[0](\sum a_{i};-b_{1},\dots,-b_{p};(-1^{s})))$ dont les résidus aux pôles sont~$r$, avec $r\in\espres[0](\mu)$ dans le premier cas et $r\in\espres[0](\mu)\setminus\left\{ 0 \right\}$ dans le second. L'éclatement du zéro d'ordre $\sum a_{i}$ (voir la proposition~\ref{prop:eclatZero}) donne une différentielle dans $\omoduli[0](\mu)$ dont les résidus sont $r$. 

Prenons maintenant une strate $\omoduli[1](\mu)$ avec $\mu:=(a;-b_{1},\dots,-b_{p};(-1^{s}))$ et $r\in\espres[1](\mu)$. Si~$p$ et~$s$ sont non nuls, alors il existe une différentielle dans $\omoduli[0](a-2;-b_{1},\dots,-b_{p};(-1^{s}))$ dont les résidus sont $r$. Avec la couture d'anse (proposition~\ref{prop:attachanse}), on obtient une différentielle dans $\omoduli[1](\mu)$ ayant pour résidus $r$. La surjectivité de l'application résiduelle dans les cas où $p=0$ ou $s=0$ a été démontrée dans les lemmes~\ref{lem:gunspe1} et~\ref{lem:gunspe2}.

Considérons maintenant les partitions $\mu:=(a;-b_{1},\dots,-b_{p};(-1^{s}))$ de $2g-2$ avec $g\geq2$. On se donne un uplet $r$ dans $\espres(\mu)$. La strate $\omoduli[1](a-2g;-b_{1},\dots,-b_{p};(-1^{s}))$ contient une différentielle dont les résidus sont $r$. Par coutures d'anse successives, on obtient une différentielle dans $\omoduli[g](\mu)$ dont les résidus sont $r$.

Enfin, considérons une partition quelconque $\mu:=(a_{1},\ldots,a_{n};-b_{1},\dots,-b_{p};(-1^{s}))$ de $2g-2$ avec $g\geq1$. L'application résiduelle de la strate $\omoduli[g](\sum a_{i};-b_{1},\dots,-b_{p};(-1^{s}))$ est surjective. La surjectivité de l'application résiduelle $\appres(\mu)$ est obtenue par éclatement du zéro  des différentielles de cette strate.
\end{proof}

\subsection{Genre zéro et les résidus sont nuls}
\label{sec:casreszero}
Nous complétons la preuve des cas
i) et iii) du théorème~\ref{thm:geq0keq1}.  Au vu du lemme~\ref{lem:abelgen}, il reste à montrer le résultat suivant.
\begin{lem}\label{lem:condgzero}
L'application résiduelle des strates $\omoduli[0](a_{1},\dots,a_{n};-b_{1},\dots,-b_{p})$ ne contient pas l'origine $(0,\dots, 0)$ si et seulement si la condition~\eqref{eq:genrezeroresiduzerofr} est satisfaite.
\end{lem}

 Rappelons que cette condition dit que si l'ordre d'un zéro d'une différentielle de genre zéro  est strictement supérieur  à $\sum_{j=1}^p b_{j}-(p+1)$, alors l'origine n'est pas contenue dans l'image de l'application résiduelle.

\begin{proof}
Nous commençons par montrer que si l'origine est dans l'image de l'application résiduelle alors tous les zéros sont d'ordres inférieurs ou égaux à $\sum_{j=1}^p b_{j}-(p+1)$.
Soit $(\PP^{1},\omega)$ une différentielle de $\omoduli[0](a_{1},\dots,a_{n};-b_{1},\dots,-b_{p})$ dont les résidus aux pôles sont nuls et $S$ la surface plate associée. Nous pouvons supposer que la surface $S$ possède un cœur dégénéré (cf proposition~\ref{prop:coeurdege}) et des liens selles horizontaux. Coupons $S$ le long de ces liens selles et de toutes les demi-droites horizontales issues des singularités. On obtient une alors une union disjointe de demi-plan positifs et négatifs. De plus, remarquons que chaque pôle d'ordre $-b_{j}$ est associé à $b_{j}-1$ demi-plan positifs et  $b_{j}-1$ demi-plan négatifs.

Considérons une singularité conique, disons $z_{1}$ d'angle $2\pi(a_{1}+1)$. Comme les résidus aux pôles sont nuls, on obtient directement du théorème des résidus que tout chemin fermé  de $S$ possède une période nulle.
Cela implique que les points correspondant à $z_{1}$ peuvent apparaître au plus une fois par demi-plan. 
Comme l'angle à chaque sommet dans chaque demi plan est $\pi$ l'angle maximal de la singularité conique $z_{1}$ est $2\pi\sum_{j=1}^p (b_{j}-1)$. Cela est équivalent au fait que l'ordre $a_{1}$ de $z_{1}$ est inférieur ou égal à $\sum_{j=1}^p b_{j}-(p+1)$.
\smallskip
\par
Nous montrons maintenant que si tous les zéros d'une différentielle sont d'ordres inférieurs ou égaux à $\sum b_{j}-(p+1)$ alors l'origine est dans l'image de l'application résiduelle.
Considérons tout d'abord les strates  $\omoduli[0]\left(a_{1},a_{2};-b_{1},\dots,-b_{p}\right)$ ayant deux zéros avec $p-1\leq a_{1},a_{2}\leq \sum b_{j} -(p+1)$. Pour tous les pôles nous prenons une partie polaire triviale $S_{i}$ d'ordre $b_{i}$ et de type~$\tau_{i}$ associée à $(1;1)$. Nous choisissons les~$\tau_{i}$ tels que $\sum_{i}\tau_{i}=a_{1}+1$. Ce choix est possible car pour chaque $i$ l'inégalité $1\leq \tau_{i} \leq b_{i}-1$ implique en sommant sur les pôles que $p\leq\sum_{i}\tau_{i}\leq \sum b_{j}-p$. 

Ensuite, nous collons les bords des parties polaires de manière cyclique. Plus précisément, nous collons le segment supérieur de $S_{i}$ au segment inférieur de $S_{i+1}$ modulo $p$. Une telle construction est représentée sur la figure~\ref{fig:ordremax}. La surface plate ainsi obtenue possède deux singularités coniques et est de genre nul. De plus, l'angle de la singularité conique à gauche des liens selles est d'angle $2\pi\sum \tau_{i}$. La différentielle ainsi construite appartient à la strate  $\omoduli[0]\left(a_{1},a_{2};-b_{1},\dots,-b_{p}\right)$ et n'a pas de résidus aux pôles.

 \begin{figure}[htb]
\center
\begin{tikzpicture}

\begin{scope}[xshift=-6cm]
\fill[fill=black!10] (0,0) coordinate (Q) circle (1.1cm);

\draw[] (0,0) coordinate (Q) -- (1.1,0) coordinate[pos=.5](a);

\node[above] at (a) {$1$};
\node[below] at (a) {$3$};

\draw[] (Q) -- (-.5,0) coordinate (P) coordinate[pos=.5](c);

\fill (Q)  circle (2pt);

\fill[color=white!50!] (P) circle (2pt);
\draw[] (P) circle (2pt);
\node[above] at (c) {$a$};
\node[below] at (c) {$b$};
\end{scope}

\begin{scope}[xshift=-3.7cm]
\fill[fill=black!10] (0,0) coordinate (Q) circle (1.1cm);

\draw[] (0,0) -- (1.1,0) coordinate[pos=.5](a);

\node[above] at (a) {$2$};
\node[below] at (a) {$1$};
\fill (Q)  circle (2pt);
\end{scope}
\begin{scope}[xshift=-1.4cm]
\fill[fill=black!10] (0,0) coordinate (Q) circle (1.1cm);

\draw[] (0,0) -- (1.1,0) coordinate[pos=.5](a);

\node[above] at (a) {$3$};
\node[below] at (a) {$2$};
\fill (Q)  circle (2pt);
\end{scope}

\begin{scope}[xshift=2cm]
\fill[fill=black!10] (0,0)  circle (1.1cm);

\draw[] (0,0) coordinate (Q) -- (1.1,0) coordinate[pos=.5](a);

\node[above] at (a) {$4$};
\node[below] at (a) {$5$};

\draw[] (0,0) -- (-.5,0) coordinate (P) coordinate[pos=.5](c);

\fill (Q)  circle (2pt);

\fill[color=white!50!] (P) circle (2pt);
\draw[] (P) circle (2pt);
\node[above] at (c) {$b$};
\node[below] at (c) {$a$};
\end{scope}

\begin{scope}[xshift=4.3cm]
\fill[fill=black!10] (0,0) coordinate (Q) circle (1.1cm);
    
\draw[] (0,0) -- (1.1,0) coordinate[pos=.5](a);

\node[above] at (a) {$5$};
\node[below] at (a) {$4$};
\fill (Q)  circle (2pt);
\end{scope}
\end{tikzpicture}
\caption{Une différentielle dans $\omoduli[0](4,1;-3,-4)$ avec résidus nuls} \label{fig:ordremax}
\end{figure}
\smallskip
\par
Nous traitons maintenant les strates $\omoduli[0](a_{1},a_{2},a_{3};-b_{1},\dots,-b_{p})$, avec $a_{1}, a_{2}\leq a_{3}$. Il y a deux cas à considérer suivant que  $a_{1},a_{2}>\sum b_{i}-p-1$ ou non.
\par
Si $a_{1}+a_{2}\leq \sum b_{i}-p-1$ (et donc $p-1\leq a_{3} \leq \sum b_{i}-p-1$), alors il existe une différentielle à résidus nuls dans la strate $\omoduli[0](a_{1}+a_{2},a_{3};-b_{1},\dots,-b_{p})$. Donc en éclatant le zéro d'ordre $a_{1}+a_{2}$ en deux zéros d'ordres $a_{1}$ et $a_{2}$ par la proposition~\ref{prop:eclatZero}, nous obtenons la différentielle souhaitée.
\par
Supposons maintenant que $a_{1}+a_{2} > \sum b_{i}-p-1$ (ou de manière équivalente $a_{3}<p-1$). Dans le reste de la preuve nous noterons $b:=\sum b_{i}$. Remarquons que  $a_{1},a_{2}\leq a_{3}$ et $a_{1}+a_{2}+a_{3}=b-2$ impliquent que  $ 3a_{3} \geq b-2$. Nous obtenons donc que $3p>b+1$, ce qui implique que l'un des pôles est d'ordre~$-2$.
\par
Nous donnons maintenant la description d'une différentielle ayant les propriétés attendues. Cette construction est illustrée dans la figure~\ref{fig:troiszerossansres} dans le cas de  $\omoduli[0](3^3,-2^4,-3)$. Dans un premier temps, nous décrivons le procédé, puis nous ajusterons les constantes pour obtenir les singularités coniques souhaitées. 

Pour le pôle d'ordre $-2$, prenons une partie polaire triviale d'ordre $2$ associée à $(v_{1},v_{2};v_{3})$ avec $v_{3}=v_{1}+v_{2}$.  Pour chaque pôle $P_{i}$ nous prenons une partie polaire triviale d'ordre~$b_{i}$ de type $\tau_{i}$ associée  à  $(v_{j_{i}};v_{j_{i}})$ pour un $j_{i}\in\lbrace 1,2,3 \rbrace$. Puis nous collons le segment $v_{j_{i}}$ du triangle au segment $v_{j_{i}}$ correspondant au domaine basique négatif de cette partie polaire.  Cette opération ajoute une contribution angulaire de $2\pi\tau_{i}$ et $2\pi(b_{i}-\tau_{i})$ aux singularités coniques correspondant aux sommets du segment $v_{j_{i}}$. Nous faisons de même pour tous les pôles jusqu'à obtenir une surface plate $S_{1}$ dont le bord est composé des trois segments $v_{j}$. Nous prenons maintenant le triangle $v_{1}v_{2}v_{3}$ et collons par translation ses trois arêtes au bord de $S_{1}$.
La surface plate ainsi obtenue est de genre zéro, possède trois zéros distincts et n'a pas de résidus aux pôles. 
\par
Il reste à ajuster le choix  des segments $v_{j_{i}}$ et des types $\tau_{i}$ pour chaque pôle afin d'obtenir les  angles souhaités. La remarque clé est que chaque pôle $P_{i}$ contribue à exactement deux singularités coniques et que la contribution à chacune d'elles est d'angle compris entre $2\pi$ et $2\pi(b_{i}-1)$. Réciproquement, si cette condition est satisfaite, alors la construction précédente permet de construire la différentielle souhaitée.
\par
La situation peut donc être modélisée par un graphe biparti $\Gamma$ dont trois sommets~$A_{i}$ représentent les trois sommets du triangle $v_{1}v_{2}v_{3}$ et $p-1$ autres sommets $B_{i}$ représentent les pôles $P_{i}$ distincts du pôle avec la partie polaire associée au triangle. Il y a une arête entre les sommets $B_{i}$ et~$A_{j}$ pour chaque multiple de $2\pi$ de la contribution de $P_{i}$ à la singularité conique~$A_{j}$. Un exemple est schématisé dans la  figure~\ref{fig:troiszerossansres}. 
\begin{figure}[htb]
\center
\begin{tikzpicture}

\begin{scope}[xshift=-1.3cm]
\fill[fill=black!10] (0,0)  circle (1.3cm);
   
 \coordinate (A) at (1/1.71,0);
  \coordinate (B) at (120:1/1.71);
    \coordinate (C)  at (240:1/1.71);
    
    \fill (A)  circle (2pt);
\filldraw[fill=white] (B) circle (2pt);
\filldraw[fill=red] (C)  circle (2pt);

   \fill[color=white]     (A) -- (B) -- (C) --cycle;
 \draw[]     (A) -- (B) coordinate[pos=.4](a);
 \draw (B) -- (C) coordinate[pos=.5](b);
 \draw (C) -- (A)  coordinate[pos=.6](c);

\node[above] at (a) {$v_{1}^{1}$};
\node[left] at (b) {$v_{2}^{1}$};
\node[below ] at (c) {$v_{3}^{1}$};

\end{scope}

\begin{scope}[xshift=2cm]
 \coordinate (A) at (1/1.71,0);
  \coordinate (B) at (120:1/1.71);
    \coordinate (C)  at (240:1/1.71);
    
    \fill[fill=black!10]     (A) -- (B) -- (C) --cycle;

 \draw[]     (A) -- (B) coordinate[pos=.4](a);
 \draw (B) -- (C) coordinate[pos=.5](b);
 \draw (C) -- (A)  coordinate[pos=.6](c);

\node[above] at (a) {$v_{1}^{2}$};
\node[left] at (b) {$v_{2}^{2}$};
\node[below ] at (c) {$v_{3}^{3}$};


    \fill (A)  circle (2pt);

\filldraw[fill=white] (B) circle (2pt);
\filldraw[fill=red] (C)  circle (2pt);

\end{scope}

 \begin{scope}[xshift=-5cm]
\fill[fill=black!10] (0,0)  circle (1.1cm);
 
\draw[](0,-.7)coordinate(P) --(0,.3) coordinate[pos=.5](b)coordinate (Q) -- (0,1.1) coordinate[pos=.5](a);

\node[left] at (a) {$1$};
\node[right] at (a) {$2$};

\filldraw[fill=red] (P)  circle (2pt);

\filldraw[fill=white] (Q) circle (2pt);
\node[left] at (b) {$v_{2}^{2}$};
\node[right] at (b) {$v_{2}^{1}$};
\end{scope}
\begin{scope}[yshift=-2.4cm]

\begin{scope}[xshift=-5cm]
\fill[fill=black!10] (0,0)  circle (1.1cm);

\draw[] (0,0) coordinate (Q1) -- (0,1.1) coordinate[pos=.5](a);

\node[left] at (a) {$2$};
\node[right] at (a) {$1$};


\filldraw[fill=white] (Q1) circle (2pt);

\end{scope}


\begin{scope}[xshift=-2.5cm]
\fill[fill=black!10] (0,0) coordinate (Q) circle (1.1cm);
    
    \begin{scope}[rotate=120]
\draw[]     (0,-1/2) -- (0,1/2)coordinate[pos=.6](a);
\node[above] at (a) {$v_{3}^{1}$};
\node[xshift=5pt,yshift=-5pt] at (a) {$v_{3}^{2}$};

    \fill (0,-1/2)  circle (2pt);
    \filldraw[fill=red] (0,1/2)  circle (2pt);

    \end{scope}

\end{scope}

\begin{scope}[xshift=0cm]
\fill[fill=black!10] (0,0) coordinate (Q) circle (1.1cm);
    
    \begin{scope}[rotate=120]
\draw[]     (0,-1/2) -- (0,1/2)coordinate[pos=.6](a);
\node[above] at (a) {$v_{3}^{2}$};
\node[xshift=5pt,yshift=-5pt] at (a) {$v_{3}^{3}$};

    \fill (0,-1/2)  circle (2pt);
    \filldraw[fill=red] (0,1/2)  circle (2pt);

    \end{scope}

\end{scope}

\begin{scope}[xshift=2.5cm]
\fill[fill=black!10] (0,0) coordinate (Q) circle (1.1cm);
      
    \begin{scope}[rotate=240]
\draw[]     (0,-1/2) -- (0,1/2)coordinate[pos=.4](a);
\node[below] at (a) {$v_{1}^{1}$};
\node[xshift=5pt,yshift=5pt] at (a) {$v_{1}^{2}$};

\fill (0,1/2)  circle (2pt);
\filldraw[fill=white] (0,-1/2) circle (2pt);
    \end{scope}
\end{scope}
\end{scope}

\begin{scope}[xshift=6cm,yshift=-.5cm]
\filldraw[fill=red] (-1,0)coordinate (A1) circle (2pt);
\fill (0,0)coordinate (A2)  circle (2pt);
\filldraw[fill=white] (1,0)coordinate (A3) circle (2pt);

\fill (-1.5,-2)coordinate (B1)  circle (2pt);
\fill (-.5,-2)coordinate (B2)  circle (2pt);
\fill (.5,-2)coordinate (B3)  circle (2pt);
\fill (1.5,-2)coordinate (B4)  circle (2pt);

\draw (A1) -- (B1);
\draw (A1) -- (B2);
\draw (A1) -- (B3);
\draw (A2) -- (B2);
\draw (A2) -- (B3);
\draw (A2) -- (B4);
\draw (A3) -- (B4);
\draw (A3) .. controls ++(-90:.5) and ++(0:.5) .. (B1);
\draw (A3) .. controls ++(180:.5) and ++(90:.5) .. (B1);

\filldraw[fill=red] (A1) circle (2pt);
\filldraw[fill=white] (A3) circle (2pt);
    \end{scope}

\end{tikzpicture}
\caption{Différentielle dans $\omoduli[0]((3)^3;(-2)^4,-3)$ avec résidus nuls et son graphe $\Gamma$.} \label{fig:troiszerossansres}
\end{figure}

Il suffit de montrer qu'il est possible de distribuer les arêtes de telle façon que les sommets~$B_{i}$ soient de valence $b_{i}$, connectés à précisément deux sommets $A_{j}$, et que la valence de $A_{j}$ soit $a_{j}$.
Une telle distribution peut être obtenue de la façon suivante. Rappelons que le plus grand zéro (que nous supposerons être $z_{3}$) est d'ordre $a_{3}\leq p-1$. Nous partons du graphe où tous les sommets $B_{i}$ sont connectés à $A_{1}$ par exactement une arête et toutes les autres arêtes connectent $A_{2}$. Prenons un sommet $B_{i}$ quelconque. Il y a $b_{i}-1$ arêtes entre $B_{i}$ et $A_{2}$. Si la valence de $A_{2}$ moins $b_{i}-1$ est supérieure ou égale à~$a_{2}$, alors nous remplaçons les $b_{i}-1$ arêtes entre $B_{i}$ et  $A_{2}$ par $b_{i}-1$ arêtes entre $B_{i}$ et~$A_{3}$. Nous recommençons alors cette opération jusqu'à l'indice $i_{0}$ tel que la valence de $A_{2}$ moins $b_{i_{0}}-1$ soit strictement inférieure à~$a_{2}$. Dans ce cas, nous remplaçons des arêtes entre $B_{i_{0}}$ et $A_{2}$ par des arêtes entre $B_{i_{0}}$ et~$A_{3}$, de telle façon que la valence de $A_{2}$  soit égale à~$a_{2}$. Le sommet $B_{i_{0}}$ est alors connecté aux trois sommets $A_{i}$. Donc nous enlevons l'arête entre $A_{1}$ et $B_{i_{0}}$ pour la mettre entre $A_{3}$ et $B_{i_{0}}$. Comme $a_{1}$ est strictement plus petit que $p-1$, cette opération est toujours possible.  Pour terminer, nous remplaçons autant d'arêtes que nécessaire entre $A_{1}$ et les $B_{j}$ pour $j > i_{0}$ pour les connecter à $A_{3}$ afin que la valence de $A_{1}$ soit~$a_{1}$. Notons qu'il existe toujours assez d'arêtes entre $A_{1}$ et les $B_{j}$ pour obtenir la valence $a_{3}$ à $A_{3}$ car il y a un total de $p-1$ sommets~$B_{i}$ et $a_{3}\leq p-1$.

Pour conclure le cas des différentielles de genre zéro avec des résidus nuls, nous considérons les strates ayant $n\geq4$ zéros.  Soient~$a_{1}$ et~$a_{2}$ les zéros de plus petits ordres. En notant  $b=\sum_{i} b_{i}$ on~a  
\begin{equation*}
 a_{1}+a_{2} \leq \frac{2}{n}\left(b-2\right)\leq \sum_{i}(b_{i}-1) -\frac{4}{n} < b -p,
\end{equation*}
 où la deuxième égalité s'obtient en remarquant que  $\tfrac{2}{n} \leq \tfrac {1}{2}$ et $b_{i}\geq2$ impliquent que $\tfrac{2b_{i}}{n}\leq b_{i}-1$. On en déduit que $ a_{1}+a_{2} \leq b -p - 1$ et donc que le cas $n\geq 4$ s'obtient en éclatant un zéro d'une différentielle ayant $n-1$ zéros d'ordres $a_{1}+a_{2}$, $a_{i}$ pour $i\geq 3$ et $p$ pôles d'ordres~$-b_{j}$.
\end{proof}

\subsection{Genre zéro avec pôles simples}
\label{sec:caslier}
Dans cette section nous nous intéressons au cas des différentielles en genre zéro n'ayant que des pôles simples. Nous prouvons les propositions \ref{prop:gzeropolesimples} et \ref{prop:g0p-1plusieurszero}. De plus nous montrons que les uplets qui ne sont pas dans l'image de l'application résiduelle sont en nombre fini pour chaque strate (à multiplication par un facteur complexe près).

Nous considérons d'abord le cas des strates $\omoduli[0](s-2;(-1^{s}))$. Afin de se familiariser avec les éléments de cette preuve, le lecteur peut consulter l'exemple~\ref{ex:graphedeconnexion}.
\begin{proof}[Preuve de la proposition~\ref{prop:gzeropolesimples}]
Supposons que les éléments du $s$-uplet $r:=(r_{1},\dots,r_{s})$ ne soient pas colinéaires. Dans ce cas, le polygone résiduel $\mathfrak{P}(r)$ (introduit au début du paragraphe~\ref{sec:casgen}) est un polygone convexe non dégénéré. On obtient alors la différentielle souhaitée en collant des demi-cylindres infinis aux arêtes de ce polygone.

Supposons maintenant que les $r_{i}$ soient colinéaires et qu'il existe un graphe $\Gamma$ associé à $r$ qui soit un graphe de connexion. On construit une différentielle de la façon suivante. Pour chaque résidu $r_{i}$ on prend une partie polaire d'ordre $1$ associée à $r_{i}$. Considérons une feuille de $\Gamma$. On peut coller le segment au bord de la partie polaire correspondante au segment au bord de la partie polaire correspondant à l'autre sommet de l'arête. Puis on enlève la feuille du graphe et le poids de cette feuille à l'autre sommet. On recommence cette procédure pour une feuille du nouveau graphe.  Cette opération est faite de manière inductive jusqu'à ce que le graphe soit réduit à un sommet.

Nous justifions maintenant que cette opération est toujours possible. Supposons tout d'abord que le graphe possède strictement plus de deux sommets. Comme le poids d'une feuille est toujours strictement plus petit que le poids du sommet auquel elle est reliée (point~${\rm (ii)}$), le collage est toujours possible. Enfin le fait que la différence des poids est nulle (point~${\rm (i)}$), implique que la surface obtenue est sans bord. On vérifiera facilement que cette surface plate est de genre zéro et possède un unique zéro. La surface est de genre zéro car s'il existait un lacet fermé homotopiquement non trivial, cela impliquerait un chemin fermé sur le graphe de connexion. Or celui-ci est un arbre. L'arbre étant connexe, la surface l'est également. Il reste à montrer que les extrémités des liens selles constituent une unique singularité conique. En tant que graphe plongé dans la surface, le graphe des liens selles est dual de celui défini par le graphe de connexions avec des arcs reliant les pôles. Le graphe de connexions étant un arbre, il ne définit qu'une seule face. Par conséquent, le graphe des liens selles n'a qu'un seul sommet.
\smallskip
\par
Supposons maintenant qu'il existe une différentielle $\omega$ dans $\omoduli[0](s-2;(-1^{s}))$ dont les résidus sont $(r_{1},\dots,r_{s})$. Supposons que les résidus soient colinéaires, nous les supposerons réels sans perte de généralité. Nous construisons un graphe associé aux $r_{i}$ qui est un graphe de connexions. Les sommets de $\Gamma_{+}$ (resp. $\Gamma_{-}$) sont associés aux pôles de $\omega$ dont le résidu est positif (resp. négatif). Les poids sont les valeurs absolues des résidus. Enfin deux sommets sont connectés si et seulement si le bord de leurs domaines polaires respectifs contiennent un même lien selle.

Le fait que $\omega$ soit de genre zéro et ne possède qu'un zéro implique clairement que ce graphe biparti est un arbre. Le théorème des résidus implique directement que la différence des poids est nulle (point (i)). Regardons maintenant l'effet de l'opération qui enlève une feuille au graphe de connexions (dans le cas où il possède au moins deux arêtes). Cette opération revient à couper la différentielle $\omega$ le long d'un lien selle $v$ dont l'un des côté est un unique domaine polaire. De plus, le domaine polaire adjacent à ce lien selle est remplacé par le domaine polaire associé aux vecteurs précédents privés de~$v$ (qui est un ensemble non vide car $\omega$ est non singulière). Les autres domaines polaires et les identifications restent les mêmes. 
\`A chaque étape, cette opération produit une différentielle de genre zéro avec un unique zéro et des pôles simples dont tous les résidus sont non nuls. Cela implique clairement le point (ii) des graphes de connexions.
\end{proof}

Avant d'étudier le cas des strates avec plus de zéros, nous montrons que les $s$-uplets qui ne sont pas dans l'image de $\appres[0](s-2;(-1^{s}))$ sont en nombre fini et que ceux-ci sont commensurables entre eux.
\begin{prop}\label{prop:finietcommens}
Soient $r:=(x_{1},\dots,x_{s_{1}},-y_{1},\dots,-y_{s_{2}})$ avec $x_{i}$ et $y_{j}$ réels strictement positifs. Si le $s$-uplet $r$ n'appartiennent pas à l'image de $\appres[0](s-2;(-1^{s}))$, alors les $x_{i}$ et $y_{j}$ sont commensurables entre eux. De plus, si les $x_{i}$ et $y_{j}$ sont entiers et premiers entre eux, alors:
\begin{equation}\label{eq:borneliee}
\sum_{i=1}^{s_{1}} x_{i} = \sum_{j=1}^{s_{2}} y_{j} \leq \frac{s_{1}s_{2}}{2}.
\end{equation}
En particulier, il n'y a, à homothétie près, qu'un nombre fini de $s$-uplets qui ne sont pas dans l'image de $\appres[0](s-2;(-1^{s}))$.
\end{prop}

\begin{proof}
Nous procédons à une démonstration par récurrence.
Si on a $s_{1}=1$ ou $s_{2}=1$, toutes les configurations de résidus sont réalisables. Si $s_{1}=s_{2}=2$, les seuls uplets qui ne sont pas dans l'image de l'application résiduelle sont proportionnels à $(1,1,-1,-1)$.

\`A présent, on suppose que la proposition est démontrée pour tous les couples $(s_{1},s_{2})$ tels que $s_{1}\leq a$, $s_{2}\leq b$ et au moins l'une des inégalités est stricte. Nous considérons un uplet qui n'est pas dans l'image  avec $s_{1}$ nombres positifs et $s_{2}$ négatifs. Il y a deux possibilités. La première est que tous les $x_{i}$ et $y_{j}$ soient égaux. Auquel cas on a $s_{1}=s_{2}$ et le uplet est proportionnel à $(1,\dots,1,-1,\dots,-1)$. Les résidus sont commensurables et la somme des résidus d'une même série respecte la borne de l'équation~\eqref{eq:borneliee}.\newline
La seconde possibilité est qu'il existe deux résidus, disons $x_{s_{1}}$ et $-y_{s_{2}}$, tels que $y_{s_{2}}<x_{s_{1}}$. Le $(s-1)$-uplet obtenu en retirant le résidu $-y_{s_{2}}$ et en remplaçant $x_{s_{1}}$ par $x_{s_{1}}-y_{s_{2}}$ n'est pas réalisable. En effet, si ce nouveau uplet était réalisable, il existerait un graphe de connexions qui lui serait associé. Il suffirait d'ajouter à ce graphe une branche avec comme poids $y_{s_{2}}$ au sommet de poids  $x_{s_{1}}-y_{s_{2}}$ et de remplacer le poids $x_{s_{1}}-y_{s_{2}}$ par $x_{s_{1}}$. Ce graphe serait un graphe de connexions pour la configuration initiale, ce qui est absurde. 

Ainsi, quitte à changer les signes, tout $s$-uplet non réalisable avec $a$ résidus positifs et~$b$ négatifs s'obtient à partir d'un $s-1$-uplet non réalisable $(x_{1},\dots,x_{a},-y_{1},\dots,-y_{b-1})$ auquel on ajoute un résidu $y_{b}$ et on remplace un résidu $x_{i}$ par $x_{i}+y_{b}$. On peut supposer, quitte à changer l'ordre que $i=a$ et on note $x'_{a}:=x_{a}+y_{b}$. 
On cherche donc pour quelles valeurs de $y_{b}$ le $s$-uplet $(x_{1},\dots,x'_{a},-y_{1},\dots,-y_{b})$ est non réalisable. Par hypothèse de récurrence, on normalise ces nombres de telle sorte que les $x_{i}$ et les $y_{j}$ avec $j\neq b$ soient des entiers premiers entre eux. 
Si $y_{b}$ n'était pas un entier, alors ces résidus seraient dans l'image. En effet, un graphe de connexions serait obtenu de la façon suivante. On permute~$y_{1}$ et~$y_{b}$. On prend $s_{1}$ sommets en haut et $s_{2}$ sommets en bas. Le sommet $i_{0}$ en haut est relié au sommet $j_{0}$ en bas si et seulement si pour $J=j_{0}$ ou $J=j_{0}+1$, on a
 $$ \sum_{i\leq i_{0}-1}x_{i}\leq \sum_{j\leq J} y_{j} \leq  \sum_{i\leq i_{0}}x_{i}\, .$$
Les poids des sommets sont évidemment donnés par les $y_{j}$, les $x_{i}$ pour $i<a$ et $x'_{a}$ pour $i=s_{1}$.
Par conséquent, pour obtenir une configuration non réalisable, il est nécessaire que~$y_{b}$ soit un entier. En effet, dans le cas contraire, les sommes partielles ne peuvent pas coïncider car l'une des familles de sommes partielles n'est pas constituée d'entiers et donc l'opération consistant à retirer une feuille au graphe peut toujours s'effectuer. Nous avons donc démontré par récurrence que les éléments des $s$-uplets non réalisables sont commensurables entre eux.

On montre maintenant la borne donnée par l'équation~\eqref{eq:borneliee}. On travaille à nouveau avec le $s$-uplet défini par $(x_{1},\dots,x_{s_{1}},-y_{1},\dots,-y_{s_{2}})$. Si les résidus sont tous égaux, la borne est clairement satisfaite. Sinon, il y a un élément $x_{i}$ strictement plus grand qu'un élément $y_{j}$. On supposera que $i=s_{1}$ et $j=s_{2}$ et que $y_{s_{2}}\leq y_{j}$ pour tout $j\leq s_{2}$. En enlevant le résidu $-y_{s_{2}}$ et en l'ajoutant à $x_{s_{1}}$, on obtient une nouvelle configuration non réalisable. D'après l'hypothèse de récurrence, elle vérifie la borne donnée par l'équation~\eqref{eq:borneliee} avec $s_{2}-1$ termes. Comme $y_{s_{2}}$ est inférieur ou égal aux autres $y_{j}$, on a
$$\sum_{j=1}^{s_{2}} y_{j} \leq \dfrac{s_{2}}{s_{2}-1}\cdot\sum_{j=1}^{s_{2}-1} y_{j} \leq  \dfrac{s_{2}}{s_{2}-1}\cdot\dfrac{s_{1}\cdot (s_{2}-1)}{2} \leq \dfrac{s_{1}\cdot s_{2}}{2}.$$
Ceci démontre par récurrence la borne de l'équation~\eqref{eq:borneliee}.

Enfin, tous les résidus non réalisables sont proportionnels à un $s$-uplet formé de nombres entiers inférieurs ou égaux à la borne de l'équation~\eqref{eq:borneliee}. Comme celle-ci ne dépend que du nombre de pôles, les droites complexes n'appartenant pas à l'image de l'application résiduelle d'une strate donnée sont en nombre fini.
\end{proof}

Nous illustrons maintenant les concepts introduits dans un exemple.
\begin{ex} \label{ex:graphedeconnexion}
 Nous illustrons tout d'abord la correspondance entre une différentielle de $\omoduli[0](5;(-1^7))$ et le graphe de connexion associé. La figure~\ref{fig:totliermaisdansimage} montre que l'image de l'application résiduelle de  cette strate contient les résidus $(3,(1^{3}),(-2^{3}))$.
 
\begin{figure}[htb]
\begin{tikzpicture}

\begin{scope}[xshift=-6cm]
\coordinate (a) at (-1,1);
\coordinate (b) at (2,1);
\coordinate (c) at (0,1);
\coordinate (d) at (1,1);

    \fill[fill=black!10] (a)  -- (c)coordinate[pos=.5](f)-- (d)coordinate[pos=.5](g)-- (b)coordinate[pos=.5](j) -- ++(0,1.2) --++(-3,0) -- cycle;
    \fill (a)  circle (2pt);
\fill[] (b) circle (2pt);
    \fill (c)  circle (2pt);
\fill[] (d) circle (2pt);
 \draw  (a) -- (b);
 \draw (a) -- ++(0,1.1) coordinate (d)coordinate[pos=.5](h);
 \draw (b) -- ++(0,1.1) coordinate (e)coordinate[pos=.5](i);
 \draw[dotted] (d) -- ++(0,.2);
 \draw[dotted] (e) -- ++(0,.2);
\node[below] at (f) {$1$};
\node[below] at (g) {$2$};
\node[below] at (j) {$3$};
\end{scope}

\begin{scope}[xshift=-2.5cm]
\coordinate (a) at (-1,1);
\coordinate (b) at (0,1);

    \fill[fill=black!10] (a)  -- (b)coordinate[pos=.5](f) -- ++(0,1.2) --++(-1,0) -- cycle;
    \fill (a)  circle (2pt);
\fill[] (b) circle (2pt);
 \draw  (a) -- (b);
 \draw (a) -- ++(0,1.1) coordinate (d)coordinate[pos=.5](h);
 \draw (b) -- ++(0,1.1) coordinate (e)coordinate[pos=.5](i);
 \draw[dotted] (d) -- ++(0,.2);
 \draw[dotted] (e) -- ++(0,.2);
\node[below] at (f) {$4$};
\end{scope}

\begin{scope}[xshift=-1cm]
\coordinate (a) at (-1,1);
\coordinate (b) at (0,1);

    \fill[fill=black!10] (a)  -- (b)coordinate[pos=.5](f) -- ++(0,1.2) --++(-1,0) -- cycle;
    \fill (a)  circle (2pt);
\fill[] (b) circle (2pt);
 \draw  (a) -- (b);
 \draw (a) -- ++(0,1.1) coordinate (d)coordinate[pos=.5](h);
 \draw (b) -- ++(0,1.1) coordinate (e)coordinate[pos=.5](i);
 \draw[dotted] (d) -- ++(0,.2);
 \draw[dotted] (e) -- ++(0,.2);
\node[below] at (f) {$5$};
\end{scope}

\begin{scope}[xshift=.5cm]
\coordinate (a) at (-1,1);
\coordinate (b) at (0,1);

    \fill[fill=black!10] (a)  -- (b)coordinate[pos=.5](f) -- ++(0,1.2) --++(-1,0) -- cycle;
    \fill (a)  circle (2pt);
\fill[] (b) circle (2pt);
 \draw  (a) -- (b);
 \draw (a) -- ++(0,1.1) coordinate (d)coordinate[pos=.5](h);
 \draw (b) -- ++(0,1.1) coordinate (e)coordinate[pos=.5](i);
 \draw[dotted] (d) -- ++(0,.2);
 \draw[dotted] (e) -- ++(0,.2);
\node[below] at (f) {$6$};
\end{scope}
\begin{scope}[xshift=-5.75cm,yshift=-1cm]
\coordinate (a) at (-1,1);
\coordinate (b) at (1,1);
\coordinate (c) at (0,1);

    \fill[fill=black!10] (a)  -- (b)coordinate[pos=.25](f)coordinate[pos=.75](g) -- ++(0,-1.2) --++(-2,0) -- cycle;
    \fill (a)  circle (2pt);
\fill[] (b) circle (2pt);
    \fill (c)  circle (2pt);
 \draw  (a) -- (b);
 \draw (a) -- ++(0,-1.1) coordinate (d)coordinate[pos=.5](h);
 \draw (b) -- ++(0,-1.1) coordinate (e)coordinate[pos=.5](i);
 \draw[dotted] (d) -- ++(0,-.2);
 \draw[dotted] (e) -- ++(0,-.2);
\node[above] at (f) {$1$};
\node[above] at (g) {$4$};
\end{scope}

\begin{scope}[xshift=-3.25cm,yshift=-1cm]
\coordinate (a) at (-1,1);
\coordinate (b) at (1,1);
\coordinate (c) at (0,1);

    \fill[fill=black!10] (a)  -- (b)coordinate[pos=.25](f)coordinate[pos=.75](g) -- ++(0,-1.2) --++(-2,0) -- cycle;
    \fill (a)  circle (2pt);
\fill[] (b) circle (2pt);
    \fill (c)  circle (2pt);
 \draw  (a) -- (b);
 \draw (a) -- ++(0,-1.1) coordinate (d)coordinate[pos=.5](h);
 \draw (b) -- ++(0,-1.1) coordinate (e)coordinate[pos=.5](i);
 \draw[dotted] (d) -- ++(0,-.2);
 \draw[dotted] (e) -- ++(0,-.2);
\node[above] at (f) {$2$};
\node[above] at (g) {$5$};
\end{scope}

\begin{scope}[xshift=-.75cm,yshift=-1cm]
\coordinate (a) at (-1,1);
\coordinate (b) at (1,1);
\coordinate (c) at (0,1);

    \fill[fill=black!10] (a)  -- (b)coordinate[pos=.25](f)coordinate[pos=.75](g) -- ++(0,-1.2) --++(-2,0) -- cycle;
    \fill (a)  circle (2pt);
\fill[] (b) circle (2pt);
    \fill (c)  circle (2pt);
 \draw  (a) -- (b);
 \draw (a) -- ++(0,-1.1) coordinate (d)coordinate[pos=.5](h);
 \draw (b) -- ++(0,-1.1) coordinate (e)coordinate[pos=.5](i);
 \draw[dotted] (d) -- ++(0,-.2);
 \draw[dotted] (e) -- ++(0,-.2);
\node[above] at (f) {$3$};
\node[above] at (g) {$6$};
\end{scope}

\begin{scope}[xshift=3cm,yshift=1.5cm]
\filldraw[fill=white] (-1,-2)coordinate (A1) circle (2pt);\node[below] at (A1) {$2$};
\filldraw[fill=white] (0,-2)coordinate (A2)  circle (2pt);\node[below] at (A2) {$2$};
\filldraw[fill=white] (1,-2)coordinate (A3) circle (2pt);\node[below] at (A3) {$2$};

\fill (-1.5,0)coordinate (B1)  circle (2pt);\node[above] at (B1) {$3$};
\fill (-.5,0)coordinate (B2)  circle (2pt);\node[above] at (B2) {$1$};
\fill (.5,0)coordinate (B3)  circle (2pt);\node[above] at (B3) {$1$};
\fill (1.5,0)coordinate (B4)  circle (2pt);\node[above] at (B4) {$1$};

\draw (A1) -- (B1);
\draw (A2) -- (B1);
\draw (A3) -- (B1);
\draw (A1) -- (B2);
\draw (A2) -- (B3);
\draw (A3) -- (B4);

\filldraw[fill=white] (-1,-2)coordinate (A1) circle (2pt);
\filldraw[fill=white] (0,-2)coordinate (A2)  circle (2pt);
\filldraw[fill=white] (1,-2)coordinate (A3) circle (2pt);
    \end{scope}

\end{tikzpicture}
\caption{Différentielle dans $\omoduli[0](5;(-1^7))$ avec résidus $\left(3,(1^{3}),(-2^{3})\right)$ et son graphe de connexions.} \label{fig:totliermaisdansimage}
\end{figure}
\par
Nous montrons maintenant que l'application résiduelle de $\omoduli[0](s-2;(-1^{s}))$ n'est pas surjective dès que $s\geq4$. En effet, soit $r:=(r_{1},\dots,r_{s})$ des nombres entiers dans $\espres[0](s-2;(-1^{s}))$. Notons  $r_{+}$ la somme des nombres positifs. Si les nombres $\ell_{+}$ et $\ell_{-}$ de $r_{i}$ égaux respectivement à $1$ et $-1$ sont supérieurs ou égaux à  $r_{+}/2$, alors $r$ n'est pas dans l'image de l'application résiduelle de cette strate. En effet, considérons un graphe~$\Gamma$ associé à~$r$. Comme $\ell_{\pm}\geq r_{+}/2$ il existe une arête de~$\Gamma$ connectant un sommet de poids $1$ à un sommet de poids~$-1$. Donc le point (ii) de la définition des graphes de connexions n'est jamais satisfait.
\par
Enfin nous montrons que si $\ell_{-}+\ell_{+}\geq r_{+}+1$ et $s\geq4$, alors le uplet $r$ n'est pas dans l'image de l'application résiduelle $\espres[0](s-2;(-1^{s}))$. Le cas $s=4$ correspond au cas où $r=(1,1,-1,-1)$ dans la strate $\omoduli[0](2;(-1^{4}))$ qui n'est pas réalisable par le paragraphe précédent. Supposons maintenant que $s>4$. Alors l'opération décrite dans le point (ii) de la définition des graphes de connexions fait diminuer de $1$ le nombre $s$ ainsi que $r_{+}$. En revanche la somme $\ell_{-}+\ell_{+}$ soit diminue de $1$ soit reste constante. Donc par récurrence sur $s$ on obtient le résultat.
\end{ex}

Nous passons maintenant au cas des strates ayant plus de deux zéros.

\begin{proof}[Preuve de la proposition~\ref{prop:g0p-1plusieurszero}]
Soient $\mu:=(a_{1},\dots,a_{n};(-1^{s}))$ une partition de $-2$ avec $n\geq2$ et $r:=(r_{1},\dots,r_{s})\in\espres[0](\mu)$. 
Supposons que $r$ soit dans l'image de l'application résiduelle de la strate $\omoduli[0](\mu)$. Montrons l'existence d'une différentielle stable vérifiant les conditions de la proposition~\ref{prop:g0p-1plusieurszero}. Soit $\omega$ une différentielle de $\omoduli[0](\mu)$ ayant pour résidus~$r$. Quitte à perturber $\omega$ sans changer les résidus, on peut supposer qu'il n'existe pas de liens selles horizontaux entre deux singularités coniques distinctes. En effet, le lieu de la strate $\omoduli[0](\mu)$ où les résidus sont $r$ est une variété orbifold de dimension~$n-1$ munie d'un atlas où les coordonnées sont les périodes des cycles de l'homologie relative (voir \cite{BCGGM3}). Comme tous les résidus sont réels, toute demi-droite horizontale issue d'une singularité conique heurte cette même singularité en temps fini. Coupons la surface plate associée à $\omega$ le long de ces liens selles. On obtient une union disjointe de cylindres et de demi-cylindres infinis.

La hauteur des cylindres d'aire finie peut être choisie arbitrairement sans changer la strate et les résidus. En faisant tendre toutes les hauteurs de ces cylindres vers l'infini, on obtient une différentielle stable. De plus, comme chaque zéro est relié à un autre par un cylindre, il y a précisément un zéro sur chaque composante irréductible de cette différentielle.

L'autre implication est claire: une différentielle dans la strate $\omoduli[0](\mu)$ ayant les résidus~$r$ est  donnée par le lissage de la différentielle stable comme expliqué au lemme~\ref{lem:lisspolessimples}.
\end{proof}

On donne une conséquence intéressante de ce résultat.
\begin{cor}\label{cor:gzerores1}
Les résidus $((1^{s+1}),(-1^{s+1}))$ n'appartiennent pas à l'image de l'application résiduelle des strates $\omoduli[0](a_{1},\dots,a_{n};(-1^{2s+2}))$ où $s\geq 1$ et $a_{1} > \frac{1}{2}(3s-1)$.
\end{cor}

\begin{proof}
Supposons par l'absurde qu'il existe une différentielle ayant ces invariants locaux. Notons $(X,\omega)$ la différentielle stable donnée par la proposition~\ref{prop:g0p-1plusieurszero}. Tout d'abord remarquons que tous les résidus  de $\omega$ aux nœuds de $X$ sont des entiers. Cela est clair pour les composantes qui n'ont qu'un nœud par le théorème des résidus. On en déduit le cas de toutes les composantes de proche en proche en utilisant le fait que la somme des résidus  de $\omega$ à un nœud est nulle.
\par
Considérons la composante~$X_1$ de $X$ qui contient le zéro d'ordre~$a_{1}$. 
Soient $\ell_-$ et $\ell_+$ les nombres de résidus égaux  respectivement à $1$ et à $-1$ aux pôles de~$\omega_1$.  Nous allons montrer que l'inégalité $a_{1} > \frac{1}{2}(3s-1)$ implique  que  $\ell_{+}+\ell_{-}\geq r_{+}+1$, où $r_{+}$ la somme des résidus positifs  de $\omega_1$. Or, ces résidus ne sont pas dans l'image de l'application résiduelle comme montré à la fin de l'exemple~\ref{ex:graphedeconnexion}. 
\par
Montrons que $\ell_{+}+\ell_{-}\geq r_{+}+1$.  Tout d'abord le nombre $\mathcal{N}_{1}$ de nœuds adjacents à la composante~$X_{1}$ est inférieur ou égal à $n-1$. Comme $a_{1} > \frac{1}{2}(3s-1)$ et $a_{1}+\dots+a_{n}=2s$ on a
\[\mathcal{N}_{1} \leq n-1 < 2s - \frac{1}{2}(3s-1) = \frac{1}{2} (s+1)\,.\]
Comme $\omega_{1}$ possède $a_{1}+2=\ell_{+} +\ell_{-}+\mathcal{N}_{1}$ pôles simples, on en déduit que $\ell_{+} +\ell_{-} > s+1$. De manière analogue au fait que les résidus aux nœuds sont des entiers, on montre que $r_{+}\leq s+1$. On obtient alors 'inégalité souhaitée.
\end{proof}

L'existence de telles différentielles entrelacées peut se reformuler dans le langage des graphes. Nous donnons maintenant une telle description puis illustrons ces constructions dans quelques exemples.

\begin{defn}\label{def:graphedeliant}
Soient $\mu:=(a_{1},\dots,a_{n};(-1^{s}))$ une partition de $-2$ et $\lambda:=(\lambda_{1},\dots,\lambda_{s})$ dans $\espres[0](\mu)$. Un {\em graphe déliant $G$} de type $\mu$ et de poids $\lambda$ est un arbre muni de marquages aux sommets et de pondérations aux demi-arêtes de la façon suivante.
\begin{enumerate}[1)]
\item $G$ a $n$ sommets $T_{1},\dots,T_{n}$.
\item Il y a $s$ marquages aux sommets de poids respectifs $\lambda_{1},\dots,\lambda_{s}$.
\item Les demi-arêtes $A_{i,1},A_{i,2}$ formant une arête ont pour poids $\Lambda_{i}$ et $-\Lambda_{i}$ avec $\Lambda_{i}\neq0$.
\end{enumerate}
De plus, les relations suivantes sont vérifiées. 
\begin{enumerate}[i)]
\item La somme du nombre de marquages et de la valence est $a_{i}+2$.
\item La somme des poids des marquages et du poids des demi-arêtes contenant $T_{i}$ est nul.
\item Le uplet formé des poids des marquages et du poids des demi-arêtes appartient à l'image de $\appres[0](a_{i};(-1^{a_{i}+2}))$.
\end{enumerate}
\end{defn}

\begin{lem}
 Soit $\omoduli[0](a_{1},\dots,a_{n};(-1^{s}))$ une strate de genre zéro avec $s$ pôles simples et $n$ zéros, alors les résidus $(r_{1},\ldots,r_{s})$ sont dans l'image de l'application résiduelle si et seulement si il existe un graphe déliant  de type $(a_{1},\dots,a_{n};(-1^{s}))$ et de poids $(r_{1},\ldots,r_{s})$.
\end{lem}

\begin{proof}
On part d'une différentielle entrelacée $(X,\omega)$ donnée par la proposition~\ref{prop:g0p-1plusieurszero}. Le graphe $G$ est donn\'e par le graphe dual de $X$. Les marquages correspondent aux pôles et les poids sont les résidus associés. Ce graphe vérifie les trois conditions d'un graphe déliant. Par exemple, la formule de Gauss-Bonnet appliquée à une composante irréductible de $X$ s'écrit $a_{i}-v_{i}-\lambda_{i}=-2$ (où $v_{i}$ est la valence du sommet correspondant) et nous donne la condition (i) de la définition~\ref{def:graphedeliant}.

Réciproquement, étant donn\'e un graphe déliant, on peut former une différentielle entrelacée en associant \`a chaque sommet une différentielle de $\omoduli[0](\mu_{i})$ o\`u le résidu aux pôles est donn\'e par le poids au marquage ou \`a la demi-arête. Le fait que cette différentielle entrelacée est lissable est une conséquence directe du lemme~\ref{lem:lisspolessimples}.
\end{proof}

\begin{ex}\label{ex:graphedeliant}
Dans cet exemple, nous donnons des graphes déliant de types $(1,3;(-1^{6}))$ et $(2,2;(-1^{6}))$ pour les r\'esidus $(2,1,1,-1,-1,-2)$. Tout d'abord on peut remarquer que ces r\'esidus ne sont pas dans l'image de $\appres[0](4;(-1^{6}))$. Ces graphes et les surfaces plates correspondantes sont repr\'esent\'es dans la figure~\ref{fig:graphdeliant}.

\begin{figure}[htb]
\center
\begin{tikzpicture}[scale=1.2]
    
    \begin{scope}[xshift=-4cm]
\fill[fill=black!10] (-4,-1.9)  -- (-2,-1.9)-- (-2,.9)-- (-4,.9) -- cycle;

   \foreach \i in {1,2,...,5}
  \coordinate (a\i) at (-1.5-\i/2,-1); 
       \foreach \i in {1,2,3,5}
   \fill (a\i)  circle (2pt);

     \foreach \i in {1,2,3,5}
     \draw (a\i) -- ++(0,-1);
     
   \foreach \i in {1,2,...,5}
  \coordinate (b\i) at (-1.5-\i/2,0); 
      
        \foreach \i in {1,2,3,5}
     \draw (b\i) -- ++(0,1);
     
     \draw (a1) -- (b1);
     \draw (a5) -- (b5);

      \foreach \i in {1,2,3,5}
   \fill[white] (b\i)  circle (2pt);
          \foreach \i in {1,2,3,5}
     \draw (b\i)  circle (2pt);
     
     \node[yshift=-.6cm, xshift=.1cm] at (a5) {$1$};
     \node[yshift=-.6cm, xshift=-.1cm] at (a3) {$1$};
      \node[yshift=-.6cm, xshift=.1cm] at (a3) {$2$};
     \node[yshift=-.6cm, xshift=-.1cm] at (a2) {$2$};
      \node[yshift=-.6cm, xshift=.1cm] at (a2) {$3$};
     \node[yshift=-.6cm, xshift=-.1cm] at (a1) {$3$};

     \node[yshift=.6cm, xshift=.1cm] at (b5) {$4$};
     \node[yshift=.6cm, xshift=-.1cm] at (b3) {$4$};
      \node[yshift=.6cm, xshift=.1cm] at (b3) {$5$};
     \node[yshift=.6cm, xshift=-.1cm] at (b2) {$5$};
      \node[yshift=.6cm, xshift=.1cm] at (b2) {$6$};
     \node[yshift=.6cm, xshift=-.1cm] at (b1) {$6$};
      
   \node[yshift=.5cm, xshift=.1cm] at (a5) {$0$};    
   \node[yshift=.5cm, xshift=-.1cm] at (a1) {$0$};    

\coordinate  (x1)  at  (0,0);
\fill (0,-1) coordinate (x2) circle (2pt); 
\draw (x1) -- ++(.5,0) coordinate (r1);
\draw (x1) -- ++(0,.5) coordinate (r2);
\draw (x1) -- ++(-.5,0) coordinate (r3);
\node[right] at (r1) {$1$};
\node[above] at (r2) {$1$};
\node[left] at (r3) {$2$};

\draw[] (x1) -- (x2) coordinate[pos=.3](R1)coordinate[pos=.7](R2);
\node[left] at (R2) {$4$};\node[left] at (R1) {$-4$};

\draw (x2) -- ++(.5,0) coordinate (r4);
\draw (x2) -- ++(0,-.5) coordinate (r5);
\draw (x2) -- ++(-.5,0) coordinate (r6);
\node[right] at (r4) {$-1$};
\node[below] at (r5) {$-1$};
\node[left] at (r6) {$-2$};

 \fill[white] (x1)  circle (2pt);
     \draw (x1)  circle (2pt);
\end{scope}

   \begin{scope}[xshift=3cm]
\fill[fill=black!10] (-4,-1.9)  -- (-2,-1.9)-- (-2,-1)-- (-1.8,0)--(-2.8,0) -- (-3,-1) -- (-3,.9) -- (-4,.9) -- cycle;

   \foreach \i in {1,2,...,5}
  \coordinate (a\i) at (-1.5-\i/2,-1); 
       \foreach \i in {1,2,...,5}
   \fill (a\i)  circle (2pt);

     \foreach \i in {1,2,4,5}
     \draw (a\i) -- ++(0,-1);
     \draw (a1) -- ++(80:1);
     \draw (a3) -- ++(80:1);
     
   \foreach \i in {3,4,5}
  \coordinate (b\i) at (-1.5-\i/2,0); 

        \foreach \i in {3,4,5}
     \draw (b\i) -- ++(0,1);
     
     \draw (a3) -- (b3);
     \draw (a5) -- (b5);

          \foreach \i in {3,4,5}
   \fill[white] (b\i)  circle (2pt);
          \foreach \i in {3,4,5}
     \draw (b\i)  circle (2pt);  
     
        \node[yshift=-.6cm, xshift=.1cm] at (a5) {$1$};
     \node[yshift=-.6cm, xshift=-.1cm] at (a4) {$1$};
      \node[yshift=-.6cm, xshift=.1cm] at (a4) {$2$};
     \node[yshift=-.6cm, xshift=-.1cm] at (a2) {$2$};
      \node[yshift=-.6cm, xshift=.1cm] at (a2) {$3$};
     \node[yshift=-.6cm, xshift=-.1cm] at (a1) {$3$};

     \node[yshift=.6cm, xshift=.1cm] at (b5) {$4$};
     \node[yshift=.6cm, xshift=-.1cm] at (b4) {$4$};
      \node[yshift=.6cm, xshift=.1cm] at (b4) {$5$};
     \node[yshift=.6cm, xshift=-.1cm] at (b3) {$5$};
      \node[yshift=.6cm, xshift=.25cm] at (a3) {$6$};
     \node[yshift=.6cm, xshift=0cm] at (a1) {$6$};
      
   \node[yshift=.5cm, xshift=.1cm] at (a5) {$0$};    
   \node[yshift=.5cm, xshift=-.1cm] at (a3) {$0$};     
     
\coordinate  (x1)  at  (0,0);
\fill (0,-1) coordinate (x2) circle (2pt); 
\draw (x1) -- ++(45:.5) coordinate (r1);
\draw (x1) -- ++(135:.5) coordinate (r2);
\node[above] at (r1) {$1$};
\node[above] at (r2) {$1$};
 
\draw[] (x1) -- (x2)  coordinate[pos=.3](R1)coordinate[pos=.7](R2);
\node[left] at (R2) {$2$};\node[left] at (R1) {$-2$};

\draw (x2) -- ++(0:.5) coordinate (r3);
\draw (x2) -- ++(-60:.5) coordinate (r4);
\draw (x2) -- ++(-120:.5) coordinate (r5);
\draw (x2) -- ++(-180:.5) coordinate (r6);
\node[right] at (r3) {$2$};
\node[below] at (r4) {$-1$};
\node[below left] at (r5) {$-2$};
\node[left] at (r6) {$-1$};

 \fill[white] (x1)  circle (2pt);
     \draw (x1)  circle (2pt);
     
\end{scope}
\end{tikzpicture}
\caption{Différentielles dans $\omoduli[0](2,2;(-1^6))$ et $\omoduli[0](1,3;(-1^6))$ avec résidus $(2,1,1,-1,-1,-2)$ et leurs graphes déliant.} \label{fig:graphdeliant}
\end{figure}

\end{ex}

\subsection{Composantes connexes}
\label{sec:casCC}
Les strates sont en général non connexes et leur composantes connexes ont été classifiées par Boissy dans les théorèmes 1.1 et 2.2 de \cite{Bo}.
Il est naturel de se demander si l'application résiduelle est surjective pour chaque composante connexe. 
Comme  les strates sont connexes en genre zéro, la question se pose pour $g\geq1$. Dans cette section nous prouvons la surjectivité de l'application résiduelle restreinte à chaque composante connexe.

La clé de la preuve est le fait que chaque composante connexe de la strate minimale peut s'obtenir à partir des strates minimales de genre zéro en cousant des anses (voir proposition 6.1 de \cite{Bo}). Cela nous permettra de restreindre notre étude au cas où tous les résidus sont nuls en genre un. C'est pourquoi nous nous contentons de rappeler la classification de Boissy dans ce cas.

En genre un, les composantes connexes des strates sont caractérisées par le nombre de rotation $\rot(S)$ des surfaces plates. Pour une surface plate $S$ définie par une différentielle méromorphe de $\omoduli[1](a_{1},\dots,a_{n};-b_{1},\dots,-b_{p})$ avec une base symplectique de lacets lisses de l'homologie $(\alpha,\beta)$ le {\em nombre de rotation} est $$\rot(S):=\pgcd(a_{1},\dots,a_{n};b_{1},\dots,b_{p},\ind(\gamma),\ind(\delta)).$$ On a le résultat suivant dû à Boissy.
\begin{itemize}
 \item[(i)] Si $n=p=1$, la strate est $\omoduli[1](a;-a)$ avec $a\geq2$ et chaque composante connexe correspond à un nombre de rotation qui est un diviseur strict de $a$.
 \item[(ii)] Sinon, il existe une composante connexe correspondant à chaque nombre de rotation qui est un diviseur de $\pgcd(a_{1},\dots,a_{n};b_{1},\dots,b_{p})$.
\end{itemize}

Nous montrons maintenant la surjectivité de l'application résiduelle restreinte à chaque composante connexe de genre supérieur ou égal à un. 
\begin{proof}[Preuve de la proposition~\ref{prop:CC}]
Il suffit de traiter le cas des strates minimales (avec un seul zéro). En effet, l'éclatement des zéros ne modifie pas les résidus aux pôles. De plus, cette opération permet d'atteindre toutes les composantes connexes à partir des strates minimales (voir la proposition 7.1 de \cite{Bo}). \`A partir de maintenant, nous considérons les strates minimales.
La proposition~6.1 de \cite{Bo} montre que chaque composante connexe d'une strate de genre $g \geq 1$ peut être obtenue par l'ajout d'une anse à une surface de genre $g-1$. 
Comme la couture d'anse est une opération locale, si l'application résiduelle est surjective en genre~$g$, elle l'est aussi  pour tous les genres supérieurs ou égaux à~$g$. Ainsi, il suffit de prouver la proposition pour les strates minimales de genre un.

En genre un, les strates ayant des pôles simples sont connexes. Nous supposerons donc que les strates ne paramètrent que des différentielles avec des pôles d'ordres inférieurs ou égaux à~$-2$. La proposition est trivialement vraie lorsqu'il y a un unique pôle, donc on considérera $p\geq2$ dans tout ce qui suit. Dans ce cas, le théorème~\ref{thm:geq0keq1} implique que la couture d'anse à partir des strates de genre zéro permet d'obtenir une différentielle dont les résidus sont $(r_{1},\dots,r_{p})\in \espres[0](\mu)\setminus\left\{(0,\dots,0)\right\}$. Dans la suite, nous construisons dans chaque composante connexe des strates $\omoduli[1](a;-b_{1},\dots,-b_{p})$ une différentielle dont tous les résidus sont nuls.

Nous considérons la construction suivante. Pour tous les pôles, nous prenons une partie polaire de type $b_{i}$ associée aux vecteurs $(1;1)$. On colle le bord supérieur de $P_{i}$ au bord inférieur de $P_{i+1}$.
Il reste deux liens selles homologues que l'on relie par un cylindre. La surface obtenue possède les invariants locaux souhaités.

Une base de l'homologie est donnée par une géodésique périodique $\alpha$ du cylindre (donc d'indice zéro) et le lacet $\beta$ suivant. Il coupe $\alpha$ dans le cylindre puis le lien selle au bord du domaine polaire $P_{p}$, puis tourne à gauche avant de ressortir de ce domaine polaire en coupant l'autre lien selle et ainsi de suite. 
Remarquons que changer le type $\tau$ d'une partie polaire change d'autant l'indice de $\beta$. On peut donc obtenir pour $\beta$ tous les indices  dans $J=\left[ p,-p+\sum_{i=1}^{p} b_{i}\right]$. \`A moins que la totalité des pôles ne soient d'ordre~$-2$, on obtient ainsi toutes les composantes connexes de la strate. En effet, si $\mu=(a;-3,-2\dots,-2)$, alors la strate est connexe. Il suffit donc de montrer que la longueur de l'intervalle $J$ est supérieur ou égale à  $\min_{i} b_{i}$, si $(b_{1},\dots,b_{p})\neq (2,\dots,2)$ et $(b_{1},\dots,b_{p})\neq (3,2,\dots,2)$. Cette inégalité est clairement satisfaite dès que $p\geq2$  et qu'il existe un $b_{i}\geq4$ ou deux $b_{j}\geq3$.
\smallskip
\par
Dans une strate minimale avec uniquement des pôles d'ordre~$-2$, il y a exactement deux composantes connexes. La construction qui précède ne permet d'obtenir que la composante dont le nombre de rotation a la même parité que $p$.
On propose alors une deuxième construction. On prend $p-1$ parties polaires et on les colle  comme précédemment. La dernière partie polaire est associée aux vecteurs $(i,1;1,i)$.
On identifie les vecteurs $i$ entre eux et les deux autres bords comme précédemment.
Le lacet $\beta$ est défini comme précédemment et  a pour indice $p$. Le lacet $\alpha$ connecte le milieu de des segments $v$ sans sortir du domaine polaire de~$P_{p}$. Son indice est donc~$1$. On construit une différentielle avec n'importe quels résidus dans la composante connexe dont le nombre de rotation est $1$.

Il reste le cas des composantes connexes de $\omoduli[1](2p;(-2^{p}))$ avec $p$ impair pour lesquelles le nombre de rotation est $2$. On reprend la construction précédente pour les $p-2$ premiers pôles. On associe au pôle $P_{p-1}$ partie polaire est associée aux vecteurs $(i,1;1,i)$. On associe à~$P_{p}$ la partie polaire associée à $(i;i)$. On identifie les bords comme précédemment.  Ainsi, les lacets analogues à ceux de la construction précédente auront  pour indices respectifs~$p-1$ et~$2$. Le nombre de rotation de la surface est donc~$2$.
\end{proof}

\section{Pluridifférentielles en genre zéro}
\label{sec:k-diff}

Dans cette section, nous considérons les $k$-différentielles de genre zéro pour $k\geq2$. Cette section est organisée de la façon suivante.
Le cas des strates de $k$-différentielles ayant un pôle d'ordre non divisible par $k$ est traité dans la section~\ref{sec:avecnondiv}. La section~\ref{sec:pasdenondiv} traite des strates ayant uniquement des pôles d'ordres divisibles par $k$ strictement inférieurs à~$-k$. La section section~\ref{sec:pasdenondivaveck} traite du cas des strates dont tous les pôles sont d'ordres divisibles par $k$, certains strictement inférieur à $-k$ et d'autres égaux à~$-k$. Enfin, la section~\ref{sec:juste-k} traite des strates qui n'ont que des pôles d'ordre $-k$.

\subsection{Différentielles ayant un pôle d'ordre inférieur non divisible par~$k$}
\label{sec:avecnondiv}

Rappelons que la strate $\komoduli[0](\mu)$ paramètre les $k$-différentielles primitives de type~$\mu$. En genre zéro, beaucoup de ces strates sont vides comme le montre le résultat suivant.
\begin{lem}\label{lem:puissk}
Soient $\mu=(m_{1},\dots,m_{t})$ un $t$-uplet tel que $\sum m_{i}=-2k$ et $d=\pgcd(\mu,k)$. Toutes les $k$-différentielles de type $\mu$ sont la puissance $d$-ième d'une $k/d$-différentielle primitive de $\Omega^{k/d}\moduli[0](\mu/d)$.
\end{lem}

\begin{proof}
Une $k$-différentielle $\xi$ sur $\PP^{1}$ de type $\mu$ est donnée par la formule 
$$\xi=\prod_{i=1}^{t}(z-z_{i})^{m_{i}}(dz)^{k}=\left( \prod_{i=1}^{t}(z-z_{i})^{m_{i}/d}(dz)^{k/d}\right)^{d}. $$
\end{proof}

A partir de maintenant nous ne considérons que des strates non vides.
Dans le cas où il existe des pôles d'ordres non divisibles par $k$,  nous commençons par caractériser l'image de l'application $k$-résiduelle pour les strates ayant un unique zéro.
\begin{lem}\label{lem:g=0gen1}
 Soit $\komoduli[0](a;-b_{1},\dots,-b_{p};-c_{1},\dots,-c_{r};(-k^{s}))$ une strate de genre zéro telle que $r\neq0$. L'image de l'application résiduelle est
 \begin{itemize}
 \item[i)] $\espresk[0](\mu)$ si $r\geq2$ ou $s\geq1$,
 \item[ii)]  $\espresk[0](\mu)\setminus\left\{(0,\dots,0)\right\}$ si $r=1$ et $s=0$.
\end{itemize}  
\end{lem}

\begin{proof}
Nous commençons par donner la construction d'une $k$-différentielle dans la strate  $\komoduli[0](\mu)$ avec $\mu=(a;-b_{1},\dots,-b_{p};-c_{1},\dots,-c_{r};(-k^{s}))$ dont les $k$-résidus sont donnés par $(R_{1},\dots,R_{p+s})$ dans $\espresk[0](\mu)\setminus\left\{(0,\dots,0)\right\}$.
\par 
Pour les pôles $P_{p+i}$ d'ordres $-k$, nous prenons une $k$-partie polaire d'ordre~$k$ associée à une racine $k$-ième $r_{p+i}$ de $R_{p+i}$. Pour chaque pôle $P_{i}$ d'ordre $-b_{i}=-k\ell_{i}$ tel que $R_{i}\neq0$, nous prenons une $k$-partie polaire non triviale d'ordre $b_{i}$ associée à $(r_{i};\emptyset)$. Si $R_{i}$ n'est pas un imaginaire pur ou $k\neq2$  nous choisissons une racine $r_{i}$ avec une partie réelle positive. Dans le cas où $k=2$ et $R_{i}$ est un réel négatif, nous choisissons la racine $r_{i}$ de partie imaginaire positive. Pour les pôles d'ordre $-b_{i}$ tels que $r_{i}=0$ nous prenons une $k$-partie polaire triviale d'ordre $b_{i}$ associée à $(r_{j_{i}};r_{j_{i}})$ où $r_{j_{i}}$ est l'une des racines choisie précédemment.
\par
Maintenant, pour tous les pôles d'ordre $-c_{i}$ sauf un, disons $P_{1}$ d'ordre $-c_{1}$, nous prenons une $k$-partie polaire de type $c_{i}$ associée à $(1;\emptyset)$. Pour le dernier pôle $P_{1}$ d'ordre $-c_{1}$, nous prenons la $k$-partie polaire de type $c_{1}$, sans points d'intersection (voir le lemme~\ref{lem:noninter}), associée à $(\emptyset;(1^{r-1}),r_{1},\dots,r_{l})$ où $l$ est le nombre de résidus non nuls. 
\par
La surface est obtenue par les recollements suivant. Nous collons le segment inférieur~$r_{j_{i}}$ de chaque $k$-partie polaire triviale au bord de la partie non triviale du pôle $P_{j_{i}}$. Nous faisons les collages similaires, pour chacun des pôles d'ordre divisible par $k$ dont le résidu est nul. Ensuite nous collons par translation les bords des $k$-parties polaires différentes de $P_{1}$ aux segments correspondants du bord de la $k$-partie polaire de $P_{1}$. Cette construction est illustrée par la figure~\ref{ex:8,8,12res}.
La $k$-différentielle $\xi$ associée à cette surface plate possède les ordres de pôles et les $k$-résidus souhaités.

\begin{figure}[htb]
\begin{tikzpicture}


\begin{scope}[xshift=-1.5cm]
      
 \fill[black!10] (-1,0)coordinate (a) -- (1.5,0)-- (a)+(1.5,0) arc (0:120:1.5)--(a)+(120:1.5) -- cycle;

   \draw (a)  -- node [below] {$2$} (0,0) coordinate (b);
 \draw (0,0) -- (1,0) coordinate[pos=.5] (c);
 \draw[dotted] (1,0) --coordinate (p1) (1.5,0);
 \fill (a)  circle (2pt);
\fill[] (b) circle (2pt);
\node[below] at (c) {$a$};

 \draw (a) -- node [above,rotate=120] {$a$} +(120:1) coordinate (d);
 \draw[dotted] (d) -- coordinate (p2) +(120:.5);

     \end{scope}
     
\begin{scope}[xshift=1.5cm,yshift=.5cm]
\fill[fill=black!10] (0.5,0)coordinate (Q)  circle (1.1cm);
    \coordinate (a) at (0,0);
    \coordinate (b) at (1,0);

     \fill (a)  circle (2pt);
\fill[] (b) circle (2pt);
    \fill[white] (a) -- (b) -- ++(0,-1.1) --++(-1,0) -- cycle;
 \draw  (a) -- (b);
 \draw (a) -- ++(0,-1);
 \draw (b) -- ++(0,-1);

\node[above] at (Q) {$1$};
    \end{scope}

\begin{scope}[xshift=6.5cm,yshift=1.5cm,rotate=180]

 \fill[black!10] (-1,0)coordinate (a) -- (1,0)-- (1,0) arc (0:120:2) -- cycle;
\draw (-1,0) coordinate (a) -- node [below,xshift=-1] {$2$} (0,0) coordinate (b);
\draw (b) -- node [below] {$1$} +(1,0) coordinate (c);
\draw (c) -- node [below] {$b$} +(1,0) coordinate (d);
\draw[dotted] (d) -- +(.5,0);
\fill (a)  circle (2pt);
\fill[] (b) circle (2pt);
\fill[] (c) circle (2pt);

\draw (a) -- node [above,rotate=120] {$b$} +(120:1) coordinate (e);
\draw[dotted] (e) -- +(120:.5);

\end{scope}

\end{tikzpicture}
\caption{Une $3$-différentielle de $\Omega^{3}\moduli[0](8;-4,-4;-6)$ avec un résidu non nul au pôle d'ordre $-6$.} \label{ex:8,8,12res}
\end{figure}

Il reste donc à montrer que le genre de la surface est zéro et que la différentielle $\xi$ possède un unique zéro;
Pour cela, il suffit de vérifier que si l'on coupe la surface le long d'un lien selle, alors on sépare cette surface en deux parties. C'est une conséquence du fait que les liens selles correspondent aux bords des domaines polaires. 
\smallskip
\par
Pour terminer la preuve du lemme~\ref{lem:g=0gen1}, il reste à considérer le cas où les $k$-résidus sont égaux à $(0,\dots,0)$. Remarquons que s'il existe un pôle d'ordre $-k$, alors $(0,\dots,0)$ n'est pas dans l'espace résiduel de la strate. Nous supposerons donc que $s=0$ dans la suite de la preuve. Nous montrons maintenant que l'origine appartient à l'image de l'application $k$-résiduelle d'une strate avec $r\geq1$ si et seulement si $r\geq2$.
\par
Montrons que l'origine appartient à l'image de l'application résiduelle s'il existe au moins deux pôles $P_{1}$ et $P_{2}$ d'ordres respectifs $-c_{1}$ et $-c_{2}$ non divisible par $k$. Pour tous les autres pôles nous associons la même $k$-partie polaire que précédemment. Plus précisément, pour chaque pôle d'ordre $-b_{i}$ divisible par $k$ nous prenons une $k$-partie polaire d'ordre $b_{i}$ associée à $(1;1)$. Pour les pôles d'ordres non divisibles par $k$ distincts de $P_{1}$, nous prenons la $k$-partie polaire associée à $(1;\emptyset)$. Pour $P_{1}$ nous prenons la $k$-partie polaire d'ordre $c_{1}$ associée à $(\emptyset;(1^{r}))$.
\par
Les collages sont les suivants. Nous collons le bord inférieur de la $k$-partie polaire du $i$-ième pôle d'ordre divisible par $k$ au bord d'en haut du $(i+1)$-ième pôle d'ordre divisible par $k$. Le bord inférieur de la $k$-partie polaire associée au pôle $P_{p}$ est collé au bord du segment du pôle~$P_{2}$. Enfin, tous les segments restant sont collés au bord de la $k$-partie polaire associée à~$P_{1}$. Cette construction est illustrée par la figure~\ref{ex:8,8,12}. On vérifie facilement que cette surface possède les propriétés souhaitées.
\begin{figure}[htb]
\begin{tikzpicture}

\begin{scope}[xshift=-1.5cm,yshift=.15cm]
      
 \fill[black!10] (-1,0)coordinate (a) -- (1.5,0)-- (a)+(1.5,0) arc (0:120:1.5)--(a)+(120:1.5) -- cycle;

   \draw (a)  -- node [below] {$1$} (0,0) coordinate (b);
 \draw (0,0) -- (1,0) coordinate[pos=.5] (c);
 \draw[dotted] (1,0) --coordinate (p1) (1.5,0);
 \fill (a)  circle (2pt);
\fill[] (b) circle (2pt);
\node[below] at (c) {$a$};

 \draw (a) -- node [above,rotate=120] {$a$} +(120:1) coordinate (d);
 \draw[dotted] (d) -- coordinate (p2) +(120:.5);

     \end{scope}

\begin{scope}[xshift=1.5cm,yshift=.8cm]
\fill[fill=black!10] (0.5,0)  circle (1.1cm);

 \draw (0,0) coordinate (a) -- coordinate (c) (1,0) coordinate (b);

 \fill (a)  circle (2pt);
\fill[] (b) circle (2pt);
\node[below] at (c) {$1$};
\node[above] at (c) {$2$};

    \end{scope}

\begin{scope}[xshift=6cm,rotate=180,yshift=-1.3cm]
      
 \fill[black!10] (-1.5,0)coordinate (a) -- (0,0)-- (60:1.5) arc (60:180:1.5) -- cycle;

   \draw (-1,0) coordinate (a) -- node [above] {$2$} (0,0) coordinate (b);
 \draw (a) -- +(-1,0) coordinate[pos=.5] (c) coordinate (e);
 \draw[dotted] (e) -- +(-.5,0);
 \fill (a)  circle (2pt);
\fill[] (b) circle (2pt);
\node[below] at (c) {$b$};

 \draw (b) -- node [above,rotate=-120] {$b$} +(60:1) coordinate (d);
 \draw[dotted] (d) -- +(60:.5);
\end{scope}

\end{tikzpicture}
\caption{Une $3$-différentielle de $\Omega^{3}\moduli[0](8;-4,-4;-6)$ avec un résidu nul.} \label{ex:8,8,12}
\end{figure}

Nous montrons enfin que s'il n'existe qu'un seul pôle d'ordre $-c$ non divisible par $k$, alors l'origine n'appartient pas à l'image de l'application résiduelle. Soit  $\xi_{0}$ une $k$-différentielle de $\komoduli[0](a;-b_{1},\dots,-b_{p};-c)$. Nous considérons le revêtement canonique $(\whX_{0},\wh\omega_{0},\pi)$ de~$\xi_{0}$. On déduit facilement du fait que $\pi$ soit ramifié au dessus d'exactement deux points (le zéro d'ordre $a$ et le pôle d'ordre $-c$) que $\whX_{0}$ est isomorphe à~$\PP^{1}$. Plus précisément, la proposition 2.4 de \cite{BCGGM3} implique que $\wh\omega_{0}$ appartient à la strate
$$\omoduli[0]\left( a+k-1;k-c-1,\left(\left(\frac{-b_{1}}{k}\right)^{k}\right),\dots,\left(\left(\frac{-b_{p}}{k}\right)^{k}\right) \right).$$
Le théorème~\ref{thm:geq0keq1} implique que les résidus de $\wh\omega_{0}$ ne sont pas tous nuls. Par conséquent, les $k$-résidus de $\xi_{0}$ ne sont pas tous nuls.
\end{proof}

Nous traitons maintenant le cas des strates ayant au moins~$2$ zéros~$a_{i}$.
\begin{lem}\label{lem:g=0gen2}
Considérons la partition $\mu=(a_{1},\dots,a_{n};-b_{1},\dots,-b_{p};-c_{1},\dots,-c_{r},(-k^{s}))$ avec $n\geq2$. Si $r\geq2$ ou $s\geq1$, alors l'application $k$-résiduelle $\appresk[0](\mu)$ est surjective. Dans le cas où $s=0$ et $r=1$ alors l'application contient le complémentaire de l'origine. De plus si au moins deux $a_{i}$ ne sont pas divisibles par $k$ ou si la somme des $a_{i}$ divisibles par $k$ est supérieure ou égale à~$kp$ alors l'application $k$-résiduelle contient l'origine.
\end{lem}

\begin{proof}
Si $r\geq2$ ou $r=1$ et $s\geq1$, on peut simplement éclater le zéro des différentielles données par le lemme~\ref{lem:g=0gen1}. On suppose maintenant $r=1$ et $s=0$. Par éclatement de zéros, il suffit de montrer que l'origine appartient à l'image de l'application résiduelle des strates $\komoduli[0](a_{1},a_{2};-b_{1},\dots,-b_{p};-c)$ dans le cas où $a_{i}$ n'est pas de la forme $kl_{i}$ avec $l_{i}<p$. Nous écrivons $a_{i}=kl_{i}+\bar{a_{i}}$ avec $-k<\bar{a_{1}},\bar{a_{2}}\leq 0$. 
Nous considérons deux cas selon que $l_{1},l_{2}\geq p$ ou il existe $l_{i}<p$.
\smallskip
\par
Dans le cas où $l_{1},l_{2}\geq p$, on associe aux pôles d'ordres $-b_{i}$ les $k$-parties polaires d'ordres~$b_{i}$ et de types~$\tau_{i}$ associées à $(1;1)$. De plus, on choisit les types $\tau_{i}$ de telle sorte que la somme $\sum_{i}\tau_{i}$ soit inférieure ou égale à $ l_{1}$ et maximale pour cette propriété. On note $\bar{l_{1}}=l_{1}-\sum_{i}\tau_{i}$. Pour le pôle d'ordre $-c$ on procède de la façon suivante.  On prend la $k$-partie polaire d'ordre~$c$ associée à $\left(\emptyset;1,\exp\left(i\pi+\bar{a_{2}}\tfrac{2i\pi}{k}\right)\right)$ et de type $\bar{l_{1}}+1$. Notons que cette construction est possible même si $\bar{a_{2}} = 0$.
On obtient la différentielle souhaitée en identifiant le bord inférieur de la $k$-partie polaire associée à $P_{i}$ au bord supérieur de celle de $P_{i+1}$. Le bord supérieur de $P_{1}$ est identifié au segment $1$ de la $k$-partie polaire d'ordre~$c$. Le bord inférieur de la $k$-partie polaire d'ordre $b_{p}$ est identifié par rotation au segment $\exp\left(\bar{a_{2}}\tfrac{2i\pi}{k}\right)$. Cette construction est illustrée à gauche de la figure~\ref{fig:rgeq1ngeq2}.
\par
Considérons le cas où il existe $l_{i}<p$ et $a_{i}\neq kl_{i}$. Nous supposerons sans perte de généralité qu'il s'agit de $l_{1}$. On associe au pôle d'ordre $-c$ la $k$-partie polaire d'ordre $c$ associée à $(\emptyset;1)$. Pour l'un des pôles d'ordre $-b_{i}$, disons $-b_{1}$, on associe la $k$-partie polaire d'ordre $b_{1}$ associée à $(1;v_{1},v_{2})$, avec $v_{i}$ de même longueur, $v_{1}+v_{2}=1$ et l'angle (dans la partie polaire) entre ces deux étant $2\pi + \tfrac{2\bar{a_{1}}\pi}{k}$. Comme $\bar{a_{1}} \neq 0$, cette partie polaire est non dégénérée. 
Pour $l_{1}$ pôles d'ordre $-b_{i}$, on associe des $k$-parties triviales associées à $(v_{1};v_{1})$ et de type $\tau_{i}=b_{i}-1$. On colle ces $k$-parties polaires de manière cyclique à $v_{1}$ et~$v_{2}$. Le point correspondant à l'intersection entre $v_{1}$ et $v_{2}$ est la singularité d'ordre~$a_{1}$.  On associe aux autres pôles d'ordre $-b_{i}$ la $k$-partie polaire triviale associée à $(1;1)$. On colle ces parties polaires de manière cyclique aux parties polaires d'ordres $b_{1}$ et $c$  pour obtenir la surface plate souhaitée.  Cette construction est illustrée à droite de la figure~\ref{fig:rgeq1ngeq2}.
\begin{figure}[htb]
\center
 \begin{tikzpicture}

\begin{scope}[xshift=-6cm]
    \fill[fill=black!10] (0,0) ellipse (1.5cm and .7cm);
\draw[] (0,0) coordinate (Q) -- (-1.5,0) coordinate[pos=.5](a);

\node[above] at (a) {$1$};
\node[below] at (a) {$2$};

\draw[] (Q) -- (.7,0) coordinate (P) coordinate[pos=.5](c);
\draw[] (Q) -- (P);

\fill (Q)  circle (2pt);

\fill[color=white!50!] (P) circle (2pt);
\draw[] (P) circle (2pt);
\node[above] at (c) {$v_{2}$};
\node[below] at (c) {$v_{1}$};

    \fill[fill=black!10] (0,-1.5) ellipse (1.1cm and .5cm);
\draw[] (.5,-1.5) coordinate (Q) --++ (-1.5,0) coordinate[pos=.5](b);

\fill (Q)  circle (2pt);

\node[above] at (b) {$2$};
\node[below] at (b) {$1$};

\end{scope}
\begin{scope}[xshift=-3cm]

 \fill[black!10] (-.5,0)coordinate (a) -- (a) + (80:1.5) coordinate (b)-- (b) arc (80:-40:1.5)--(a)+(-40:1.5) coordinate (c) -- cycle;
\draw[] (a) -- node [left,rotate=-10] {$5$}  (b);
\draw[] (a) -- node [right,rotate=-130] {$5$} (c);
\draw (a) -- ++(.7,0)coordinate[pos=.5](e) coordinate (d);
\draw (d) -- ++(.8,0)coordinate[pos=.5](f);
\fill (a) circle (2pt);
\fill[color=white] (d) circle (2pt);
\draw[] (d) circle (2pt);
\node[above] at (f) {$4$};
\node[below] at (f) {$3$};
\node[above] at (e) {$v_{1}$};
\node[below] at (e) {$v_{2}$};

\fill[fill=black!10] (0,-1.5) ellipse (1.1cm and .5cm);
\draw[] (-.5,-1.5) coordinate (Q) --++ (1.5,0) coordinate[pos=.5](b);
\filldraw[fill=white] (Q)  circle (2pt);

\node[above] at (b) {$4$};
\node[below] at (b) {$3$};
\end{scope}

\begin{scope}[xshift=2cm]
\begin{scope}[yshift=1cm,xshift=-.5cm]
    \fill[fill=black!10] (.5,-.3) ellipse (1.1cm and .9cm);
   
 \coordinate (A) at (0,0);
  \coordinate (B) at (1,0);
    \coordinate (C)  at (-30:.6);
 
 \filldraw[fill=white] (A)  circle (2pt);
\filldraw[fill=white] (B) circle (2pt);
    \fill (C)  circle (2pt);

   \fill[color=white]     (A) -- (B) -- (C) --cycle;
 \draw[]     (A) -- (B) coordinate[pos=.5](a);
 \draw (B) -- (C) coordinate[pos=.3](b);
 \draw (C) -- (A)  coordinate[pos=.6](c);

\node[above] at (a) {$1$};
\node[below] at (c) {$v_{1}$};
\node[below] at (b) {$v_{2}$};
\end{scope}

    \fill[fill=black!10] (0,-1.5) ellipse (1.1cm and .5cm);
\draw[] (-.5,-1.5) coordinate (Q) -- (.5,-1.5) coordinate[pos=.5](b) coordinate (P);

\filldraw[fill=white] (Q) circle (2pt);
\filldraw[fill=white] (P) circle (2pt);

\node[above] at (b) {$2$};
\node[below] at (b) {$1$};

\end{scope}
\begin{scope}[xshift=5cm]
\begin{scope}[rotate=180,yshift=-1.3cm]
      
 \fill[black!10] (-1.5,0)coordinate (a) -- (0,0)-- (60:1.5) arc (60:180:1.5) -- cycle;

   \draw (-1,0) coordinate (a) -- node [above] {$2$} (0,0) coordinate (b);
 \draw (a) -- +(-1,0) coordinate[pos=.5] (c) coordinate (e);
 \draw[dotted] (e) -- +(-.5,0);
  \draw (b) -- node [above,rotate=-120] {$b$} +(60:1) coordinate (d);
 \draw[dotted] (d) -- +(60:.5);
\filldraw[fill=white] (a) circle (2pt);
\filldraw[fill=white] (b) circle (2pt);
\node[below] at (c) {$b$};
\end{scope}

\begin{scope}[yshift=-1.3cm]
\fill[fill=black!10] (0,0) coordinate (O) ellipse (1.1cm and .6cm);
\coordinate (P) at (-30:.3);
\coordinate (Q) at (150:.3);
\draw[] (Q) -- (P)coordinate[pos=.5](b);
\filldraw[fill=white] (Q) circle (2pt);
\fill (P)  circle (2pt);

\node[below] at (b) {$v_{2}$};
\node[above] at (b) {$v_{1}$};
\end{scope}
\end{scope}
\end{tikzpicture}
\caption{Une $3$-différentielle de $\Omega^{3}\moduli[0](4,6;-9;-7)$ (à gauche) et de $\Omega^{3}\moduli[0](2,8;(-6^{3});-4)$ (à droite) dont tous les $3$-résidus sont nuls.} \label{fig:rgeq1ngeq2}
\end{figure}
\end{proof}

Pour terminer la preuve du théorème~\ref{thm:g=0gen1} il suffit de montrer que l'origine n'est pas dans l'image de l'application résiduelle dans les derniers cas.
\begin{lem}
Considérons la partition $\mu=(a_{1},kl_{2},\dots,kl_{n};-b_{1},\dots,-b_{p};-c)$ où $\sum_{i=1}^{n} l_{i}<p$. L'application résiduelle $\appresk[0](\mu)$ ne contient pas $\lbrace 0 \rbrace$.
\end{lem}

\begin{proof}
\'Etant donné une $k$-différentielle de la strate $\komoduli[0](\mu)$ dont tous les $k$-résidus sont nuls, le revêtement canonique définit une différentielle méromorphe $(\widehat X, \widehat\omega)$.  Par les résultats de la section~2.1 de \cite{BCGGM3}, le revêtement est uniquement ramifié au zéro d'ordre $a_{1}$ et au pôle d'ordre~$c$. Donc la surface $\widehat X$ est une sphère de Riemann. De plus la différentielle $\widehat \omega$ possède  $k$ zéros d'ordres $l_{i}$ pour $i\geq 2$, un zéro d'ordre $a_{1}+k-1$, de plus $k$ pôles d'ordres $-b_{j}/k$ pour tout $j\geq1$ et un pôle d'ordre $c-k+1$. Notons de plus que le résidu de chaque pôle est nul. D'après le Théorème~\ref{thm:geq0keq1}, cela est possible si et seulement si l'ordre de tous les zéros est inférieur ou égal à $$(c+1-k)+k\sum_{j=1}^{p} \frac{b_{j}}{k} - (kp+2)\,.$$ 
Cette condition implique que 
$$ a_{1} \leq c + \sum b_{i} -2k -kp \,.$$
Le fait que la différence entre la somme des ordres des zéros et la somme des ordres des pôles est égal à~$-2k$ implique le résultat. 
\end{proof}

\subsection{Les strates avec $r=s=0$ et $p\neq0$}
\label{sec:pasdenondiv}

Cette section est dédiée aux strates  de la forme $\komoduli[0](a_{1},\dots,a_{n};-b_{1},\dots,-b_{p})$. Insistons sur le fait que dans toute cette section, nous ne considérons que les strates non vide. D'après le lemme~\ref{lem:puissk}, cela revient à supposer que $\pgcd((a_{i}),k)=1$.
Enfin nous dénotons tout au long de cette section $b_{i}:=k\ell_{i}$ et $a_{i}:=kl_{i}+\bar{a_{i}}$ avec $-k<\bar{a_{i}}\leq 0$. 

Nous commençons par le cas le plus simple et le plus intéressant de strates n'ayant qu'un pôle.

\begin{lem}\label{lem:unpolesdivk}
Soit $\komoduli[0](a_{1},\ldots,a_{n};-k\ell)$ une strate non vide de genre $0$. L'image de l'application résiduelle de cette strate est $\CC$ si $n\geq 3$ et $\CC^{\ast}$ si $n=2$. 
\end{lem}

\begin{proof}
Considérons tout d'abord les strates $\komoduli[0](\mu)$ où $\mu=(a_{1},a_{2};-k\ell)$ à deux zéros d'ordres premiers avec $k$. Il suffit de montrer que l'image de l'application $k$-résiduelle $\appresk[0](\mu)$ ne contient pas $0$. 

Soit $(X,\xi)$ une $k$-différentielle dans $\komoduli[0](a_{1},a_{2};-k\ell)$. Il existe un lien selle $\gamma$ entre les deux zéros de $\xi$. Sur $\PP^{1}\setminus\left\{\gamma\right\}$, la différentielle $\xi^{1/k}$ est une différentielle abélienne. Le bord de cette surface plate est constitué de deux segments $\gamma_{i}$ qui sont identifiés l'un à l'autre pour former le lien selle $\gamma$ de $\xi$. Comme $\xi$ est primitive, le collage fait intervenir une rotation d'angle $\frac{2l\pi}{k}$, avec $\pgcd(l,k)=1$. Le $k$-résidu au pôle est la puissance $k$-ième de la somme des périodes des chemins $\gamma_{1}$ et $\gamma_{2}$, qui est donc non nulle. 
\smallskip
\par
Nous considérons maintenant les strates  $\komoduli[0](a_{1},\ldots,a_{n};-k\ell)$ possédant $n\geq3$ zéros. Nous commençons par montrer le point crucial: l'image de l'application résiduelle $\appresk[0](a_{1},a_{2},a_{3};-k\ell)$ contient~$0$. Comme $\pgcd(a_{i},k)=1$, nous pouvons supposer que $a_{1}$ et~$a_{2}$ ne sont pas divisibles par~$k$.
\par
Nous commençons par construire une $k$-différentielle dans la strate $\komoduli[0](\bar{a_{1}},\bar{a_{2}},a_{3}';-2k)$, où $a_{3}'$ est donné par la condition $\bar{a_{1}}+\bar{a_{2}}+a_{3}'=0$. Notons qu'il existe un $l_{3}'$ tel que $a_{3}=l_{3}'k+a_{3}'$. On se donne la $k$-partie polaire triviale d'ordre $2k$ associée à $(v_{1},v_{2};w_{1},w_{2})$ où les $v_{i},w_{j}$ sont définis comme suit. On a $v_{1}+v_{2}=w_{1}+w_{2}=1$ et l'angle formé au point de concaténation de $v_{1}$ et $v_{2}$ (resp. $w_{1}$ et $w_{2}$) est $(k+\bar{a_{1}})\frac{2\pi}{k}$ (resp.  $(k+\bar{a_{2}})\frac{2\pi}{k}$). Dans la phrase précédente, il va sans dire que l'angle est calculé dans les $k$-parties polaires $D^{+}(v_{1},v_{2})$ et $D^{-}(w_{1},w_{2})$. Nous obtenons une surface plate $S_{0}$ en collant (par rotation) $v_{1}$ avec $v_{2}$ et~$w_{1}$ avec $w_{2}$. Cette construction (ainsi que la suivante) est illustrée par la figure~\ref{fig:kdiffgzerounpol}. On peut vérifier que les choix fait dans la construction impliquent que la différentielle associée possède les invariants locaux souhaités.

\begin{figure}[hbt]
\begin{tikzpicture}
\begin{scope}[xshift=-6cm,yshift=.5cm]

     \foreach \i in {1,2,...,4}
  \coordinate (a\i) at (\i,0); 
  \coordinate (b) at (2.5,1.71/2);
   \fill[black!10] (a1) -- (a4) -- ++(0,0) arc (0:180:1.5) -- cycle;
      \fill[white] (b)  circle (2pt);
      \draw (b) circle (2pt);
 \fill[white] (a2) -- (a3) -- (b) -- cycle;
       \foreach \i in {2,3}
   \fill (a\i)  circle (2pt);
   \draw (a1) -- (a2) coordinate[pos=.5](e1) -- (b) coordinate[pos=.5](e2) -- (a3) coordinate[pos=.5](e3) -- (a4) coordinate[pos=.5](e4);

     \foreach \i in {1,2,...,4}
  \coordinate (c\i) at (\i,-1); 
  \coordinate (d) at (2.5,-1/2);
   \fill[black!10] (c1) -- (c2) -- (d) -- (c3) -- (c4) -- ++(0,0) arc (0:-180:1.5) -- cycle;
     \foreach \i in {2,3}
   \fill (c\i)  circle (2pt);
   \draw (c1) -- (c2) coordinate[pos=.5](f1) -- (d) coordinate[pos=.5](f2) -- (c3) coordinate[pos=.5](f3) -- (c4) coordinate[pos=.5](f4);
     \fill[red] (d)  circle (2pt);
      \draw (d) circle (2pt);

\node[above] at (e1) {$1$};
\node[below] at (f1) {$1$};
\node[above] at (e4) {$2$};
\node[below] at (f4) {$3$};

\draw (d) -- ++(0,-2)coordinate[pos=.5](e);
\node[right] at (e) {$5$};
\node[left] at (e) {$4$};

\node[above,rotate=60] at (e2) {$6$};
\node[below,rotate=120] at (e3) {$6$};
\node[above,rotate=30] at (f2) {$7$};
\node[below,rotate=150] at (f3) {$7$};
\end{scope}

\begin{scope}[xshift=0cm]
\fill[fill=black!10] (0,0) coordinate (Q) circle (1cm);

\draw[] (Q) -- (0,-1) coordinate[pos=.5](a);
     \fill[red] (Q)  circle (2pt);
      \draw (Q) circle (2pt);
\node[right] at (a) {$4$};
\node[left] at (a) {$5$};

\end{scope}

\begin{scope}[xshift=3cm]
\fill[fill=black!10] (0,0) coordinate (Q) circle (1cm);

\draw[] (0,0) -- (1,0) coordinate[pos=.5](a);

\node[above] at (a) {$3$};
\node[below] at (a) {$2$};
\fill (Q) circle (2pt);

\end{scope}

\end{tikzpicture}
\caption{Pluridifférentielle de $\Omega^{6}\moduli[0](-1,2,5;-18)$ dont le $6$-résidu est nul.} \label{fig:kdiffgzerounpol}
\end{figure}
 
\`A présent, nous construisons une différentielle dans $\komoduli[0](a_{1},a_{2},a_{3};-k\ell)$ dont le $k$-résidu est nul. Nous partons de l'union disjointe de la surface plate $S_{0}$ de $\komoduli[0](\bar{a_{1}},\bar{a_{2}},a_{3}',-2k)$ construite au paragraphe précédent et de $l_{1}+l_{2}+l_{3}'$ plans. Dans chacun de ces plans, nous faisons une demi-fente respectivement vers le haut pour $l_{1}$ plans, vers le bas pour $l_{2}$ plans et vers la droite pour les $l'_{3}$ plans restants. À chaque singularité conique de $S_{0}$ nous faisons une demi-fente dans la direction correspondante. La surface plate $S$ est obtenue en collant cycliquement les bords des fentes par translation. Il n'est pas difficile de vérifier que cette surface plate possède les propriétés souhaitées.

Enfin, pour $n>3$ zéros, nous procédons par récurrence sur $n$ en utilisant l'éclatement des zéros (Propriété~\ref{prop:eclatZero}). Même si, dans certains cas, il n'est pas possible d'obtenir une $k$-différentielle en éclatant un zéro d'une $k$-différentielle primitive, nous montrerons que dans le cas que nous traitons, toutes les $k$-différentielles peuvent s'obtenir en éclatant un zéro d'une $k/d$-différentielle primitive avec $d<k$.
Nous venons  de montrer que l'application résiduelle est $\CC^{\ast}$ pour $n=2$ et contient $0$ pour $n=3$. Il suffit donc de montrer qu'il existe $a_{i}$ et $a_{j}$ tels que $a_{i}+a_{j}>-k$ et $d:=\pgcd(\mu\setminus\left\{a_{i},a_{j}\right\}\cup\left\{a_{i}+a_{j}\right\},k)\neq k$. En effet, dans ce cas, on peut éclater le zéro d'ordre $a_{i}+a_{j}$ des $k$-différentielles qui sont la puissance $d$-ième d'une $(k/d)$-différentielle de la strate $\Omega^{k/d}\moduli[0]\left(\frac{\mu}{d}\setminus\left\{\frac{a_{i}}{d},\frac{a_{j}}{d}\right\}\cup\left\{\frac{a_{i}+a_{j}}{d}\right\}\right).$ La multiplicativité du $k$-résidu (cf. équation~\eqref{eq:multiplires}) permet alors de conclure.

Montrons que l'on peut toujours trouver deux ordres $a_{i}$ et $a_{j}$ comme ci-dessus.
Supposons tout d'abord qu'il existe  $a_{i},a_{j}$ tels que $a_{i}+a_{j}<-k$. En particulier,  $-k<a_{i},a_{j}<0$ et donc l'addition de $a_{i}$ à n'importe lequel des ordres positifs vérifie les conditions ci-dessus. Maintenant nous supposons que la somme de toutes les paires d'ordres $a_{i},a_{j}$ sont strictement supérieurs à $-k$. Dans le cas où $d=k$ pour un choix $a_{i},a_{j}$, alors pour tout $l\neq i,j$ les~$a_{l}$ sont divisibles par $k$. Donc $a_{i}$ et $a_{j}$ sont premiers avec $k$. En particulier, on peut choisir~$a_{i}$ et~$a_{l}$ pour $l\neq j$ pour obtenir la condition ci-dessus.
\end{proof}

\smallskip
\par
Nous traitons maintenant le cas de $p\geq2$ pôles.  Nous commençons par traiter le cas du complémentaire de l'origine. Le résultat suivant donne le résultat pour $k\geq3$ et couvre la majeure partie du cas $k=2$.

\begin{lem}\label{lem:polesdivksanszero}
L'image de  $\appresk[0](a_{1},\ldots,a_{n};-k\ell_{1},\ldots,-k\ell_{p})$ avec $n\geq2$ et $p\geq1$ contient l'espace $\espresk[0](\mu)\setminus\lbrace(0,\dots,0)\rbrace$ si $k\geq3$. L'image de $\appresk[0][2](a_{1},\ldots,a_{n};-2\ell_{1},\ldots,-2\ell_{p})$  avec $n\geq2$ et $p\geq1$ contient le complémentaire des droites engendrées par les vecteurs qui ne contiennent que des~$0$ et des~$1$.
\end{lem}

Pour traiter les cas difficiles du lemme~\ref{lem:polesdivksanszero}, nous montrons deux résultats préliminaires. Le lecteur pourra les sauter et revenir les consulter lors de leur utilisation à la fin de la preuve du lemme~\ref{lem:polesdivksanszero}.
\begin{lem}\label{lem:nonintersecp}
 Soient $k\geq2$, $\zeta$ une racine $k$-ième primitive de l'unité,  $R_{1},R_{2}\in\CC$ non tous deux nuls et $R \in \CC^{\ast}$. Il existe des racines $k$-ième  $r_{i}$ et $r$ respectivement de $R_{i}$ et  $R$ et il existe un vecteur $v_{1} \in \CC^{\ast}$ satisfaisant $v_{1} +  r_{1}  - \zeta v_{1} - r_{2} =r$, tel que la $k$-partie polaire (généralisée) associée aux vecteurs $(v_{1},r_{1};r_{2},\zeta v_{1})$   existe.  


 
\end{lem}

\begin{proof}
  On choisit $r_{1}$ et $r_{2}$ avec une partie réelle positive et une racine $r$ de $R$ telle que $r_{1}-r_{2} \neq r$. Il existe alors un unique vecteur $v_{1}\in\CC^{\ast}$ tel que $v_{1} +  r_{1}  - \zeta v_{1} - r_{2} =r$. Considérons la concaténation de $v_{1}$ puis de~$r_{1}$. Nous considérons la droite polygonale en concaténant une demi-droite partant du point initial de~$v_{1}$ au vecteur~$v_{1}$, puis à $r_{1}$ et  enfin une autre demi-droite partant du point final. La première demi-droite va vers les réels négatifs, mais peut être non horizontale afin d'être en dessous de $v_{1}$. Nous  formons alors le demi-plan à bord polygonal au dessus de cette droite polygonale.   On fait la même construction pour le demi-plan inférieur à la concaténation de $r_{2}$ et $v_{2}$. En collant les deux demi-plans on obtient la $k$-partie polaire souhaitée.
\end{proof}
\begin{lem}\label{lem:nonintersecp3}
 Soit $k\geq2$, $\zeta$ une racine primitive $k$-ième de l'unité et $R\in \mathbb{S}^{1}$ différente de $1$ si $k=2$. Il existe $v_{1},v_{2}\in\CC^{\ast}$ et une racine $k$-ième $r$ de $R$ satisfaisant $v_{1}+v_{2}=1$ et $-v_{1} + \zeta v_{2} = -r$ tels que les $k$-parties polaires $(\emptyset;v_{1},v_{2})$ et $(-v_{1},\zeta v_{2};\emptyset)$ existent.
\end{lem}

\begin{proof}
 Les deux équations impliquent que $v_{2}=\tfrac{1-r}{1+\zeta}$ et $v_{1}=\tfrac{\zeta+r}{1+\zeta}$. Par hypothèse $\zeta$ est différente de $1$ et donc ces vecteurs sont bien définis. De plus, par hypothèse, on peut choisir~$r$ distinct de $1$ et de $\zeta$.  Il est alors clair que les $k$-parties polaires du lemme existent avec ces vecteurs.
\end{proof}

Nous voudrions faire une remarque importante sur les lemmes~\ref{lem:nonintersecp} et~\ref{lem:nonintersecp3}.
\begin{rem}\label{rem:importexist}
 Il est intéressant de noter que dans les notations du lemme~\ref{lem:nonintersecp}, si les résidus quadratiques $(R_{1},R_{2},R) $ sont proportionnels à $\left((x+1)^{2},x^{2},1 \right)$, alors il n'est pas possible que la somme des angles entre $r_{1}$ et respectivement $v_{1}$ et $\zeta v_{1}$ soit égal à~$2\pi$.  En effet, dans ce cas, la construction du lemme~\ref{lem:nonintersecp} ne peut fonctionner que pour $r$ égal à~$-1$. Cela implique que $v_{1}=-1$ et donc la somme des angles entre $r_{1}$ et respectivement $v_{1}$ et $-v_{1}$ est égal à~$4\pi$.
 \par
 Le lemme~\ref{lem:nonintersecp3} ne permettra pas de réaliser les invariants manquants dans le cas quadratique et les $2$-résidus spéciaux.  Cela est un reflet du fait qu'il existe des  cas quadratiques particuliers dans le théorème~\ref{thm:r=0s=0} et le lemme~\ref{lem:surjr=0sneq03}. 
\end{rem}

Nous passons maintenant à la preuve du lemme~\ref{lem:polesdivksanszero}.
 Nous voudrions remarquer que la construction donnée dans la preuve permet d'obtenir un résultat un peu meilleur dans le cas quadratique. Toutefois nous traiterons uniformément ces cas dans les lemmes suivants. 
\begin{proof}[Démonstration du lemme~\ref{lem:polesdivksanszero}]
Nous commençons par le cas des strates de la forme $\komoduli[0](\mu)$ avec $\mu=(a_{1},a_{2};-k\ell_{1},\dots,-k\ell_{p})$. 
 Nous supposerons que les pôles dont le $k$-résidu est nul sont $P_{1},\dots,P_{t}$ et que les $k$-résidus $R_{t+1},\dots,R_{p}$ aux pôles $P_{t+1},\dots,P_{p}$ sont non nuls. 
\par
Supposons tout d'abord que $-k<a_{1}<0$. Soit $R=(R_{1},\dots,R_{p})\neq (0,\dots,0)$, on construit une différentielle dont les $k$-résidus sont~$R$ de la manière suivante. Pour $i>t$, nous prenons une racine $k$-ième~$r_{i}$ de~$R_{i}$ telle que $\Re(r_{i})\geq0$ et si $\Re(r_{i})=0$, alors $\Im(r_{i})>0$. Nous définissons~$v_{1}$ et $v_{2}$ deux vecteurs de même longueur, tels que l'angle  entre $v_{1}$ et $v_{2}$ est  $2\pi +\frac{2a_{1}\pi}{k}$, et de somme égale à $S=\sum_{i\geq2}r_{i} -r_{1}$. Ces vecteurs sont non nuls si et seulement si $S\neq0$. Si $r_{1}=0$, cette condition est une conséquence de l'hypothèse sur les parties réelles et imaginaires des $r_{i}$. Si $r_{1}\neq0$, alors quitte à permuter les pôles ou à changer une des racines~$r_{i}$ par une autre, on montre que l'on peut satisfaire cette condition sauf dans le cas de la strate $\Omega^{2}\moduli[0](-1,b_{1}+b_{2}-3;-b_{1},-b_{2})$ avec les résidus quadratiques égaux à $(1,1)$ (qui n'est pas considéré dans ce lemme).
\par
Considérons le cas distinct de $\Omega^{2}\moduli[0](-1,b_{1}+b_{2}-3;-b_{1},-b_{2})$ avec les résidus quadratiques égaux à $(1,1)$. Prenons le pôle $P_{1}$ d'ordre~$-k\ell_{1}$ et de $k$-résidu~$R_{1}$. On associe alors à ce pôle la $k$-partie polaire triviale d'ordre $k\ell_{1}$ associée à $(v_{1},v_{2};r_{t+1},\dots,r_{p})$ si $R_{1}=0$ et la $k$-partie polaire non triviale associée à $(v_{1},v_{2};r_{2},\dots,r_{p})$ si $R_{1}\neq 0$. Notons que l'angle (calculé dans~$D^{+}$) au point d'intersection des $v_{i}$ est $2\pi +\frac{2a_{1}\pi}{k}$.  Pour les autres pôles $P_{i}$, on prend une $k$-partie polaire associée à $(r_{i};\emptyset)$ si $R_{i}\neq0$ et $(r_{j_{i}};r_{j_{i}})$ avec $j_{i}>t$ si $R_{i}=0$. Il reste à identifier tous les segments par translation à exception des $v_{i}$ que nous identifions par rotation. On vérifie de manière analogue à précédemment que cette surface plate possède les invariants locaux souhaités.
\smallskip
\par
Nous supposerons maintenant que $a_{1}$ et $a_{2}$ sont supérieurs ou égaux à~$0$. De plus, nous supposons que $l_{1}\leq l_{2}$ dans l'écriture $a_{i}=kl_{i}+\bar{a_{i}}$.
La construction dépend de l'existence, ou non, d'un entier $m\leq t$ tel que
\begin{equation}\label{eq:preuver=0s=0}
\ell_{m}^{-} :=\sum_{i=1}^{m-1}(\ell_{i}-1) < l_{1} \leq  \sum_{i=1}^{m}(\ell_{i}-1).
\end{equation}
\par
Supposons qu'il existe un $m$ satisfaisant l'équation~\eqref{eq:preuver=0s=0}.  On associe au pôle $P_{m}$ la $k$-partie polaire d'ordre $k(\ell_{m}-(l_{1}-\ell_{m}^{-})+1)$ et de type $\ell_{m}-(l_{1}-\ell_{m}^{-})$ associée à $(v_{1},v_{2};r_{t+1},\dots,r_{p})$ où les $v_{i}$ sont de même longueur, $v_{1}+v_{2}=\sum r_{i}$ et l'angle à leur intersection est $2\pi+\frac{2\bar{a_{1}}\pi}{k}$. Puis on coupe cette partie polaire par une demi-droite commençant au point d'intersection de $v_{1}$ et~$v_{2}$. Enfin on colle de manière cyclique $(l_{1}-\ell_{m}^{-}-1)$ domaines basiques triviaux à cette demi droite. En prenant les $r_{i}$ tels que $\Re(r_{i})\geq0$ et $\Im(r_{i})>0$ si $\Re(r_{i})=0$, la ligne brisée formée des $r_{i}$ ne possède pas de points d'auto-intersection. Ensuite, on associe aux $m-1$ premiers pôles la $k$-partie polaire d'ordre $b_{i}$ et de type $1$ associée à $(v_{1};v_{1})$. On colle les  bords de ces $k$-parties polaires entre elles par translation de manière cyclique. Les segments~$v_{1}$ et~$v_{2}$ de la $k$-partie polaire spéciale sont collés aux bords restant par rotation et translation. Pour les autres pôles $P_{i}$ d'indices $m<i\leq t$, on prend la $k$-partie polaire associée à $(r_{j_{i}};r_{j_{i}})$ avec $j_{i}>t$. Les identifications sont données comme précédemment. Cette construction est illustrée par la figure~\ref{fig:kdiffgzerounres}. Cela définit une $k$-différentielle avec les invariants locaux souhaités. 

\begin{figure}[hbt]
\center
\begin{tikzpicture}
\begin{scope}[xshift=-6cm]
    \fill[fill=black!10] (0,0) ellipse (1.5cm and .7cm);
\draw[] (0,0) coordinate (Q) -- (1.5,0) coordinate[pos=.5](a);

\node[above] at (a) {$1$};
\node[below] at (a) {$2$};

\draw[] (Q) -- (-1,0) coordinate (P) coordinate[pos=.5](c);
\draw[] (Q) -- (P);

\fill (Q)  circle (2pt);

\fill[color=white!50!] (P) circle (2pt);
\draw[] (P) circle (2pt);
\node[above] at (c) {$v_{2}$};
\node[below] at (c) {$v_{1}$};

    \fill[fill=black!10] (0,-1.5) ellipse (1.1cm and .5cm);
\draw[] (-.5,-1.5) coordinate (Q) --++ (1.5,0) coordinate[pos=.5](b);

\fill (Q)  circle (2pt);

\node[above] at (b) {$2$};
\node[below] at (b) {$1$};

\end{scope}

\begin{scope}[xshift=0cm]
\fill[fill=black!10] (-.5,0) ellipse (2cm and 1.2cm);
\draw[] (0,0) coordinate (Q) -- (1.5,0) coordinate[pos=.5](a);

\node[above] at (a) {$5$};
\node[below] at (a) {$6$};

\draw[] (Q) -- (-1,0) coordinate (P);
\draw[] (Q) -- (P);
\coordinate (R)  at (-120:1);
\draw[dashed] (R) -- ++(1.4,0);

\fill (Q)  circle (2pt);

\fill[color=white!50!] (P) circle (2pt);
\draw[] (P) circle (2pt);
\fill[color=white!50!] (R) circle (2pt);
\draw[] (R) circle (2pt);
\fill[color=white] (Q) --(P) --(R) --(Q);
 \draw[]     (P) -- (Q) coordinate[pos=.5](a);
 \draw (Q) -- (R) coordinate[pos=.5](b);
 \draw (R) -- (P)  coordinate[pos=.6](c);
 \node[above] at (a) {$v_{1}$};
\node[below,rotate=60] at (b) {$v_{2}$};
\node[above,rotate=120] at (c) {$r$};
 
\begin{scope}[xshift=1cm,yshift=.5cm]
    \fill[fill=black!10] (.6,-2) ellipse (.9cm and .5cm);

\draw[] (0,-2) coordinate (Q) --++ (1.5,0) coordinate[pos=.5](b);
\fill (Q)  circle (2pt);

\node[above] at (b) {$6$};
\node[below] at (b) {$5$};

\draw[] (P)  -- ++(-1.5,0) coordinate[pos=.5](a);
\fill[color=white!50!] (P) circle (2pt);
\draw[] (P) circle (2pt);

\node[above] at (a) {$3$};
\node[below] at (a) {$4$};
\end{scope}

\begin{scope}[xshift=-1cm,yshift=.5cm]
    \fill[fill=black!10] (-1.6,-2) ellipse (.9cm and .5cm);

\draw[] (-1,-2) coordinate (P) -- ++(-1.5,0) coordinate[pos=.5](b);

\node[above] at (b) {$4$};
\node[below] at (b) {$3$};

\fill[color=white!50!] (P) circle (2pt);
\draw[] (P) circle (2pt);
\end{scope}
\end{scope}

\begin{scope}[xshift=4cm,yshift=-.5cm]
\fill[fill=black!10] (0.5,0)coordinate (Q)  circle (1.2cm);
    \coordinate (a) at (0,0);
    \coordinate (b) at (1,0);

     \filldraw[fill=white] (a)  circle (2pt);
\filldraw[fill=white](b) circle (2pt);
    \fill[white] (a) -- (b) -- ++(0,1.2) --++(-1,0) -- cycle;
 \draw  (a) -- (b);
 \draw (a) -- ++(0,1.1);
 \draw (b) -- ++(0,1.1);

\node[below] at (Q) {$r$};
    \end{scope}

\end{tikzpicture}
\caption{Une $6$-différentielle dans $\Omega^{6}\moduli[0](17,25;-12,-18,-24)$ avec un unique $6$-résidu non nul au pôle d'ordre $-12$.} \label{fig:kdiffgzerounres}
\end{figure}
\par
Supposons qu'il n'existe pas d'entier $m$ satisfaisant à l'équation~\eqref{eq:preuver=0s=0}. Nous commençons par supposer qu'il existe deux pôles dont les $k$-résidus sont non nuls. Par simplifier, nous supposerons qu'il n'existe pas d'autres pôles et que ces deux pôles sont d'ordre~$-2k$. Il suffit de traiter les cas où $l_{i}\geq1$ pour $i\in\{1,2\}$. Si les deux $k$-résidus ne sont pas de même norme, alors le lemme~\ref{lem:nonintersecp} permet de construire les $k$-différentielles avec les invariants souhaités en considérant les $k$-parties polaires suivantes. Celle associée à $(v_{1},r_{2};\zeta v_{1})$ tels que la somme $v_{1}+r_{2}-\zeta v_{1}$ est égale à une racine du $k$-résidu $R_{1}$. Si les normes sont égale il suffit d'utiliser les $k$-parties polaires présentées dans le lemme~\ref{lem:nonintersecp3}. Comme cela, on obtient tout les invariants sauf dans le cas quadratique les zéros d'ordres $(1,3)$ si les résidus sont égaux à $(1,1)$. Si les deux pôles ne sont pas d'ordres $-2k$, on obtient tout les invariants en coupant ces $k$-différentielles le long de demi-droites partant des zéros et en ajoutant des plans de manière cyclique. Notons que cette méthode permet d'obtenir tout les invariants, même dans le cas quadratique. Enfin, s'il existe des pôles dont le $k$-résidu est nul, on coupe ces $k$-différentielles le long de liens selles et on ajoute les $k$-parties polaires correspondant à ces pôles.
\par
Maintenant, on considère le cas où au moins trois pôles ont un $k$-résidu non nul. On prend le pôle, disons $P_{p}$, d'ordre $-b_{p}$ minimal parmi les pôles ayant un $k$-résidu non nul. On partitionne les autres pôles avec un $k$-résidu non nul en deux parties $A_{i}$.  Ces deux parties sont telles que les deux sommes $S_{i}:=\sum_{P_{j}\in A_{i}}\ell_{j}$ soient respectivement inférieures ou égales à~$l_{i}$. Notons que cette condition peut toujours être satisfaite. En effet, l'égalité $l_{1}+l_{2}=1+\sum_{i}\ell_{i}$ et la minimalité de $-\ell_{p}$ impliquent que $\ell_{p-1}\leq (l_{1}+l_{2})/2$ et donc que $\ell_{p-1}\leq l_{2}$. On conclut par récurrence sur le nombre de pôles dont le résidu est non nul. On notera $A_{1}=\left\{P_{t+1},\dots,P_{p_{0}}\right\}$ et $A_{2}=\left\{P_{p_{0}+1},\dots,P_{p-1}\right\}$. Pour les pôles de $A_{i}$, prenons des racines $r_{j}$ des $R_{j}$ telles que $(-1)^{i}\Re(r_{j})\geq0$ et $(-1)^{i}\Im(r_{j})\geq0$ en cas d'égalité.
\par
Nous donnons maintenant la construction des $k$-différentielles, qui est illustrée par la figure~\ref{fig:kdiffgzerounres2} avec $l_{1}=3$, $l_{2}=4$ et $S_{1}=S_{2}=1$.
On associe au pôle $P_{p}$ la $k$-partie polaire d'ordre~$k\ell_{p}$ et de type $\tau$ associée à  associée à  $\left(v_{1},-\sum_{j=t+1}^{p_{0}} r_{j}; \sum_{j=p_{0}+1}^{p-1} r_{j},v_{2}\right)$  où les $v_{i}$ sont tels que $v_{1}-v_{2} = \sum_{j=t+1}^{p} r_{j}$ où $r_{p}$ est une racine de $R_{p}$   et $v_{2}=\exp(2\bar{a}_{2}i\pi/k)v_{1}$. Nous spécifierons par la suite les valeurs de $\tau$ et $r_{p}$.  De plus, sauf dans le cas où $k=2$ et tous les $2$-résidus sont positivement proportionnels, on peut choisir des racines des $R_{i}$ telles que la différence des deux sommes ne soient pas égale à la norme de~$r_{p}$.   Le lemme~\ref{lem:nonintersecp}  montre l'existence de cette  $k$-partie polaire sauf dans le cas $k=2$ et tous les  $2$-résidus sont positivement proportionnels.  Nous supposerons tout d'abord que nous ne sommes pas dans ce cas. 

%
 Le type $\tau$ est choisi tel que les inégalités $S_{1}+\tau\leq l_{1}$ et $S_{2}+(\ell_{p}-1-\tau)\leq l_{2}$ soient satisfaites.
Notons que s'il n'y a pas de pôle dont le résidu est nul, alors ces inégalités sont des égalités. Dans le cas général, ces inégalités peuvent être strictes.  On associe aux pôles de $A_{1}$ les parties polaires de type $b_{i}$ associées à $(\emptyset;-r_{i})$ et à ceux de $A_{2}$ celles de type $b_{j}$ associées à $(r_{j};\emptyset)$.
 Enfin on procède comme précédemment avec les pôles dont le $k$-résidu est nul. Notons que l'on peut coller ces pôles de façon a ajouter $2$ à l'ordre d'un des deux zéros ou $1$ à l'ordre de chaque zéro. On choisit de les coller de telle façon que l'on obtienne les zéros d'ordres désirés. Maintenant on considère des polygones donnés par la concaténation des vecteurs $\left(\sum_{j=t+1}^{p_{0}} r_{j},-r_{t+1},\dots,-r_{p_{0}} \right)$ et $\left( -\sum_{j=p_{0}+1}^{p-1} r_{j},r_{p_{0}+1},\dots,r_{p-1}\right)$. On les colle aux arrêtes correspondantes de la $k$-partie polaire associée à~$P_{p}$. Notons que la construction du lemme~\ref{lem:nonintersecp} est telle que la surface ainsi formée est non dégénérée (même si sa projection dans le plan peut contenir des intersections). On procède enfin aux collages des différentes $k$-parties polaires et des segments $v_{1}$ et~$v_{2}$.

 \begin{figure}[hbt]
\center
\begin{tikzpicture}

\begin{scope}[xshift=-7cm]
\fill[fill=black!10] (0,0) coordinate (O) ellipse (1.1cm and .6cm);
\coordinate (P) at (-.4,0);
\coordinate (Q) at (.4,0);
\draw[] (Q) -- (P)coordinate[pos=.5](b);
\filldraw[fill=white] (Q) circle (2pt);
\fill (P)  circle (2pt);

\node[below] at (b) {$v_{1}$};
\node[above] at (b) {$v_{2}$};
\end{scope}

\begin{scope}[xshift=4cm]
\fill[fill=black!10] (0,0)coordinate (Q)  ellipse (2cm and 1.2cm);

\draw[] (Q) -- (-.8,0) coordinate (P)coordinate[pos=.5](a);
\draw[] (Q) -- (60:1)  coordinate (R)coordinate[pos=.5](b);
\draw[] (P) -- ++(-70:1)  coordinate (S)coordinate[pos=.5](c);
\draw[] (S) -- ++(60:.8)  coordinate (T)coordinate[pos=.5](d);

\fill(P) circle (2pt);
\filldraw[fill=white] (Q) circle (2pt);
\filldraw[fill=white] (R) circle (2pt);
\fill (S) circle (2pt);
\filldraw[fill=white] (T) circle (2pt);
\fill[color=white] (T) --(S) --(P) --(Q) -- (R) -- ++(2,0) --++(0,-1.095) -- cycle;

\draw[] (Q) -- (P);
\draw[] (Q) -- (R);
\draw[] (P) -- (S);
\draw[] (S) -- (T);
 \node[above] at (a) {$v_{1}$};
\node[below,rotate=60] at (b) {$-r_{1}$};
\node[above,rotate=120] at (c) {$r_{2}$};
\node[above,rotate=-120] at (d) {$v_{2}$};

\draw(R) -- ++(1.4,0) coordinate[pos=.6](e);
\draw(T) -- ++(1.4,0) coordinate[pos=.6](f);
\draw[dashed](P) -- ++(-1.2,0);

\node[above] at (e) {$3$};
\node[below] at (f) {$3$};

\draw[dotted] (e) -- (f) coordinate[pos=.5] (g);
\node[right] at (g) {$r_{3}$};
\end{scope}

\begin{scope}[xshift=-4cm,yshift=-.5cm]
\fill[fill=black!10] (0.4,.6)coordinate (Q)  circle (1.2cm);
    \coordinate (a) at (0,0);
    \coordinate (b) at (60:1);
    \draw[dashed](b) -- ++(1,0);

     \filldraw[fill=white] (a)  circle (2pt);
\filldraw[fill=white](b) circle (2pt);
    \fill[white] (a) -- (b) -- ++(180:1.3)--++(-120:1)-- cycle;
 \draw  (a) -- (b);
 \draw (a) -- ++(180:.7) coordinate[pos=.5] (a);
 \draw (b) -- ++(180:1.2) coordinate[pos=.5] (b);

\node[below] at (Q) {$-r_{1}$};
\node[below] at (a) {$1$};
\node[above] at (b) {$1$};
    \end{scope}
    
\begin{scope}[xshift=0cm,yshift=.5cm]
\fill[fill=black!10] (0.1,-.5)coordinate (Q)  circle (1.2cm);
    \coordinate (a) at (0,0);
    \coordinate (b) at (-70:1);
    \draw[dashed](b) -- ++(1,0);
    
     \fill (a)  circle (2pt);
\fill (b) circle (2pt);
    \fill[white] (a) -- (b) -- ++(-180:1.4) --++(110:1) -- cycle;
 \draw  (a) -- (b);
 \draw (a) -- ++(-180:.9) coordinate[pos=.5] (a);
 \draw (b) -- ++(-180:1.1)coordinate[pos=.5] (b);

\node[right] at (Q) {$r_{2}$};
\node[above] at (a) {$2$};
\node[below] at (b) {$2$};

    \end{scope}

\end{tikzpicture}
\caption{Une $6$-différentielle dans $\Omega^{6}\moduli[0](17,19;(-12^{4}))$ dont les $6$-résidus sont $(0,r_{1}^{6},r_{2}^{6},r_{3}^{6})$.} \label{fig:kdiffgzerounres2}
\end{figure}
\par
 Nous traitons maintenant le cas où $k=2$ et les sommes $-\sum_{j=t+1}^{p_{0}} r_{j}$ et  $\sum_{j=p_{0}+1}^{p-1} r_{j}$ soient égalent à $x+1$ et $x$ avec $R_{p} = 1$. Si tous les $2$-résidus ne sont pas positivement proportionnels, nous pouvons changer le signe de l'un et utiliser la construction précédente. Nous supposerons donc que les $2$-résidus sont réels mais non identiques.  Supposons que les $R_{j}$ ne soient pas tous égaux pour $j \in \{ t+1 , \dots, p_{0}\}$. Dans ce cas, on suppose que $R_{t+1} > R_{t+2}$. Pour le pôle $p$, on utilise alors les vecteurs $- r_{t+1}+r_{t+2},-r_{t+3},\dots, r_{p_{0}}$. Pour les pôles $t+1$ et $t+2$ on considère les parties $2$-polaires respectivement associées à $(\emptyset;- r_{t+1}+r_{t+2},-r_{t+2} )$ et $(-r_{t+2};\emptyset )$. Cela permet d'utiliser le lemme~\ref{lem:nonintersecp} et de conclure en utilisant la construction du paragraphe précédent. Considérons le cas où tous les $r_{j}$ sont égaux entre eux dans chacune des deux sommes. Nous supposerons que ces résidus quadratiques sont tous égaux à~$1$. Nous présentons le cas de trois pôles, qui s'adapte directement à plus de pôles. Dans ce cas, les $2$-résidus sont de la forme $\left(n^{2},m^{2},(n+m)^{2}\right)$ avec $n,m\neq0$. Nous associons aux pôles les $2$-parties polaires associées respectivement à $(n;\emptyset)$, puis à $(n;n,m)$ et enfin à $(n,m;\emptyset)$. Nous collons le vecteur $n$ de la première partie $2$-polaire au $n$ inférieur de la seconde. Puis nous collons le $n$ supérieur et le $m$ de cette partie polaire aux $n$ et $m$ de la troisième respectivement par rotation et par translation. Cela permet d'utiliser le lemme~\ref{lem:nonintersecp} afin d'obtenir les invariants souhaités. 
\smallskip
\par
Enfin, le cas avec $n\geq3$ zéros se déduit du cas que nous venons de traiter avec deux zéros par éclatement de zéros. 
\end{proof}

Nous traitons les cas de des strates quadratiques qui ne sont pas couverts par le lemme~\ref{lem:polesdivksanszero}. Dans ce cas, il existe des strates dont l’application $2$-résiduelle ne contient pas le complémentaire de l'origine.

\begin{lem}\label{lm:realquadcolib}
L'application $2$-résiduelle de la strate $\Omega^{2}\moduli[0](a_{1},a_{2};-b_{1},\ldots,-b_{p})$ avec $a_{i}\geq-1$ impairs et $p\geq2$ nombres pairs $b_{i}\geq 4$ contient tous les uplets contenant un nombre impair d'éléments égaux à~$1$ et un nombre arbitraire égaux à~$0$.
\end{lem}

\begin{proof}
Dans un premier temps, on considère le cas où tous les résidus quadratique sont nuls à l'exception d'un seul que l'on supposera égal à~$4$. Nous allons donner une construction qui fournit une différentielle quadratique de la strate considérée avec les résidus quadratiques souhaités. 
\par
Nous recollons des domaines polaires standards (voir la section~\ref{sec:briques}) le long de liens selles selon un graphe d'incidence. Ce graphe est composé d'une boucle formée par $s$ sommets de valence $2$ et collée à un sommet de valence $3$ auquel est attaché une chaîne composé de $p-s-1$ sommets se terminant par un sommet de valence~$1$.  Le cas extrême $p=s$ correspond à un graphe cyclique, mais la construction reste la même.
Les $s$ sommets de la boucle correspondent à des domaines polaires associés aux vecteurs $(1;1)$, tandis que les $p-s-1$ sommets intermédiaires de la chaîne correspondent à des domaines polaires associés aux vecteurs $(2;2)$. Le sommet de valence $3$ correspond à un domaine polaire associé aux vecteurs $(1,1;2)$ et le dernier sommet de la chaîne correspond à un domaine polaire associé aux vecteurs $(2;\emptyset)$. Si $s=p$, un domaine polaire est associé aux vecteurs $(1,1;\emptyset)$ et les autres à $(1;1)$. La différentielle quadratique obtenue en recollant les parties polaires selon le graphe d'incidence possède un $2$-résidu égal à~$4$ et $p-1$ nuls.
\par
Nous montrons maintenant que la différentielle ainsi formée peut avoir les zéros d'ordres souhaités.
Cela se fait grâce au choix du type des domaines polaires (voir la section~\ref{sec:briques}). On peut remarquer que l'une des singularités coniques n'est adjacent qu'aux domaines polaires de la boucle tandis que l'autre est adjacent à tous les domaines polaires. On note $a_{1}$ l'ordre de la première singularité.
L'angle de cette singularité est $\pi+2s\pi$ auquel on ajoute des multiples de $2\pi$ correspondant aux types des domaines polaires qui contribuent à cette singularité. Sachant que l'on peut mettre n'importe quel sommet de résidu nul dans la boucle, l'angle total de cette singularité peut aller de $\pi$ (si $s=0$ et que le type du domaine correspondant au sommet de valence $3$ est choisi de façon adéquate) jusqu'à $-\pi+\pi\sum_{j=1}^{p}(b_{j} -2)$. L'ordre de $a_{1}$ peut donc aller de $-1$ jusqu'à $-3+\sum_{j=1}^{p}(b_{j} -2)$. Ceci couvre tous les ordres possibles et donc toutes les strates considérées.
\smallskip
\par
Nous traitons maintenant le cas général d'un nombre $t\geq 3$ impair de résidus quadratiques égaux entre eux.
Pour obtenir les configurations de résidus avec un nombre impair de résidus égaux, il suffit en partant de la différentielle construite au paragraphe précédent de changer des paires de pôles à résidus nul en paires de pôles à résidus égaux à~$4$. Pour cela, on considère deux domaines polaires à résidu nul séparés par un lien selle d'une certaine longueur. On accroît cette longueur jusqu'à ce que chacun des deux résidus résidus quadratiques soit égal à~$4$. Ce processus ne modifie pas l'ordre des zéros. On applique cette construction à $(t-1/2)$ paires de domaines polaires consécutives. Notons que la forme du graphe de connexion implique qu'il existe un choix pour ces paires. 
On obtient ainsi les différentielles quadratiques ayant les invariants locaux souhaités.
\end{proof}

Nous traitons maintenant les cas des strates quadratiques  $\Omega^{2}\moduli[0](a_{1},a_{2};-b_{1},\ldots,-b_{p})$ contenant un nombre pair de résidus quadratiques égaux entre eux.

\begin{lem}\label{lem:realquadzeroa}
L'application $2$-résiduelle de la strate $\Omega^{2}\moduli[0](a_{1},a_{2};-b_{1},\ldots,-b_{p})$ avec $a_{i}\geq-1$ impairs et $p\geq2$ nombres pairs $b_{i}\geq 4$ contient tous les uplets contenant un nombre pair supérieur ou égal à $4$ d'éléments égaux à~$1$ et un nombre arbitraire égaux à~$0$.
\end{lem}

\begin{proof}
Nous commençons par montrer que la configuration $(1,1,1,1)$ est réalisable dans les strates $\Omega^{2}\moduli[0](7,5,-4^{4})$, $\Omega^{2}\moduli[0](9,3,-4^{4})$, $\Omega^{2}\moduli[0](11,1,-4^{4})$ et $\Omega^{2}\moduli[0](13,-1,-4^{4})$.
\par
D'après la figure~\ref{fig:exeptpasexept}, il existe une différentielle quadratique dans $\Omega^{2}\moduli[0](7,5,-4^{4})$ dont les résidus quadratiques sont $(1,1,1,1)$.
\begin{figure}[htb]
\begin{tikzpicture}

\begin{scope}[xshift=-6cm]
\fill[fill=black!10] (0,0)coordinate (Q)  circle (1.1cm);
    \coordinate (a) at (-.25,0);
    \coordinate (b) at (.25,0);
    \coordinate (c) at (-.75,0);
     \fill (a)  circle (2pt);
\fill[] (b) circle (2pt);
  \fill[] (c) circle (2pt);

    \filldraw[fill=white] (-.19,0)  arc (0:180:2pt);

    \fill[white] (a) -- (b) -- ++(0,-1.2) --++(-.5,0) -- cycle;
 \draw  (b)-- (a);
 \draw (a)-- (c) coordinate[pos=.5](d);
 \draw (a) -- ++(0,-1.05);
 \draw (b) -- ++(0,-1.05);

\node[above] at (Q) {$v_{2}$};
\node[above] at (d) {$v_{1}$};
\node[below] at (d) {$v_{3}$};
    \end{scope}
    
\begin{scope}[xshift=-3cm]
\fill[fill=black!10] (0,0)coordinate (Q)  circle (1.1cm);
    \coordinate (a) at (-.25,0);
    \coordinate (b) at (.25,0);
    \coordinate (c) at (-.75,0);
    
     \fill (a)  circle (2pt);
\fill[] (b) circle (2pt);
    \filldraw[fill=white] (-.19,0)  arc (0:-180:2pt); 

    \fill[white] (a) -- (b) -- ++(0,1.2) --++(-.5,0) -- cycle;
 \draw  (a) -- (b);
 \draw (a) -- ++(0,1.05);
 \draw (b) -- ++(0,1.05);
 \draw (a)-- (c) coordinate[pos=.5](d);
  \filldraw[fill=white] (c) circle (2pt);

\node[below] at (Q) {$v_{2}$};
\node[below,rotate=180] at (d) {$v_{1}$};
\node[below] at (d) {$v_{4}$};
    \end{scope}

\begin{scope}[xshift=0cm]
\fill[fill=black!10] (0,0)coordinate (Q)  circle (1.1cm);
    \coordinate (a) at (-.25,0);
    \coordinate (b) at (.25,0);

     \fill (a)  circle (2pt);
\fill[] (b) circle (2pt);
    \fill[white] (a) -- (b) -- ++(0,-1.2) --++(-.5,0) -- cycle;
 \draw  (a) -- (b);
 \draw (a) -- ++(0,-1.05);
 \draw (b) -- ++(0,-1.05);

\node[above] at (Q) {$v_{3}$};
    \end{scope}

\begin{scope}[xshift=3cm]
\fill[fill=black!10] (0,0)coordinate (Q)  circle (1.1cm);
    \coordinate (a) at (-.25,0);
    \coordinate (b) at (.25,0);
  \filldraw[fill=white] (a) circle (2pt);
  \filldraw[fill=white] (b) circle (2pt);

    \fill[white] (a) -- (b) -- ++(0,-1.2) --++(-.5,0) -- cycle;
 \draw  (a) -- (b);
 \draw (a) -- ++(0,-1.05);
 \draw (b) -- ++(0,-1.05);

\node[above] at (Q) {$v_{4}$};
    \end{scope}

\end{tikzpicture}
\caption{Une $2$-différentielle de $\Omega^{2}\moduli[0](7,5;(-4^{4}))$ dont les $2$-résidus sont $(1,1,1,1)$.} \label{fig:exeptpasexept}
\end{figure}
\par
Pour les strates de la forme $\Omega^{2}\moduli[0](13-2k,-1+2k;(-4^{4}))$ avec $0 \leq k \leq 2$, nous allons recoller des parties polaires d'ordre $4$ le long de liens selles selon un graphe d'incidence. Ce graphe est composé d'une boucle formée par $k$ sommets de valence $2$ et collée à un sommet de valence $3$ auquel est attaché une chaîne composée de $3-k$ sommets comme montré dans la figure~\ref{fig:graphesincidence}.
\begin{figure}[htb]
\begin{tikzpicture}
\begin{scope}[xshift=-5cm]
\coordinate (z1) at (0,0);\node[right] at (z1) {$(1,1;3)$};
\coordinate (z2) at (0,-1);\node[right] at (z2) {$(3;2)$};
\coordinate (z3) at (0,-2);\node[right] at (z3) {$(2;1)$};
\coordinate (z4) at (0,-3);\node[right] at (z4) {$(1;\emptyset)$};

\draw (z1) -- (z2) -- (z3) -- (z4);
\foreach \i in {1,2,3,4}
\fill (z\i) circle (2pt); 
\draw (z1) .. controls ++(40:1) and ++(140:1) .. (z1);
\end{scope}

\begin{scope}[xshift=0cm]
\coordinate (z1) at (0,0);\node[right] at (z1) {$(1.5;0.5)$};
\coordinate (z2) at (0,-1);\node[right] at (z2) {$(0.5,1.5;2)$};
\coordinate (z3) at (0,-2);\node[right] at (z3) {$(2;1)$};
\coordinate (z4) at (0,-3);\node[right] at (z4) {$(1;\emptyset)$};

\draw (z2) -- (z3) -- (z4);
\foreach \i in {1,2,3,4}
\fill (z\i) circle (2pt); 
\draw (z1) .. controls ++(-40:.5) and ++(40:.5) .. (z2);
\draw (z1) .. controls ++(-140:.5) and ++(140:.5) .. (z2);
\end{scope}

\begin{scope}[xshift=5cm]
\coordinate (z1) at (1,-1);\node[right] at (z1) {$(2;1)$};
\coordinate (z2) at (-1,-1);\node[left] at (z2) {$(2;1)$};
\coordinate (z3) at (0,-2);\node[below right] at (z3) {$(1,1;1)$};
\coordinate (z4) at (0,-3);\node[right] at (z4) {$(1;\emptyset)$};

\draw (z3) -- (z4);
\foreach \i in {1,2,3,4}
\fill (z\i) circle (2pt); 
\draw (z1) .. controls ++(140:.9) and ++(40:.9) .. (z2);
\draw (z2) .. controls ++(-140:.6) and ++(180:.6) .. (z3);
\draw (z3) .. controls ++(10:.6) and ++(-40:.6) .. (z1);
\end{scope}
\end{tikzpicture}
\caption{Les graphes d'incidence pour les exemples dans les strates $\Omega^{2}\moduli[0](13-2k,-1+2k;(-4^{4}))$ avec $0 \leq k \leq 2$ de gauche à droite.} \label{fig:graphesincidence}
\end{figure}
La $2$-partie polaire correspondant au sommet de valence $1$ est associée au vecteur $(1;\emptyset)$.
Les parties $2$-polaires correspondant aux sommets de valence $2$ sont associées à des vecteurs de la forme $(v+1;v)$. Enfin, la partie $2$-polaire correspondant au sommet de valence $3$ est associée à des vecteurs de type $(v_{1},v_{2};v_{1}+v_{2}-1)$ ou $(v_{1},v_{2};v_{1}+v_{2}+1)$ tels que les vecteurs $v_{i}$ correspondent au arêtes de la boucle. Cela est montré dans chaque cas dans la figure~\ref{fig:graphesincidence}. On vérifie sans difficulté que le recollement de ces parties polaires donne des différentielles quadratiques aux singularités souhaitées.
\par
Dans le cas où tous les pôles sont d'ordres $-4$, on ajoute aux différentielles quadratiques construites aux paragraphes précédents des parties $2$-polaires comme suit. Pour obtenir des pôles dont les résidus quadratiques sont égaux à~$1$, resp. $0$, on 
peut ajouter une partie polaire associée à des vecteurs de la forme $(v;v+1)$, resp. $(v,v)$, à la boucle ou à la chaîne. La première opération ajoute $2$ à l’ordre de chaque zéro et la seconde $4$ à l'ordre du zéro le plus grand. Cela permet d'obtenir toutes les strates considérées.
\par
Dans le cas général, il suffit de coller des plans de manière cyclique à une demi-fente infinie partant de l'un des deux zéros. Un calcul simple montre que l'on peut obtenir tous les ordres des deux zéros.
\end{proof}

Il reste à présent le cas des strates quadratiques  $\Omega^{2}\moduli[0](a_{1},a_{2};-b_{1},\ldots,-b_{p})$ dont exactement deux résidus quadratiques sont non nuls et égaux entre eux. On traite d'abord les cas pour lesquels un pôle d'ordre $-6$ ou moins a un résidu nul.

\begin{lem}\label{lem:realquadzerob}
L'image de l'application $2$-résiduelle de la strate $\Omega^{2}\moduli[0](a_{1},a_{2};-b_{1},\ldots,-b_{p})$ avec $a_{i}\geq-1$ impairs et $p\geq2$ nombres pairs $b_{i}\geq 4$ contient $(1,1,0,\dots,0)$ dès lors qu'un pôle d'ordre inférieur ou égal à $-6$ possède un résidu nul.
\end{lem}

\begin{proof}
Nous montrons tout d'abord que $(1,1,0)$ est dans l'image de l'application $2$-résiduelle des strates $\Omega^{2}\moduli[0](5,5;-4,-4,-6)$, $\Omega^{2}\moduli[0](7,3;-4,-4,-6)$, $\Omega^{2}\moduli[0](9,1;-4,-4,-6)$ et $\Omega^{2}\moduli[0](11,-1;-4,-4,-6)$ si le pôle d'ordre $-6$ a un $2$-résidu nul.
\par
On commence par la strate $\Omega^{2}\moduli[0](5,5;-4,-4,-6)$. Pour le pôle d'ordre $-6$, on colle par translation les demi-fentes infinies de deux domaines polaires ouverts à droite associés au vecteur $(1;1)$. Pour les pôles d'ordre $-4$ on considère deux parties $2$-polaires d'ordre $4$ associés à $(1;\emptyset)$. La différentielle quadratique est obtenue en collant les deux segments des parties $2$-polaires d'ordre $-4$ aux segments inférieurs de celle d'ordre $6$ par translation et le deux segments supérieurs entre eux par rotation d'angle~$\pi$. 
\par
Pour la strate $\Omega^{2}\moduli[0](11,-1;-4,-4,-6)$, on considère une partie $2$-polaire d'ordre $6$ associée à $(1,1;1,1)$ et deux parties $2$-polaires d'ordre $4$ associées à $(1,\emptyset)$. On colle les deux segments supérieurs de la  $2$-polaire d'ordre $6$  par translation et les autres segments par translations. 
\par
Pour les strates $\Omega^{2}\moduli[0](7,3;-4,-4,-6)$ et $\Omega^{2}\moduli[0](9,1;-4,-4,-6)$, nous considérons des parties $2$-polaires d'ordre $4$ associées respectivement à $(v_{2},v_{3};v_{1})$ et $v_{3}$ où $v_{1}=v_{2}=1/2$ et $v_{3}=1$. Pour les pôles d'ordre $-6$, on considère la partie $2$-polaire d'ordre $6$ associée à $(v_{1},v_{2};\emptyset)$ dans le premier cas et à $(v_{2},v_{1};\emptyset)$ dans le second. Les recollements des sommets $v_{i}$ entre eux donnent les différentielles quadratiques voulues.
\par
Comme dans la preuve du lemme~\ref{lem:realquadzeroa}, l'ajout de parties $2$-polaires et l'augmentation de l'ordre de certains pôles en collant cycliquement des plans avec des  demi-fentes infinies, permet d'obtenir les $2$-résidus $(1,1,0,\dots,0)$ dans tout les cas considérés.
\end{proof}

Il reste donc à considérer les cas avec deux $2$-résidus non nuls égaux entre eux et avec des $2$-résidus nuls uniquement à des pôles d'ordre $-4$.

\begin{lem}\label{lm:realquadzerounb}
Soit $\Omega^{2}\moduli[0](a_{1},a_{2};-b_{1},\ldots,-b_{p})$ une strate avec $a_{1},a_{2}\geq-1$ impairs et $p\geq2$ nombres pairs $b_{i}\geq4$. L'application $2$-résiduelle de cette strate contient le complémentaire de l'origine de l'espace $2$-résiduel sauf pour  les deux familles exceptionnelles suivantes. Dans les deux cas, $b$ est un entier pair supérieur ou égal à~$4$.
\begin{enumerate}
 \item L'image de l'application $2$-résiduelle de $\Omega^{2}\moduli[0](2p+b-5,2p+b-5;-b,-b-2,(-4^{p-2}))$ est le complémentaire de la droite engendrée par $(1,1,(0^{p-2}))$ dans l'espace $2$-résiduel.
 \item L'image de l'application $2$-résiduelle de $\Omega^{2}\moduli[0](2p+b-7,2p+b-5;-b,-b,(-4^{p-2}))$ est le complémentaire de la droite engendrée par $(1,1,(0^{p-2}))$ dans l'espace $2$-résiduel.
\end{enumerate}
\end{lem}

\begin{proof}
D'après les lemmes~\ref{lem:polesdivksanszero} à~\ref{lem:realquadzerob} il reste à traiter les $2$-résidus de la forme $(1,1,0,\dots,0)$ dans les strates de la forme $\Omega^{2}\moduli[0](a_{1},a_{2};-b_{1},-b_{2},(-4^{p-2}))$, où tout les pôles d'ordre $-4$ ont un résidu nul. Nous allons étudier de façon exhaustive toutes les façons de réaliser ces singularités pour ces strates. Nous supposerons que $a_{1} \leq a_{2}$ et $b_{1} \leq b_{2}$.
\par
Si une telle $2$-différentielle existe, alors elle possède un cœur dégénéré. Elle est donc constituée de~$p$ domaines polaires séparés par $p$ liens selles horizontaux. Ces domaines polaires connectés selon un graphe d'incidence  avec $p$ sommets et de $p$ arêtes (le graphe est donc de genre $1$). Comme un sommet correspondant à un pôle de résidu nul ne peut pas être de valence $1$ (le lien selle correspondant serait de longueur nulle), il y a au plus $2$ sommets de valence~$1$. Nous allons étudier successivement  les cas avec $0$, $1$ et~$2$ sommets de valence~$1$.
\par
S'il n'y a aucun sommet de valence $1$, le graphe d'incidence est cyclique. Il faut alors considérer les longueurs des liens selles le long du cycle. De part et d'autre d'un domaine polaire de $2$-résidu nul, les liens selles conservent la même longueur. Les longueurs des liens selles au bord des domaines polaires de $2$-résidu égal à $1$ satisfont $\lambda+\mu=1$ ou $\mu-\lambda=1$. La $2$-différentielle étant primitive, chaque pôle satisfait une égalité différente. Cela donne une contradiction et une telle $2$-différentielle ne peut pas exister.
\par
S'il y a exactement un sommet de valence $1$, le graphe d'incidence est du même type que ceux de la figure~\ref{fig:graphesincidence}, avec un nombre arbitraire de sommets sur la boucle et sur la chaîne. La construction décrite dans la preuve du lemme~\ref{lem:realquadzeroa} s'applique de manière analogue. Nous donnons maintenant la description des strates qui  peuvent s'obtenir de cette façon. 
Il s'agit de répartir les domaines polaires de façon à réaliser les ordres des zéros souhaités. L'ordre de la singularité conique au bord des domaines polaires de la chaîne est supérieur ou égal à $2p-3+b_{1}$ (si $b_{1}$ est au bout de la chaîne). En effet, chaque domaine polaire contribue à cette singularité conique d'un angle  supérieur ou égal à $2\pi$, auxquels se rajoutent $2\pi$ pour le sommet de valence 3 et $(b_{1}-3)\pi$ pour celui de valence~$1$. On peut donc attribuer à l'une des deux singularités n'importe quelle valeur (impaire) à partir de $2p-3+b_{1}$. Comme $a_{1}+a_{2}=-4+b_{1}+b_{2}+4(p-2)$, cela signifie que l'autre singularité prendra au plus la valeur $-4+b_{1}+b_{2}+4(p-2) -(2p-3+b_{1})$, c'est-à-dire $2p-9+b_{2}$. Donc, si $b_{2} \geq b_{1}+6$, tous les ordres de singularités sont réalisables. Si $b_{2}=b_{1}+4$, alors on a $a_{2}\geq b_{1}+2p-1>2p-3+b_{1}$, et tout les ordres de zéros sont donc possible.
Dans les cas avec $b_{2}=b_{1}$ ou $b_{2}=b_{1}+2$ on vérifie de la même façon que l'on peut obtenir tout les ordres de zéros à l’exception des deux cas listés dans le lemme~\ref{lm:realquadzerounb}.
\par
Il reste donc a montrer pour ces deux cas il n'existe pas de différentielle quadratique avec pour $2$-résidus $(1,1,0,\dots,0)$. 
Il suffit de montrer qu'une construction avec un graphe d'incidence qui a deux sommets de valence $1$ ne fonctionne pas. 
Il faut distinguer deux cas selon s'il existe un sommet de valence $4$ ou deux sommets de valence $3$ parmi les pôles de résidu nul. 
\par
Le cas avec deux sommets de valence $3$ est équivalent au cas cyclique, les longueurs des liens selles le long du cycle contenu dans le graphe sont $\lambda$ et $\mu$ dans les tronçons découpés par les sommets de valence $3$. Les relations $\lambda+\mu=1$ et $\mu-\lambda=1$ sont contradictoires et la différentielle ne peut donc pas exister.
\par
Considérons enfin le cas où il y a un unique sommet de valence~$4$. Le bord du domaine polaire associé possède quatre liens selles consécutifs horizontaux. Deux liens selles correspondent au cycle contenu dans le graphe. Ils ont donc la même longueur et la même orientation, car la différentielle est primitive et que  tout les pôles du cycle ont un résidu nul. Les deux autres segments sont nécessairement de longueur $1$ et de l'autre orientation.
La géométrie de la surface est ainsi entièrement déterminée. Les deux branches contribuent à la même singularité conique. Chaque domaine polaire du cycle ajoute $2\pi$ à chaque singularité conique. Enfin, dans le sommet de valence $4$, sur l'angle total de $6\pi$, il y a $5\pi$ pour la singularité conique à laquelle contribuent les deux branches et $\pi$ pour l'autre singularité. L'écart d'angle total entre les deux singularités coniques est donc au minimum de $4\pi$. L'écart d'ordre est donc supérieur ou égal à~$4$. Ainsi, les résidus quadratiques $(1,1,0,\dots,0)$ ne sont pas réalisables dans les strates considérées.
\end{proof}
\par
Enfin nous traitons le cas des strates  $\Omega^{2}\moduli[0](a_{1},\dots,a_{n};-b_{1},\ldots,-b_{p})$ avec $n\geq3$.
\begin{lem}
 L'application $2$-résiduelle des strates  $\Omega^{2}\moduli[0](a_{1},\dots,a_{n};-b_{1},\ldots,-b_{p})$ avec $n\geq3$ contient $\espresk[2][0](a_{1},\dots,a_{n};-b_{1},\ldots,-b_{p})\setminus (0,\dots,0)$.
\end{lem}

\begin{proof}
La preuve de ce lemme se fait par éclatement de zéros. Il suffit donc de montrer que lorsque la strate contient trois zéros d'ordres $a_{i}$, l'application $2$-résiduelle de cette strate continent le complémentaire de l'origine. Cela est une conséquence immédiate du lemme~\ref{lm:realquadzerounb} pour toutes les strates, sauf pour $\Omega^{2}\moduli[0](2p+b-7,2p+b-7,2;-b,-b,(-4^{p-2}))$ avec $p\geq2$.
\par
La figure~\ref{fig:deuxdiffresunun} montre une différentielle quadratique de  $\Omega^{2}\moduli[0](1,1,2;-4,-4)$ dont les $2$-résidus sont $(1,1)$.
\begin{figure}[hbt]
\center
\begin{tikzpicture}
\begin{scope}[xshift=-2cm]
\fill[fill=black!10] (1,0)coordinate (Q)  circle (1.2cm);
    \coordinate (a) at (0,0);
    \coordinate (b) at (2,0);

     \filldraw[fill=white] (a)  circle (2pt);
\filldraw[fill=white](Q) circle (2pt);
\filldraw[fill=red] (1.07,0)  arc (0:180:2pt); 
\fill (b) circle (2pt);
    \fill[white] (a) -- (Q) -- ++(0,-1.2) --++(-1,0) -- cycle;
 \draw  (a) -- (b) coordinate[pos=.25](c) coordinate[pos=.75](d);
 \draw (a) -- ++(0,-.5);
 \draw (Q) -- ++(0,-1.1);

\node[above] at (c) {$1$};
\node[above] at (d) {$2$};
\node[below] at (d) {$3$};

    \end{scope}
\begin{scope}[xshift=2cm]
\fill[fill=black!10] (1,0)coordinate (Q)  circle (1.2cm);
    \coordinate (a) at (0,0);
    \coordinate (b) at (2,0);

     \filldraw[fill=red] (a)  circle (2pt);
\filldraw[fill=red](Q) circle (2pt);
\filldraw[fill=white] (1.07,0)  arc (0:180:2pt); 
\fill (b) circle (2pt);
    \fill[white] (a) -- (Q) -- ++(0,-1.2) --++(-1,0) -- cycle;
 \draw  (a) -- (b) coordinate[pos=.25](c) coordinate[pos=.75](d);
 \draw (a) -- ++(0,-.5);
 \draw (Q) -- ++(0,-1.1);

\node[below,rotate=180] at (c) {$1$};
\node[above] at (d) {$3$};
\node[below] at (d) {$2$};

    \end{scope}
    \end{tikzpicture}
\caption{Différentielle quadratique de $\Omega^{2}\moduli[0](1,1,2;-4,-4)$ dont les $2$-résidus sont $(1,1)$.} \label{fig:deuxdiffresunun}
\end{figure}
Le cas des strates de la forme $\Omega^{2}\moduli[0](b-3,b-3,2;-b,-b)$ s'obtiennent de manière similaire en utilisant des $2$-parties polaires d'ordre $b$ et de type  $b-1$ associés aux même vecteurs que dans la figure~\ref{fig:deuxdiffresunun}.
Enfin, le cas général se déduit de celui-ci en coupant le long du lien selle noté~$1$ et en collant de manière cyclique des $2$-parties polaires d'ordre~$4$ associées à $(1;1)$.  
\end{proof}

\smallskip
\par
Nous voudrions maintenant comprendre les cas où l'application résiduelle contient l'origine. Le cas d'un pôle a déjà été traité, ainsi nous supposerons que le nombre de pôles est supérieur ou égal à  deux. 
Dans le cas où $n=2$ on a le résultat suivant dont la preuve est analogue à celle de \cite[Lemma~3.6]{BCGGM}.
\begin{lem}
Pour tout $k\geq2$ et $p\geq1$ l'image de $\appresk[0](a_{1},a_{2};-b_{1},\dots,-b_{p})$  ne contient pas $(0,\dots,0)$.
\end{lem}

\begin{proof}
Supposons par l'absurde qu'il existe une $k$-différentielle~$\xi_{0}$ dans la strate $\komoduli[0](a_{1},a_{2};-b_{1},\dots,-b_{p})$ dont les $k$-résidus soient nuls à tous les pôles~$P_{i}$.
 Pour tous les $i\in\lbrace 1,\dots,p \rbrace$, nous créons une $k$-différentielle entrelacée de la façon suivante. Si  $b_{i}/k$ est pair, nous collons une $k$-différentielle $\xi_{i}$ de genre $g_{i}=\tfrac{b_{i}}{2k}$ avec un unique zéro d'ordre $b_i-2k$ au point d'attachement~$z_i$ identifié avec~$P_{i}$. Si $b_{i}/k$ est impair, nous collons une $k$-différentielle~$\xi_{i}$ de genre $g_{i}=\tfrac{b_{i}+k}{2k}$ avec un zéro d'ordre $b_i-2k$ au point d'attachement~$z_i$ identifié avec~$P_{i}$ et un autre zéro d'ordre~$k$.
De plus, pour tout $i\in \left\{1,\dots,p\right\}$ les $k$-différentielles $\xi_{i}$ sont les puissances $k$-ième de différentielles abéliennes holomorphes.
\par
La $k$-différentielle entrelacée ainsi obtenue est de type $(a_{1},a_{2},(k^{d}))$, où $d$ est le nombre de~$b_{i}$ tels que $b_{i}/k$ est impaire. D'après le théorème 1.5 de \cite{BCGGM3} (qui, dans ce cas, est une généralisation directe de la propriété~\ref{lem:lissdeuxcomp}), cette $k$-différentielle entrelacée est lissable dans la strate $\PP\omoduli(a_{1},a_{2},(k^{d}))$, avec $g=\sum_{i}g_{i}$.
\par
Dénotons par $S$ le lieu de toutes les $k$-différentielles entrelacées que nous venons de construire, où $\xi_{0}$ et le $p$-uplet $(\xi_{1},\dots,\xi_{p})$ sont chacun considérés à multiplication par~$\CC^{\ast}$ près. Le fait que les $k$-résidus des pôles de~$\xi_{0}$ soient nuls implique que~$\xi_{0}$ est uniquement déterminée modulo $\CC^{\ast}$. Donc la dimension de $S$ est donnée par
$$ -1 + \sum_{i>0\,,\, b_{i}/k\text{ pair}}2g_{i}+\sum_{i>0\,,\,b_{i}/k\text{ impaire}}
(2g_{i}+1)=2g+d-1,$$
car les $k$-différentielles $\xi_{i}$ sont les puissances de différentielles abéliennes holomorphes et le~$-1$ provient du fait que nous considérons le uplet $(\xi_{1},\dots,\xi_{p})$ modulo multiplication par~$\CC^{\ast}$. 
\par
Remarquons que
$\dim \PP\omoduli(a_{1},a_{2},(k^{d})) = 2g+d-1.$
Le lieu $S$ a la même dimension que la strate $\PP\omoduli(a_{1},a_{2},(k^{d}))$. Ainsi, il ne peut pas être au bord de celle-ci. Cela contredit le fait que toutes les $k$-différentielles de $S$ soient lissables dans $\PP\omoduli(a_{1},a_{2},(k^{d}))$.
\end{proof}

Dans le cas des différentielles avec $n\geq3$ zéros, nous commençons par considérer le cas où les ordres de deux zéros ne sont pas divisibles par $k$ et tous les autres ordres sont divisibles par~$k$. Dans ce cas, le lemme suivant donne une restriction afin que l'origine soit dans l'image de l'application $k$-résiduelle.

\begin{lem}\label{lem:condnecngeq3}
Soit  $\mu=(a_{1}\dots,a_{n};-b_{1},\dots,-b_{p})$ une partition telle que $a_{1}$ et $a_{2}$ ne sont pas divisibles par $k$ et  $a_{i}$ est divisible par $k$ pour $i \geq 3$. Si l'application résiduelle $\appresk[0](\mu)$ contient $\lbrace 0 \rbrace$, alors $\sum_{i=3}^{n} a_{i} \geq kp$.
\end{lem}

\begin{proof}
Une $k$-différentielle dans la strate $\komoduli[0](\mu)$ dont les $k$-résidus sont nuls ne possède pas de lien selle reliant les singularités coniques d'ordres $a_{1}$ et $a_{2}$. En effet, supposons qu'il existe une $k$-différentielle $\xi$ possédant un tel lien selle~$\gamma$.  Considérons un voisinage de~$\gamma$, alors sur le complémentaire $\xi$ est la puissance $k$-ième d'une différentielle abélienne sur un disque. Par le théorème des résidu, l'intégral de cette forme le long du bord est la somme des résidus, donc nulle. Cela implique en passant à la limite que $\xi$ est donne une surface de translation, ce qui est absurde.  
Il est également impossible qu'une $k$-différentielle à $k$-résidus nuls possède un lien selle fermé contenant $a_{1}$ ou $a_{2}$. En effet, une telle courbe découperait~$\PP^{1}$ en deux composantes bordées par un unique vecteur. La restriction de $\xi$ sur l'une des composantes serait alors la puissance $k$-ième d'une différentielle abélienne dont résidus sont nuls. Cela implique que ce lien selle est de période nulle, ce qui est absurde.
\par
Supposons par l'absurde qu'il existe une $k$-différentielle à résidus nuls dans une strate telle que $\sum_{i=3}^{n} a_{i} < kp$. Toute classe d'homotopie d'arcs topologiques reliant $a_{1}$ à $a_{2}$ admet un représentant formé d'une chaîne $\gamma=(\gamma_{1},\dots,\gamma_{r})$ de $r\geq2$ liens selles  reliant une succession de singularités coniques intermédiaires. Si l'on découpe la surface le long de $\gamma$, on obtient une surface de translation
dont le bord est formé de $2r$ segments. Prenons une racine $k$-ième de $\xi$ sur le complémentaire de $\gamma$ et supposons, quitte à déformer $\xi$ que ces segments ne possède pas de somme partielle de leurs périodes nulle. Rappelons que les deux segments correspondant à un lien selle diffère en général par une rotation d'angle~$\tfrac{2\pi}{k}$.

Si une singularité conique n'est pas sur $\gamma$, alors on peut déformer la différentielle sans modifier les périodes du bord de la surface ni les résidus aux pôles. En effet, on peut modifier les périodes relatives contenant au moins une singularité de $\PP^{1}\setminus \gamma$. On peut donc dégénéré la différentielle en fusionnant deux singularités d'ordres $a_{i}$ pour $i\geq3$ dont l'une n'est pas sur~$\gamma$.  Notons que pendant se processus, il peut arriver que d'autres liens selles dégénèrent. Dans ce cas,  leurs classes d'homologie ne diffèrent que par celles de lacets simples autour des pôles de période nulle ou de segments de bord dont la somme des périodes est nulle. Ces liens selles ayant la même période, ils sont dans la même direction et ne se coupent donc pas. Si l'ensemble des liens selles dégénérant ne partitionnent pas la surface en plusieurs composantes, alors cette dégénération est une fusion de singularités coniques. Comme l'ordre des singularités ainsi obtenue est la somme des ordres des singularités qui se fusionnent, on obtient ainsi une différentielle avec strictement moins de singularités coniques. De plus, la borne $\sum_{i=3}^{n} a_{i} < kp$ reste vérifiée.   
\par
Dans le cas où ces liens selles découpent la surface en plusieurs composantes, on peut couper ces liens selles. Il y a alors deux cas à considérer. Si l'une de ces composantes possède une singularité dans son intérieur, alors on peut recommencer ce processus avec cette composante. L'autre cas est lorsque aucune composante ne possède de singularité à l'intérieur d'une composante. Dans ce cas, la somme des périodes de chaque composante est nulle. Donc l'hypothèse qu'aucune somme partielle des périodes est nulle, implique que le bord de ces composantes est formée exactement de deux liens selles reliant les singularités~$A$ et~$B$. Comme il n'y a pas de singularité intérieure dans de telles composantes, les liens selles éventuels entre $A$ et $B$ découpent la composante en autant de domaines polaires  et possèdent deux angles chacun multiple de~$2\pi$. Nous supprimons alors ces domaines polaires et recollons les segments qui bordaient ces domaines ensemble. Comme pour chaque pôle ainsi supprimé, la somme  $\sum_{i=3}^{n} a_{i}$ diminue d'au moins~$k$, cette chirurgie donne un nouveau contre-exemple~$\xi$ dans lequel la chaîne $\gamma$ passe une unique fois par chacune des singularités coniques de la surface.
\par
Considérons cette $k$-différentielle dotée d'une chaîne formée de $n-1$ liens selles.  
Nous pouvons compléter ces liens selles pour former un système maximal d'arcs géodésiques ne se coupant pas. Par \cite[lemme~4.10]{tahar} un tel système possède au moins $n+p-2$ liens selles. Notons que si $p=1$, alors $\sum_{i=3}^{n} a_{i} = 0$ et donc $n=2$. Comme ce cas a été traité, nous supposerons $p\geq2$. Il existe donc des liens selles qui n'appartiennent pas à la chaîne~$\gamma$. S'il  existe un tel lien selle $\alpha$ qui relie deux singularités coniques qui ne sont pas reliées par un $\gamma_{i}\in\gamma$, nous pouvons remplacer au moins deux $\gamma_{i}$ de $\gamma$ par~$\sigma$. On peut alors itérer le processus décrit précédemment. Si un tel lien selle $\sigma$ relie deux singularités coniques reliées par un $\gamma_{i}$, alors $\sigma$ et $\gamma_{i}$ bordent des domaines polaires que nous pouvons alors éliminer. 
On en déduit que pour une  $k$-différentielle $\xi$ minimale pour le nombre et l'ordre des singularités, les seuls liens selles qui ne sont pas dans $\gamma$ sont fermés sur les singularités coniques intermédiaires de~$\gamma$. Enfin, s'il existe deux tels liens selles, ils découpent ensemble un cylindre. On peut éliminer ce cylindre en recollant  et recoller ses bords l'un sur l'autre. 
On en déduit que la $k$-différentielle $\xi$ minimale possède un système maximal d'arcs géodésiques formé des $n-1$ liens selles $\gamma_{i}$ et un lien selle fermé $\sigma$ sur une singularité d'ordre $a_{i}$ pour $i\geq3$. 
\par
La $k$-différentielle minimale que nous avons obtenue compte donc $n$ liens selles ne se coupant pas,  ce qui implique par \cite[Lemme 4.10]{tahar} que $p=2$ et que le \coeur est dégénéré. La surface est donc formée de deux domaines polaires collés le long d'un lien selle fermé. Les $n-1$ liens selles $\gamma_{i}$ ont chacun de leurs deux côtés donnant sur le même domaine polaire. On peut alors  fusionner les singularités coniques jusqu'à obtenir trois singularités coniques.  Les deux singularités coniques d'ordre non divisibles par $k$, chacune dans son domaine polaire, et une troisième d'ordre divisible par $k$ et bordant un lien selle fermé séparant les deux domaines polaires. La condition $a_{3}<2k$ implique que $a_{3}=k$ et donc l'angle de la singularité conique correspondante est~$4\pi$. Pour conclure la preuve il suffit donc de montrer que cela est impossible. En effet, un angle donnant sur un domaine polaire est supérieur ou égal à~$\pi$. Le cas d'égalité implique que le domaine de pôle d'ordre $-b_{1}$ a ses trois angles respectivement égaux à $\pi$, $\pi$ et $(2d-1)\pi$. Cela exclut que le pôle correspondant est de résidu nul. 
\end{proof}

Nous montrons maintenant que dans le cas où uniquement deux zéros sont d'ordres non divisibles par $k$, la condition du lemme~\ref{lem:condnecngeq3} est suffisante pour que l'origine soit dans l'image de l'application $k$-résiduelle.

\begin{lem}\label{lem:condsufngeq3}
Soit  $\mu=(a_{1}\dots,a_{n};-b_{1},\dots,-b_{p})$ une partition telle que $a_{1}$ et $a_{2}$ ne sont pas divisibles par $k$ et  $a_{i}$ est divisible par $k$ pour $i \geq 3$.  Si $\sum_{i=3}^{n} a_{i} \geq kp$, alors l'application résiduelle $\appresk[0](\mu)$ contient $\lbrace 0 \rbrace$.
\end{lem}

\begin{proof}
Nous allons démontrer le cas $n=3$ et en éclatant $a_{3}$ on en déduit  le cas général.
Nous notons $a_{i} = kl_{i} + \bar{a}_{i}$ avec $0\geq\bar{a}_{i}>-k$. Nous considérerons deux cas, selon qu'il existe $a_{i}$ avec $i\in \lbrace 1,2\rbrace$ tel que modulo permutation des indices $\ell_{i}>\sum_{i=1}^{p-1}(b_{i}-1)$ ou non.
\par
Dans le premier cas, nous supposerons que l'inégalité est satisfaite pour $i=1$. On prend une $k$-partie polaire d'ordre $2k$ associé à $(v_{1}, \alpha v_{1};v_{1},\alpha v_{1})$ avec $\alpha$ tel que l'angle au dessus du point d'intersection entre $v_{1}$ et $\alpha v_{1}$ est égal à $(k+\bar{a}_{1})\tfrac{2\pi}{k}$.
Pour tout les autres pôles on prend une  $k$-partie polaire d'ordre $b_{i}$ associée à $(v_{1};v_{1})$ et de type $1$. On colle ces parties polaires cycliquement aux vecteurs supérieurs $v_{1}$ et $\alpha v_{1}$ de la première $k$-partie polaire. Enfin nous coupons cette $k$-partie polaire des singularités par des demi-droites infinies et collons le nombre de plan nécessaire pour obtenir les ordres $a_{i}$. Cette construction est représentée dans la  figure~\ref{fig:polesdivkaveczero}.  
 \begin{figure}[hbt]
\begin{tikzpicture}
\begin{scope}[xshift=-6cm,yshift=.5cm]

     \foreach \i in {1,2,...,4}
  \coordinate (a\i) at (\i,0); 
  \coordinate (b) at (2.5,.25);
   \fill[black!10] (a1) -- (a4) -- ++(0,0) arc (0:180:1.5) -- cycle;
      \fill[white] (b)  circle (2pt);
      \draw (b) circle (2pt);
 \fill[white] (a2) -- (a3) -- (b) -- cycle;
       \foreach \i in {2,3}
   \fill (a\i)  circle (2pt);
   \draw (a1) -- (a2) coordinate[pos=.5](e1) -- (b) coordinate[pos=.5](e2) -- (a3) coordinate[pos=.5](e3) -- (a4) coordinate[pos=.5](e4);

     \foreach \i in {1,2,...,4}
  \coordinate (c\i) at (\i,-1); 
  \coordinate (d) at (2.5,-.75);
   \fill[black!10] (c1) -- (c2) -- (d) -- (c3) -- (c4) -- ++(0,0) arc (0:-180:1.5) -- cycle;
     \foreach \i in {2,3}
   \fill (c\i)  circle (2pt);
   \draw (c1) -- (c2) coordinate[pos=.5](f1) -- (d) coordinate[pos=.5](f2) -- (c3) coordinate[pos=.5](f3) -- (c4) coordinate[pos=.5](f4);
     \fill[red] (d)  circle (2pt);
      \draw (d) circle (2pt);

\node[above] at (e1) {$1$};
\node[below] at (f1) {$1$};
\node[above] at (e4) {$2$};
\node[below] at (f4) {$2$};

\node[above,rotate=30] at (e2) {$3$};
\node[below,rotate=150] at (e3) {$4$};
\node[above,rotate=30] at (f2) {$5$};
\node[below,rotate=150] at (f3) {$5$};
\end{scope}

\begin{scope}[xshift=1cm]
\fill[fill=black!10] (0,0) coordinate (Q) circle (1.5cm);

\coordinate (b1) at (-.25,-1/8);
\coordinate (b2) at (.25,1/8);
\draw (b1) -- (b2) coordinate[pos=.5](c);

\fill (b1)  circle (2pt);
\fill[white] (b2)  circle (2pt);
\draw (b2) circle (2pt);
\node[below,rotate=30] at (c) {$3$};
\node[above,rotate=30] at (c) {$4$};
\end{scope}

\end{tikzpicture}
\caption{Pluridifférentielle de $\Omega^{3}\moduli[0](6,2,-2;-6,-6)$ dont les $3$-résidus sont nuls.} \label{fig:polesdivkaveczero}
\end{figure}
\par
Dans le second cas, nous considérons une $k$-partie polaire d'ordre $2k$ associée à $(v_{1},\alpha v_{1}; v_{2}$ et une $k$-partie polaire d'ordre $2k$ associée à $(v_{2};v_{1},\alpha v_{1}$ telles que  l'angle au dessus du point d'intersection entre $v_{1}$ et $\alpha v_{1}$ est égal à $(k+\bar{a}_{1})\tfrac{2\pi}{k}$ et $v_{2}= v_{1}+\alpha v_{1}$. Le reste de la construction est similaire à précédemment, en considérant trois types de $k$-parties polaires, celles associées à $(v_{1};v_{1})$ que l'on colle cycliquement soit  à la première soit à la seconde $k$-partie polaire et celles  associées à $(v_{2};v_{2})$ que l'on colle cycliquement entre ces deux $k$-partie polaires. 
\end{proof}

Enfin nous montrons que s'il existe au moins trois zéros d'ordres non divisibles par $k$, alors l'origine est toujours dans l'image de l'application $k$-résiduelle.

\begin{lem}
Considérons la partition $\mu=(a_{1},\dots,a_{n};-b_{1},\dots,-b_{p})$. Si au moins trois des ordres $a_{1},\dots,a_{n}$ ne sont pas divisibles par $k$, alors l'application résiduelle $\appresk[0](\mu)$ contient $\lbrace 0 \rbrace$.
\end{lem}

\begin{proof}
Nous supposerons tout d'abord que $k\geq3$. Nous démontrons le lemme dans le cas $n=3$. L'éclatement des singularités coniques permet d'en déduire le cas général.
 Nous notons $a_{i} = kl_{i} + \bar{a}_{i}$ avec $0>\bar{a}_{i}>-k$ et $b_{i}=k\ell_{i}$. Si $l_{3}$ est strictement supérieur à~$p$, alors les constructions données dans la preuve du lemme~\ref{lem:condsufngeq3} permettent alors d'obtenir les $k$-différentielles avec les invariants souhaités.
 \par

Nous supposons maintenant que nous $l_{i}\leq p$ pour $i=1,2,3$. Dans ce cas, nous donnons deux constructions selon que $\bar{a}_{1}+\bar{a}_{2}+\bar{a}_{3}=-2k$ ou  $\bar{a}_{1}+\bar{a}_{2}+\bar{a}_{3}=-k$.
%
\par
Dans le cas où  $\bar{a}_{1}+\bar{a}_{2}+\bar{a}_{3}=-2k$, ces $k$-différentielles sont formées de deux triangles plats, dont les côtés sont reliés deux à deux par des chaînes de domaines polaires de valence~$2$. Dans l'une de ces chaîne les identifications sont des translations et dans les deux autres chaîne, ce sont des translations et une rotation.  Chaque zéro correspond à un sommet de chaque triangle et à l'un des deux sommets de chaque domaine polaire dans deux des trois chaînes.  Comme les sommets des triangles contribuent à chaque singularité d'un angle strictement inférieure à $2\pi$ cet angle est égal à $(k+\bar{a}_{i})\frac{2\pi}{k}$. Comme l'angle total des deux triangles est $2\pi$, nous en déduisons que $\bar{a}_{1}+\bar{a}_{2}+\bar{a}_{3}=-2k$. Chaque domaine polaire contribue aux angles de deux des trois singularités en ajoutant $\ell_{i}$ fois $2\pi$ à ceux-ci. Maintenant, on mets tout les domaines polaires sur les deux chaînes qui bordent~$z_{3}$, où l'on suppose $l_{1} \leq l_{2} \leq l_{3}$. Dans ce cas, chacun des $p$ domaines polaires contribuera d'un angle d'au moins $2\pi$ à chaque zéro. Donc la condition implique que $l_{1}+l_{2} \geq p$. Cette condition est vérifiée car comme $kl_{1}+kl_{2}+kl_{3} = \sum_{j=1}^{p} b_{j}$ implique que $l_{1}+l_{2}+l_{3} \geq 2p$ et donc si $l_{1}+l_{2} < p$ alors on aurait $l_{3}>p$.
\par
Pour finir, il reste à traiter le cas où $\bar{a}_{1}+\bar{a}_{2}+\bar{a}_{3}=-k$. Ce cas se traite de manière similaire au cas précédent en remplaçant chaque triangle par une partie polaire associées au complémentaire du triangle. 
\smallskip
\par
Dans le cas quadratique, il suffit de considérer le cas de $4$ zéros d'ordres impaires. Le cas avec plus de zéros d'ordres impaires se déduit par éclatement. Prenons une strate avec quatre zéros impairs.  On sait par le lemme~\ref{lem:condsufngeq3} que si l'origine peut s'obtenir en éclatant le zéro d'ordre pair d'une différentielle quadratique avec deux zéros impairs et un zéro d'ordre $a_{3}$  pair, alors  $a_{3} \geq 2p$. Comme les pôles sont d'ordres $b_{i}\geq 4$, on en déduit que $a_{1}+a_{2}+a_{3} \geq 4p - 4$.
Donc on peut obtenir tout les cas sauf ceux des strates de la forme $\Omega^{2}\moduli[0]((p-1)^4;(-4)^p)$ avec $p$ pair. 
\par
Dans le cas $\Omega^{2}\moduli[0](1,1,1,1;-4,-4)$ on considère deux $2$-parties polaires d'ordre $4$ associées à $(v_{1},v_{2};v_{3},v_{4})$ et  à $(v_{4},v_{2};v_{3},v_{1})$ respectivement, avec $v_{i}=1$. On colle les segments de même nom ensemble pour obtenir la différentielle quadratique souhaitée. Enfin pour obtenir les $2$-différentielles dans les strates   $\Omega^{2}\moduli[0]((p-1)^4;(-4)^p)$ avec $p$ pair, on considère les  $2$-parties polaires d'ordre $4$ associées à $(1;1)$. Puis on  colle cycliquement la moitié de ces $2$-parties polaires  aux segments $v_{1}$ et l'autre aux segments $v_{3}$.
\end{proof}

\subsection{Les strates avec $r=0$, $p\neq0$ et $s\neq0$.}
\label{sec:pasdenondivaveck}

Cette  section est dédiée à la preuve  du théorème~\ref{thm:r=0sneq0}, i.e.   aux strates  $\komoduli[0](a_{1},\dots,a_{n};-b_{1},\dots,-b_{p};(-k^{s}))$ avec $p\neq0$ et $s\neq0$. Rappelons que dans toute cette section, nous ne considérons que les strates non vide. D'après le lemme~\ref{lem:puissk}, cela revient à supposer que $\pgcd((a_{i}),k)=1$.
Enfin nous définissons tout au long de cette section $b_{i}:=k\ell_{i}$ et $a_{i}:=kl_{i}+\bar{a_{i}}$ avec $-k<\bar{a_{i}}\leq 0$.

On commence par traiter les cas des strates exceptionnelles suivantes.

\begin{lem} \label{lem:exceptquadpsneq0}
 L'image de l'application résiduelle de $\Omega^{2}\moduli[0](2s'-1,2s'+1;-4;(-2^{2s'}))$ avec $s'\geq1$ ne contient pas  $\CC^{\ast}\cdot(0;1,\dots,1)$.

L'image de l'application résiduelle des strates $\Omega^{2}\moduli[0](2a-1,2a+1;(-4^{a});(-2^{2}))$ avec $a\geq0$ ne contient pas $\CC^{\ast}\cdot(0,\dots,0;1,1)$.

  L'application $2$-résiduelle de la strate $\Omega^{2}\moduli[0](1+2s,1+2s;-4;(-2^{2s+1}))$ ne contient pas l'élément $(1;1\dots,1)$ pour tout $s\geq0$.
  
   L'application $2$-résiduelle de la strate $\Omega^{2}\moduli[0](2a-1,2a-1;(-4^a);-2)$ ne contient pas l'élément $(1,0,\dots,0;1)$ pour tout $s\geq0$.
\end{lem}

\begin{proof}
Supposons qu'une différentielle $\xi_{0}$ de $\Omega^{2}\moduli[0](2s'-1,2s'+1;-4;(-2^{2s'}))$ possède les résidus quadratiques $(0,1,\dots,1)$. Nous formons alors une différentielle entrelacée en attachant une courbe elliptique munie d'une différentielle quadratique holomorphe au pôle d'ordre $-4$ de $\xi_{0}$. Cette différentielle quadratique entrelacée peut être lissée sans modifier les résidus aux pôles doubles d'après le lemme~\ref{lem:lissdeuxcomp}. Cela aboutit à une contradiction avec le lemme~\ref{lem:nonsurjdeuxzero} où nous montrons (évidemment sans utiliser ce résultat) que les strates de la forme $\Omega^{2}\moduli[1](2s'-1,2s'+1;(-2^{2s'}))$ ne possèdent pas de différentielle avec résidu quadratique égal à un à tous les pôles.
\smallskip
\par
Le second cas se traite de manière similaire. Si $a\geq 1$ nous utilisons le lemme~\ref{lem:geq1six} disant que les strates $\Omega^{2}\moduli[1](2a-1,2a+1;(-4^{a}))$ ne contiennent pas de différentielle quadratique dont les résidus quadratique sont $(0,\dots,0)$.  Si $a=0$, nous utilisons le fait classique que la strate $\komoduli[1](1,-1)$ est vide. La contradiction s'obtient en considérant la différentielle entrelacée obtenue en collant ensemble les deux pôles d'ordre $-2$ d'une $2$-différentielle de la strate  $\Omega^{2}\moduli[0](2a-1,2a+1;(-4^{a});(-2^{2}))$.
\smallskip
\par
Nous considérons maintenant   les strates $\Omega^{2}\moduli[0](1+2s,1+2s;-4;(-2^{2s+1}))$ avec les $2$-résidus $(1;1,\dots,1)$. On suppose par l'absurde qu'il existe une différentielle quadratique $\eta$ possédant ces invariants. Alors par flot contractant (voir la section~\ref{sec:coeur}), on peut supposer que $\eta$ possède un \coeur dégénéré. En particulier, $\eta$ possède exactement $2s+2$ liens selles, connectant $2s+2$ domaines polaires. 
Le graphe d’incidence défini dans la section~\ref{sec:coeur} est donc formé d’un cycle contenant le pôle d'ordre $-4$ sur lequel sont colles des arbres. Ces arbres contribuent à une unique singularité conique et le bord de la surface correspondante est un unique lien selle d'holonomie entière. Comme les cylindres ont une circonférence égale à~$1$, ces arbres sont en fait formés d’un unique domaine polaire d’ordre~$2$ avec un lien selle de bord. Ceux-ci sont tous collés sur l’unique domaine de pole d’ordre $4$. 
Si $s\geq1$, alors nous pouvons simplifier ce graphe.  Dans la chaîne, les domaines polaires peuvent être éliminés par paire.  En effet, deux domaines polaires consécutifs possède un lien selle en commun et donc les liens selles aux bords respectifs de deux domaines possèdent la même holonomie. On peut donc supprimer les domaines et coller les deux segments au bord de la surface obtenue ensemble. On peut donc supposer que cette chaîne possède soit zéro soit un unique domaine polaire. 

S’il  y a un sommet sur la chaîne, alors pour obtenir les invariants souhaités le domaine polaire d’ordre $4$ doit être constitué d’une cicatrice où sont collés $2s$ cylindres et deux segments $\gamma,\gamma'$  de longueur $1/2$ de part et d’autre de la cicatrice qui relient ce domaine à un domaine d'ordre~$2$. Comme le résidu quadratique de ce pôle est égal à~$1$, la parité du nombre de cylindres de part et d’autre de la cicatrice est distincte. Cela est impossible car il y a un nombre impair de pôles d’ordre~$2$.
S’il  n'y a pas de sommet sur la chaîne, alors le domaine polaire d’ordre $4$ possède une cicatrice avec deux segments identifiés par rotation du même côté de la cicatrice. Entre ces deux segments il y a $k \leq s$ bords de cylindres et donc une  singularité conique est d’angle $(2k+1)\pi$. Cette singularité ne peut donc pas correspondre à un zéro d'ordre~$2s+1$.
\smallskip
\par
Nous considérons enfin les strates $\Omega^{2}\moduli[0](2a-1,2a-1;(-4^a);-2)$ avec les $2$-résidus $(1,0,\dots,0;1)$. On suppose par l'absurde qu'il existe différentielle quadratique $\eta$ possédant ces invariants. En utilisant le flot contractant, on peut supposer que $\eta$ possède un \coeur dégénéré. Donc $\eta$ possède $a+1$ liens selles, reliant $a+1$ domaines polaires. Le graphe d’incidence de $\eta$ admet au plus deux sommets de valence $1$ correspondant aux pôles  dont le résidu quadratique est non nul. Donc ce graphe est un cycle sur lequel sont colles des arbres. Les domaines polaires correspondant aux sommets de ces arbres sont bordés par une unique singularité conique. De plus, ces domaines polaires sont bordés par un unique lien selle de longueur entière. Le cylindre correspondant au pôle d’ordre $-2$ ayant une circonférence égale à $1$, les arbres sont collés sur des sommets correspondant à des pôles d'ordre~$-4$. Nous distinguons les cas selon s’il existe $0$, $1$ ou $2$ sommets de valence $1$ dans le graphe d'incidence.
\par
S’il n’en existe aucun, le graphe d’incidence est un cycle. Les domaines polaires d’ordre~$4$ correspondant a des sommets de valence $2$ peuvent être supprimés, puis on recolle les bords ensemble. On obtient ainsi une différentielle quadratique de $\Omega^{2}\moduli[0](1,1;-4;-2)$ dont les résidus quadratiques sont $(1,1)$. Cette différentielle  n’existe pas par le cas précédent de ce lemme.
S’il existe un sommet de valence $1$, alors le cycle du graphe d'incidence contient exactement un sommet de valence~$3$. Considérons le domaine polaire dont le $2$-résidu est non nul ne correspondant pas au sommet de valence~$1$ et le sommet correspondant~$v_{0}$. Si $v_{0}$ n’appartient pas au cycle ou si $v_{0}$ est le sommet de valence $3$, alors, après avoir supprimer les sommets du cycle comme précédemment, on peut supposer que le cycle possède uniquement le sommet~$v_{0}$. Le domaine polaire correspondant admet une cicatrice avec trois segments dont deux sont identifiés par rotation. Donc une singularité conique  possède un angle inférieur ou égal à~$2\pi$, ce qui est absurde. Si le sommet $v_{0}$ appartient au cycle mais n'est pas le sommet de valence~$3$, alors, après suppression des sommets du cycle, on obtient un cycle avec deux sommets. La somme et la différence de l'holonomie des liens selles au bord de ce domaine polaire sont égales à~$1$, ce qui est absurde.
Enfin, S’il  existe deux sommet de valence~$1$, alors le cycle est  constitué de domaines polaires d’ordre $4$ dont le résidu est nul. On élimine les sommets de valence $2$ et on obtient un cycle constitué soit de deux sommets de valence 3, soit d'un unique sommet de valence~$4$. Dans ce dernier cas, le domaine polaire correspondant possède deux segments identifiés par rotation du même cote de la cicatrice.  Les deux autres segments sont de longueur $1$, car le résidu du pôle correspondant est nul. Cela implique que les deux segments identifies sont adjacents et qu’une des singularités coniques est d’angle $\pi$. Si le cycle possède deux sommets de valence~$3$, alors les longueurs des deux segments qui les relient sont de somme et de différence égales à~$1$, ce qui est impossible comme précédemment. 
\end{proof}

%

Nous considérons maintenant le cas des strates ayant à la fois des pôles d'ordres~$-k$ et des pôles d'ordres inférieurs divisibles par~$k$.

\begin{lem}\label{lem:surjr=0sneq01}
 Les applications résiduelles des strates $\komoduli[0](a_{1},a_{2};-k\ell;-k)$ avec $\ell\geq2$ sont surjectives sauf dans le cas $\Omega^{2}\moduli[0](1,1;-4,-2)$ où l'application $2$-résiduelle contient le complémentaire de $(1;1)$.
\end{lem}

\begin{proof}
 Nous traitons tout d'abord le cas où $\ell=2$. Soit il existe $a_{i}$, disons $a_{1}$ tel que $-k<a_{1}<0$, soit on a $0<a_{1},a_{2}<k$. On dénote le $k$-résidu~$R$, resp.~$R_{1}$, au pôle d'ordre~$-2k$, resp.~$-k$. Quitte à multiplier les $k$-résidus par une constante, on peut supposer que $R_{1}=1$ et on considérera la racine $r_{1}=1$ de~$R_{1}$. Les trois cas ci-dessous sont représentés dans la figure~\ref{fig:surjr=0sneq01}.
 
 Dans le cas où $-k<a_{1}<0$, on associe au pôle d'ordre $-2k$ la $k$-partie polaire d'ordre~$2k$ associée à $(v_{1},v_{2};r_{1})$ telle que l'angle au point d'intersection de la concaténation des $v_{i}$ soit $2\pi + \tfrac{2a_{1}\pi}{k}$ et que la longueur des $v_{i}$ soit identique et $v_{1}+v_{2}-r_{1}$ est égale à une racine de $R$ dont la partie réelle est positive. On colle alors au segment $r_{1}$ une $k$-partie polaire d'ordre $k$ associée à~$r_{1}$ et les $v_{i}$ ensemble. La surface obtenue vérifie clairement les propriétés souhaitées.
 
 Dans le cas où $0<a_{1},a_{2}<k$ la construction est la suivante. Si $|R|>1$  on considère la $k$-partie polaire d'ordre $2k$ associée à $(\emptyset; v_{1},r_{1},v_{2})$ telle que les $v_{i}$ soient d'égale longueur, la somme des angles aux sommets au bord de $r_{1}$ égale à $\pi+ \tfrac{2a_{1}\pi}{k}$ et la somme est l'opposée d'une racine de $R$ de partie réelle négative (ou de partie imaginaire négative si $R$ est un réel négatif dans le cas quadratique).    On colle alors au segment $r_{1}$ une $k$-partie polaire d'ordre $k$ associée à~$r_{1}$ et les $v_{i}$ ensemble. La surface obtenue vérifie clairement les propriétés souhaitées.
 
 Si $|R|\leq1$, on considère la $k$-partie polaire d'ordre $k$ associée à $(v_{1},v_{2})$ avec $v_{1}+v_{2}=1$ et $-\exp \left( \tfrac{2a_{1}i\pi}{k} \right) v_{2} - v_{1}$ est une racine~$r$ de~$R$.  Comme la solution de ces équations est donnée par $v_{1}= (\exp \left( \tfrac{2a_{1}i\pi}{k} \right)+r) /(\exp \left( \tfrac{2a_{1}i\pi}{k} \right) -1)$, on peut choisir $r$ telle que $v_{1}$ est de partie imaginaire strictement négative ou dans le segment $\left] 0,1\right[$ sauf dans le cas exceptionnel du lemme~\ref{lem:surjr=0sneq01}. On considère maintenant la $k$-partie polaire d'ordre $2k$ associée à $\left(-\exp \left( \tfrac{2a_{1}i\pi}{k} \right) v_{2};v_{1}\right)$. En collant les $v_{i}$ entre eux par rotations et translations on obtient une surface plate ayant les invariants souhaités. En effet, considérons la singularité correspondant au point d'intersection entre $v_{1}$ et $v_{2}$ dans la partie polaire d'ordre~$k$ (le point noir sur la figure~\ref{fig:surjr=0sneq01} à droite).  Comme $v_{1}$ est identifié par translation et~$v_{2}$ par une rotation d'angle $2a_{1}\pi/k$ (de la partie polaire d'ordre $k$ à celle d'ordre $2k$) on en déduit que l'angle de cette singularité est $2\pi + 2a_{1}\pi/k$. C'est donc un zéro d'ordre $a_{1}$ comme souhaité.
 \begin{figure}[hbt]
\center
\begin{tikzpicture}[scale=.8]

\begin{scope}[xshift=-5cm]
 \fill[fill=black!10] (0,0) ellipse (1.5cm and 1.2cm);
   
\coordinate (A) at (-1,0);
\coordinate (B) at (1,0);
\coordinate (C)  at (0,1);
    
\draw [dotted] (A)-- +(-.5,0);
\draw [dotted] (B)-- +(.5,0);
 
\filldraw[fill=white] (A)  circle (2pt);
\filldraw[fill=white] (B) circle (2pt);
\fill (C)  circle (2pt);

\fill[color=white]     (A) -- (B) -- (C) --cycle;
\draw[]     (A) -- (B) coordinate[pos=.5](a);
\draw (B) -- (C) coordinate[pos=.3](b);
\draw (C) -- (A)  coordinate[pos=.6](c);

\node[below] at (a) {$r_{1}$};
\node[above,rotate=45] at (c) {$v_{1}$};
\node[above,rotate=-45] at (b) {$v_{2}$};

\begin{scope}[xshift=-3cm,yshift=-2.5cm]
\coordinate (a) at (2,-.5);
\coordinate (b) at (4,-.5);

    \fill[fill=black!10] (a)  -- (b)coordinate[pos=.5](f) -- ++(0,1.2) --++(-2,0) -- cycle;
 \draw  (a) -- (b);
 \draw (a) -- ++(0,1.1) coordinate (d)coordinate[pos=.5](h);
 \draw (b) -- ++(0,1.1) coordinate (e)coordinate[pos=.5](i);
 \draw[dotted] (d) -- ++(0,.2);
 \draw[dotted] (e) -- ++(0,.2);
     \filldraw[fill=white] (a)  circle (2pt);
\filldraw[fill=white] (b) circle (2pt);
\node[above] at (f) {$r_{1}$};
\end{scope}
\end{scope}

\begin{scope}[xshift=-1cm]
\clip (-4,1) rectangle (4,-3);
\foreach \i in {1,2,3}
\coordinate (a\i) at (\i,0); 
\fill[black!10] (a1) -- (a3) coordinate[pos=.25] (b1) coordinate[pos=.75] (b2)  -- ++(0,0) arc (0:180:1) -- cycle;

\draw (a1) -- (a3) ;
\fill (a2)  circle (2pt);
\node[above] at (b1) {$1$};
\node[above] at (b2) {$2$};

     \foreach \i in {0,1,...,4}
  \coordinate (c\i) at (\i,-2); 
    \coordinate (d1) at (1.5,-1.5);
     \coordinate (d2) at (2.5,-1.5); 
    \coordinate (e1) at (1.5,-.6);
     \coordinate (e2) at (2.5,-.6); 
     
   \fill[black!10] (c0) -- (c1) coordinate[pos=.5] (b1) -- (d1) coordinate[pos=.5] (b3) --(e1) --(e2) -- (d2) -- (c3)  coordinate[pos=.5] (b4) -- (c4) coordinate[pos=.5] (b2) -- ++(0,0) arc (0:-180:2) -- cycle;
     \foreach \i in {1,3}
   \fill (c\i)  circle (2pt);
   \draw (c0) -- (c1) coordinate[pos=.5](f1) -- (d1) -- (e1);
     \draw (e2) -- (d2) -- (c3) coordinate[pos=.5](f4) -- (c4);
     \draw[dotted] (d1) -- (d2) coordinate[pos=.5] (b5);
     \foreach \i in {1,2}
   \filldraw[fill=white] (d\i)  circle (2pt);     

\node[below] at (b1) {$1$};
\node[below] at (b2) {$2$};
\node[above,rotate=45] at (b3) {$v_{1}$};
\node[above,rotate=-45] at (b4) {$v_{2}$}; 
\node[below] at (b5) {$r_{1}$};
\end{scope}

\begin{scope}[xshift=6cm]
    \fill[fill=black!10] (0,0) ellipse (1.5cm and 1.2cm);
    
\coordinate (A) at (-.5,-.5);  
\coordinate (B) at (0,-1);
\coordinate (C) at (-.5,.5);

\fill[white] (A) -- (B) -- (C)  -- ++(-1.2,0) -- ++(0,-1) -- cycle;
     
    \draw [dotted] (A)-- +(-.9,0);
    \draw [dotted] (B)-- +(.8,0);
   \draw [dotted] (C)-- +(-.9,0);

 \draw[]     (A) -- (B) coordinate[pos=.5](a) -- (C) coordinate[pos=.5](b);
\node[below,rotate=-45] at (a) {$v_{1}$};\node[above,rotate=-45] at (b) {$v_{2}$};

 \filldraw[fill=white] (A)  circle (2pt);
\fill (B) circle (2pt);
 \filldraw[fill=white] (C)  circle (2pt);

\coordinate (a) at (2,-.5);
\coordinate (b) at (4,-.5);
\coordinate (c) at (2.5,-1);

    \fill[fill=black!10] (a)  -- (c)coordinate[pos=.5](f1)-- (b)coordinate[pos=.5](f2) -- ++(0,1.2) --++(-2,0) -- cycle;
 \draw  (a) -- (c) -- (b);
 \draw (a) -- ++(0,1.1) coordinate (d)coordinate[pos=.5](h);
 \draw (b) -- ++(0,1.1) coordinate (e)coordinate[pos=.5](i);
 \draw[dotted] (d) -- ++(0,.2);
 \draw[dotted] (e) -- ++(0,.2);
     \filldraw[fill=white] (a)  circle (2pt);
\filldraw[fill=white] (b) circle (2pt);
\fill (c) circle (2pt);

\node[above,rotate=-45] at (f1) {$v_{1}$};
\node[above,rotate=45] at (f2) {$v_{2}$};
\end{scope}
\end{tikzpicture}
\caption{Différentielles quartiques dans $\Omega^{4}\moduli[0](-1,5;-8;-4)$ à gauche et dans $\Omega^{4}\moduli[0](1,3;-8;-4)$ au centre et à droite. Les résidus quartiques sont respectivement $(0,1)$, $(16,1)$ et $(1/8;1)$.} \label{fig:surjr=0sneq01}
\end{figure}

 \smallskip
 \par
 Nous traitons maintenant le cas où $\ell\geq 3$. On part  d'une surface plate  ayant les $k$-résidus souhaités dans la strate $\komoduli[0](a_{1}',a_{2}';-2\ell;-k)$  telle que $a_{i}=a_{i}' +l_{i}'k$ avec $l_{i}'\geq0$. On coupe alors une demi-droite partant du zéro d'ordre $a_{i}'$ et on colle de manière cyclique $l_{i}'$ plans possédant la même fente. 
\end{proof}

Nous déterminons maintenant l'image de l'application résiduelle des strates de la forme $\omoduli[0](a_{1},a_{2};-k\ell;-k,-k)$.
\begin{lem}\label{lem:surjr=0sneq02}
L'image de l'application résiduelle de la strate $\Omega^{2}\moduli[0](3,1;-4;-2,-2)$ contient 
 $\espresk[0](3,1;-4;-2,-2)\setminus \CC^{\ast}\cdot (0;1,1)$.
 Les applications résiduelles des autres strates de la forme $\komoduli[0](a_{1},a_{2};-k\ell;-k,-k)$ 
 sont surjectives pour tout $\ell\geq2$.
\end{lem}
 
\begin{proof}
 Nous commençons par traiter le cas $\ell=2$. Dans le cas où l'un des zéros d'ordre $-k<a_{i}<0$, la construction est identique à celle du lemme~\ref{lem:surjr=0sneq01} et nous laissons les modifications triviales au lecteur. Nous nous concentrerons donc sur le cas où $0<a_{2}<k<a_{1}<2k$. Dans ce cas, nous séparons les cas où $R_{1} \neq (-1)^{k}R_{2}$ et $R_{1} = (-1)^{k}R_{2}$.
 \par
 Dans le premier cas  nous supposons que $|R_{1}| \geq |R_{2}|$. Nous commençons par choisir des racines $r_{i}$ telles que $r_{1}$, resp $r_{2}$, est de partie réelle négative, resp. positive (ou de partie imaginaire négative, resp. positive, si la partie réelle est nulle). Nous prenons une $k$-partie polaire d'ordre $-2k$ associée à $(-r_{1};v_{1},r_{2},\alpha v_{1})$ telle que $\alpha=\exp\left((2a_{1}-k)i\pi/k\right)$ et la somme $r_{1}+v_{1}+r_{2}+\alpha v_{1}$ est égale à l'opposé d'une racine $r$ de~$R$. La solution de cette équation est $v_{1}=-\tfrac{r+r_{1}+r_{2}}{1+\alpha}$. On peut clairement choisir $r$ telle que la concaténation de $v_{1}$ avec $r_{2}$ et $\alpha v_{1}$ soit sans point d’intersection. La surface avec les invariants souhaités s'obtient alors en collant les segments $v_{i}$ entre eux et des $k$-parties polaires d'ordre $k$ associées à $r_{i}$ aux segments correspondants.
 \par
 Dans le cas où $R_{1} = (-1)^{k}R_{2}$ et que $R\neq0$ la même construction que précédemment donne la $k$-différentielle souhaitée. Il reste donc à traiter le cas des $k$-résidus égaux à $(0;1,(-1)^{k})$ dans $\komoduli[0](a_{1},a_{2};-k\ell;-k,-k)$ avec $0<a_{2}<k<a_{1}<2k$ pour $k\geq3$.  Dans les cas $(a_{1},a_{2})\neq(k+1,k-1)$, la construction est similaire à la construction du paragraphe précédent. Nous prenons une $k$-partie polaire d'ordre $-2k$ associée à $(-r_{1};v_{1},1,\alpha v_{1})$ telle que $\alpha=\exp\left((2a_{1}-k)i\pi/k\right)$, la racine $r_{1}$ est égale respectivement à $\exp(2i\pi/k)$ si $k$ est pair et $\exp(i\pi/k)$ si $k$ est impair et tels que $1+v_{1}+r_{1}+\alpha v_{1}=0$. On a alors $v_{1}=-\tfrac{1+r_{1}}{1+\alpha}$ et on peut vérifier que la concaténation de $v_{1}$ avec $1$ et $ \alpha v_{1}$ est sans point d'intersection. On peut alors conclure la construction comme précédemment.
Enfin, dans le cas où $(a_{1},a_{2})=(k+1,k-1)$, on procède à la construction analogue avec la $k$-partie polaire d'ordre $-2k$ associée à $(\exp(2i\pi/k),1;\exp(2i\pi/k),1)$. On colle par rotation l'un des segments $1$ à l'un des segments $\exp(2i\pi/k)$ et des demi-cylindres infinis aux autres segments. On vérifiera alors facilement que la différentielle obtenue possède les invariants souhaités.
 \smallskip
 \par
 Pour montrer la surjectivité des applications résiduelles $\appresk[0](a_{1},a_{2};-k\ell;-k,-k)$, on peut essentiellement procéder comme dans la fin de la preuve du lemme~\ref{lem:surjr=0sneq01}. En effet, 
 on peut partir des surfaces obtenues aux paragraphes précédents et ajouter des multiples de $k$ aux zéros de ces strates. Pour cela on coupe ces surfaces le long de demi-droites infinies partant des singularités et on recolle $\ell-2$ plans de manière à obtenir les ordres souhaités. De cette façon on obtient toutes ces strates sauf  la strate $\Omega^{2}\moduli[0](3,3;-6;-2,-2)$.    De plus, cette technique donne que l'image de $\appresk[0][2](3,3;-6;-2,-2)$ contient $\espresk[0][2](3,3;-6,-2,-2)\setminus \CC^{\ast}\cdot (0,1,1)$. Une différentielle quadratique ayant ces ordres et dont les $k$-résidus sont $(0,1,1)$ est représentée dans la figure~\ref{fig:surjr=0sneq02}.
  \begin{figure}[hbt]
\center
\begin{tikzpicture}[scale=1]

\begin{scope}[xshift=-2cm]
    \fill[fill=black!10] (0,0) ellipse (1.5cm and 1cm);

\draw[] (-.7,0) coordinate (q1) -- (.3,0) coordinate (q2) coordinate[pos=.5](a);

\node[above] at (a) {$r_{1}$};
\node[below] at (a) {$v_{1}$};

\draw[] (q2) -- (1.5,0) coordinate[pos=.5](e);

\node[above] at (e) {$1$};
\node[below] at (e) {$2$};

\fill (q1) circle (2pt);
\fill (q2) circle (2pt);
\filldraw[fill=white] (.37,0)  arc (0:-180:2pt); 
\end{scope}

\begin{scope}[xshift=2cm]
    \fill[fill=black!10] (0,0) ellipse (1.5cm and 1cm);

\draw[] (-.7,0) coordinate (q1) -- (.3,0) coordinate (q2) coordinate[pos=.5](a);

\node[above] at (a) {$r_{2}$};
\node[above,rotate=180] at (a) {$v_{1}$};

\draw[] (q2) -- (1.5,0) coordinate[pos=.5](e);

\node[below] at (e) {$1$};
\node[above] at (e) {$2$};

\filldraw[fill=white] (q1) circle (2pt);
\fill (q2) circle (2pt);
\filldraw[fill=white] (.37,0)  arc (0:180:2pt); 
\end{scope}

\begin{scope}[xshift=-6cm,yshift=.5cm]
\coordinate (a) at (0,0);
\coordinate (b) at (1,0);

    \fill[fill=black!10] (a)  -- (b)coordinate[pos=.5](f) -- ++(0,-1.3) --++(-1,0) -- cycle;
 \draw  (a) -- (b);
 \draw (a) -- ++(0,-1.2) coordinate (d)coordinate[pos=.5](h);
 \draw (b) -- ++(0,-1.2) coordinate (e)coordinate[pos=.5](i);
 \draw[dotted] (d) -- ++(0,-.2);
 \draw[dotted] (e) -- ++(0,-.2);
     \fill (a)  circle (2pt);
\fill (b) circle (2pt);
\node[below] at (f) {$r_{1}$};
\end{scope}

\begin{scope}[xshift=5cm,yshift=.5cm]
\coordinate (a) at (0,0);
\coordinate (b) at (1,0);

    \fill[fill=black!10] (a)  -- (b)coordinate[pos=.5](f) -- ++(0,-1.3) --++(-1,0) -- cycle;
 \draw  (a) -- (b);
 \draw (a) -- ++(0,-1.2) coordinate (d)coordinate[pos=.5](h);
 \draw (b) -- ++(0,-1.2) coordinate (e)coordinate[pos=.5](i);
 \draw[dotted] (d) -- ++(0,-.2);
 \draw[dotted] (e) -- ++(0,-.2);
     \filldraw[fill=white] (a)  circle (2pt);
\filldraw[fill=white] (b) circle (2pt);
\node[below] at (f) {$r_{2}$};
\end{scope}
\end{tikzpicture}
\caption{Une différentielle quadratique dans $\Omega^{2}\moduli[0](3,3;-6;-2,-2)$ dont les résidus quadratiques sont $(0;1,1)$.} \label{fig:surjr=0sneq02}
\end{figure}
\end{proof}

Nous considérons maintenant le cas des strates ayant un unique pôle d'ordre $-k\ell$ et un nombre arbitraire de pôles d'ordre~$-k$.

 \begin{lem}\label{lem:surjr=0sneq03}
Les applications résiduelles des strates $\komoduli[0](a_{1},a_{2};-k\ell;(-k^{s}))$ avec $s\geq3$ distinctes des strates $\Omega^{2}\moduli[0](2s'-1;2s'+1;-4;(-2^{2s'}))$ avec $s=2s'$ et  $\Omega^{2}\moduli[0](2s'+1;2s'+1;-4;(-2^{2s'+1}))$ avec $s=2s'+1$ sont surjectives.
L'image de l'application résiduelle  de $\Omega^{2}\moduli[0](2s'-1;2s'+1;-4;(-2^{2s'}))$  contient le complémentaire de $\CC^{\ast}\cdot(0;1,\dots,1)$ dans l'espace résiduel et celle de $\Omega^{2}\moduli[0](2s'+1;2s'+1;-4;(-2^{2s'+1}))$  contient le complémentaire de $\CC^{\ast}\cdot(1;1,\dots,1)$
\end{lem}

Afin de prouver ce résultat, nous utiliserons un résultat  similaire à celui du lemme~\ref{lem:nonintersecp} que nous énonçons maintenant.
\begin{lem}\label{lem:nonintersecp2}
 Soient $k\geq2$, $\zeta$ une racine $k$-ième primitive de l'unité,  $R_{1},R_{2}\in\CC^{ \ast}$ non tous deux nuls. Il existe des racines $k$-ième  $r_{i}$ respectivement de $R_{i}$ et il existe $v_{1} \in \CC^{\ast}$ tels que $v_{1} +  r_{1}  - \zeta v_{1} - r_{2} =0$, tels que les $k$-parties polaires (généralisées) associées respectivement aux vecteurs $(v_{1},r_{1};r_{2},\zeta v_{1})$  et $(v_{1},r_{1};\zeta v_{1},r_{2})$  existent.  
\end{lem}

\begin{proof}
La preuve est en tout point similaire à celle du lemme~\ref{lem:nonintersecp}. Nous donnerons ici les grandes lignes de l'argument.
  On choisit $r_{1}$ et $r_{2}$ avec une partie réelle positive.  Il existe alors un unique $v_{1}\in\CC^{\ast}$ tel que $v_{1} +  r_{1}  - \zeta v_{1} - r_{2} =0$. Considérons la concaténation de~$v_{1}$ puis de $r_{1}$. Si $v_{1}$ et $r_{1}$ sont proportionnels mais de sens opposé, nous considérons que~$v_{1}$ est en dessous de $r_{1}$. Puis nous formons le demi-plan à bord polygonal au dessus de ces vecteurs et des demi-droites partant des points initial de $v_{1}$ et final de $r_{1}$. On fait la même construction pour le demi-plan inférieur à la concaténation de $r_{2}$ et $v_{2}$. En collant les deux demi-plans on obtient la $k$-partie polaire souhaitée. Cette construction est présentée sur la figure~\ref{fig:excosntrudiff}.
  \begin{figure}[hbt]
\center
\begin{tikzpicture}

\begin{scope}
  \clip (-1,-.5) rectangle (10,1.5);
\draw (0,0) coordinate (a1) --  (0:1.5) coordinate[pos=.5] (b1) coordinate (a2) --  ++(0:1) coordinate[pos=.5] (b2) coordinate (a3) -- ++(120:1)coordinate[pos=.5] (b3) coordinate (a4) -- ++(0:1.5) coordinate[pos=.5] (b4) coordinate (a5); 
 \fill[black!10] (a1) -- (a2) -- (a3) -- (a4) -- (a5) --++ (0:.7) arc (0:195:2.2) -- cycle;  
\draw (0,0) coordinate (a1) --  (a2) --   (a3) --  (a4) --  (a5); 
 
\filldraw[fill=white] (a2) circle (1pt);
\fill (a3) circle (1pt);
\fill (a4) circle (1pt);

\draw[dotted] (a1)-- ++(180:.4);
\draw[dotted] (a5)-- ++(0:.4);

\node[below] at (b1) {$1$};
\node[above] at (b4) {$2$};
 \node[below] at (b2) {$v_{1}$};
 \node[above,rotate=-60] at (b3) {$r_{1}$};
\end{scope}

 \begin{scope}[xshift=7cm,yshift=.5cm]
\clip (-1,1.5) rectangle (10,-.5);
\draw (0,0) coordinate (a1) --  (0:1.5) coordinate[pos=.5] (b1) coordinate (a2) --  ++(0:1) coordinate[pos=.5] (b2) coordinate (a3) -- ++ (120:1)coordinate[pos=.5] (b3) coordinate (a4) -- ++(0:1.5) coordinate[pos=.5] (b4) coordinate (a5); 
 \fill[black!10] (a1) -- (a2) -- (a3) -- (a4) -- (a5) --++ (0:.7) arc (10:-165:2.2) -- cycle;  
\draw (0,0) coordinate (a1) --  (a2) --   (a3) --  (a4) --  (a5); 
 
\filldraw[fill=white] (a2) circle (1pt);
\filldraw[fill=white]  (a3) circle (1pt);
\fill (a4) circle (1pt);

\draw[dotted] (a1)-- ++(180:.4);
\draw[dotted] (a5)-- ++(0:.4);

\node[above] at (b1) {$1$};
\node[above] at (b4) {$2$};
 \node[below,rotate=-60] at (b3) {$\zeta v_{1}$};
 \node[below] at (b2) {$r_{2}$};
 \end{scope}

\end{tikzpicture}
\caption{La  $3$-partie polaire associée aux vecteurs $(v_{1},r_{1};r_{2},\exp(2i\pi/3)v_{1}) $ égaux à $ (\tfrac{2\sqrt{3}}{3} \exp(i\pi/6),-1;1,\tfrac{2\sqrt{3}}{3} \exp(5i\pi/6))$} \label{fig:excosntrudiff}
\end{figure}
\end{proof}

Avant de passer à la preuve du lemme~\ref{lem:surjr=0sneq03}, nous faisons une remarque importante.
\begin{rem}\label{rem:importexist2}
 Notons que dans le cas où $k=2$ et  $R_{1}=R_{2}=1$, alors la première $2$-partie polaire donne le carré d'une différentielle quadratique et dans la seconde donne une différentielle quadratique avec un zéro d'ordre~$-1$. On ne peut pas obtenir de cette façon une différentielle quadratique dont la somme des angles entre $r_{1}$ et respectivement $v_{1}$ et $\zeta v_{1}$ soit égal à~$2\pi$. Cela est la raison pour laquelle le théorème~\ref{thm:r=0sneq0} possède des cas exceptionnels.
\end{rem}

Nous pouvons maintenant procéder à la preuve du lemme~\ref{lem:surjr=0sneq03}.
\begin{proof}
  Nous traiterons le cas $\ell=2$, les cas $\ell\geq3$ s'obtenant comme dans les lemmes~\ref{lem:surjr=0sneq01} et~\ref{lem:surjr=0sneq02}.
 Rappelons que nous noterons $a_{i}=kl_{i}+\bar{a_{i}}$ avec $-k<\bar{a_{i}}<0$ et $l_{1}+l_{2}=s+1$.
 \par
  Si $a_{2}<0$, alors on considère la partie polaire d'ordre $2k$ associée à $(v_{1},\alpha v_{1};r_{1},\dots,r_{s})$ où les $r_{i}$ sont des racines des $R_{i}$ de partie réelles positives, le coefficient $\alpha=\exp\left(-(2a_{2}+k)i\pi/k\right)$ et la somme $\sum r_{i}- (1+\alpha) v_{1}$ égale à une racine~$r$ du $k$-résidu $R$ au pôle d'ordre~$-2k$. Puis on fait les collages similaire à ceux des lemmes précédents.  Nous considérerons par la suite les cas $a_{i}>0$.
  \par
  On partitionne les $k$-résidus en deux ensembles $S_{1}\sqcup S_{2}$ de cardinaux $l_{1}$ et $l_{2}-1$ satisfaisant aux conditions suivantes.  Nous notons par $R_{i}$ le $k$-résidu d'un pôle d'ordre $-k$ appartenant à $S_{1}$ si $i\leq l_{1}$ et à $S_{2}$ si $i>l_{1}$. On choisit une racine $r_{i}$ des $k$-résidus que l'on supposera de partie réelle négative pour les pôles de~$S_{1}$ et positive pour ceux de~$S_{2}$. On choisit une racine $r$ du $k$-résidu $R$ du pôle d'ordre $-2k$, telles que les hypothèses du lemme~\ref{lem:nonintersecp} soit satisfaite. Notons $\tilde{r}_{1}= \sum_{i=1}^{l_{1}} r_{i}$ et $\tilde{r}_{2}=\sum_{i=l_{1}+1}^{s}$.   Nous considérons alors la $k$-partie polaire d'ordre $2k$ associée à $(v_{1},-\tilde{r}_{1},v_{2}; \tilde{r}_{2},\zeta v_{1})$, avec  $v_{1}$ comme dans le lemme~\ref{lem:nonintersecp} si $r=0$ et le lemme~\ref{lem:nonintersecp2} si $r\neq0$.  On ajoute deux polygones dont les arêtes sont respectivement  $(\tilde{r}_{1},-r_{1},\dots,-r_{l_{1}})$ et  $(-\sum_{i=l_{1}+1}^{s} r_{i},r_{l_{1}+1}\dots,r_{s})$ que l'on colle à la $k$-partie polaire. Par exemple, pour obtenir les $3$-résidus $(1,1,1)$ dans la strate $\Omega^{3}\moduli[0](5,4;-6;(-3^{3}))$, il suffit de coller à la $3$-partie polaire de la figure~\ref{fig:excosntrudiff} un triangle équilatéral à $r_{1}$. Notons que la projection de chaque surface peut avoir des points d'auto-intersection dans le plan. Ce n'est donc pas une $k$-partie polaire, mais la surface ainsi obtenue est bien définie. Enfin, on identifie le segment $v_{1}$ à $\zeta v_{1}$  par rotation (et translation) et on colle des demi-cylindres infinis aux autres segments. Cette construction permet d'obtenir les $k$-différentielles avec les invariants locaux souhaités sauf dans le cas quadratique si $(\tilde{r}_{1},\tilde{r}_{2},r)$ est soit de la forme $(x+1,x,1)$ (modulo multiplication par une constante), soit de la forme $(x,x,0)$ (voir les remarques~\ref{rem:importexist} et~\ref{rem:importexist2}).   Nous traitons maintenant ces cas.
  \par
 Il est facile d'adapter la preuve si tous les $2$-résidus ne sont pas positivement proportionnels. On supposera donc que les $2$-résidus non nuls sont réels mais non tous identiques. Dans le cas  $(x+1,x,1)$, supposons tout d'abord les $R_{j}$ ne soient pas tous égaux pour $j \leq l_{1}$. Dans ce cas, on suppose que $R_{1} > R_{2}$. Pour le pôle $p$, on utilise alors les vecteurs $- r_{1}+r_{2},-r_{3},\dots, r_{l_{1}}$. Pour les pôles $1$ et $2$ on considère les parties $2$-polaires d'ordre $2$ respectivement associées à $(\emptyset;- r_{1}+r_{2},-r_{2} )$ et $(-r_{2};\emptyset )$. Cela permet d'utiliser le lemme~\ref{lem:nonintersecp} et de conclure en utilisant la construction du paragraphe précédent. Considérons le cas où tous les $r_{j}$ sont égaux entre eux dans chacune des deux sommes. Dans ce cas, on considère la construction du lemme~\ref{lem:surjr=0sneq02} et on y colle les demi-cylindres.
 \par
 Enfin il reste à considérer les cas $(x,x,0)$. Si tous les $2$-résidus ne sont pas égaux entre eux, alors la première construction du paragraphe précédent permet de conclure.
\end{proof}

Enfin nous traitons le cas général avec deux zéros.

\begin{lem}\label{lem:surjr=0sneq04}
Les applications résiduelles des strates $\komoduli[0](a_{1},a_{2};-k\ell_{1},\dots,-k\ell_{p};(-k^{s}))$ avec $p\geq2$ et $s\neq0$ distinctes des strates  quadratiques $\Omega^{2}\moduli[0](2a-1;2a+1;(-4^{a});(-2^{2}))$ et $\Omega^{2}\moduli[0](2a-1,2a-1;(-4^a);-2)$  sont surjectives.
L'image de l'application résiduelle de la strate $\Omega^{2}\moduli[0](2a-1;2a+1;(-4^{a});(-2^{2}))$ contient le complémentaire de $\CC^{\ast}\cdot(0,\dots,0;1,1)$ dans l'espace résiduel et celle de  $\Omega^{2}\moduli[0](2a-1,2a-1;(-4^a);-2)$ contient le complémentaire de $\CC^{\ast}\cdot(1,0,\dots,0;1)$.
\end{lem}

\begin{proof}
 Nous commençons en supposant que tous les $k$-résidus aux pôles $P_{i}$ d'ordres $-k\ell_{i}$ sont distincts de zéro. Avec la notation usuelle $a_{i}=kl_{i}+\bar{a_{i}}$, on a $l_{1}+l_{2}+1=s + \sum_{i=1}^{p} \ell_{i}$. On déduit du fait que $s\geq1$ et $p\geq2$ qu'il existe un $\ell_{i}$ tel que $\ell_{i}\leq \max\left\{l_{1},l_{2}\right\}$. Par récurrence et quitte à permuter les indices, on peut partitionner les pôles en trois ensembles $E_{0}=\left\{P_{1}\right\}$, $E_{1}=\left\{P_{2},\dots,P_{t}\right\}$ et $E_{2}=\left\{P_{t+1},\dots,P_{p}\right\}$ tels que $\sum_{i=2}^{t}\ell_{i} \leq l_{1}$ et $\sum_{i=t+1}^{p}\ell_{i} \leq l_{2}$. On fait alors une construction similaire à celles des lemmes~\ref{lem:surjr=0sneq01},~\ref{lem:surjr=0sneq02} et~\ref{lem:surjr=0sneq03}. Plus précisément, dans le cas où $l_{1}=0$, la construction est la même. Dans le cas où $l_{1},l_{2}\neq 0$, on considère les constructions précédentes avec les racines des $k$-résidus des pôles de $E_{1}$ avec $l_{1}-\sum_{i=2}^{t}\ell_{i}-1$ racines des $k$-résidus des pôles d'ordre $-k$ d'une part et les racines des éléments de $E_{2}$ et des autres pôles d'ordre $-k$ pour l'inférieur d'autre part. Par le lemme~\ref{lem:nonintersecp}, cette construction donne un $k$-différentielle avec les invariants souhaités, sauf éventuellement dans le cas où $k=2$, tous les pôles sont d'ordres $-2$ ou~$-4$ avec un $2$-résidu égal à $1$ et les zéros sont du même ordre. Dans ce cas, il suffit de permuter le rôle d'un pôle d'ordre $-2$ avec un pôle d'ordre $-4$ pour pouvoir utiliser cette construction.
 \par
 Supposons maintenant qu'il existe des pôles  dont les $k$-résidus sont non nuls et des pôles dont les $k$-résidus sont nuls.  Il suffit alors d'utiliser la construction du Lemme~\ref{lem:surjr=0sneq03} grâce au lemme~\ref{lem:nonintersecp2} qui donne la $k$-partie polaire associée à un pôle dont le $k$-résidu est nul. Cela permet d'obtenir les invariants souhaités sauf dans le cas quadratique les strates $\Omega^{2}\moduli[0](2a+2s-1,2a+2s-1;(-4^a),(-2^{2s+1}))$ avec $a\geq2$, $s\geq1$ et les résidus quadratiques sont de la forme $(1,0,\dots,0;1,\dots,1)$. Dans ce cas, la construction est la suivante.
 On considère la $2$-partie polaire d'ordre $4$ associée aux vecteurs $((1^{s+2});(1^{s+1}))$. Pour tout les autres pôles d'ordre $-4$, on associe la $2$-partie polaire associée à $(1;1)$.  A l'une segment en bas de la cicatrice on colle un segment d'une autre $2$-partie polaire d'ordre~$4$.  Puis on colle les cylindres à tous les segments à l’exception du premier et du dernier du demi-plan supérieur. On colle alors de manière cyclique les $2$-parties polaires d'ordre $4$ à ces segments puis on identifie les deux segments non identifiés par rotation. On vérifie facilement que la différentielle quadratique ainsi formée possède les invariants souhaités.
 \par
  Enfin nous supposons que tous les $k$-résidus aux pôles d'ordres $-k\ell_{i}$ sont nuls. Comme au paragraphe précédent, on peut se servir du lemme~\ref{lem:nonintersecp2} afin de traiter tous les à l’exception des strates $\Omega^{2}\moduli[0](2a+2s-1,2a+2s-3;(-4^a);(-2^{2s}))$  avec $a,s\geq2$ et  les $k$-résidus de la forme $(0,\dots,0;1,\dots,1)$. Des différentielles quadratiques avec ces invariants sont obtenues de la façon suivante. On prend une $2$-partie polaire d'ordre~$4$ associée aux vecteurs $((1^{s+1});1^{s+1})$ et les autres $2$ parties polaires d'ordres $4$ sont associées à $(1;1)$.  On colle $s+1$ cylindres aux segments en haut de la $4$-partie polaire.  On colle le segment supérieur d'une $k$-partie polaire d'ordre $4$ à un segment inférieur qui n'est ni l’initial ni le final. Puis on colle les autres domaines d'ordre $4$ de manière cyclique au segment initial. On colle alors le dernier segment de ces domaines au segment final du domaine polaire d'ordre $4$ spécial. Enfin on colle des cylindres aux segments restant. On peut facilement vérifier que la différentielle quadratique ainsi formée possède les invariants souhaités. 
%
 \end{proof}

Enfin, nous traitons le cas des strates avec $n\geq3$ zéros.
\begin{lem}\label{lem:surjr=0sneq05}
L'application résiduelle des strates $\komoduli[0](a_{1},\dots,a_{n};-k\ell_{1},\dots,-k\ell_{p};(-k^{s}))$ avec $s\geq1$ est surjective pour $n\geq3$.
\end{lem}

\begin{proof}
La preuve se fait par éclatement de zéros. Dans certains cas, il n'est pas possible d'obtenir une $k$-différentielle dans une strate en éclatant un zéro d'une $k$-différentielle primitive. Toutefois, grâce à la multiplicativité des $k$-résidus (voir l'équation~\eqref{eq:multiplires}) il suffit de vérifier que l'on peut partir d'une strate dont les éléments ne sont pas la puissance $k$-ième de différentielles abéliennes. Cela revient à vérifier que l'on peut additionner deux~$a_{i}$ de telle façon que la somme ne soit pas divisible par~$k$. Nous laissons le lecteur vérifier cette propriété élémentaire. De plus, dans les strates de la forme $\Omega^{2}\moduli[0](a_{1},a_{2},a_{3};(-4^{4});(-2^{2s+1}))$, on peut toujours ajouter le zéro d'ordre pair à l'un des zéro d'ordre impair pour obtenir une strate où l'application $2$-résiduelle est surjective. On obtient donc la surjectivité de l'application résiduelle de toutes les strates distinctes de $\Omega^{2}\moduli[0](a_{1},\dots,a_{n};-4;(-2^{2s'}))$ et  $\Omega^{2}\moduli[0](a_{1},\dots,a_{n};(-4^{a});(-2^{2}))$, en combinant l'éclatement des zéros avec les résultats précédents de cette section.
\par
Considérons les strates $\Omega^{2}\moduli[0](a_{1},\dots,a_{n};-4;(-2^{2s'}))$ et $\Omega^{2}\moduli[0](a_{1},\dots,a_{n};(-4^{a});(-2^{2}))$. Il est clair qu'il suffit de montrer la surjectivité pour les strates telles que $n=3$ puis d'éclater les zéros. Un calcul élémentaire montre que si $(a_{1},a_{2},a_{3})$ est respectivement distinct de $(2,2s'-1,2s'-1)$ et $(2,2a-1,2a-1)$, alors on peut additionner deux~$a_{i}$ de telle façon que les ordres obtenus ne soit pas ceux d'une strate exceptionnelle et qu'un ordre soit impaire. Il suffit donc de montrer que l'application $2$-résiduelle des strates $\Omega^{2}\moduli[0](2,2s'-1,2s'-1;-4;(-2^{2s'}))$ et  $\Omega^{2}\moduli[0](2,2a-1,2a-1;(-4^{a});(-2^{2}))$ contient respectivement $(0;1,\dots,1)$ et $(0,\dots,0;1,1)$. La figure~\ref{fig:extroiszero} donne une telle différentielle quadratique avec ces invariants locaux pour $s'=a=1$ et $a=2$. 
\begin{figure}[hbt]
\center
\begin{tikzpicture}

\begin{scope}[xshift=-4cm]
    \fill[fill=black!10] (0,0) ellipse (1.2cm and .7cm);

\draw[] (-.75,0)coordinate (Q) -- (.75,0) coordinate (P) coordinate[pos=.5](c) coordinate[pos=.15](d) coordinate[pos=.85](e) coordinate[pos=.35](R) coordinate[pos=.65](S);
\draw[] (Q) -- (P);

\filldraw[fill=white] (R)  arc (0:180:2pt); 
\filldraw[fill=white] (S)  arc (180:0:2pt); 
\fill (R)  arc (0:-180:2pt); 
\fill (S)  arc (-180:0:2pt);

\filldraw[fill=red] (P) circle (2pt);
\filldraw[fill=red] (Q) circle (2pt);

\node[above] at (c) {$2$};
\node[below] at (c) {$1$};
\node[above] at (d) {$3$};
\node[below] at (d) {$4$};
\node[below, rotate=180] at (e) {$3$};
\node[above, rotate=180] at (e) {$4$};

\begin{scope}[xshift=2cm]
\coordinate (a) at (-.5,-.5);
\coordinate (b) at (0,-.5);

    \fill[fill=black!10] (a)  -- (b)coordinate[pos=.5](f) -- ++(0,1.2) --++(-.5,0) -- cycle;
    \fill (a)  circle (2pt);
\fill[] (b) circle (2pt);
 \draw  (a) -- (b);
 \draw (a) -- ++(0,1.1) coordinate (d)coordinate[pos=.5](h);
 \draw (b) -- ++(0,1.1) coordinate (e)coordinate[pos=.5](i);
 \draw[dotted] (d) -- ++(0,.2);
 \draw[dotted] (e) -- ++(0,.2);
\node[below] at (f) {$1$};
\end{scope}

\begin{scope}[xshift=3cm]
\coordinate (a) at (-.5,-.5);
\coordinate (b) at (0,-.5);

    \fill[fill=black!10] (a)  -- (b)coordinate[pos=.5](f) -- ++(0,1.2) --++(-.5,0) -- cycle;
 \draw  (a) -- (b);
 \draw (a) -- ++(0,1.1) coordinate (d)coordinate[pos=.5](h);
 \draw (b) -- ++(0,1.1) coordinate (e)coordinate[pos=.5](i);
 \draw[dotted] (d) -- ++(0,.2);
 \draw[dotted] (e) -- ++(0,.2);
 \filldraw[fill=white] (a)  circle (2pt); 
\filldraw[fill=white] (b)  circle (2pt); 
\node[above, rotate=180] at (f) {$2$};
\end{scope}
\end{scope}

\begin{scope}[xshift=5cm]
     \fill[fill=black!10] (0,0) ellipse (1.1cm and .7cm);

\draw[] (-.75,0)coordinate (Q) -- (.75,0) coordinate (P) coordinate[pos=.5](c) coordinate[pos=.15](d) coordinate[pos=.85](e) coordinate[pos=.35](R) coordinate[pos=.65](S);
\draw[] (Q) -- (P);

\filldraw[fill=red] (R)  arc (0:-180:2pt); 
\filldraw[fill=red] (S)  arc (-180:0:2pt);

\filldraw[fill=white] (P) circle (2pt);
\filldraw[fill=white] (Q) circle (2pt);

\node[above] at (c) {$2$};
\node[below] at (c) {$5$};
\node[below] at (d) {$4$};
\node[above, rotate=180] at (e) {$4$};

\begin{scope}[xshift=-2.4cm]
    \fill[fill=black!10] (0,0) ellipse (1.1cm and .7cm);

\draw[] (-.75,0)coordinate (Q) -- (.75,0) coordinate (P) coordinate[pos=.5](c) coordinate[pos=.15](d) coordinate[pos=.85](e) coordinate[pos=.35](R) coordinate[pos=.65](S);
\draw[] (Q) -- (P);

\filldraw[fill=red] (R)  arc (0:180:2pt); 
\filldraw[fill=red] (S)  arc (180:0:2pt);

\fill (P) circle (2pt);
\fill (Q) circle (2pt);

\node[above] at (c) {$5$};
\node[below] at (c) {$1$};
\node[above] at (d) {$3$};
\node[below, rotate=180] at (e) {$3$};
\end{scope}

\begin{scope}[xshift=2cm,yshift=-.7cm]
\coordinate (a) at (-.75,0);
\coordinate (b) at (.75,0);

    \fill[fill=black!10] (a)  -- (b)coordinate[pos=.5](f) -- ++(0,.6) --++(-1.5,0) -- cycle;
    \fill (a)  circle (2pt);
\fill[] (b) circle (2pt);
 \draw  (a) -- (b);
 \draw (a) -- ++(0,.5) coordinate (d)coordinate[pos=.5](h);
 \draw (b) -- ++(0,.5) coordinate (e)coordinate[pos=.5](i);
 \draw[dotted] (d) -- ++(0,.15);
 \draw[dotted] (e) -- ++(0,.15);
\node[above] at (f) {$1$};
\end{scope}

\begin{scope}[xshift=2cm,yshift=.7cm]
\coordinate (a) at (-.75,0);
\coordinate (b) at (.75,0);

    \fill[fill=black!10] (a)  -- (b)coordinate[pos=.5](f) -- ++(0,-.6) --++(-1.5,0) -- cycle;

 \draw  (a) -- (b);
 \draw (a) -- ++(0,-.5) coordinate (d)coordinate[pos=.5](h);
 \draw (b) -- ++(0,-.5) coordinate (e)coordinate[pos=.5](i);
 \draw[dotted] (d) -- ++(0,-.15);
 \draw[dotted] (e) -- ++(0,-.15);
     \filldraw[fill=white](a)  circle (2pt);
   \filldraw[fill=white](b) circle (2pt);
\node[below] at (f) {$2$};
\end{scope}
\end{scope}
\end{tikzpicture}
\caption{Deux différentielles quadratiques dans $\Omega^{2}\moduli[0](2,1,1;-4;-2,-2)$ et $\Omega^{2}\moduli[0](2,3,3;-4,-4;-2,-2)$ dont les $2$-résidus sont respectivement $(0;1,1)$ et $(0,0;1,1)$.} \label{fig:extroiszero}
\end{figure}
\par
Pour les strates $\Omega^{2}\moduli[0](2,2s'-1,2s'-1;-4;(-2^{2s'}))$ on procède à la construction suivante. Soit $v\in\CC^{\ast}$, on associe aux pôles d'ordres $-2$ les $2$-parties polaires d'ordre~$2$ associées à~$(v)$. Pour le pôle d'ordre~$-4$, on associe la $2$-partie polaire d'ordre $4$ associée à $((v^{s'+2});(v^{s'+2}))$. La différentielle quadratique est obtenue en collant par rotation d'angle $\pi$ les premiers vecteur~$v$ de la $2$-partie polaire d'ordre $4$ aux derniers et les autres aux bords des $2$-parties polaires d'ordre~$2$. Cette construction est représentée à gauche de la figure~\ref{fig:extroiszero}.
\par
Pour les strates $\Omega^{2}\moduli[0](2,2a-1,2a-1;(-4^{a});(-2^{2}))$ on considère deux cas selon la parité de~$a$. Si $a$ est impair, on note $v$ l'holonomie des liens selles $v_{1}$ et $v_{2}$ labellisés par $1$ et $2$ dans la différentielle quadratique $\eta$ représentée à gauche de la figure~\ref{fig:extroiszero}. On associe à $a-1$ pôles d'ordre $-4$ la $2$-partie polaire d'ordre $4$ associée à $(v;v)$. Puis, on coupe $\eta$ le long des liens selles $v_{1}$ et $v_{2}$. Puis on colle cycliquement par translation $(a-1)/2$ des $2$-parties polaires, respectivement aux segments obtenus en coupant les liens selles $v_{1}$ et~$v_{2}$. Si $a$ est pair, on procède de manière similaire en partant de la différentielle quadratique à droite de la  figure~\ref{fig:extroiszero}.
\end{proof}

\subsection{$k$-différentielles dont les pôles sont d'ordre $-k$.}
\label{sec:juste-k}

Ce cas est, comme dans le cas abélien, assez subtil. En particulier, le cas  quadratique possède de nombreuses différences avec le cas des $k$-différentielles pour $k\geq3$. C'est pourquoi nous présentons d'abord quelques généralités communes à tous les cas. Puis nous traitons le cas des différentielles quadratiques. Enfin nous étudierons le cas des $k$-différentielles pour $k\geq3$.

\smallskip
\par
\subsubsection{\bf Généralités.}
\label{sec:juste-kgen}

Dans les strates ayant un unique zéro, toutes les singularités ont un ordre divisible par $k$. Elles ne sont donc pas primitives (lemme~\ref{lem:puissk}).
Nous nous concentrerons sur le cas des strates avec deux zéros. Le cas des strates avec plus de zéros se traiteront par éclatement.

Nous notons $S$ la surface plate associée à une $k$-différentielle $\xi$ d'une strate $\komoduli[0](\mu)$ avec $\mu:=(a_{1},a_{2};(-k^{s}))$ et dont les $k$-résidus sont $(R_{1},\dots,R_{s})$. 
Nous supposerons que les $a_{i}=kl_{i}+\bar{a_{i}}$ vérifient $-k<\bar{a_{1}}\leq-k/2\leq\bar{a_{2}}<0$. Remarquons que la condition $\pgcd(a_{1},a_{2},k)=1$ implique que le cas $\bar{a_{1}}=\bar{a_{2}}=-k/2$ est possible si et seulement si $k=2$. De plus, on a l'égalité $l_{1}+l_{2}+1=s$.

Nous présentons un procédé de construction de $k$-différentielle dans $\komoduli[0](a_{1},a_{2};(-k^{s}))$ qui sera très utile.  
Soit $E=(r_1,\dots,r_s)$ un $s$-uplet formé de racines $k$-ième des $R_{i}$, on va construire $\calG$ un $s+2$-gone  qui a pour arêtes les $r_{i}$  et deux vecteurs non nuls $v_{1},v_{2}$ avec $v_{2}=\zeta v_{1}$ pour une racine $k$-ième primitive de l'unité~$\alpha$. Les arêtes~$v_{i}$ partitionnement les $r_{i}$ en deux uplets $E_{1}$ et $E_{2}$ de cardinaux respectifs $e_{i}\in\left\{0,\dots, s\right\}$. Par définition le polygone~$\calG$ est la concaténation de $v_{1},r_{1},\dots,r_{e_1},v_{2},r_{e_{1}+1},\dots,r_{s}$ dans cet ordre (si $e_{1}=0$, alors~$v_{2}$ suit directement $v_{1}$).
 Enfin nous supposons que le vecteur normal $n_{v_{1}}$ à $v_{1}$ pointant vers l’extérieur de $\calG$ est tel que  $(n_{v_{1}},v_{1})$ forme une base directe.
Nous supposerons que l'un des deux cas suivant est satisfait.
\begin{itemize}
\item[(C1)] On a $e_{1}=l_{1}$ et $e_{2}=l_{2}+1$ et $\zeta=\exp\left((k-2a_{1})i\pi/k\right)$. 
\item[(C2)] On a $e_{1}=l_{1}+1$ et $e_{2}=l_{2}$ et $\zeta=\exp\left((2a_{1}-k)i\pi/k\right)$.
\end{itemize}
\par
La {\em construction (Ci)} désigne le fait de prendre un polygone de type (Ci) et de former une $k$-différentielle en collant les deux segments $v_{i}$ par rotation et des demi-cylindres infinis aux segments $r_{i}$. On vérifiera facilement que la $k$-différentielle ainsi obtenue est de type $\mu$ et que les $k$-résidus aux pôles d'ordre $-k$ sont les $R_{i}$.  Les deux cas sont illustrés par la figure~\ref{fig:deuxconstructions}, où les segments $v_{i}$ sont tracés normalement, les $r_{i}$ sont en segments hachés et les demi-cylindres sont pointillés. 
\begin{figure}[hbt]

\begin{tikzpicture}[scale=1,decoration={
    markings,
    mark=at position 0.5 with {\arrow[very thick]{>}}}]
\begin{scope}[xshift=-3cm,yshift=2cm]

 \fill[fill=black!10] (0,0) coordinate (p0) --  ++(-1,-1) coordinate[pos=.5] (t1) coordinate (p1)-- ++(2,-.5)  coordinate[pos=.5] (r1) coordinate (p2) -- ++(2,0.5) coordinate[pos=.5] (r2) coordinate (p3)-- ++(-1,1) coordinate[pos=.5] (t2) coordinate (p4)-- ++(0,1)  coordinate[pos=.5] (r3)coordinate (p5) --  (p0) coordinate[pos=.5] (r4);

\draw[postaction={decorate}]  (p0) -- (p1);
\draw[dashed] (p1)  --  (p2);
\draw[dashed] (p2) --(p3);
\draw[postaction={decorate}] (p3) -- (p4); 
\draw[dashed] (p4) -- (p5);
\draw[dashed] (p5) --(p0);

\draw[dotted] (p1) --++ (-.2,-1.5);
\draw[dotted] (p2)-- ++ (-.2,-1.5);
\draw[dotted] (p2)-- ++ (.2,-1.5);
\draw[dotted] (p3) --++ (.2,-1.5);
\draw[dotted] (p4) --++ (1.5,0);
\draw[dotted] (p5) --++ (1.5,0);
\draw[dotted] (p5) --++ (-1,1);
\draw[dotted] (p0) --++ (-1,1);

\filldraw[fill=white] (p0) circle (2pt);
\filldraw[fill=white] (p4) circle (2pt);
\filldraw[fill=white] (p5) circle (2pt);
\fill[] (p1) circle (2pt);
\fill[] (p2) circle (2pt);
\fill[] (p3) circle (2pt);

\node[above,rotate=360/8] at (t1) {$v_{2}$};
\node[above,rotate=-360/8] at (t2) {$v_{1}$};
\node[below] at (r1) {$r_{3}$};
\node[below] at (r2) {$r_{4}$};
\node[right] at (r3) {$r_{1}$};
\node[left] at (r4) {$r_{2}$};
\node at ($ (p2) +(0,.8)$) {$\calG$};
\end{scope}

\begin{scope}[xshift=4cm]

 \fill[fill=black!10] (0,0) coordinate (p1) -- ++(1,1) coordinate[pos=.5] (t1) coordinate (p2)-- ++(1,0)  coordinate[pos=.5] (r1) coordinate (p3) -- ++(1,-1) coordinate[pos=.5] (t2) coordinate (p4)-- ++(0,3) coordinate[pos=.5] (r2) coordinate (p5)-- ++(-3,-1) coordinate[pos=.5] (r3)coordinate (p6) --  (p1) coordinate[pos=.5] (r4);

\draw[postaction={decorate}] (p1) --  (p2);
\draw[dashed] (p2) -- (p3);
\draw[postaction={decorate}] (p3) --  (p4); 
\draw[dashed] (p4) -- (p5);
\draw[dashed] (p5) -- (p6);
\draw[dashed] (p6) -- (p1);

\draw[dotted] (p2) --++ (0,-1.5);
\draw[dotted] (p3)-- ++ (0,-1.5);
\draw[dotted] (p4)-- ++ (1.5,0);
\draw[dotted] (p5) --++ (1.5,0);
\draw[dotted] (p5) --++ (-1,1);
\draw[dotted] (p6) --++ (-1,1);
\draw[dotted] (p6) --++ (-1.5,0);
\draw[dotted] (p1) --++ (-1.5,0);

\fill(p2) circle (2pt);
\fill (p3) circle (2pt);
\filldraw[fill=white] (p4) circle (2pt);
\filldraw[fill=white]  (p5) circle (2pt);
\filldraw[fill=white]  (p6) circle (2pt);
\filldraw[fill=white]  (p1) circle (2pt);

\node[above,rotate=360/8] at (t1) {$v_{2}$};
\node[above,rotate=-360/8] at (t2) {$v_{1}$};
\node[below] at (r1) {$r_{4}$};
\node[right] at (r2) {$r_{1}$};
\node[above] at (r3) {$r_{2}$};
\node[left] at (r4) {$r_{3}$};
\node at ($ (r1) +(0,.6)$) {$\calG$};
\end{scope}

\end{tikzpicture}

\caption{Les constructions (C1) à gauche et (C2) à droite sont illustrées dans le cas de la strate $\Omega^{4}\moduli[0](5,3;(-4^{4}))$.}\label{fig:deuxconstructions}
\end{figure}

En général, un tel polygone $\calG$ n'existe pas et les constructions (Ci) ne sont pas possibles. Nous donnons maintenant une condition suffisante (non nécessaire) sur l'existence de~$\calG$, ce qui implique l’existence de $k$-différentielles correspondantes. Nous commençons par  le cas où  $-k<a_{2}<0$ (en particulier on a $-k/2\leq a_{2}<0$).

\begin{lem}\label{lem:constravecneg}
Soit $(R_{1},\dots,R_{s})$ un $s$-uplet de nombres complexes non nuls et $r_{i}$ une racine $k$-ième de~$R_{i}$. Si la somme
$t:=\sum_{i\neq1}r_{i}$
satisfait les propriétés suivantes:
\begin{enumerate}[i)]
\item  $0<|t|\leq |r_{1}|$,
\item $t\neq -\zeta r_{1}$ pour toutes les racines $k$-ième de l'unité $\zeta$,
\item si $s\geq 3$, alors les $r_{i}$ ne sont pas tous colinéaires pour $i\geq 2$, 
\end{enumerate}
alors il existe une $k$-différentielle dans la strate $\komoduli[0](a_{1},a_{2};(-k^{s}))$ avec $-k/2\leq a_{2}<0$ dont les $k$-résidus sont $(R_{1},\dots,R_{s})$.
\end{lem}

\begin{proof}
Considérons des racines $r_{i}$ de $R_{i}$ satisfaisants aux conditions du lemme~\ref{lem:constravecneg}. Nous construisons un polygone de type (Ci) de la façon suivante.  Quitte à multiplier les $R_{i}$ par un complexe non nul et les $r_{i}$ par une racine $k$-ième de l'unité, nous pouvons supposer que $-r_{1}$ est réel positif et que l'argument de~$t$ appartient à $\left[-\frac{2\pi}{k},0\right[$. Pour $i\geq 2$, prenons les arguments des~$r_{i}$ dans  $\left[-\pi,\pi\right[$ et rangeons les $r_{i}$ par argument croissant.
Nous construisons un polygone intermédiaire en concaténant les $r_{i}$ depuis l'origine jusqu'au point final $t$ dans cet ordre. Puis nous relions les points~$t$ à $-r_{1}$ par deux segments $v_{1}$ et~$v_{2}$ d'égale longueur qui font un angle $\alpha=(a_{2}+k)\frac{2\pi}{k}$ (voir la figure~\ref{fig:constructionmoinsk} pour lequel des deux angles à leur intersection est~$\alpha$). On supposera que $v_{1}$ part de $t$ et $v_{2}$ termine en $-r_{1}$. Le segment $v_{1}$ peut couper le segment $-r_{1}$. Toutefois, c'est l'unique intersection qui peut se produire. En effet, si l'argument de~$t$ appartient à $\left]-\frac{2\pi}{k},0\right[$, comme $\alpha \geq \pi$ ni $v_{1}$ ni $v_{2}$ ne rencontrent d'autres arêtes ou un sommet du polygone. Si l'argument de~$t$ est  $-\frac{2\pi}{k}$, on obtient la même conclusion en tenant compte du fait que $|t|<|r_{1}|$. 

Si $v_{1}$ ne rencontrent pas $-r_{1}$, alors le polygone obtenue est de type (C2) (voir la figure~\ref{fig:constructionmoinsk} à gauche). Si  $v_{1}$ coupe $-r_{1}$,  le polygone de type (C1) est obtenu en concaténant $v_{1}$ puis $r_{1}$, $v_{2}$ et enfin les  $r_{i}$ pour $i\geq2$ (voir la figure~\ref{fig:constructionmoinsk} au centre ou à droite). Notons que le point (iii) implique que ce polygone n'est pas dégénéré.  Les construction (Ci) permettent d'obtenir la $k$-différentielle souhaitée. 
\begin{figure}[hbt]

\begin{tikzpicture}[scale=2.5,decoration={
    markings,
    mark=at position 0.5 with {\arrow[very thick]{>}}}]
\begin{scope}[xshift=-3cm,yshift=2cm]

 \fill[fill=black!10] (0,0) coordinate (p0) --  ++(1,0) coordinate[pos=.5] (t1) coordinate (p1)-- ++(-.3,-.1)  coordinate[pos=.5] (r1) coordinate (p2) -- ++(-.3,-.1) coordinate[pos=.5] (r2) coordinate (p3)-- ++(-.8,-.6) coordinate[pos=.5] (t2) coordinate (p4)-- ++(-.3,.5)  coordinate[pos=.5] (r3)coordinate (p5) --  (p0) coordinate[pos=.5] (r4);

\draw[dashed,postaction={decorate}]  (p0) -- (p1);
\draw[postaction={decorate}] (p2)  --  (p1);
\draw[postaction={decorate}] (p3) --(p2);
\draw[dashed,postaction={decorate}] (p4) -- (p3); 
\draw[dashed,postaction={decorate}] (p5) -- (p4);
\draw[dashed,postaction={decorate}] (p0) --(p5);
\draw[dotted] (p0) --(p3)coordinate[pos=.6](t);

\fill[] (p0) circle (.5pt);
\fill[] (p4) circle (.5pt);
\fill(p5) circle (.5pt);
\fill[] (p1) circle (.5pt);
\filldraw[fill=white] (p2) circle (.5pt);
\fill[] (p3) circle (.5pt);

\node[above] at (t1) {$-r_{1}$};
\node[below] at (t2) {$r_{4}$};
\node[below] at (r1) {$v_{2}$};
\node[below] at (r2) {$v_{1}$};
\node[right] at (r3) {$r_{3}$};
\node[left] at (r4) {$r_{2}$};
\node[] at (t) {$t$};

\draw[->] (.65,-.115) arc  (185:30:.07); \node at (.7,.05) {$\alpha$};
\end{scope}


\begin{scope}[xshift=-.5cm,yshift=2cm]

 \draw[dashed] (0,0) coordinate (p0) --  ++(1,0) coordinate[pos=.3] (t1) coordinate (p1)-- ++(-.4,.2)  coordinate[pos=.5] (r1) coordinate (p2) -- ++(-.2,-.4) coordinate[pos=.5] (r2) coordinate (p3)-- ++(-.8,-.6) coordinate[pos=.5] (t2) coordinate (p4)-- ++(-.3,.5)  coordinate[pos=.5] (r3)coordinate (p5) --  (p0) coordinate[pos=.5] (r4);

\draw[] (p2)  --  (p1);
\draw[] (p3) --(p2);
\draw[dotted] (p0) --(p3)coordinate[pos=.6](t);

\node[above] at (t1) {$-r_{1}$};
\node[below] at (t2) {$r_{4}$};
\node[above] at (r1) {$v_{2}$};
\node[below right] at (r2) {$v_{1}$};
\node[right] at (r3) {$r_{3}$};
\node[left] at (r4) {$r_{2}$};
\node[] at (t) {$t$};

\draw[->] (.585,.17) arc  (250:-60:.07); \node at (.6,.25) {$\alpha$};
\end{scope}

\begin{scope}[xshift=1cm,yshift=2.2cm]

 \fill[fill=black!10] (0,0) coordinate (p0) --  ++(1,0) coordinate[pos=.5] (t1) coordinate (p1) -- ++(-.2,-.4) coordinate[pos=.5] (r2) coordinate (p3)-- ++(-.8,-.6) coordinate[pos=.5] (t2) coordinate (p4)-- ++(-.3,.5)  coordinate[pos=.5] (r3)coordinate (p5) -- ++(.7,.3)  coordinate[pos=.5] (r1) coordinate (p2) --  (p0) coordinate[pos=.4] (r4);

\draw[dashed]  (p0) -- (p1);
\draw[] (p3)  --  (p1);
\draw[] (p0) --(p2);
\draw[dashed,] (p4) -- (p3); 
\draw[dashed,] (p5) -- (p4);
\draw[dashed,] (p2) --(p5);

\filldraw[fill=white] (p0) circle (.5pt);
\fill[] (p4) circle (.5pt);
\fill (p5) circle (.5pt);
\filldraw[fill=white] (p1) circle (.5pt);
\fill[] (p2) circle (.5pt);
\fill[] (p3) circle (.5pt);

\node[above] at (t1) {$-r_{1}$};
\node[below] at (t2) {$r_{4}$};
\node[left] at (r1) {$r_{2}$};
\node[below right] at (r2) {$v_{1}$};
\node[right] at (r3) {$r_{3}$};
\node[left] at (r4) {$v_{2}$};
\end{scope}

\end{tikzpicture}

\caption{La construction du lemme~\ref{lem:constravecneg} lorsque le segment $v_{1}$ ne rencontre pas $-r_{1}$ à gauche et lorsqu'ils s'intersectent au centre et à droite.}\label{fig:constructionmoinsk}
\end{figure}
\end{proof}
La preuve donne un peu plus que le lemme. En effet, on a la remarque suivante.
\begin{rem}\label{rem:quadmoinsun}
 Avec les mêmes notations que dans la preuve, si l'argument de $t$ est $-\frac{2\pi}{k}$ et si $|t|=|r_{1}|$, alors la construction que nous donnons dans la preuve fonctionne sauf si $a_{2}=-1$.
\end{rem}

Nous donnons l'analogue du lemme~\ref{lem:constravecneg} pour les strates dont les zéros sont d'ordres $a_{i}$ positifs. De plus, nous nous restreindrons au cas $k\geq3$.
\begin{lem}\label{lem:constravecnegbis}
Soient  $S:=\komoduli[0](a_{1},a_{2};(-k^{s}))$ une strate de genre zéro,  $(R_{1},\dots,R_{s})\in(\CC^{\ast})^{s}$ et des racines $k$-ième $r_{i}$ de~$R_{i}$. 
 S'il existe deux sommes $s_{1}:=\sum_{i\leq l_{1}}r_{i}$ et $s_{2}:=\sum_{i>l_{1}}r_{i}$ telles que 
\begin{enumerate}[i)]
\item $0<|s_{1}|<q |s_{2}|$, avec $q=\frac{1}{\sqrt{2}}$ pour $k\geq4$ et   $q=\frac{1}{2}$ pour $k=3$,
\item  les $l_{1}$ premiers $r_{i}$ ne sont pas tous colinéaires si $l_{1}\geq 3$ respectivement,
\end{enumerate}
alors il existe une différentielle dans la strate  $S$ telle que les $k$-résidus sont $(R_{1},\dots,R_{s})$.
\end{lem}

\begin{proof}
La preuve de ce lemme est analogue à celle du lemme~\ref{lem:constravecneg}.  Toutefois, nous utiliserons une variante de la construction (C1) car en général le polygone que nous construirons ne peut pas se représenter sans intersection dans le plan. 

Prenons les racines $r_{i}$ comme dans le lemme. Nous supposerons que la somme $s_{2}$ des $l_{2}+1$ dernières racines  est réelle positive.  Quitte à multiplier tous les $r_{i}$ pour $i\leq l_{1}$ par une même racine $k$-ième de l'unité,  l'argument de~$s_{1}$ appartient au segment $\left[-\tfrac{\pi}{k},\tfrac{\pi}{k}\right]$. 
On fait  alors la construction (C1) avec les deux vecteurs $s_{1}$ et $s_{2}$. Le point (i) implique que polygone ainsi formé est non dégénéré. En effet, dans le cas $k=3$, le fait que  $|s_{1}| < \frac{1}{2} |s_{2|} $ et la condition sur l'argument de $s_{1}$ implique que la somme $s_{1} +s_{2}$ est d’argument de norme strictement inférieur à celui de $1+\exp(i\pi/3)$. Comme $v_{2} = (s_{1}+s_{2}) / (1+\exp(i\pi/3))$ on en déduit que la partie imaginaire de $v_{2}$ est strictement négative, et donc la non dégénérescence du polygone.  La non dégénérescence dans les cas $k\geq 4$ se montre similairement.

On procède alors à la construction suivante. On forme deux polygones constitués respectivement des vecteurs $-s_{i}$ et des racines $r_{i}$ qui forment les sommes. On concatène les $l_{1}$ premiers $r_{i}$, resp. $l_{2}+1$ derniers, par argument croissant en commençant par ceux d'arguments opposés à ceux de $s_{i}$ puis le vecteur $-s_{i}$. Le premier polygone est non dégénéré pour $l_{1}\geq3$ par la propriété~(ii). Si  $l_{1}=2$ et les racines sont proportionnelles, alors le polygone correspondant est dégénéré, mais la construction suivante s'étend à ce cas. Si les $l_{2}+1$ dernières racines sont proportionnelles, alors on peut supposer qu'elle sont toutes dans la même demi-droite partant de l'origine. En effet, sinon on pourrait changer les racines afin d'augmenter la somme $s_{2}$ et refaire la construction avec ces nouvelles racines. Si les $l_{2}+1$ dernières racines sont proportionnelles, le polygone est alors dégénéré, mais la construction suivante s'étend sans problèmes. La $k$-différentielle est alors obtenue en collant ces polygones au polygone initial, des demi-cylindres infinis aux vecteurs~$r_{i}$ et les deux vecteurs $v_{i}$ ensemble.
\end{proof}

Pour terminer ce paragraphe, nous voudrions insister sur le fait que les constructions ci-dessus ne sont pas les seules pour obtenir de telles $k$-différentielles. Toutefois, elles permettent de construire la majorité de celles-ci.
\smallskip
\par
\subsubsection{\bf Les strates quadratiques.}

Dans ce paragraphe, nous regardons l'application résiduelle des strates de la forme $\Omega^{2}\moduli[0](a_{1},\dots,a_{n};(-2^{s}))$ avec $a_{i}\geq -1$ et au moins deux $a_{i}$ sont impairs. Nous considérons pour commencer le cas de deux zéros. 

Dans ce cas, il existe deux familles de strates spéciales données dans le lemme suivant. Rappelons que trois nombres sont triangulaires s'ils sont le carré de nombres de somme nulle (voir définition~\ref{def:triangulaire}).
\begin{lem}\label{lem:quadexeptionnelsgzero}
L'image de l'application résiduelle des strates $\Omega^{2}\moduli[0](2s'-1,2s'+1;(-2^{2s'+2}))$ ne contient pas les résidus quadratiques de la forme $(1,\dots,1,R,R)$ avec $R\in\CC^{\ast}$.

L'image par l'application résiduelle des strates $\Omega^{2}\moduli[0](2s'-1,2s'-1;(-2^{2s'+1}))$ ne contient pas les résidus quadratiques de la forme $(R_{1},\dots,R_{1},R_{2},R_{3})$ où  $R_{1},R_{2},R_{3}$ sont triangulaires.
\end{lem}

\begin{proof}
Supposons qu'il existe une $2$-différentielle dans  $\Omega^{2}\moduli[0](2s-1,2s+1;(-2^{2s'+2}))$ dont les $2$-résidus sont $(1,\dots,1,R,R)$. On obtient une $2$-différentielle entrelacée lissable en collant les deux pôles ayant pour résidu quadratique~$R$. De plus, cette différentielle est lissable sans modifier les autres résidus (voir lemme~\ref{lem:lisspolessimples}). On obtient donc une $2$-différentielle dans la strate $\Omega^{2}\moduli[1](2s'-1,2s'+1;(-2^{2s'}))$ dont les résidus quadratiques sont $(1,\dots,1)$. Or nous montrerons qu'un telle différentielle quadratique n'existe pas dans le lemme~\ref{lem:nonsurjdeuxzero}, ce qui donne une contradiction.

Nous traitons maintenant le cas où les résidus quadratiques sont triangulaires dans les strates $\Omega^{2}\moduli[0](2s'-1,2s'-1;(-2^{2s'+1}))$. Supposons par l'absurde qu'il existe une différentielle quadratique $\xi$ dans $\Omega^{2}\moduli[0](1,1;(-2^{3}))$ avec ces invariants locaux. 
On a alors 
\[\xi=\frac{z}{(z-1)^{2}(z-a)^{2}(z-b)^{2}} (dz)^{2} \, ,\]
avec $a\neq b\in \CC\setminus\{0,1\}$ et 
\[ R_{1}=\frac{1}{(1-a)^{2}(1-b)^{2}},\text{ }  R_{2}=\frac{a}{(1-a)^{2}(a-b)^{2}}, \text{ et }  R_{3}=\frac{b}{(1-b)^{2}(a-b)^{2}}\, . \]
En notant par $\alpha$ et $\beta$ une racine de $a$ et $b$ respectivement, la condition de triangularité donne
\[  \pm (\alpha^{2}-\beta^{2}) \pm \alpha (\beta-1) \pm \beta (\alpha -1) = 0   \, .\]
Un calcul élémentaire montre que cette équation est équivalente à
\[  \pm (\alpha\pm \beta)(\beta\pm 1) (1 \pm \alpha) = 0   \, .\]
On en déduit que $a=b$ ou  $a=1$ ou  $b=1$ ce qui n'est pas permis.

Supposons qu'il existe une différentielle quadratique de $\Omega^{2}\moduli[0](2s'-1,2s'-1;(-2^{2s'+1}))$ avec $s'\geq2$ dont les $2$-résidus sont $((R_{1}^{2s'-1}),R_{2},R_{3})$. Il existe deux cas de figure (modulo multiplication de toutes les racines par $-1$) selon que:
\begin{itemize}
 \item on a les racines $r_{2}$ et $r_{3}$, 
 \item on a les racines $r_{2}$ et $-r_{3}$,
\end{itemize}
avec la condition $r_{1}+r_{2}+r_{3}=0$.

Dans le premier cas, on choisit aussi $a^{+}$ racines $r_{1}$ et $a^{-}$ racines $-r_{1}$. On a aussi deux vecteurs $v$ tels que 
$$ r_{2} + r_{3} +a^{+}r_{1} - a^{-}r_{1} +2v =0\,.$$
La différentielle quadratique est obtenue en concaténant tous ces vecteurs et en identifiant par rotation les deux vecteurs $v$ et des demi-cylindres infinis aux autres vecteurs. On déduit de l'égalité précédente que $v = \tfrac{a^{-}-a^{+} + 1}{2}r_{1}$. En particulier, le vecteur $v$ est proportionnel à $r_{1}$. Donc la concaténation est de la forme $r_{2},-r_{1},\dots,-r_{1},r_{3},v,r_{1},\dots,r_{1},v$ ou $r_{2},v,-r_{1},\dots,-r_{1},v,r_{3},r_{1},\dots,r_{1}$. Pour que la différentielle soit dans la strate considérée, il faut que $a^{+}=s'$ dans le premier cas et $a^{-}=s'$ dans le second cas. Donc dans le premier cas, on a $v=0$, ce qui implique que la différentielle est dégénérée. Dans le second cas, on a $v=r_{1}$ et on voit que la différentielle quadratique ainsi obtenue est dégénérée.

Dans le second cas, les vecteurs $v$ satisfont à l'équation 
$$ r_{2} - r_{3} +a^{+}r_{1} - a^{-}r_{1} +2v =0\,.$$
De cette équation on déduit que $v= \tfrac{a^{-}-a^{+}-1}{2}r_{1}-r_{2}$.  La concaténation des vecteurs est de la forme $r_{2}-r_{3},-r_{1},\dots,-r_{1},v,-r_{1},\dots,-r_{1},v,r_{1},\dots,r_{1}$ car sinon la concaténation serait dégénérée. On peut vérifier que dans le  cas précédent on a $s'$ vecteurs $-r_{1}$ entre les deux vecteurs $v$. On en déduit (en allant dans le sens contraire de la concaténation donnée) que le point après le premier $v$ est $a^{+}-\tfrac{s-a^{+}-1}{2}r_{1}+r_{2} -s'$. Comme $a^{+}\leq s-1$ on en déduit que ce point est donné par $r_{2} - \alpha r_{1}$ avec $\alpha$ strictement positif. On en déduit immédiatement que le polygone est dégénéré et donc que la différentielle quadratique ne peut pas exister. 
\end{proof}

Nous montrons maintenant que ces cas sont très spéciaux.
\begin{lem}\label{lm:quadgzerogen}
 Soit $\Omega^{2}\moduli[0](a_{1},\dots,a_{n};(-2^{s}))$ une strate quadratique.  L'image par l'application résiduelle de cette strate contient tous les $s$-uplets dont les éléments n'appartiennent pas à un même rayon issu de l'origine, à l’exception de ceux donnés dans le lemme~\ref{lem:quadexeptionnelsgzero}. 
\end{lem}

\begin{proof}
Nous commençons par traiter le cas où il y a deux zéros $a _{i}=2l_{i}-1$ et on suppose que $l_{1}\leq l_{2}$. L'idée de la preuve est de partitionner les pôles en deux sous-ensembles de cardinal~$l_{1}$ et $l_{2}+1$ respectivement affin d'obtenir les différentielles quadratiques souhaitées par la construction (C1) de la section~\ref{sec:juste-kgen}.
 
On commence par le cas des strates  $\Omega^{2}\moduli[0](a_{1},a_{2};(-2^{s}))$ avec $(a_{1},a_{2})\neq (2s'-1,2s'+1)$ si $s=2s'+2$ et des  résidus quadratiques $(1,\dots,1,R,R)$ où $R\notin\RR_{+}$. Soit $r$ une racine carrée de $R$, on note $v=2r + l_{2}-l_{1}-1$. On construit un polygone en concaténant les deux segments $r$ puis $l_{2}-1$ segments $1$ puis le vecteur $-v/2$ puis $l_{1}$ segments $-1$ en enfin le segment $-v/2$. Ce polygone est représenté à gauche de la figure~\ref{fig:quaddeuxresidus} pour la strate $\Omega^{2}\moduli[0](1,7;(-2^{6}))$. Notons que le point du polygone avant de concaténer le dernier vecteur $-v/2$ est de coordonnées $((l_{2}-l_{1}-1)/2,1)$ dans la base $(1,r)$ de $\RR^{2}$. Clairement le polygone est non dégénéré si et seulement si $l_{2}-l_{1}-1>0$, i.e. si $(a_{1},a_{2})\neq (2s'-1,2s'+1)$. Dans ce cas, la construction (C1) à partir de ce polygone donne la différentielle quadratique avec les invariants souhaités.

Prenons  un $s$-uplet de la forme $((1^{s_{1}}),(R^{s_{2}}))$ avec $R\notin\RR_{+}$, tel que soit $s_{2}=1$, soit $s_{i}\geq3$ pour $i=1,2$. On construit une $2$-différentielle dans  $\Omega^{2}\moduli[0](a_{1},a_{2};(-2^{s}))$ ayant ces $2$-résidus. Dans le cas où $s_{1}=1$, il suffit de faire la construction du paragraphe précédent avec un unique~$r$.  Notons que si $a_{1}=a_{2}=s/2-2$, alors le polygone est dégénéré. En effet, le sommet avant la concaténation du deuxième segment $-v/2$ est au milieu du segment~$r$. Donc le second segment $-v/2$ coïncide avec la première moitié du segment~$r$.   Toutefois, la construction~(C1) s'étend à ce cas pour donner une différentielle quadratique avec les invariants locaux souhaités en identifiant la première moitié du vecteur~$r$ avec le premier segment $-v/2$. 
Dans le cas où $s_{1},s_{2}\geq 3$, on a deux cas à traiter selon que les $s_{i}$ sont impairs ou sont pairs.
Si  les~$s_{i}$ sont impairs, on utilise la construction~(C1) avec un uplet~$E$ formé de $\left[\tfrac{s_{1}}{2}\right]$ vecteurs $1$, de $\left[\tfrac{s_{2}}{2}\right]$ vecteurs $r$, de $ \lceil\tfrac{s_{1}}{2}\rceil$ vecteurs $-1$, de $\lceil\tfrac{s_{2}}{2}\rceil$ vecteurs~$-r$ et enfin avec deux vecteurs $v_{1}=v_{2}=(1+r)/2$ égaux à l’opposé de la demi-somme des autres vecteurs. On forme alors un polygone en concaténant les  $l_{1}$ premiers éléments de~$E$ puis $v_{1}$ puis les $l_{2}-1$ suivants et enfin $v_{2}$. Ce polygone est représenté au milieu de la figure~\ref{fig:quaddeuxresidus} pour la strate $\Omega^{2}\moduli[0](3,5;(-2^{6}))$ et les $2$-résidus $((1^{3}),(R^{3}))$.  Ce polygone est non dégénéré car les $l_{1}$ premier vecteurs sont d'un côté des vecteurs $v_{i}$ et les autres vecteurs strictement de l'autre côté.  La différentielle quadratique est alors obtenue par  la construction (C1).
Dans le cas où les $s_{i}$ sont tous les deux d'ordres pairs, on  utilise la construction~(C1) avec $E$ constitué de  $\tfrac{s_{1}}{2}-1$ fois $-1$, de $\tfrac{s_{2}}{2}-1$ fois $r$, de $\tfrac{s_{1}}{2}$ fois $1$, de   $\tfrac{s_{2}}{2}+1$ fois~$-r$ et enfin une fois $-1$ et enfin avec deux vecteurs $v_{1}=v_{2}=r$ égaux à l’opposé de la demi-somme de ces vecteurs. L'ensemble $E_{1}$ de la construction (C1) est constitué des $l_{1}$ premiers éléments de~$E$. Comme dans le cas précédent, le polygone est non dégénéré. En effet, comme $\tfrac{s_{2}}{2}-1\geq1$ il n'y a pas de points d'intersection au début de la concaténation. De plus, la concaténation des premiers vecteurs est d'un côté des vecteurs $v_{i}$ et les derniers de l'autre côté. Cette construction est présentée à droite de la figure~\ref{fig:quaddeuxresidus} pour la strate $\Omega^{2}\moduli[0](7,9;(-2^{10}))$ et les $2$-résidus $((1^{6}),(R^{4}))$. 
\begin{figure}[hbt]
\centering
\begin{tikzpicture}
\begin{scope}[xshift=-6cm]
 \filldraw[fill=black!10](0,0) coordinate (p1) -- ++(1,1)  coordinate[pos=.5] (q1)  coordinate (p2)-- ++(1,1) coordinate[pos=.5] (q2) coordinate (p3)-- ++(1,0) coordinate[pos=.5] (q3) coordinate (p4)-- ++(1,0)coordinate[pos=.5] (q4)  coordinate (p5)-- ++(1,0) coordinate[pos=.5] (q5) coordinate (p6)--
++ (-2,-1) coordinate [pos=.5] (q6) coordinate (p7) -- ++(-1,0)coordinate[pos=.5] (q7)  coordinate (p8) -- ++ (-2,-1) coordinate [pos=.5] (q8);

  \foreach \i in {1,2,...,6}
  \fill (p\i)  circle (2pt); 
    \foreach \i in {7,8}
  \filldraw[fill=white] (p\i)  circle (2pt); 
    \foreach \i in {1,2}
  \node[above left] at (q\i) {$r$};
      \foreach \i in {3,...,5}
  \node[above] at (q\i) {$1$};
  \node[below right] at (q8) {$v_{1}$};
    \node[above] at (q7) {$1$};
  \node[below right] at (q6) {$v_{2}$};

\end{scope}

\begin{scope}[xshift=-1cm]

\filldraw[fill=black!10] (0,0) coordinate (p1) -- ++(1,.5)  coordinate[pos=.5] (q1) coordinate (p2)-- ++(1,0) coordinate (p3)-- ++(1,1)  coordinate (p4)-- ++(1,.5)  coordinate[pos=.5] (q4) coordinate (p5)-- ++(-1,0)  coordinate (p6) --++(-1,0)  coordinate (p7)-- ++(-1,-1)coordinate (p8)-- ++(-1,-1);

  \foreach \i in {2,3,4}
  \filldraw[fill=white] (p\i)  circle (2pt); 
    \foreach \i in {1,5,6,7,8}
  \fill (p\i)  circle (2pt); 
    \node[below right] at (q1) {$v_{1}$};
  \node[below right] at (q4) {$v_{2}$};
\end{scope}

\begin{scope}[xshift=5cm,yshift=0cm]

\filldraw[fill=black!10] ((0,0) coordinate (p0) -- ++(.8,.8)coordinate[pos=.5] (q1)  coordinate (p1) -- ++(-1,0) coordinate (p4)-- ++(-1,0) coordinate (p3)-- ++(.8,.8)  coordinate (p5)-- ++(1,0)  coordinate (p6)  -- ++(.8,.8)coordinate[pos=.5] (q2)  coordinate (p12)    --++(1,0)  coordinate (p7)-- ++(1,0)coordinate (p8)-- ++(-.8,-.8) coordinate (p9)-- ++(-.8,-.8) coordinate (p10)-- ++(-.8,-.8) coordinate (p2)-- ++(-1,0) coordinate (p11);

  \foreach \i in {0,2,7,8,...,12}
  \fill (p\i)  circle (2pt); 
    \node[left] at (q1) {$v_{1}$};
  \node[left] at (q2) {$v_{2}$};
    \foreach \i in {1,3,4,5,6}
  \filldraw[fill=white] (p\i)  circle (2pt); 
\end{scope}

\end{tikzpicture}
\caption{Les polygones utilisés pour obtenir des $2$-différentielles dans  $\Omega^{2}\moduli[0](a_{1},a_{2};(-2^{s}))$ avec des $2$-résidus de la forme $((1^{s_{1}}),(R^{s_{2}}))$ avec $R\notin\RR_{+}$.}\label{fig:quaddeuxresidus}
\end{figure} 
\par
On considère maintenant le cas où les $2$-résidus sont sur deux rayons distincts non nécessairement égaux entre eux. Nous noterons ces $2$-résidus $(\alpha_{1},\dots,\alpha_{s_{1}},\beta_{1}R,\dots,\beta_{s_{2}}R)$ avec $\alpha_{i}\in\RR_{+}^{\ast}$ et $\beta_{j}\in\RR_{+}^{\ast}$. Si $s_{1}=1$ ou $s_{2}=1$, alors on peut faire la construction (C1) comme au paragraphe précédent. On supposera donc que $s_{1},s_{2}\geq2$.
\par
On commence par le cas général où il existe deux résidus quadratiques non égaux le long de chaque rayon, i.e. il existe $\alpha_{i_{1}}\neq \alpha_{i_{2}}$ et $\beta_{j_{1}}\neq \beta_{j_{2}}$. On partitionne les résidus quadratiques $(\alpha_{1},\dots,\alpha_{s_{1}})$ en deux sous-ensembles $\mathcal{R}_{i,1}$ satisfaisant les conditions suivantes. Le cardinal de $\mathcal{R}_{1,1}$ est  supérieur ou égal à $\lfloor\frac{s_{1}}{2}\rfloor$ et la somme des racines des ces éléments est strictement inférieure à celles de $\mathcal{R}_{2,1}$. 
On partitionne les éléments $(\beta_{1}R,\dots,\beta_{s_{2}}R)$ en deux sous ensembles~$\mathcal{R}_{i,2}$ avec les mêmes propriétés. On construit un polygone de type (C1) avec $v_{1}=v_{2}$ définis comme l'opposé de la demi-somme des éléments définis dans la phrase suivante, $E_{1}$, resp.~$E_{2}$, constitué des  $l_{1}$ premier, resp. $l_{2}+1$ derniers, éléments suivants. On prend les racines avec parties réelles positives des éléments de $\mathcal{R}_{1,1}$ et de ceux de $\mathcal{R}_{1,2}$ (si la partie réelle de ces éléments est nulle, la partie imaginaire est positive), puis les racines de parties réelles négatives de ceux de $\mathcal{R}_{2,1}$ et $\mathcal{R}_{2,2}$.  Le polygone ainsi obtenu est clairement non dégénéré et la construction (C1) donne la différentielle quadratique désirée.
\par
Considérons le cas où tous les résidus quadratiques sont égaux entre eux sur un rayon, i.e. les résidus quadratiques sont de la forme $(\alpha_{1},\dots,\alpha_{s_{1}},R,\dots,R)$ avec les $\alpha_{i}\in\RR_{+}^{\ast}$ non tous égaux entre eux. Dans ce cas la seule situation où la  construction précédente ne fonctionne pas est lorsque $s_{2}$ est pair.
  Dans ce cas, on fait une construction similaire à celle illustrée à gauche de la figure~\ref{fig:quaddeuxresidus}. On considère $s_{2}/2+1$ vecteurs $r$ puis  $ \lceil\tfrac{s_{1}}{2}\rceil$ racines positives des $\alpha_{i}$ puis $s_{2}/2-1$ vecteurs $-r$ puis les racines négatives des autres $\alpha_{i}$. Si nous supposons que les $ \lceil\tfrac{s_{1}}{2}\rceil$ premiers $\alpha_{i}$ possèdent les normes les plus grandes, alors un calcul similaire au cas  illustré à gauche de la figure~\ref{fig:quaddeuxresidus} montre que le polygone ainsi formé est non dégénéré. On conclut en utilisant la construction (C1).
\smallskip
\par 
Considérons maintenant le cas où les $R_{i}$ sont sur trois rayons différents. Dans un premier temps nous supposerons que les $s=s_{1}+s_{2}+s_{3}$ résidus quadratiques sont de la forme  $((R_{1}^{s_{1}}),(R_{2}^{s_{2}}),(R_{3}^{s_{3}}))$ avec les $R_{i}$ triangulaires. De plus, nous choisissons  les racines $r_{i}$ de $R_{i}$ telles que $r_{1}+r_{2}+r_{3}=0$. Si $s_{1}=s_{2}=s_{3}$, on définit  $v_{1}=v_{2}= r_{1}$ et on construit le polygone de type (C1) suivant. Soit $E$ le $s$-uplet dont les éléments sont  $s_{1}-1$ fois~$r_{1}$, puis $s_{2}$ fois $r_{2}$, puis une fois $-r_{1}$ et enfin $s_{3}$ fois $r_{3}$.  On concatène $v_{1}$ puis les $l_{1}$ premiers éléments l'ensemble $E$ puis $v_{2}$ et enfin les $l_{2}+1$ derniers éléments de~$E$.  Le polygone obtenu pour $\Omega^{2}\moduli[0](3,5;(-2^{6}))$ est représenté à gauche de la figure~\ref{fig:quadtroisresidus}. On vérifiera facilement que le polygone ainsi obtenu est non dégénéré sauf dans le cas de la strate $\Omega^{2}\moduli[0](1,1;(-2^{3}))$. La construction (C1) donne donc les différentielles quadratiques souhaitées. 
\par
Si  $s_{1}\geq s_{2}\geq s_{3}$ avec au moins une inégalité stricte on fait la construction (C1). Nous donnons tout d'abord  une construction qui fonctionne sauf dans le cas des strates  de la forme $\Omega^{2}\moduli[0](2s'-1,2s'-1;(-2^{2s'+1}))$ lorsque $s_{1}>s_{2}=s_{3}$. 
\par
Le uplet~$E_{1}$ est constitué des~$l_{1}$ premiers éléments suivants: $\lfloor\frac{s_{1}}{2}\rfloor$ fois $r_{1}$, puis~$s_{3}$ fois $r_{3}$, puis~$\lceil\frac{s_{1}}{2}\rceil$ fois $-r_{1}$ et enfin $s_{2}$ fois~$r_{2}$. On définit $v_{1}=v_{2}$ égale à la moitié de l'opposé de la somme des éléments précédents. On construit alors le polygone de type (C1) avec ces vecteurs, le cas $\Omega^{2}\moduli[0](5,7;(-2^{8}))$ est représenté au centre de la figure~\ref{fig:quadtroisresidus} avec $s_{1}=5$, $s_{2}=2$ et $s_{3}=1$.  Il reste à justifier que le polygone ainsi construit est non dégénéré sauf si $s_{2}=s_{3}=1$ et $a_{1}=a_{2}=2s'-1$ avec $s=2s'+1$. Supposons que $s_{2}=s_{3}=1$ et que $s$ est pair. On a alors  $a_{1}\leq s-3$ et donc $l_{1}\leq \lfloor\frac{s_{1}}{2}\rfloor$. Comme $v_{1}=v_{2}=r_{1}/2$, on en déduit immédiatement que le polygone est non dégénéré. Le cas $s$ impair se traite de la même façon dès que  $a_{1}\neq a_{2}$. Les cas avec $s_{2}=s_{3}\geq2$ et soit $s$ pair, soit $s$ impair et $a_{1}\neq a_{2}$  se traitent comme précédemment. Enfin le cas $s_{2} > s_{3}$  se montre de manière similaire.
\par
Il reste à traiter le cas des strates  $\Omega^{2}\moduli[0](2s'-1,2s'-1;(-2^{2s'+1}))$ avec $s_{1}>s_{2}=s_{3}>1$. Dans ce cas, on définie $E_{1}$ par une fois $r_{2}$, puis $(s_{1}-1)/2$ fois $r_{1}$, et $s_{3}-1$ fois $r_{3}$. Le uplet $E_{2}$ est définie par $r_{3}$, puis $(s_{1}+1)/2$ fois $-r_{1}$ et enfin  $s_{2}-1$ fois $r_{2}$. Il suit facilement du fait que $v_{1}=v_{2}=\tfrac{1+s_{2}}{2}$ que le polygone de type (C1) ainsi obtenu est non dégénéré sauf dans le cas $(s_{1},s_{2},s_{3})=(3,2,2)$. Dans ce dernier cas, nous définissons $v_{1}=v_{2}=-r_{1}$ et nous considérons le polygone obtenu via la concaténation $r_{1},r_{2},r_{1},v_{1},-r_{3},-r_{1},r_{3},v_{2},-r_{2}$. Ce polygone est dégénéré mais comme lors du cas $((1^{s_{1}}),(R_{s_{2}})$ au début de cette preuve, la construction (C1) s'étend à ce cas et donne la différentielle quadratique souhaitée.  

\begin{figure}[hbt]
\centering
\begin{tikzpicture}
\begin{scope}[xshift=-4cm]
  \foreach \a in {1}
\filldraw[fill=black!10] ( (0,0) coordinate (p0) -- ++(\a,0)  coordinate[pos=.5] (q1)  coordinate (p1) -- ++(\a,0)   coordinate (p2)-- ++(-\a,\a)  coordinate (p3)-- ++(\a,0) coordinate[pos=.5] (q2)  coordinate (p4)-- ++(-\a,\a)  coordinate (p5)-- ++(-\a,0)  coordinate (p6)-- ++(0,-\a)  coordinate (p7) -- ++(0,-\a)  coordinate (p8);

   \foreach \i in {0,4,5,6,7,8}
   \fill (p\i)  circle (2pt); 
  \foreach \i in {1,2,3}
  \filldraw[fill=white] (p\i)  circle (2pt); 
    \node[below ] at (q1) {$v_{1}$};
  \node[below right] at (q2) {$v_{2}$};
\end{scope}

\begin{scope}[xshift=-.2cm]
  \foreach \a in {.9}
\filldraw[fill=black!10] (0,0) coordinate (p1) -- ++(\a,\a) coordinate (p2)-- ++(\a,\a) coordinate (p3)-- ++(\a,0)  coordinate (p4)-- ++(\a,0)  coordinate (p5)-- ++(\a,0)  coordinate (p6)  --++(-3*\a/2,-\a/2)  coordinate[pos=.5] (q2)  coordinate (p10) --++(0,-\a)  coordinate (p7)-- ++(-\a,0)coordinate (p8)-- ++(-\a,0) coordinate (p9) --++(-3*\a/2,-\a/2) coordinate  coordinate[pos=.5] (q1);

   \foreach \i in {1,...,6}
   \fill (p\i)  circle (2pt); 
  \foreach \i in {7,...,10}
  \filldraw[fill=white] (p\i)  circle (2pt); 
    \node[below ] at (q1) {$v_{1}$};
  \node[below] at (q2) {$v_{2}$};
\end{scope}

\begin{scope}[xshift=6cm,yshift=1cm]

\filldraw[fill=black!10]  (0,0) coordinate (p0) -- ++ (.25,.05)  coordinate[pos=.5] (q1) coordinate (p1)-- ++(1,-1) coordinate (p2)-- ++(1,0) coordinate (p3) -- ++ (.25,.05)  coordinate[pos=.5] (q2)  coordinate (p8) -- ++(1,1.1)  coordinate (p4)-- ++(-1.1,1.1)  coordinate (p5)-- ++(-1.2,0)  coordinate (p6) --++(-1.2,-1.31)  coordinate (p7);

  \foreach \i in {0,4,5,6,7,8}
   \fill (p\i)  circle (2pt); 
  \foreach \i in {1,2,3}
  \filldraw[fill=white] (p\i)  circle (2pt); 
    \node[below ] at (q1) {$v_{1}$};
  \node[below] at (q2) {$v_{2}$};

\end{scope}
\end{tikzpicture}
\caption{Les polygones utilisés pour construire les $2$-différentielles dont les $2$-résidus sont sur trois rayons distincts.}\label{fig:quadtroisresidus}
\end{figure} 
\smallskip
\par
Nous traitons maintenant le cas des résidus quadratiques sur trois rayons distincts mais ne sont pas triangulaires.
Nous commençons par le cas où les résidus quadratiques sont $((R_{1}^{s_{1}}),(R_{2}^{s_{2}}),(R_{3}^{s_{3}})$ avec les $R_{i}$ sur trois rayons distincts. Commençons par le cas où $s_{i}=1$ pour tout~$i$. Dans ce cas on choisit les racines $r_{i}$ telles que la somme est minimal et quitte à permuter les indexes, on peut supposer que  le point final de la concaténation des $r_{i}$ pour $i$ croissant est soit strictement au dessus de la droite engendrée par $r_{1}$, soit appartient au segment~$r_{1}$. On définit $v_{i}=-(r_{1}+r_{2}+r_{3})/2$  et on obtient les deux strates via la construction (C1) soit en concaténant les $r_{i}$ puis les $v_{i}$ soit en concaténant $v_{1}$ puis $r_{1}$ puis $v_{2}$ puis les autres $r_{i}$. Notons que le polygone est dégénéré si et seulement si le point final de la concaténation est sur $r_{1}$. Toutefois la construction (C1) s'étend à ce cas en identifiant les segments correspondants aux $v_{i}$ bord du demi-cylindre collé à $r_{1}$ par rotation et translation.  Dans le cas où il y a plus de pôle on modifie cette construction de la façon suivante. Dans les cas où certains $s_{i}=2$ et les autres $s_{j}=1$ on concatènent deux fois les résidus correspondant et on réalise la même construction. Dans le cas où tous les $s_{i}=2$, alors cette construction peut ne pas fonctionner car le point final de la concaténation peut être~$r_{1}$. Dans ce cas on considère $r_{1}$ puis $2r_{2}$ puis $-r_{1}$ puis $2r_{3}$ et la construction est identique. Lorsque un $s_{i}$ est supérieur ou égal à $3$ on concatène une fois $r_{i}$ et une fois $-r_{i}$ à la construction précédente. Comme cela on obtient les différentielles quadratiques souhaitées.
\par
Supposons maintenant que les $2$-résidus sont sur trois rayons distincts, mais pas tous égaux sur au moins un rayon. On les notera $(R_{1,1},\dots,R_{1,s_{1}},R_{2,1},\dots,R_{2,s_{2}},R_{3,1},\dots,R_{3,s_{3}})$, où les $R_{i,j}$ sont sur un même rayon à $i$ fixé et au moins un $s_{i}\geq2$.  On considère deux cas selon que $s_{i}\geq2$ pour $i=2,3$ ou que $s_{2}=s_{3}=1$.
Dans le cas où $s_{i}\geq2$ pour $i=2,3$, pour chaque $i$, on partitionne $\left\{R_{i,1},\dots,R_{i,s_{i}}\right\}$ en deux sous-ensembles $\mathcal{R}_{i,1}$ et $\mathcal{R}_{i,2}$ vérifiant les propriétés suivantes. La norme de la  somme des racines des éléments de  $\mathcal{R}_{i,1}$ est strictement inférieure à celle des éléments de~$\mathcal{R}_{i,2}$. De plus la cardinalité de~$\mathcal{R}_{i,1}$ est supérieure ou égale à~$\frac{s_{i}}{2}-1$. Il est clair que l'on peut toujours trouver de telles partitions (ce qui n'est en général pas le cas si les résidus quadratiques sont triangulaires).  Le uplet $E$ est constitué des racines carrées $r_{i}$ des éléments des $\mathcal{R}_{i,1}$ qui ont une partie réelle positive (ou une partie imaginaire si la partie réelle est nulle) dans l'ordre des arguments croissant puis des racines $-r_{i}$ de partie réelle négative des~$\mathcal{R}_{i,2}$. Notons que si $s_{1}=1$, alors $\mathcal{R}_{1,1}=\emptyset$ et donc on commence avec les éléments de $\mathcal{R}_{2,1}$. Les vecteurs~$v_{i}$ sont définis comme précédemment comme étant la moitié de l'opposée de la somme de ces racines.  
La construction que nous allons décrire est représentée à droite de la figure~\ref{fig:quadtroisresidus} dans le cas de la strate $\Omega^{2}\moduli[0](3,5;(-2^{6}))$.   Supposons que quitte à multiplier tous les résidus quadratiques par un même nombre complexe, on peut supposer que les racines des éléments de $\mathcal{R}_{1,1}$ sont réels positifs. Notons que la concaténation est sans points d'intersections et que le point final de celle-ci est d'ordonnée strictement négative. On en déduit que l'on peut prendre pour $E_{1}$ les $l_{1}$ premier éléments de~$E$ si ceux-ci sont dans les ensembles $\mathcal{R}_{i,1}$. Enfin dans le cas où $l_{1}$ est supérieur que le nombre de ces éléments, on peut ajouter soit les premières racines de $\mathcal{R}_{1,2}$ ou les dernières de  $\mathcal{R}_{3,2}$ selon que l'argument des $-v_{i}$ est supérieur ou inférieur à celui des racines des $\mathcal{R}_{2,2}$. On conclut alors avec la construction (C1).  

Dans le cas où $s_{2}=s_{3}=1$  la construction est la suivante. Nous commençons par concaténer les racines $r_{1,i}$ des $R_{1,i}$ par norme croissante puis $r_{2}$ puis $r_{3}$. Les vecteurs $v_{i}$ sont définis comme précédemment. Par hypothèse, la somme des $\left[ s_{1}/2 \right]$ premières racines $r_{1,i}$ est strictement à la somme des dernières. Donc le polygone obtenu en concaténant $v_{1}$ puis $l_{1}\leq \left[ s_{1}/2 \right]$  des premières racines $-r_{1,i}$ puis $v_{2}$ puis les autres racines $r_{1,i}$ et enfin $r_{2}$ et $r_{3}$ est non dégénéré.  Donc on peut obtenir toutes les strates sauf celle où $a_{1}=a_{2}$ via la construction (C1).
Dans le cas $a_{1}=a_{2}$ la construction est la suivante.  Dans ce l'ensemble $E_{1}$ contient la racine de partie réelle négative de~$r_{1,1}$ puis les racines de partie réelle positive des éléments des derniers~$r_{1,i}$. On suppose de plus que le premier élément de partie réelle positive est de norme strictement supérieur à~$r_{1,1}$ (ce qui est possible par hypothèse).  Le polygone de type (C1) ainsi formé est dégénéré. Toutefois, la construction (C1) qui lui est associée donne une différentielle quadratique avec les invariants souhaités en collant le demi cylindre correspondant à~$r_{1,1}$ à celui qui correspond au premier élément de partie réelle positive. 

\smallskip
\par 
Considérons enfin les cas où les $R_{i}$ appartiennent à au moins quatre rayons différents. 
Considérons le cas où les $s$ résidus quadratiques $R_{i}$ sont sur $s$ rayons distincts et vérifient $\sum r_{i}=0$ pour un choix de racines $r_{i}$ de $R_{i}$. Nous supposerons de plus que la concaténation des $r_{i}$ pour $i$ croissant est un polygone convexe. Un exercice élémentaire permet de vérifier qu'il existe au moins $s_{1}:=\lceil\tfrac{s}{2}\rceil$ racines $r_{i}$ dont les arguments sont contenus dans un segment de longueur $\pi$. Nous supposerons qu'il s'agit des résidus $r_{1},\dots,r_{s_{1}}$. Nous considérons le uplet~$E$ constitué de  $r_{2},\dots,r_{s_{1}-1}$, puis $-r_{1}$, puis $r_{s_{1}+1},\dots,r_{s}$ et enfin $-r_{s_{1}}$. Les vecteurs $v_{i}$ et les uplets $E_{i}$ sont définis comme précédemment. Cette construction est représentée dans la figure~\ref{fig:quatreplusresidus}. La condition sur les angles assure que le polygone de type (C1) ainsi obtenu est non dégénéré pour toute les strates. Les différentielles quadratiques s'obtiennent avec la construction~(C1).
\begin{figure}[hbt]
\centering
\begin{tikzpicture}
\begin{scope}[xshift=-4cm]
  \foreach \a in {.75}
\filldraw[fill=black!10]  (0,0) coordinate (p1) -- ++(\a,2*\a)   coordinate[pos=.5] (q4)  coordinate (p2)-- ++(\a,\a)   coordinate[pos=.5] (q3) coordinate (p3)-- ++(\a,0)   coordinate[pos=.5] (q2) coordinate (p4)-- ++(0,-\a)   coordinate[pos=.5] (q1) coordinate (p5)-- ++(-3*\a,-2*\a)   coordinate[pos=.5] (q5) coordinate (p6);

\node[right ] at (q1) {$r_{1}$};
\node[above] at (q2) {$r_{2}$};
\node[left ] at (q3) {$r_{3}$};
\node[left] at (q4) {$r_{4}$};
\node[right] at (q5) {$r_{5}$};

\end{scope}

\begin{scope}[xshift=5cm,yshift=2cm]
  \foreach \a in {.75}
\filldraw[fill=black!10]  (0,0) coordinate (p1)  -- ++(-\a,0)  coordinate[pos=.5] (q1)  coordinate (p7) -- ++(-\a,0) coordinate[pos=.5] (q3) coordinate (p5) -- ++(0,-\a) coordinate[pos=.5] (q4) coordinate (p4)  -- ++(-\a,0)  coordinate[pos=.5] (q2)  coordinate (p8) -- ++(-\a,-2*\a) coordinate[pos=.5] (q5) coordinate (p2)-- ++(3*\a,2*\a) coordinate[pos=.5] (q6)  coordinate (p6)-- ++(\a,\a)  coordinate[pos=.5] (q7)  coordinate (p3);

  \foreach \i in {1,2,3,6,8}
  \fill (p\i)  circle (2pt); 
  \foreach \i in {4,7,5}
  \filldraw[fill=white] (p\i)  circle (2pt); 
    \node[above] at (q1) {$v_{1}$};
  \node[below] at (q2) {$v_{2}$};
  
  \node[above ] at (q3) {$r_{2}$};
\node[left] at (q4) {$-r_{1}$};
\node[left ] at (q5) {$r_{4}$};
\node[right] at (q6) {$r_{5}$};
\node[right] at (q7) {$-r_{3}$};
\end{scope}
\end{tikzpicture}
\caption{Les polygones utilisés pour construire une $2$-différentielle de $\Omega^{2}\moduli[0](3,5;(-2^{5}))$ dont les $2$-résidus sont sur $5$ rayons distincts et il existe des racines de somme nulle.}\label{fig:quatreplusresidus}
\end{figure} 
\par
Les cas où il existe une somme de racines des résidus quadratiques nulle et où il y a plusieurs résidus quadratiques sont sur le même rayon se traite de manière analogue au cas précédent. 
\par
Le cas où les résidus quadratiques sont sur au moins $4$ rayons et il n'existe pas de somme des racines $r_{i}$ qui soit nulle se traite de manière similaire au cas de $3$ rayons. On sépare le cas où tout les résidus quadratiques sont égaux entre eux sur chaque rayon et le cas où il existe au moins deux résidus quadratiques distinctes sur un rayon. Les constructions sont les généralisations directes de celles présentées dans ces cas.
\smallskip
\par
Nous considérons maintenant le cas des strates avec $n\geq3$ zéros. Le résultat peut se déduire des strates avec deux zéros par éclatement de zéros, à l’exception du cas des strates $\Omega^{2}\moduli[0](2s'-1,2s'-1,2;(-2^{2s'+2}))$ avec les résidus quadratiques de la forme $((1^{2s'}),R,R)$ avec $R\notin\RR_{+}$. Dans ce cas la construction, représentée sur la figure~\ref{fig:quadtroiszeros}, est la suivante. On construit un triangle en concaténant le vecteur $r$ avec $s'$ fois le vecteur $1$ puis trois fois le vecteur $-\tfrac{r+s'}{3}$ dans le premier cas. On construit un second triangle en concaténant les mêmes vecteurs en commençant par les vecteurs $1$ puis $r$. Les trois derniers vecteurs du premier triangle sont notés $v_{1},v_{2},v_{3}$ dans cet ordre et ceux du second par $w_{i}$. La différentielle quadratique est obtenue en collant des demi-cylindres infinis aux vecteurs $1$ et $r$ et en identifiant $v_{1}$ à $v_{3}$, $w_{1}$ à $w_{3}$ et $v_{2}$ à $w_{2}$.
\begin{figure}[hbt]
\centering
\begin{tikzpicture}
 \filldraw[fill=black!10](0,0) coordinate (p1) -- ++(1,1)  coordinate[pos=.5] (q1)  coordinate (p2)-- ++(1,0) coordinate[pos=.5] (q2) coordinate (p3)-- ++(1,0)coordinate[pos=.5] (q3)  coordinate (p4) -- ++ (-1,-1/3) coordinate [pos=.5] (q4) coordinate (p5) -- ++(-1,-1/3)coordinate[pos=.5] (q5)  coordinate (p6) -- ++(-1,-1/3)coordinate [pos=.5] (q6);

  \foreach \i in {1,2,...,4}
  \fill (p\i)  circle (2pt); 
    \foreach \i in {5,6}
  \filldraw[fill=white] (p\i)  circle (2pt); 
  \node[above left] at (q1) {$r$};
      \foreach \i in {2,3}
  \node[above] at (q\i) {$1$};
  \node[below ] at (q4) {$v_{1}$};
    \node[below] at (q5) {$v_{2}$};
  \node[below ] at (q6) {$v_{3}$};

  \begin{scope}[xshift=1.5cm,yshift=-.5cm]
 \filldraw[fill=black!10](0,0) coordinate (p1) -- ++(1,0) coordinate[pos=.5] (q2) coordinate (p3)-- ++(1,0)coordinate[pos=.5] (q3)  coordinate (p4) -- ++(1,1)  coordinate[pos=.5] (q1)  coordinate (p2) -- ++ (-1,-1/3) coordinate [pos=.5] (q4) coordinate (p5) -- ++(-1,-1/3)coordinate[pos=.5] (q5)  coordinate (p6) -- ++(-1,-1/3)coordinate [pos=.5] (q6);

  \foreach \i in {1,2,...,4}
  \filldraw[fill=red] (p\i)  circle (2pt); 
    \foreach \i in {5,6}
  \filldraw[fill=white] (p\i)  circle (2pt); 
  \node[below right] at (q1) {$r$};
      \foreach \i in {2,3}
  \node[below] at (q\i) {$1$};
  \node[above ] at (q4) {$w_{1}$};
    \node[above] at (q5) {$w_{2}$};
  \node[above ] at (q6) {$w_{3}$};
\end{scope}

\end{tikzpicture}
\caption{Les polygones utilisés pour construire un élément de  $\Omega^{2}\moduli[0](2s'-1,2s'-1,2;(-2^{2s'+2}))$ dont les $2$-résidus sont de la forme $((1^{2s'}),R,R)$.}\label{fig:quadtroiszeros}
\end{figure} 
\end{proof}

Nous nous intéressons maintenant au cas des résidus quadratiques qui appartiennent tous à un même rayon issu de l'origine. Nous énonçons tout d'abord le cas de deux zéros, dont l'un d'ordre~$-1$. Rappelons que la notion de graphe associé à été introduite dans le paragraphe précédent la proposition~\ref{prop:gzeropolesimples}.
\begin{prop}\label{prop:quadmoinsun}
Soient $\Omega^{2}\moduli[0](-1,2s-3;(-2^{s}))$ et $R:=(R_{1},\dots,R_{s})\in(\RR_{+}^{\ast})^{s}$. Le $s$-uplet $R$ appartient à l'image de l'application résiduelle de cette strate si et seulement il existe un graphe associé aux $r:=(r_{1},\dots,r_{s};-r_{s},\dots,-r_{1})$, où  $r_{i}$ est la racine positive de $R_{i}$, qui est un graphe de connexion symétrique.
\end{prop}

\begin{proof}
Supposons qu'il existe un graphe associé aux résidus $r$ qui soit un graphe de connexion symétrique. Dans ce cas il existe une différentielle symétrique d'ordre $2$ dans la strate $\omoduli[0](2s-2;(-1^{2s}))$ dont les résidus sont~$r$ (voir proposition~\ref{prop:gzeropolesimples}). Il est clair que la symétrie à comme point fixe le zéro d'ordre $2s-2$ et un point non polaire de $\PP^{1}$. Le quotient de cette différentielle par la symétrie fournit donc la différentielle quadratique souhaitée.
\par
Réciproquement, supposons que $(r_{1}^{2},\dots,r_{s}^{2})$ est dans l'image de l'application résiduelle de $\Omega^{2}\moduli[0](-1,2s-3;(-2^{s}))$. Alors le revêtement canonique de la différentielle quadratique correspondante est une différentielle abélienne de $\omoduli[0](2s-2;(-1^{2s}))$ dont les résidus sont~$r$. Le graphe de connexion associé est clairement symétrique.
\end{proof}

Le cas des strates où les deux zéros sont positifs et les résidus sont sur le même rayon est très technique. On peut l'aborder en généralisant les graphes de connexion.
Toutefois la combinatoire devient extrêmement complexe et une description de ces cas nécessiterait un autre article.
Nous nous contenterons dans la suite de résultats partiels pour $n\geq3$ zéros. Nous pouvons montrer la surjectivité de l'application résiduelle dans quelques strates, laissant le cas général ouvert.

Nous prouvons la proposition~\ref{prop:quadsurjbcpimp} qui donne la surjectivité de l'application résiduelle de certaines strates $\Omega^{2}\moduli[0](a_{1},\dots,a_{n};(-2^{s}))$.
Rappelons la notation $a_{i}=2l_{i}+\bar{a_{i}}$ avec $\bar{a_{i}}\in\left\{0,-1\right\}$.
\begin{proof}[Preuve de la proposition~\ref{prop:quadsurjbcpimp}.]
Nous commençons par montrer la surjectivité de l'application $2$-résiduelle dans le cas où $n\geq4$ et au moins quatre $a_{i}$ sont impairs. On commence pour supposer que $n=4$ et les $a_{i}$ sont impairs. Les cas avec $n>4$ s'obtiennent alors par éclatement de zéros. Le lemme~\ref{lm:quadgzerogen} implique qu’il suffit de traiter le cas où les résidus quadratiques sont réels positifs. On partitionne ces résidus en quatre sous-ensembles $\mathcal{R}_{i}$ de cardinaux respectifs~$l_{i}$ (on notera que $l_{1}+\dots+l_{4}=s$). On choisit un réel strictement positif $s_{1}$ tel que $$S:=2s_{1}+\sum_{R_{i}\in\mathcal{R}_{1}}\sqrt{R_{i}}-\sum_{j=2,3,4}\sum_{R_{i}\in\mathcal{R}_{j}}\sqrt{R_{i}}>0. $$ Enfin on choisit deux complexes $s_{2}$ et $s_{3}$ de partie réelle égale à  $-\frac{S}{4}$ et de partie imaginaire respectivement strictement positive et strictement négative, tels que $$2\sum_{j=1,2,3}s_{j}+\sum_{R_{i}\in\mathcal{R}_{1}}\sqrt{R_{i}}-\sum_{j=2,3,4}\sum_{R_{i}\in\mathcal{R}_{j}}\sqrt{R_{i}}=0.$$ On forme alors le polygone schématisé à gauche de la figure~\ref{fig:troisquatrezeros}. Les pointillés sont constitués de la concaténations des~$\pm\sqrt{R_{j}}$ dans les $\mathcal{R}_{i}$. On obtient la différentielle quadratique souhaitée en collant les $s_{i}$ deux à deux et des demi-cylindres infinis aux segments $r_{i}$.
\begin{figure}[hbt]
\centering
\begin{tikzpicture}

\begin{scope}[xshift=-3.5cm]
  \foreach \a in {1.2}
\draw[dotted] (0,0) coordinate (p1)-- ++(\a,0)  coordinate (p3) -- ++(2*\a,0) coordinate (p2)coordinate[pos=.5] (r1)  -- ++(\a,0)   coordinate (p4)-- ++(-\a/2,\a)  coordinate (p5)-- ++(-\a/2,0)  coordinate (p6)coordinate[pos=.5] (r2)-- ++(-\a/2,\a)  coordinate (p7)  -- ++(-\a/3,0)  coordinate (p10)coordinate[pos=.5] (r3)-- ++(-\a/2,-\a)  coordinate (p8)-- ++(-7*\a/6,0)  coordinate (p9)coordinate[pos=.5] (r4)-- ++(-\a/2,-\a);

  \draw (p1) -- (p3) coordinate[pos=.5] (q1);
  \draw (p2) -- (p4) coordinate[pos=.5] (q2);

  \draw (p4) -- (p5) coordinate[pos=.5] (q3);
  \draw (p6) -- (p7) coordinate[pos=.6] (q4);
 \draw (p8) -- (p10) coordinate[pos=.5] (q5);

  \draw (p9) -- (p1) coordinate[pos=.5] (q6);  

  \node[above] at (q1) {$s_{1}$}; \node[above] at (q2) {$s_{1}$};
   \node[right] at (q3) {$s_{2}$}; \node[right] at (q4) {$s_{2}$};
     \node[left] at (q5) {$s_{3}$}; \node[left] at (q6) {$s_{3}$};
     
       \node[above] at (r1) {$\mathcal{R}_{1}$}; \node[above] at (r2) {$\mathcal{R}_{2}$};
   \node[below] at (r3) {$\mathcal{R}_{3}$}; \node[above] at (r4) {$\mathcal{R}_{4}$};
  
    \fill (p3)  circle (2pt); \fill (p2)  circle (2pt); 
    \filldraw[fill=white] (p5)  circle (2pt); \filldraw[fill=white] (p6)  circle (2pt); 
  \filldraw[fill=red] (p8)  circle (2pt); \filldraw[fill=red] (p9)  circle (2pt); 
    \filldraw[fill=blue] (p1)  circle (2pt); \filldraw[fill=blue] (p10)  circle (2pt);
    \filldraw[fill=blue] (p4)  circle (2pt); \filldraw[fill=blue] (p7)  circle (2pt); 
  \end{scope}
  
  \begin{scope}[xshift=3.5cm]
  \foreach \a in {1.2}
\draw[dotted] (0,0) coordinate (p1) -- ++(4*\a,0) coordinate (p2)coordinate[pos=.5] (r1)  -- ++(-\a/2,\a)  coordinate (p3)-- ++(-\a,0)  coordinate (p4)coordinate[pos=.5] (r2)-- ++(-\a/2,\a)  coordinate (p5) -- ++(-\a/2,-\a)  coordinate (p6)-- ++(-\a,0)  coordinate (p7)coordinate[pos=.5] (r3)-- ++(-\a/2,-\a);

  \draw (p2) -- (p3) coordinate[pos=.5] (q1);
  \draw (p4) -- (p5) coordinate[pos=.5] (q2);

  \draw (p5) -- (p6) coordinate[pos=.5] (q3);
  \draw (p7) -- (p1) coordinate[pos=.5] (q4);

  \node[right] at (q1) {$s_{1}$}; \node[right] at (q2) {$s_{1}$};
   \node[left] at (q3) {$s_{2}$}; \node[left] at (q4) {$s_{2}$};

       \node[above] at (r1) {$\mathcal{R}_{3}$}; \node[above] at (r2) {$\mathcal{R}_{1}$};
   \node[above] at (r3) {$\mathcal{R}_{2}$}; 
  
    \fill (p1)  circle (2pt); \fill (p2)  circle (2pt); \fill (p5)  circle (2pt); 
    \filldraw[fill=white] (p3)  circle (2pt); \filldraw[fill=white] (p4)  circle (2pt); 
  \filldraw[fill=red] (p6)  circle (2pt); \filldraw[fill=red] (p7)  circle (2pt); 

  \end{scope}

\end{tikzpicture}
\caption{Le polygone pour obtenir une différentielle quadratique de $\Omega^{2}\moduli[0](a_{1},\dots,a_{4};(-2^{s}))$ à gauche et  de $\Omega^{2}\moduli[0](a_{1},\dots,a_{3};(-2^{s}))$ avec $a_{3}$ pair à droite. }\label{fig:troisquatrezeros}
\end{figure} 
\smallskip
\par
Nous prouvons maintenant la surjectivité de l'application $2$-résiduelle dans le cas où $n=3$ et $a_{1}+a_{2}< a_{3}$ avec $a_{3}$ pair. Il suffit de considérer les résidus quadratiques réels positifs.
On forme trois sous-ensembles $\mathcal{R}_{i}$ des résidus de cardinaux respectifs $l_{1}$, $l_{2}$ et $l_{3}+1$. Ces ensembles sont choisis de telle sorte que 
$$ \sum_{R_{i}\in\mathcal{R}_{3}}\sqrt{R_{i}}>\sum_{R_{i}\in\mathcal{R}_{1}\cup\mathcal{R}_{2}}\sqrt{R_{i}}.$$
Une telle partition est possible car  $a_{1}+a_{2}< a_{3}$ implique $l_{3}+1>l_{1}+l_{2}$. Puis on forme le polygone schématisé à droite de la figure~\ref{fig:troisquatrezeros}. Les pointillés sont constitués de la concaténations des $\pm\sqrt{R_{j}}$ dans les $\mathcal{R}_{i}$. On obtient la différentielle quadratique souhaitée en collant les $s_{i}$ deux à deux et des demis-cylindres infinis aux segments $r_{i}$.
\end{proof}
Notons que la preuve donne quelques renseignements supplémentaires.
\begin{rem}
 Avec les notations de la proposition~\ref{prop:quadsurjbcpimp}, on peut considérer le cas où $a_{1}+a_{2}=a_{3}$ et les $a_{i}$ ne sont pas tous pairs. Dans ce cas, si les $R_{i}$ ne sont pas tous égaux entre eux, alors ils sont dans l'image de l'application $2$-résiduelle.
\end{rem}


 \smallskip
\par
\subsubsection{\bf Les strates de $k$-différentielles avec $k\geq3$.}

Dans ce paragraphe, nous traitons le cas des $k$-différentielles avec $k\geq3$. Nous montrons que l'application $k$-résiduelle est surjective sauf dans quelques cas sporadiques liés à la structure de réseau des racines $k$-ième de l'unité pour $k=3,4,6$. Nous montrons tout d'abord l'existence des cas sporadiques puis nous traitons les cas avec deux zéros dont un d'ordre négatif, puis avec deux zéros positifs et enfin le cas des strates avec au moins trois zéros.

Nous prouvons la non-réalisabilité des $k$-résidus dans les cas sporadiques du théorème~\ref{thm:geq0kspe}.  Cette preuve étant très longue, nous avons décidé de traiter tout d'abord le cas (1) et les cas des différentielles quartiques puis le cas des différentielles cubiques et sextiques. Les cas quartiques sont similaires entre eux et reposent sur le fait que les liens selles relient des entiers de Gauss entre eux, i.e. des éléments du réseaux $\ZZ \oplus i \ZZ$.

\begin{proof}[Existence des cas sporadiques (1) et quartiques du  théorème~\ref{thm:geq0kspe}.]
Nous commençons par établir le cas (1), c'est à dire prouver que les $k$-résidus $(1,(-1)^{k})$ n'appartiennent pas à l'image de l'application $k$-résiduelle $\appresk[0](-1,1;-k,-k)$. Si une telle $k$-différentielle existait, la $k$-différentielle entrelacée obtenue en collant les deux pôles d'ordre $-k$ serait lissable (voir lemme~\ref{lem:lisspolessimples}). La $k$-différentielle obtenue par lissage serait dans la strate $\komoduli[1](1,-1)$, qui est vide. 
\smallskip
\par
Nous considérons maintenant le cas (10), c'est à dire que le $4$-résidu $(1,1,1,1)$ n'est pas dans l'image de  $\appresk[0][4](-1,9;(-4^{4}))$. 
\par
Supposons par l'absurde qu'il existe  une $4$-différentielle $\xi$ ayant ces invariants locaux. Elle aurait $4$ demi-cylindres infinis dont les circonférences seraient dans l'ensemble $\lbrace\pm1,\pm i\rbrace$. Il y a deux cas à considérer selon que le zéro d'ordre $-1$ borde un domaine polaire ou non.
\par
Supposons que le zéro d'ordre $-1$ est au bord d'un domaine polaire. Nous pouvons couper la surface le long du lien selle le plus court entre les deux zéros. On obtient une surface plate à bord avec deux segments $v_{1}$ et~$v_{2}$ au bord. Comme le zéro est d'ordre $-1$, l'angle de la singularité correspondante est de  $\frac{3\pi}{2}$. Comme le domaine polaire contribue d'un angle $\pi$ à cette singularité, la surface est obtenue localement en collant au lien selle $\gamma$ au bord du domaine polaire un quadrilatère de côtés $\gamma$, $v_{1}$, $\gamma'$ et $v_{2}$ dont la somme des angles entre $\gamma$ et les $v_{i}$ est égale à~$\tfrac{\pi}{2}$. Donc la longueur de $\gamma'$ est strictement inférieure à~$1$. La surface plate est obtenue en concaténant entre les extrémités de $\gamma'$ les segments correspondant aux liens selles des trois pôles d'ordre $-4$ restant. Comme ces segments appartiennent à $\lbrace\pm1,\pm i\rbrace$ la somme des segments est dans $\ZZ \oplus i\ZZ$, ce qui implique que $\gamma'$ est nul. Et donc la surface plate est singulière.
\par
Supposons maintenant que le zéro d'ordre $-1$ ne borde pas de domaine polaire. On peut alors couper la surface le long du lien selle de longueur minimale entre les deux zéros. La surface obtenue est une surface de translation à bord ayant un angle de $\frac{3\pi}{2}$. Le complément des domaines polaires est alors un hexagone. Toutefois il n'est pas possible de construire un hexagone dont quatre segments appartiennent à $\left\{\pm1,\pm i\right\}$ et les deux derniers sont adjacent, de même longueur et forment un angle interne de~$\frac{3\pi}{2}$. Cela montre que $\appresk[0][4](-1,9;(-4^{4}))$ ne contient pas  $(1,1,1,1)$.
\par 
Dans le cas (8) des $4$-résidus $(1,1,-2)$ dans la strate $\Omega^{4}\moduli[0](-1,5;(-4^{3}))$,  les circonférences des cylindres demi-infinis sont soit dans $\{\pm 1,\pm i\}$ soit une diagonale du carré dont les arêtes sont de longueur~$1$. Ce cas se montre donc de manière complètement similaire au précédent.  
\smallskip
\par
Nous traitons maintenant les cas (9) et (11) où les  $4$-résidus sont égaux à~$1$, respectivement dans les strates $\Omega^{4}\moduli[0](3,5;(-4^{4}))$ et $\Omega^{4}\moduli[0](3,13;(-4^{6}))$. Nous raisonnons comme dans le cas (10) et nous donnerons les principales différences avec celui-ci.

Supposons qu'il existe une $4$-differentielle $\xi$ possédant ces invariants. Découpons $\xi$ le long d’un lien selle reliant les deux singularités coniques pour obtenir une surface de translation de genre zero avec deux bords $\gamma$ et $\gamma'$ tels que $\gamma' = \pm i\gamma$. Cette surface possède un nombre pair de cylindres demi-infinis de circonférences dans $\{ \pm 1,\pm i\}$. Notons que cela implique que les parties réelles et imaginaires de $\gamma$ sont soit toutes les deux paires, soit toutes les deux impaires. De plus, on en déduit que $\gamma$ est un élément de $\ZZ \oplus i\ZZ$. Les vecteurs d'holonomie des liens selle appartiennent donc au réseau~$\ZZ[i]$.
Ceci implique en particulier que les cylindres sont bordés par un unique lien selle, de longueur~$1$. La somme des angles des deux singularités coniques est égale à $8\pi$ dans le cas (9) et à $12\pi$ dans le cas~(11).  On en déduit que le \coeur, qui est connexe, est respectivement un hexagone et un octogone. Deux arêtes de ces polygones sont $\gamma$ et $\gamma'$, tandis que les autres sont collées aux cylindres correspondant aux pôles d'ordre~$-4$.

Dans le cas de la strate $\Omega^{4}\moduli[0](3,5;(-4^{4}))$, les angles aux singularités coniques sont $7\pi/2$ et $9\pi/2$.  Si $\gamma$ et $\gamma'$ sont consécutifs, alors la singularité correspondante est d'angle au plus~$2\pi$, ce qui est exclu. 
S'il  existe exactement un pôle entre $\gamma$ et $\gamma'$, alors la singularité bordant les autres cylindres est composée d'un angle de $3\pi$ pour les cylindres et au moins $\pi$ pour les deux angles dans l'hexagone.  Elle correspond donc au  zéro d’ordre~$5$. Il reste alors un angle de $\pi/2$ pour les deux angles entre ces cylindres et respectivement $\gamma$ et $\gamma'$.  Cela est impossible pour un polygone dont les sommets sont des points entiers du plan.  Enfin supposons que~$\gamma$ et~$\gamma'$ séparent les pôles en deux sous-ensembles de cardinal~$2$. Chaque singularité est composée d'un angle de $2\pi$ pour les cylindres et respectivement $3\pi/2$ et $5\pi/2$ dans l’hexagone.  Pour les angles participant au zéro d’ordre $3$, l'un de ceux-ci se situe entre les deux cylindres. Cet angle ne peut être égal à $\pi$ car la somme des deux angles restant serait de $\pi/2$. L’angle entre les deux cylindres est donc égal à $\pi/2$ et la somme des deux autres angles est égal à~$\pi$. On vérifie facilement que cela est également impossible dans un polygone dont les sommets sont entiers.
\par
Maintenant, dans la strate $\Omega^{4}\moduli[0](3,13;(-4^{6}))$ les angles des singularités sont respectivement $7\pi/2$ et $17\pi/2$. Nous distinguons les cas selon la répartition des six cylindres entre $\gamma$ et~$\gamma'$. La singularité conique d'angle $7\pi/2$ ne peut être adjacente qu'au plus à deux cylindres et doit être adjacente à au moins un cylindre. Les deux répartitions possibles sont donc en sous-ensembles de cardinaux respectifs $(5,1)$ et  $(4,2)$. Avec la répartition $(5,1)$, deux angles de l’octogone contribuent au zéro d’ordre $3$, pour un angle total de $5\pi/2$.  Donc les angles entre l'arête $r$ correspondant au cylindre et les arêtes $\gamma$ et $\gamma'$ peuvent être de la forme $(2\pi,\pi/2)$ et $(3\pi/2,\pi)$. Donc modulo translation les arêtes $(\gamma,r,\gamma')$ sont de la forme $(-1,1,i)$ et $(i,1,1)$. Dans les deux cas, il est facile de vérifier qu'il faut au moins $7$ arêtes dans $\{\pm 1,\pm i\}$ pour connecter les points initial et final de la concaténation (tel que l'octogone est au dessus du vecteur~$1$).  Enfin, avec la répartition $(4,2)$, la singularité conique d'angle $\tfrac{7\pi}{2}$  de $2\pi$ provenant des cylindres et $3\pi/2$ venant de trois angles de l’octogone. Cela implique que le point initial de $\gamma$ coïncide avec le point final de $\gamma'$. La surface ainsi formée est donc singulière. 
\end{proof}

Nous passons maintenant aux cas cubiques et sextiques  du  théorème~\ref{thm:geq0kspe}. Ces cas sont similaires entre eux et la preuve repose sur le fait que les racines $3$-ièmes et $6$-ièmes de l'unité engendre le réseau $\ZZ \oplus \exp(2i\pi/3)\ZZ$.
%
%

\begin{proof}[Existence des cas sporadiques cubiques et sextiques du  théorème~\ref{thm:geq0kspe}.]
Nous commençons par les cas~(2) et~(3). Ce sont les cas des strates $\Omega^{3}\moduli[0](-1,4;(-3^{3}))$ et $\Omega^{3}\moduli[0](1,2;(-3^{3}))$ avec  les $3$-résidus égaux à  $(1,1,1)$. Supposons qu'il existe une telle $3$-différentielle~$\xi$. Les liens selles de bord des trois domaines polaires sont des racines $3$-ième de l'unité. Cela implique qu'ils sont collés sur le \coeur de~$\xi$. La somme des angles des singularités coniques de $\xi$ étant $6\pi$, et les angles au bord de chaque domaine polaire étant de $\pi$, on en déduit que la géométrie de la surface plate est de deux types possibles. Dans le premier cas, chaque domaine polaire est bordé par un unique lien selle de longueur $1$ tandis que le \coeur est un pentagone dont deux côtés $\gamma$ et $\gamma'$ sont identifiés. 
Dans le second cas, deux domaines polaires sont bordés par un lien selle de longueur $1$ tandis qu'un troisième est bordé par deux liens selles dont la somme des longueurs est~$1$. Le \coeur est alors un quadrilatère dont les arêtes appartiennent aux directions des racines $3$-ième de l'unité. Cela implique que le quadrilatère possède deux arêtes consécutives formant un angle de~$\pi$. La singularité associée à ce  sommet est alors d'angle  $2\pi$, ce qui est absurde.
\par
Dans le cas où le \coeur est un pentagone avec deux arêtes $\gamma,\gamma'$  identifiés et trois arêtes correspondant au bord des cylindres. Il y a deux cas à considérer selon que $\gamma$ et $\gamma'$ sont consécutifs ou il existe une arête $r$ entre les deux.

Dans le premier cas,  l'angle entre $\gamma$ et $\gamma'$ inférieur à $2\pi$, ce qui implique que $\xi$ est dans la strate $\Omega^{3}\moduli[0](-1,4;(-3^{3})$.
Donc l'angle de cette singularité est $\frac{4\pi}{3}$. Il est alors facile de voir qu'il faut au moins $4$ racines $3$-ième de l'unité pour relier le point initial de $\gamma$ au point final de $\gamma'$ sans créer de points d’intersection.

Nous considérons maintenant le cas où il existe une arête $r$ entre $\gamma$ et $\gamma'$. Dans le cas de la strate $\Omega^{3}\moduli[0](-1,4;(-3^{3}))$, la singularité conique d'ordre $-1$ est d'angle $\frac{4\pi}{3}$ et il s'agit nécessairement de celle située aux deux extrémités de~$r$. La somme des deux angles bordant ce côté dans le pentagone est donc $\frac{\pi}{3}$. Cela implique que l'angle entre $\gamma$ et $r$ est nul et donc la surface est singulière. Nous supposerons donc que $\xi$ est dans la strate $\Omega^{3}\moduli[0](1,2;(-3^{3}))$. Il est facile de vérifier que le zéro d'ordre $1$ doit correspondre aux bord de~$r$. Supposons que $r=1$, alors la somme des angles aux sommets au bord de $r$ est $5\pi/3$, ce qui implique que la somme des deux autres arêtes ne peut pas être égale à~$1$. Donc les  arêtes doivent être de la forme $(\gamma,1,\gamma',\exp(2i\pi/3),\exp(2i\pi/3))$. Cela implique que $\gamma= \exp(-2i\pi/3)$ et $\gamma'= \exp(-i\pi/3)$. Cela donne une surface dégénérée.
\par
Notons que le cas (12) de la strate $\Omega^{6}\moduli[0](-1,7;(-6^{3})$  et des $6$-résidus égaux à $(1,1,1)$ peut se traiter de manière similaire. En effet, en coupant  le complémentaire des domaines polaires le long d'un lien selle, on obtient un pentagone. En notant $\gamma$ et $\gamma'$ les arêtes de ce pentagone identifiées par rotation, il y a deux cas à considérer selon que $\gamma$ et $\gamma'$ sont consécutives ou qu'il existe une arête entre les deux. Dans le premier cas, il faut au moins quatre racines $6$-ième de l'unité pour relier le point initial de $\gamma$ au point final de $\gamma'$. Dans le second cas, la distance entre le point initial de $\gamma$ et final de $\gamma'$ serait strictement inférieure à~$1$. Cela est impossible et on pourra consulter le cas (10) pour plus de détails. 
\smallskip
\par
Nous justifions maintenant le cas (4), i.e. il n'existe pas de $3$-différentielle dans la strate $\Omega^{3}\moduli[0](4,2;(-3^{4}))$ dont les $3$-résidus sont  $(1,1,-1,-1)$.   On montre comme précédemment que le \coeur est un hexagone et que les arêtes identifiées par rotation $\gamma,\gamma'$ séparent les pôles en deux ensembles de cardinal~$2$.  Considérons les pôles qui contribuent au zéro d'ordre~$2$.  Supposons que les $3$-résidus de ces pôles sont égaux à~$1$. Si les deux racines des $3$-résidus sont égales à~$1$, alors l'hexagone est de la forme $(\gamma=-1,1,1,\gamma'=\exp(2i\pi/3),-1,\exp(-i\pi/3))$ et  la surface associée est singulière. Dans le cas où les racines sont égales à $1$ et $\exp(2i\pi/3)$, la seule façon d'obtenir une surface telle que le point inital de $\gamma$ et le point final de $\gamma'$ soient  distance supérieure ou égale à $1$ (sans auto-intersection) est que $\gamma=\exp(-2i\pi/3)$ et $\gamma'=-1$. Donc la surface associée est singulière. Enfin, il reste à considérer le cas où les deux $3$-résidus sont distincts. Dans ce cas, la somme des deux racines est soit nulle, soit égale à $\sqrt{3}$. On déduit facilement de cela qu'il n'est pas possible d'obtenir les invariants souhaités.
\smallskip
\par
Nous justifions maintenant le cas (5), i.e. il n'existe pas de $3$-différentielle dans la strate $\Omega^{3}\moduli[0](2,7;(-3^{5}))$ dont les $3$-résidus sont  $(1^{5})$ ou $((1^{4}),-1)$.  Dans les deux cas, le \coeur est un heptagone et il est facile de vérifier que  les arêtes identifiées par rotation $\gamma,\gamma'$ séparent les pôles en deux ensembles de cardinaux respectifs~$2$ et~$3$. De plus, le sous-ensemble de cardinal $2$ contribue au zéro d'ordre~$2$. Supposons que les deux $3$-résidus sont égaux. Supposons que les deux racines sont égales entre elles, disons à~$1$. On peut voir que la somme des autres racines ne possède pas une partie imaginaire nulle. En effet, la seule possibilité est que les racines sont $\exp(2i\pi/3)$, $-1$ et $\exp(-2i\pi/3)$. Comme leur somme est de norme strictement supérieure à~$2$, on ne peut pas former la $3$-différentielle souhaité. Donc cela implique que $\gamma=1$ ou $\gamma'=1$ et donc la surface est dégénérée.   Si les deux racines sont distinctes, alors l'argument du cas~(4) fonctionne. Si les deux $3$-résidus sont distincts, alors les racines sont de la forme $1$ et $\exp(i\pi/3)$. Donc $\gamma= \exp(-2i\pi/3)$ et $\gamma'=-1$, ce qui implique que la surface est singulière.
\smallskip
\par
Le cas (6) qui dit que les $3$-résidus $((1^{3}),(-1^{3}))$ et  $(1^{6})$ ne sont pas réalisé dans  la strate $\Omega^{3}\moduli[0](2,10;(-3^{6}))$ est  en tout point similaire aux cas (4) et (5). Nous ne donnerons donc pas les détails de ce cas.
\smallskip
\par
Nous traitons enfin le cas (7). Nous montrons donc que les $3$-résidus $(1^{6})$ ne sont pas réalisés dans $\Omega^{3}\moduli[0](7,5;(-3^{6}))$. Le \coeur est un octogone dont deux arêtes $\gamma,\gamma'$ sont identifiées par rotation. Cela donne une partition des pôles et il est facile de voir que les sous-ensembles sont de cardinaux $(2,4)$ ou $(3,3)$. Dans le cas où l'un des sous-ensembles est de cardinal $2$, les pôles correspondants contribuent au zéro d'ordre~$2$.  Dans ce cas, on peut montrer qu'il faut au moins $7$ autres pôles pour relier le point initial de $\gamma$ au point final de~$\gamma'$. On   peut donc se concentrer sur le cas où les deux sous-ensembles sont de cardinal~$3$. Dans ce cas soit les $3$ racines des pôles contribuant au zéro d'ordre~$2$ sont égales à~$1$, soit $2$ sont égales à $1$ et une est égale à $\exp(2i\pi/3)$. Dans le premier cas, des considérations similaires à précédemment montrent que la somme des racines des autres pôles est horizontale. On vérifie que cela est impossible (pour cela on notera que $-1$ n'est pas une racine cubique de l'unité).  Dans le dernier cas, la somme des angles que forment le point final de $\gamma$ et le point initial de $\gamma'$ avec ces racines doit être de~$\pi$. Comme le point initial de $\gamma$ et le point final de $\gamma'$ doivent être sur le réseau $\ZZ \oplus \exp(i\pi/3)$, on en déduit que les angles sont $(\pi,0)$, $(2\pi/3,\pi/3)$, $(\pi/3,2\pi/3)$ ou $(\pi,0)$. Dans tout les cas, la surface ainsi formée est singulière.
\end{proof}

Nous prouvons maintenant le  théorème~\ref{thm:geq0kspe} dans le cas des strates où un zéro est d'ordre négatif. Rappelons que ce théorème énonce la surjectivité de l'application résiduelle, sauf dans les cas exceptionnels.
La preuve, une récurrence sur le nombre de pôles de la strate, repose sur le lemme~\ref{lem:constravecneg} donnant une condition suffisante pour construire de telles $k$-différentielles.

\begin{proof}[Preuve du théorème~\ref{thm:geq0kspe}, un zéro est d'ordre négatif.]
Nous traitons tout d'abord le cas où $-k<a_{1}<-\left[\frac{k}{2}\right]$. Prenons des racines $k$-ième $r_{i}$ de $R_{i}$ satisfaisant aux conditions suivantes. La somme des $r_{i}$ est non nulle et le $\RR$-espace vectoriel engendré par les $r_{i}$ est $\CC$. Pour $k\geq3$ ces deux conditions peuvent être facilement satisfaites simultanément. Nous supposons que les~$r_{i}$ sont ordonnés par argument décroissant.
Nous formons un polygone en concaténant les résidus par argument croissant, puis en reliant le sommet final au sommet initial par deux segments satisfaisant les deux propriétés suivantes. Ils sont de même longueur et l'angle à leur intersection est égal à $(a_{1}+k)\frac{2\pi}{k}$. Remarquons que ce polygone est sans point d'auto-intersection. La différentielle est obtenue en collant des demi-cylindres aux arêtes $r_{i}$ et les deux arêtes spéciales entre elles par rotation.
\smallskip
\par
Nous traitons maintenant le cas plus subtil où $0>a_{2}\geq-\left[\frac{k}{2}\right]$. L'application $k$-résiduelle des strates avec un unique pôle d'ordre~$-k$ est trivialement surjective. 
\par
Dans le cas de deux pôles d'ordre~$-k$, le lemme~\ref{lem:constravecneg} implique sans difficultés que tous les $k$-résidus qui ne sont pas proportionnels à $(1,(-1)^{k})$ sont dans l'image de l'application résiduelle des strates $\komoduli[0](a_{1},a_{2};-k,-k)$. Supposons maintenant que les résidus soient $(1,(-1)^{k})$ et que $(a_{1},a_{2})\neq(1,-1)$. Alors par la remarque~\ref{rem:quadmoinsun}, on peut obtenir une $k$-différentielle avec ces invariants.
\smallskip
\par
Nous considérons maintenant les cas avec $s\geq3$ pôles d'ordre~$-k$. Nous montrons par récurrence sur $s$ que les hypothèses du lemme~\ref{lem:constravecneg} sont presque toujours vérifiées. Les quelques cas restant seront traités à la main. 
\par
Considérons tout d'abord le cas où tous les $k$-résidus $R_{i}$ sont de norme $1$. Afin d'utiliser le lemme~\ref{lem:constravecneg} il suffit de montrer  qu'il existe $s-1$ racines $r_{i}$ des $R_{i}$ telles que la somme appartient au disque ouvert épointé $\Delta^{\ast}$. Nous supposerons dans un premier temps que si $k=4,6$, alors les $R_{i}$ ne sont pas tous égaux entre eux et que si $k=3$, alors les $3$-résidus $R_{i}$ ne sont pas égaux ou opposés entre eux. S'il y a $s=3$ pôles, il suffit de prendre deux racines $r_{1}$ et $r_{2}$ telles que les arguments $r_{1}$  et $-r_{2}$sont de différence strictement inférieure à~$\tfrac{\pi}{3}$. Cela est clairement toujours possible sous nos hypothèses.  De plus, notons que ces deux racines ne sont pas proportionnelles. 
Par récurrence on suppose que la somme de $s-1$ racines est $t_{s-1}\in\Delta_{\ast}$ et qu'au moins deux racines ne sont pas proportionnelles. Pour $r_{s} $ tel que $\left<t_{s-1}, r_{s}\right> = \min_{r|r^{k}=R_{s}}\left\{\left<t_{s-1}, r\right>\right\}$ il est aisé de vérifier que $t_{s-1}+r_{s}\in\Delta^{\ast}$. On peut donc conclure par le lemme~\ref{lem:constravecneg} à l'existence de $k$-différentielles avec les invariants souhaités.
\par
Considérons les cas des $k$-résidus de la forme $(1^{s})$ pour $k=4,6$ et $((1^{s_{1}}),(-1^{s_{2}}))$ pour $k=3$. Nous construisons une $k$-différentielle dans $\komoduli[0]((s-2)k+1,-1;(-k^{s}))$ avec ces $k$-résidus sauf dans les cas sporadiques (2), (10) et (12) du théorème~\ref{thm:geq0kspe}. Notons que ces strates sont les seules qui ne soient pas vides lorsque $0>a_{2}\geq-\lfloor \tfrac{k}{2} \rfloor$ et $k\in\lbrace 3,4,6\rbrace$ (voir lemme~\ref{lem:puissk}).
\par
Si $k=3$ et les $3$-résidus sont $(1^{s})$, nous construisons le polygone de type (C1) en concaténant $\lfloor\tfrac{s}{2}\rfloor$ segments  $\e^{i\pi/3}$ puis $\lceil\tfrac{s}{2}\rceil$ fois $-1$ et en joignant le point initial avec le point final par deux segments d'égale longueur faisant un angle de $\frac{2\pi}{3}$ à leur intersection.  Pour $s>3$, cela nous donne un polygone non dégénéré et on conclut par la construction~(C1).
\par
Si les $3$-résidus sont $((1^{s_{1}}),(-1^{s_{2}}))$ on procède de la façon suivante. Dans le cas de la strate $\Omega^{3}\moduli[0](-1,4;(-3^{3}))$ et des $3$-résidus $(1,1,-1)$ on utilise la construction (C1) avec les racines $-1$, puis $\exp(2i\pi/3)$, puis $1$ puis les deux vecteurs~$v_{i}$. Pour la strate $\Omega^{3}\moduli[0](-1,7;(-3^{4}))$ et des $3$-résidus $(1,1,-1,-1)$ on utilise la construction (C1) avec les racines $1$, puis $\exp(i\pi/3)$, puis $\exp(2i\pi/3)$, puis $-1$ puis les deux vecteurs~$v_{i}$. Dans les autres cas, on part du polygone de type (C1) du paragraphe précédent construit avec $s_{1}+\lfloor \tfrac{s_{2}}{2} \rfloor$ segments. Remarquons qu'un triangle équilatéral permet de remplacer un $3$-résidu égal à $1$ par deux $3$-résidus égaux à~$-1$, voir la figure~\ref{fig:transfork3}. Cela implique le résultat pour $s_{2}$ pair. Si $s_{2}$ est impair, on construit le polygone de type (C1) en ajoutant un segment $\exp(-2i\pi/3)$ après les $s_{1}+\lfloor \tfrac{s_{2}}{2} \rfloor$ segments de la construction précédente.
\begin{figure}[hbt]
\centering
\begin{tikzpicture}
\begin{scope}[xshift=3cm]
\filldraw[fill=black!10]  (0,0) coordinate (p1)  -- ++(120:1)  coordinate[pos=.5] (q1)  coordinate (p2) -- ++(1,0) coordinate[pos=.5] (q2) coordinate (p3) -- ++(1,0) coordinate[pos=.5] (q3) coordinate (p4)  -- ++(-120:1)  coordinate[pos=.5] (q4)  coordinate (p5) -- ++(150:.577) coordinate[pos=.5] (q5) coordinate (p6)-- ++(-150:.577) coordinate[pos=.5] (q6) ;

  \foreach \i in {1,...,5}
  \fill (p\i)  circle (2pt); 
  \filldraw[fill=white] (p6)  circle (2pt); 

  \node[below] at (q5) {$v_{1}$};
  \node[below] at (q6) {$v_{2}$};
  
  \node[left] at (q1) {$1$};
\node[above] at (q2) {$1$};
\node[above] at (q3) {$1$};
\node[right] at (q4) {$1$};
\end{scope}

\begin{scope}[xshift=-3cm]
\filldraw[fill=black!10]  (0,0) coordinate (p1)  -- ++(120:1)  coordinate[pos=.5] (q1)  coordinate (p2) -- ++(1,0) coordinate[pos=.5] (q2) coordinate (p3) -- ++(60:1) coordinate[pos=.5] (q3) coordinate (p4) -- ++(-60:1) coordinate[pos=.5] (q4) coordinate (p5)  -- ++(-120:1)  coordinate[pos=.5] (q5)  coordinate (p6) -- ++(150:.577) coordinate[pos=.5] (q6) coordinate (p7)-- ++(-150:.577) coordinate[pos=.5] (q7) ;
\draw[dotted] (p3) -- (p5);

  \foreach \i in {1,...,6}
  \fill (p\i)  circle (2pt); 
  \filldraw[fill=white] (p7)  circle (2pt); 

  \node[below] at (q6) {$v_{1}$};
  \node[below] at (q7) {$v_{2}$};
  
  \node[left] at (q1) {$1$};
\node[above] at (q2) {$1$};
\node[left] at (q3) {$-1$};
\node[right] at (q4) {$-1$};
\node[right] at (q5) {$1$};
\end{scope}
\end{tikzpicture}
\caption{Transformation d'un polygone de type (C1) pour obtenir un élément de $\Omega^{3}\moduli[0](7,-1;(-3^{4}))$ dont les $3$-résidus sont $(1^{4})$ en un élément de $\Omega^{3}\moduli[0](10,-1;(-3^{5}))$ dont les $3$-résidus sont $((1^{3}),(-1^{2}))$.}\label{fig:transfork3}
\end{figure} 
\par
Dans le cas $k=4$ nous construisons un polygone de type (C1) comme suit. Si le nombre de pôles $s$ est impair, le uplet $E$ est constitué de $\lfloor s/2 \rfloor$ fois $-i$, du segment~$1$ et de  $\lfloor s/2 \rfloor$ fois~$i$. Comme la différence entre le point initial et final de la concaténation des éléments de $E$ est~$1$, nous pouvons joindre ces deux points par deux segments de même longueur faisant un angle de~$\frac{3\pi}{2}$. Si $s$ est paire, $E$ est constitué de $\tfrac{s}{2} -1$ fois $-i$, puis deux fois~$1$ et enfin  $ \tfrac{s}{2} -1$ fois~$i$.  La différence entre le point initial et final est donc le segment $2$ et nous joignons ces deux points par deux segments de même longueur faisant un angle de  $\frac{3\pi}{2}$. Ce polygone ne possède pas de points d’auto-intersection pour $s\geq5$ et $s=3$. Ainsi la construction (C1) donne la $4$-différentielle désirée.
\par
Enfin il reste à traiter le cas $k=6$. Si $s$ est pair, $E$ est formé de  $\frac{s}{2}-1$ fois $-1$, puis $e^{2i\pi/3}$, puis $e^{i\pi/3}$ et finalement $\frac{s}{2}-1$ fois~$1$. Si $s$ est impaire, $E$ est formé des vecteurs précédents et de $e^{-i\pi/3}$ à la fin. Nous joignons le point initial au point final de la concaténation des éléments de $E$ par deux segments de même longueur et d'angle $\frac{5\pi}{3}$. Ce polygone ne possède pas de points d’auto-intersection pour $s\geq4$ et on conclut par la construction (C1).
\smallskip
\par
Nous considérons maintenant les cas où les $k$-résidus n'ont pas tous la même norme. Commençons par le cas $s=3$. Soit $R_{1}$ un $k$-résidu de plus grande norme. Si $k\neq 3,4$ et si $(R_{2},R_{3})\neq (1,-1)$ dans le cas où $k=5$, alors comme précédemment il existe des racines de $R_{2}$ et $R_{3}$ dont la somme est de norme strictement comprise entre $O$ et~$|R_{1}|$. Dans le cas où $k=5$ et $(R_{2},R_{3}) = (1,-1)$, alors on considère deux cas. Si $|R_{1}| \geq 3$, alors il suffit de prendre $r_{2}=1$ et $r_{3}= \exp(3i\pi/5)$ car la norme de $r_{2}+r_{3}$ est strictement inférieure à celle de la racine $5$-ième de $3$. Si $1<R_{1}<3$, alors on considère la racine $r_{1}$ d'angle dans  $\left] -\tfrac{\pi}{5}, \tfrac{\pi}{5}\right]$. On vérifie facilement que la norme de $r_{1}-1$ est strictement inférieure à~$1$. Cela conclut ce cas spécial et le lemme~\ref{lem:constravecneg} permet alors de conclure dans tous les cas.
\par
Dans le cas où  $k=4$,  si $R_{2}\neq R_{3}$, alors la construction du paragraphe précédent fonctionne. Si $R_{2}=R_{3}=1$ et  $|R_{1}|>1$ il y a deux cas  à considérer. Si $|R_{1}|>4$, alors en prenant $r_{2}=1$ et $r_{3}=i$ on peut utiliser le lemme~\ref{lem:constravecneg}.
Supposons maintenant que  $|R_{1}|\leq4$.  Si les $4$-résidus  sont distincts de $(1,1,-4)$ alors nous permutons ces $4$-résidus de sorte que $R_{2}$ soit le plus grand, avec $R_{2} \neq 4$ et $R_{1}=R_{3}=1$. Dans ce cas, la somme de racine quatrième  de $R'_{3}$ et $R'_{2}$ se trouve dans~$\Delta^{\ast}$ et le lemme~\ref{lem:constravecneg} permet de conclure. Notons que  cette construction ne fonctionne pas dans le cas $(-4,1,1)$ car les racines quatrièmes de  $-4$ sont $\pm 1\pm i$ et celles de~$1$ sont $\pm1$ et~$\pm i$. De fait, on a montré que ce cas est un cas sporadique du théorème~\ref{thm:geq0kspe}.
\par
Dans le cas $k=3$, une analyse similaire fonctionne.  Si $R_{2}\neq - R_{3}$, la construction du polygone de type (C1) ne pose aucun problème. Si $R_{2}= - R_{3}=1$ et  $|R_{1}|>1$, alors nous considérons deux cas. Si $|R_{1}|>3\sqrt{3}$, alors il existe deux racines troisièmes des $3$-résidus $R_{2}$ et $R_{3}$ telles que la somme est inférieure ou égale à la racine cubique de $|R_{1}|$ et on peut utiliser le lemme~\ref{lem:constravecneg}.
Supposons maintenant que  $|R_{1}|\leq 3\sqrt{3}$.  Alors on prend la racine $r_{1}$ dont l'argument est compris dans le segment $\left] -\tfrac{\pi}{3};\tfrac{\pi}{3} \right]$. Maintenant on concatène la racine $-1$ de $-1$ si l'argument est compris entre $ -\tfrac{\pi}{6}$ et $\tfrac{\pi}{6}$ et $|R_{1}|<3\sqrt{3}$. Dans le cas où l'argument est de norme strictement supérieur à $\tfrac{\pi}{6}$, on concatène la racine $\pm \exp(2i\pi/3)$ de $1$. Dans tous les cas, la somme est dans $\Delta^{\ast}$, ce qui permet d'utiliser le lemme~\ref{lem:constravecneg} pour faire la construction de type~(C1). Dans le dernier cas où $R_{1}=3i\sqrt{3}$ on concatène $1$, puis $\exp(i\pi/3)$, puis la racine d'argument $5\pi/6$ de $R_{1}$. On utilise alors la construction (C1).
%
\par
Considérons maintenant par récurrence le cas où $s\geq 4$. Comme précédemment nous considérerons tout d'abord les cas $k\neq3,4$. Soient $(R_{1},\dots,R_{s+1})$ des résidus où l'on suppose que la norme de $R_{1}$ est maximale et celle de $R_{s+1}$ minimale. Par récurrence, il existe des racines $k$-ième (non toutes proportionnelles) aux 
 $(R_{2},\dots,R_{s})$ telles que leur somme $t_{s}$ est non nulle et de longueur strictement inférieure à  $|R_{1}|^{\frac{1}{k}}$. Nous sommes maintenant ramenés au cas de trois $k$-résidus $(R_{1},t_{s}^{k},R_{s+1})$. Ainsi la construction précédente nous donne des racines qui satisfont aux hypothèses du lemme~\ref{lem:constravecneg}.  Cela implique la surjectivité de l'application $k$-résiduelle de ces strates.
\par
Dans le cas $k=3$, la construction est similaire sauf dans le cas où les $3$-résidus sont de la forme $((1^{s_{1}}),(-1^{s_{2}}),R_{s+1})$ avec $1 \leq |R_{s+1}| \leq 3\sqrt{3}$. Dans les cas où $s+1=4$ pôles, si $s_{1}=3$ on concatène $2$ fois $1$ puis $\exp(i\pi/3)$ puis la racine de $R_{4}$ avec argument $5\pi/6$. Si $s_{1}=2$ et $s_{2}=1$, alors on concatène $2$ fois $1$ puis la racine de $R_{4}$ avec argument $5\pi/6$ puis $-1$. Dans les deux cas on utilise alors la construction~(C1). Dans le cas où il y a plus de $s+1\geq5$ pôles, on concatène les vecteurs $1$ puis les vecteurs $\exp(i\pi/3)$ puis les vecteurs $-1$ tel que la partie réelle de la somme de ces vecteurs est égale à~$-1/2$. Puis on concatène la racine de $R_{4}$ dont l'argument est $3\pi/2$ et on conclut avec la construction~(C1).
\par
Dans le cas $k=4$, pour tous les $4$-résidus qui ne sont pas de la forme $(1,\dots,1,-4)$ la construction du paragraphe précédent  se généralise sans problèmes. Supposons maintenant que $R=(1,\dots,1,-4)$. S'il y a un nombre impair de~$1$, alors il existe une somme de racines (non toutes proportionnelles) de longueur $1$ qui vérifient les hypothèses du lemme~\ref{lem:constravecneg}. Si le nombre $s-1$ de~$1$ est pair, on considère le polygone de type (C1) associé à $E$ donné par $\tfrac{s-3}{2}$ fois $-i$, puis $2$ fois $1$, puis $\tfrac{s-3}{2}$ fois $i$ et enfin $\sqrt{2}\exp(\frac{3i\pi}{4})$. Ce polygone est non dégénéré et la construction (C1) donne la $4$-différentielle souhaitée.
\end{proof}

Nous prouvons maintenant le théorème~\ref{thm:geq0kspe} dans les cas où les deux zéros sont  positifs. Ce théorème énonce la surjectivité de l'application résiduelle, sauf dans des cas sporadiques.
La preuve  repose sur le lemme~\ref{lem:constravecnegbis}. Dans cette preuve, nous utilisons les notations de ce lemme, en particulier $a_{i}=kl_{i}+\bar{a_{i}}$ avec $-k<\bar{a_{1}}<\bar{a_{2}}<0$.

\begin{proof}[Preuve du théorème~\ref{thm:geq0kspe}, les deux zéros sont positifs.] 
On considère  une strate de la forme $\komoduli[0](a_{1},a_{2};(-k^{s}))$ et  un $s$-uplet $R:=(R_{1},\cdots,R_{s})$ dans $(\CC^{\ast})^{s}$.  
Nous débutons par les cas où $k\geq5$. Considérons les $l_{2}+1\geq2$ $k$-résidus de normes maximales, que l'on supposera être $(R_{l_{1}+1},\dots,R_{s})$. Nous normalisons les $k$-résidus de telle sorte que $\min_{i=l_{1}+1,\dots,s}\left\{|R_{i}|\right\}=1$. On peut choisir les racines $r_{i}$ des~$R_{i}$, telles que la somme des $l_{2}+1$ dernières $r_{i}$ soit de longueur supérieure à $|1+\exp(\tfrac{i\pi}{5})|$ et celle des  $l_{1}$ premières soit non nulle et  inférieure ou égale à~$1$. Cela se montre facilement par récurrence sur $l_{1}$ et $l_{2}$ sauf si $k=5$ et $(R_{1},R_{2})=(1,-1)$. Dans ce dernier cas, il suffit de prendre $(r_{1},r_{2})=(1,\exp(3i\pi/5))$. Comme  $|1+\exp(\tfrac{i\pi}{5})|>\sqrt{2}$ (et $|1+\exp(\tfrac{i\pi}{5})| / |1+\exp(\tfrac{3i\pi}{5})| >\sqrt{2}$ dans le dernier cas), le lemme~\ref{lem:constravecnegbis} implique l'existence d'une $k$-différentielle avec les invariants souhaités. 
\smallskip
\par
Nous considérons maintenant le cas $k=4$. Nous traitons d'abord le cas des strates $\Omega^{4}\moduli[0](a_{1},a_{2};(-4^{s}))$ où $a_{2}=4l_{2}-1$ avec $l_{2}\geq2$. Supposons que les $l_{2}+1$ derniers résidus sont de norme maximale. Nous normaliserons les $4$-résidus de telle sorte que  $\min_{i=l_{1}+1,\dots,l_{s}}\left\{|R_{i}|\right\}=1$. On peut choisir des racines des $4$-résidus~$R_{i}$ telles que la somme des $l_{2}+1$ dernières est strictement supérieure à $2$ et celle des~$l_{1}$ premières est inférieure à~$\sqrt{2}$. Le lemme~\ref{lem:constravecnegbis} permet alors de conclure.
\par
On considère maintenant les strates de la forme  $\Omega^{4}\moduli[0](4(s-3)+1,3;(-4^{s}))$. Supposons tout d'abord que les $4$-résidus ne sont pas de la forme $(1,\dots,1)$. On supposera que les deux derniers $4$-résidus sont de normes maximales et que  $|R_{s-1}|=1\leq |R_{s}|$. Si $|R_{s}|^{1/4}\geq\sqrt{2}$, alors il existe des racines $r_{s-1}$ et $r_{s}$ dont la somme est strictement supérieure à~$2$. On peut choisir les racines des autres $4$-résidus non toutes proportionnelles entre elles telles que la norme de leur somme est inférieure à~$\sqrt{2}$. On applique alors le lemme~\ref{lem:constravecnegbis} afin d'obtenir une $4$-différentielle avec les invariants souhaités. Si $|R_{s}|^{1/4}<\sqrt{2}$, il y a deux cas à traiter selon qu'il existe une somme des $s-2$ premières racines des $4$-résidus strictement inférieure à~$1$ ou non. Le premier cas ne pose pas de problème par le lemme~\ref{lem:constravecnegbis}. Si toutes les sommes sont supérieures ou égales à~$1$, alors on permute le rôle de $R_{s}$ avec $R_{1}$. On choisit alors des racines des $4$-résidus telles que la somme des $s-2$ premières est inférieure à $1$ et celle des deux dernières est supérieure ou égale à~$\sqrt{2}$. La possibilité de ce choix se montre par récurrence sur $s$ et le lemme~\ref{lem:constravecnegbis}  nous permet de conclure ce cas.
\par
Dans le cas $k=4$, il reste à considérer le cas des $4$-résidus égaux à $(1,\dots,1)$. 
Dans le cas des strates où $l_{1}$ est impair et $l_{2}=1$, on utilise la construction (C1) avec les éléments suivants. Le uplet $E_{1}$ constitué de $\tfrac{l_{1}-1}{2}$ racines égales à $-i$, puis une égale à~$1$ et $\tfrac{l_{1}-1}{2}$ racines égales à~$i$. Le uplet $E_{2}$ est constitué de $2$ racines égales à~$-1$. 
Nous traitons maintenant le cas où~$l_{1}$ est pair et $l_{2}=1$.
Si  $l_{1}\geq6$ on peut faire la construction (C2) avec les éléments suivants. Le uplet $E_{1}$ est constitué de $\tfrac{s-4}{2}$ fois $i$, de~$3$ fois $-1$ et de $\tfrac{s-4}{2}$ fois $-i$. Le uplet $E_{2}$ est le uplet~$(1)$. 
%
\smallskip
\par 
Enfin, nous considérons le cas $k=3$. Le cas des $3$-résidus qui ne sont pas de la forme $(1,\dots,1,-1,\dots,-1)$ se montre de manière similaire au cas $k=4$.  Plus précisément, on considère le cas des strates $\Omega^{3}\moduli[0](a_{1},a_{2};(-3^{s}))$ où $a_{i}=3l_{i}+\bar{a}_{i}$. Supposons que les $l_{2}+1$ derniers résidus sont de norme maximale. Nous normaliserons les $3$-résidus de telle sorte que  $\min_{i=l_{1}+1,\dots,l_{s}}\left\{|R_{i}|\right\}=1$.  Dans ce cas, il existe une somme non nulle de  $3$-racines $r_{i}$ pour $i\leq l_{1}$ de norme inférieure ou égale à~$\sqrt{3}$.  Le résultat se déduit immédiatement du lemme~\ref{lem:constravecnegbis} pour $l_{2}\geq3$. Dans le cas $l_{2}=2$, notons qu'il y a un problème seulement si les $l_{1}$ premiers $3$-résidus sont dans $\left\{ \pm 1 \right\}$ et les racines troisième des  $l_{2}+1$ derniers sont de normes strictement inférieures à~$2$. Dans ce cas, on peut échanger l'un des $l_{1}$ premiers $3$-résidus avec un des $l_{2}+1$ dernier qui n'est pas dans $\left\{ \pm 1 \right\}$ afin de pouvoir utiliser le lemme~\ref{lem:constravecnegbis}. 
\par
Dans le cas des strates avec $l_{2}=1$, 
on supposera que les deux derniers $3$-résidus sont de normes maximales. Si $|R_{i}|^{1/3} < 1/2$ pour tout $i\leq l_{1}$, alors on peut appliquer directement le lemme~\ref{lem:constravecnegbis} afin d'obtenir une $3$-différentielle avec les invariants souhaités. Si $|R_{l_{1}}|^{1/3}\geq 1/2$, 
il y a deux cas à traiter selon qu'il existe une somme des $l_{1}$ premières racines des $3$-résidus strictement inférieure à~$\sqrt{3}/2$ ou non. Le premier cas ne pose pas de problème par le lemme~\ref{lem:constravecnegbis}. Si toutes les sommes (non nulles) sont de normes supérieures ou égales à~$\sqrt{3}/2$, alors on permute le rôle de $R_{1}$ avec $R_{l_{1}+1}$. La somme des deux dernières racines est alors de norme strictement supérieure à $1$. On choisit alors des racines des $3$-résidus telles que la somme des $l_{1}$ premières est inférieure à $1/2$. La possibilité de ce choix se montre par récurrence sur $l_{1}$ et le lemme~\ref{lem:constravecnegbis}  nous permet de conclure ce cas.
\smallskip
\par
Nous considérerons maintenant le cas des $s$-uplets de la forme $((1^{t_{1}}),(-1^{t_{2}}))$.
Si $l_{2}\geq3$, alors on peut choisir des racines cubiques des $3$-résidus satisfaisant les conditions suivantes. Notons qu'au moins trois $3$-résidus sont égaux entre eux, que l'on supposera égaux à~$1$. On peut donc choisir des racines telles que la somme des $l_{2}+1$ dernières racines est de longueur supérieure à~$\left|3+ \exp(i\pi/3) \right| > \tfrac{7}{2}$. On choisit les~$l_{1}$ premières racines non toutes proportionnelles et de somme de norme~$1$ ou~$\sqrt{3}$. Le lemme~\ref{lem:constravecnegbis} permet alors de conclure à la surjectivité de l'application $3$-résiduelle dans ces cas.   
\par
Si $l_{2}=2$, i.e. $a_{2}=5$, alors on peut appliquer le lemme~\ref{lem:constravecnegbis} s'il existe une somme de $l_{1}$ racines de norme~$1$.  En particulier, il suffit de considérer le cas  $l_{1}\geq2$. On sait que l'on peut toujours obtenir une somme de longueur $1$ ou $\sqrt{3}$, donc on supposera que l'on a choisit des racines dont la somme est de norme~$\sqrt{3}$. Si les $3$-résidus ne sont pas tous égaux à~$1$, on peut échanger un $3$-résidu de cette somme par son opposé. Avec ces nouveaux $3$-résidus on peut clairement obtenir une somme des racines égale à~$1$. 
 \par
 Considérons maintenant le cas où tous les $3$-résidus sont égaux à~$1$. Si $l_{1}$ n'est pas divisible par $3$, notons $l_{1}=3l_{1}'-t$ avec $t\in\lbrace 1,2\rbrace$. Comme $l_{1}\geq2$, on a $l_{1}'\geq1$ et si $l_{1}'=1$ alors $t=1$. On fait la construction~(C1) avec $E_{2}=(-1,-1,-1)$ et $E_{1}$ constitué de $l_{1}'$ fois $\exp(\tfrac{-i\pi}{3})$, de $l_{1}'$ fois $\exp(\tfrac{i\pi}{3})$ et enfin de $l_{1}'-t$ fois~$-1$. Soit la concaténation de $E_{1}$ est égale à~$1$, soit à~$2$. On peut vérifier que le polygone de type~(C1) ainsi obtenu est non dégénéré (voir la figure~\ref{fig:a2egale5etk3} à gauche).
Dans le cas où $l_{1}$ est divisible par $3$  on utilise la construction~(C2). On traite tout d'abord le cas $l_{1}>6$. On considère le uplet $E_{2}=(-1,-1)$ et le uplet $E_{1}$ constitué de $\lfloor\tfrac{l_{1}+1}{2}\rfloor$ fois $\exp(\tfrac{-i\pi}{3})$, puis  $\lfloor\tfrac{l_{1}+1}{2}\rfloor$ fois $\exp(\tfrac{i\pi}{3})$ puis $-1$ si $l_{1}$ est pair. La construction (C2) est clairement réalisable. Comme le cas $l_{1}=5$ correspond au cas sporadique (7), il suffit de montrer que $(1^{9})$ est réalisable dans $\Omega^{3}\moduli[0](16,5;(-3^{9}))$ où $l_{1}=6$. Cela peut s'obtenir en considérant la concaténation des vecteurs
$$(-1,1,\exp(2i\pi/3),\exp(2i\pi/3),\exp(-2i\pi/3),\exp(-2i\pi/3),\exp(-2i\pi/3),1,1,1,\exp(2i\pi/3) )\,.$$
On colle alors par rotation le premier vecteur $1$ au second vecteur  $\exp(2i\pi/3)$, puis on colle les cylindres aux autres segments. La différentielle cubique ainsi réalisée possède les invariants souhaités.

\begin{figure}[hbt]
\centering
\begin{tikzpicture}
\begin{scope}[xshift=-6cm]
\filldraw[fill=black!10]  (0,0) coordinate (p1)  -- ++(1,0)  coordinate[pos=.5] (q1)  coordinate (p2) -- ++(1,0) coordinate[pos=.5] (q2) coordinate (p3) -- ++(1,0) coordinate[pos=.5] (q3) coordinate (p4)  -- ++(-150:1.155)  coordinate[pos=.6] (q4)  coordinate (p5) -- ++(1,0) coordinate[pos=.5] (q5) coordinate (p6)-- ++(-120:1) coordinate[pos=.5] (q6)coordinate (p7) -- ++(-120:1) coordinate[pos=.5] (q7)coordinate (p8) -- ++(120:1) coordinate[pos=.5] (q8)coordinate (p9)-- ++(120:1) coordinate[pos=.5] (q9)coordinate (p10)-- ++(150:1.155)   coordinate[pos=.4] (q10);

  \foreach \i in {1,...,4}
    \filldraw[fill=white] (p\i)  circle (2pt); 
  \foreach \i in {5,...,10}
    \fill (p\i)  circle (2pt); 

  \node[left] at (q10) {$v_{1}$};
  \node[right] at (q4) {$v_{2}$};
  
  \node[above] at (q1) {$-1$};
\node[above] at (q2) {$-1$};
\node[above] at (q3) {$-1$};
  \node[below] at (q5) {$-1$};
\node[right] at (q6) {$\exp(\tfrac{i\pi}{3})$};
\node[right] at (q7) {$\exp(\tfrac{i\pi}{3})$};
\node[left] at (q8) {$\exp(\tfrac{-i\pi}{3})$};
\node[left] at (q9) {$\exp(\tfrac{-i\pi}{3})$};
\end{scope}

\begin{scope}[xshift=0cm, yshift=-1cm]
\filldraw[fill=black!10]  (0,0) coordinate (p1)  -- ++(1,0)  coordinate[pos=.5] (q1)  coordinate (p2) -- ++(30:.577) coordinate[pos=.5] (q2) coordinate (p3) -- ++(-120:1)  coordinate[pos=.6] (q3)  coordinate (p4) -- ++(-1,0) coordinate[pos=.5] (q4) coordinate (p5)  -- ++(120:1) coordinate[pos=.5] (q5) coordinate (p6)-- ++(-30:.577) coordinate[pos=.5] (q6);

  \foreach \i in {1,2}
    \filldraw[fill=white] (p\i)  circle (2pt); 
  \foreach \i in {3,...,6}
    \fill (p\i)  circle (2pt); 

      \node[above] at (q1) {$-1$};
\node[right] at (q3) {$\exp(\frac{i\pi}{3})$};
\node[below] at (q4) {$1$};
  \node[left] at (q5) {$\exp(\frac{-i\pi}{3})$};
    
 
\end{scope}

\begin{scope}[xshift=5cm, yshift=-1cm]
\filldraw[fill=black!10]  (0,0) coordinate (p1)  -- ++(1,0)  coordinate[pos=.5] (q1)  coordinate (p2) -- ++(30:1.73) coordinate[pos=.5] (q2) coordinate (p3) -- ++(-120:1)  coordinate[pos=.6] (q3)  coordinate (p4) -- ++(-120:1) coordinate[pos=.5] (q4) coordinate (p5)  -- ++(-1,0) coordinate[pos=.5] (q5) coordinate (p6)-- ++(-1,0) coordinate[pos=.5] (q6) coordinate (p7)-- ++(120:1) coordinate[pos=.5] (q7) coordinate (p8)-- ++(120:1) coordinate[pos=.5] (q8) coordinate (p9)  -- ++(-30:1.73) coordinate[pos=.5] (q9) ;

  \foreach \i in {1,2}
    \filldraw[fill=white] (p\i)  circle (2pt); 
  \foreach \i in {3,...,9}
    \fill (p\i)  circle (2pt); 

%
 
\end{scope}
\end{tikzpicture}
\caption{Les polygones dans le cas des strates cubiques avec $a_{2}=5$ et $l_{1}$ n'est pas divisible par $3$ à gauche, avec $a_{2}=2$ et $a_{1}=4$ pour les $3$-résidus $(1,1,1,-1)$ au centre et avec $a_{2}=2$ et $a_{1}=13$ pour les $3$-résidus $((1^{5}),(-1^{2}))$ à droite.}\label{fig:a2egale5etk3}
\end{figure} 
\par
Il reste maintenant à considérer le dernier cas où $l_{2}=1$, i.e. $a_{2}=2$.  Nous commençons par traiter chacun des cas $l_{1}\leq 5$, puis nous traitons le cas $l_{1}>5$. 
\par
 Pour $l_{1}=1$, i.e. dans la strate $\Omega^{3}\moduli[0](1,2;(-3^{3}))$, et les $3$-résidus de la forme $(1,1,-1)$, nous utilisons la construction~(C1) avec $E_{1}=(1)$ et $E_{2}=(-1,-1)$. Le cas $(1,1,1)$ correspond au cas sporadique~(3). 
\par
A présent, supposons que $l_{1}=2$, i.e. la strate $\Omega^{3}\moduli[0](4,2;(-3^{4}))$. Si les $3$-résidus sont $(1^{4})$, alors on part de la $3$-différentielle de $\Omega^{3}\moduli[0](1,2;(-3^{3}))$ dont les $3$-résidus sont $(1,1,-1)$. On remplace alors le résidu~$-1$ par deux résidus~$1$ par la construction de la figure~\ref{fig:transfork3}. Notons qu'il un pôle spécial, dont le domaine est bordé par le zéro d'ordre $2$. Nous ne pourrons pas appliquer cette opération à ce pôle sans changer l'ordre du zéro d'ordre~$2$. Nous supposerons par la suite que le $3$-résidu de ce pôle est~$1$.  Dans le cas où les $3$-résidus sont $((1^{3}),-1)$, on utilise la construction~(C2)  avec $E_{1}=(\exp(\tfrac{-i\pi}{3}),1,\exp(\tfrac{i\pi}{3}))$ et $E_{2}=(-1)$, voir la figure~\ref{fig:a2egale5etk3} centrale.  Enfin le dernier cas $(1,1,-1,-1)$ correspond au cas sporadique~(4).
\par
Les cas avec $l_{1}=3$ et $l_{1}=4$, i.e. des strates $\Omega^{3}\moduli[0](7,2;(-3^{5}))$ et $\Omega^{3}\moduli[0](10,2;(-3^{6}))$, se traitent en remplaçant un $3$-résidu par deux $3$-résidus opposés par la construction de la figure~\ref{fig:transfork3}. On obtient donc que les $3$-résidus de la forme $((1^{t_{1}}),(-1^{t_{2}}))$ qui ne sont pas de la forme $((1^{4}),-1)$,  $(1^{6})$ et  $((1^{3}),(-1^{3}))$ sont dans l'image de l'application $3$-résiduelle de ces strates. Enfin, on a montré que les cas restant ne sont pas réalisables et correspondent aux cas (5) et (6) du théorème~\ref{thm:geq0kspe}. 
\par
Enfin considérons le cas $l_{1}=5$, i.e. la strate $\Omega^{3}\moduli[0](13,2;(-3^{7}))$. On obtient les $3$-résidus distincts de $((1^{5}),-1,-1)$ via la construction de la figure~\ref{fig:transfork3}.  Montrons qu'il existe une différentielle cubique dont les $3$-résidus sont de la forme $((-1^{5}),1,1)$. Cela peut se faire grâce à la construction (C2) avec $E_{1}=(\exp(2i\pi/3),\exp(2i\pi/3),-1,-1,\exp(-2i\pi/3),\exp(-2i\pi/3))$ et $E_{2}=(1)$, voir la figure~\ref{fig:a2egale5etk3} de droite.  
\par
Comme pour $l_{1}=5$ l'application $3$-résiduelle contient les éléments de la forme $((1^{t_{1}}),(-1^{t_{2}}))$  de ces strates. On peut remplacer un $3$-résidu égal à~$1$ (resp.~$-1$) par deux $3$-résidus égaux à~$-1$ (resp.~$1$) via la construction de la figure~\ref{fig:transfork3}. On en déduit que tous les $3$-résidus de la forme $((1^{t_{1}}),(-1^{t_{2}}))$ sont dans l'image de l'application résiduelle pour $l_{1} > 5$.
\end{proof}

Pour terminer, nous montrons qu'il n'existe pas d'obstruction pour les strates de $k$-différentielles de genre zéro avec $k\geq3$, n'ayant que des pôles d'ordre $-k$ et au moins trois zéros.
\begin{prop}
L'application résiduelle de la strate $\komoduli[0](a_{1},\dots,a_{n};(-k^{s}))$ est surjective pour $k\geq3$ et  $n\geq3$.
\end{prop}

\begin{proof}
La preuve se fait par induction sur le nombre de zéros. On commence par vérifier que, partant d'une strate avec trois zéros, on peut presque toujours additionner l'ordre de deux zéros et obtenir une strate dont l'application $k$-résiduelle est surjective. Les quelques exceptions seront traités à la main. Le cas général est alors obtenu par éclatement de zéros des $3$-différentielles avec $3$ zéros. 

Remarquons que la strate que nous obtenons en additionnant deux des trois zéros peut ne pas être primitive. Toutefois, il existe une strate qui n'est pas constituée par des $k$-différentielles qui sont la puissance $k$-ième de différentielles abélienne. De plus, si $k$ est pair, comme exactement deux zéros sont d'ordres impaires, les strates obtenues ne sont pas les puissances $\tfrac{k}{2}$-ième d'une strate quadratique.  L'équation~\eqref{eq:multiplires} établit alors la surjectivité de l'application $k$-résiduelle des strate avec $3$ zéros sauf dans le cas où les strates obtenues sont sporadiques.

Il reste donc à considérer les cas $k=3,4,6$. Dans le cas $k=6$, la seule strate sporadique est la strate $\Omega^{6}\moduli[0](-1,7;(-6^{3}))$. Dans ce cas, les ordres des zéros problématiques sont de la forme $(-1,a,7-a)$ ou $(-1-b,b,7)$. Dans le premier cas, si $a$ est pair, il suffit de considérer les strates avec les zéros $(-1+a,7-a)$ et les zéros $(a,7-a-1)$ dans le cas où $a$ est impaire. La même considération permet de traiter les strates de la seconde forme.

Dans le cas où $k=4$, il y a quatre strates sporadiques. Ces cas se traitent de manière similaire à précédemment sauf dans les deux cas suivants. Les strates de $4$-différentielles $\Omega^{4}\moduli[0](-1,3,6;(-4^{4}))$ et $\Omega^{4}\moduli[0](-1,4,5;(-4^{4}))$ sont telles que les strates que l'on obtient  en rassemblant les zéros  sont  $\Omega^{4}\moduli[0](3,5;(-4^{4}))$ et  $\Omega^{4}\moduli[0](-1,9;(-4^{4}))$ et une strate qui paramètre la puissance d'une différentielle abélienne ou quadratique. Donc dans ces deux cas, il reste à montrer que les $4$-résidus $(1,1,1,1)$ sont réalisables. Des $4$-différentielles satisfaisant à ces conditions sont données dans la figure~\ref{fig:k4troiszeros}.
\begin{figure}[hbt]
\centering
\begin{tikzpicture}[scale=2]
\begin{scope}[xshift=-6cm]
\filldraw[fill=black!10]  (0,0) coordinate (p1)  -- ++(1,0)  coordinate[pos=.5] (q1)  coordinate (p2) -- ++(0,1) coordinate[pos=.5] (q2) coordinate (p3) -- ++(-.25,-.25) coordinate[pos=.5] (q3) coordinate (p4)   -- ++(-.5,0)  coordinate[pos=.6] (q4)  coordinate (p5) -- ++(-.25,.25) coordinate[pos=.5] (q7)coordinate (p8) -- cycle;

 \foreach \i in {1,2,3,8}
     \fill (p\i)  circle (1pt); 
  \foreach \i in {4,5}
    \filldraw[fill=white] (p\i)  circle (1pt); 
    
    \node[above,rotate=45] at (q3) {$1$};
   \node[below,rotate=135] at (q7) {$1$};
   \node[above] at (q4) {$3$};

\fill[fill=black!10] (0.25,1.5) coordinate (p10) -- ++(.5,0)  coordinate[pos=.6] (q10)  coordinate (p11)   -- ++(-.75,.25)  coordinate[pos=.6] (q11)  coordinate (p12) -- ++(.25,.75) coordinate[pos=.5] (q12) coordinate (p13) -- ++ (-1.1,0) coordinate[pos=.75] (q13) coordinate (p14)  -- ++(0,-1)  -- (p10) coordinate[pos=.25] (q14)coordinate (p15);
\draw (q14) -- (p11) -- (p12) -- (p13) -- (q13);
\draw[dotted] (q14) --++(-.2,0);
\draw[dotted] (q13) --++(-.2,0);

  \foreach \i in {10,11,13}
    \filldraw[fill=white] (p\i)  circle (1pt); 
  \foreach \i in {12}
    \filldraw[fill=red] (p\i)  circle (1pt); 
    
\node[above,rotate=-20] at (q11) {$2$};
\node[below,rotate=70] at (q12) {$2$};
\node[below] at (q10) {$3$};
    
\fill[fill=black!10] (-0.25,0) coordinate (p10) -- ++(0,1)  coordinate[pos=.6] (q10)  coordinate (p11)    -- ++ (-.9,0) coordinate[pos=.75] (q13) coordinate (p14)  -- ++(0,-1)  -- (p10) coordinate[pos=.25] (q14)coordinate (p15);
\draw (q14) -- (p10) -- (p11) -- (q13);
\draw[dotted] (q14) --++(-.2,0);
\draw[dotted] (q13) --++(-.2,0);

 \foreach \i in {10,11}
     \fill (p\i)  circle (1pt); 

\fill[fill=black!10] (1.25,0) coordinate (p10) -- ++(0,1)  coordinate[pos=.6] (q10)  coordinate (p11)    -- ++ (.9,0) coordinate[pos=.75] (q13) coordinate (p14)  -- ++(0,-1)  -- (p10) coordinate[pos=.25] (q14)coordinate (p15);
\draw (q14) -- (p10) -- (p11) -- (q13);
\draw[dotted] (q14) --++(.2,0);
\draw[dotted] (q13) --++(.2,0);
 \foreach \i in {10,11}
     \fill (p\i)  circle (1pt); 

\fill[fill=black!10] (0,-.25) coordinate (p10) -- ++(1,0)  coordinate[pos=.6] (q10)  coordinate (p11)    -- ++ (0,-.9) coordinate[pos=.75] (q13) coordinate (p14)  -- ++(-1,0)  -- (p10) coordinate[pos=.25] (q14)coordinate (p15);
\draw (q14) -- (p10) -- (p11) -- (q13);
\draw[dotted] (q14) --++(0,-.2);
\draw[dotted] (q13) --++(0,-.2);
 \foreach \i in {10,11}
     \fill (p\i)  circle (1pt); 
\end{scope}

\begin{scope}[xshift=-2cm]
\filldraw[fill=black!10]  (0,0) coordinate (p1)  -- ++(1,0)  coordinate[pos=.5] (q1)  coordinate (p2) -- ++(0,1) coordinate[pos=.5] (q2) coordinate (p3) -- ++(-.25,0) coordinate[pos=.5] (q3) coordinate (p4)   -- ++(-.5,0)  coordinate[pos=.6] (q4)  coordinate (p5) -- ++(-.25,0) coordinate[pos=.5] (q7)coordinate (p8) -- cycle;

 \foreach \i in {1,2,3,8}
     \fill (p\i)  circle (1pt); 
  \foreach \i in {4,5}
    \filldraw[fill=white] (p\i)  circle (1pt); 
    
    \node[above] at (q3) {$1$};
   \node[below,rotate=180] at (q7) {$1$};
   \node[above] at (q4) {$3$};

\fill[fill=black!10] (0.25,1.5) coordinate (p10) -- ++(.5,0)  coordinate[pos=.6] (q10)  coordinate (p11)   -- ++(-.75,.25)  coordinate[pos=.6] (q11)  coordinate (p12) -- ++(.25,.75) coordinate[pos=.5] (q12) coordinate (p13) -- ++ (-1.1,0) coordinate[pos=.75] (q13) coordinate (p14)  -- ++(0,-1)  -- (p10) coordinate[pos=.25] (q14)coordinate (p15);
\draw (q14) -- (p11) -- (p12) -- (p13) -- (q13);
\draw[dotted] (q14) --++(-.2,0);
\draw[dotted] (q13) --++(-.2,0);

  \foreach \i in {10,11,13}
    \filldraw[fill=white] (p\i)  circle (1pt); 
  \foreach \i in {12}
    \filldraw[fill=red] (p\i)  circle (1pt); 
    
\node[above,rotate=-20] at (q11) {$2$};
\node[below,rotate=70] at (q12) {$2$};
\node[below] at (q10) {$3$};
    
\fill[fill=black!10] (-0.25,0) coordinate (p10) -- ++(0,1)  coordinate[pos=.6] (q10)  coordinate (p11)    -- ++ (-.9,0) coordinate[pos=.75] (q13) coordinate (p14)  -- ++(0,-1)  -- (p10) coordinate[pos=.25] (q14)coordinate (p15);
\draw (q14) -- (p10) -- (p11) -- (q13);
\draw[dotted] (q14) --++(-.2,0);
\draw[dotted] (q13) --++(-.2,0);

 \foreach \i in {10,11}
     \fill (p\i)  circle (1pt); 

\fill[fill=black!10] (1.25,0) coordinate (p10) -- ++(0,1)  coordinate[pos=.6] (q10)  coordinate (p11)    -- ++ (.9,0) coordinate[pos=.75] (q13) coordinate (p14)  -- ++(0,-1)  -- (p10) coordinate[pos=.25] (q14)coordinate (p15);
\draw (q14) -- (p10) -- (p11) -- (q13);
\draw[dotted] (q14) --++(.2,0);
\draw[dotted] (q13) --++(.2,0);
 \foreach \i in {10,11}
     \fill (p\i)  circle (1pt); 

\fill[fill=black!10] (0,-.25) coordinate (p10) -- ++(1,0)  coordinate[pos=.6] (q10)  coordinate (p11)    -- ++ (0,-.9) coordinate[pos=.75] (q13) coordinate (p14)  -- ++(-1,0)  -- (p10) coordinate[pos=.25] (q14)coordinate (p15);
\draw (q14) -- (p10) -- (p11) -- (q13);
\draw[dotted] (q14) --++(0,-.2);
\draw[dotted] (q13) --++(0,-.2);
 \foreach \i in {10,11}
     \fill (p\i)  circle (1pt); 
\end{scope}

\end{tikzpicture}
\caption{Une $4$-différentielle de $\Omega^{4}\moduli[0](-1,3,6;(-4^{4}))$ à gauche et $\Omega^{4}\moduli[0](-1,4,5;(-4^{4}))$ à droite dont les $4$-résidus sont $(1,1,1,1)$.}\label{fig:k4troiszeros}
\end{figure}

Dans le cas où $k=3$, on vérifie comme précédemment que les strates qui posent problème sont les strates $\Omega^{3}\moduli[0](-1,1,3;(-3^{3}))$ et $\Omega^{3}\moduli[0](-1,2,2;(-3^{3}))$ avec trois pôles, et $\Omega^{3}\moduli[0](2,3,7;(-3^{6}))$ et   $\Omega^{3}\moduli[0](2,5,5;(-3^{6}))$ avec six pôles. Dans ces cas, les $3$-résidus qui posent problèmes sont $(1^{3})$ et $(1^{6})$ respectivement. Des $3$-différentielles dans les strates $\Omega^{3}\moduli[0](-1,1,3;(-3^{3}))$ et $\Omega^{3}\moduli[0](-1,2,2;(-3^{3}))$ avec $3$-résidus égaux à $(1,1,1)$ sont représentés dans la figure~\ref{fig:k3troiszeros}.
\begin{figure}[hbt]
\centering
\begin{tikzpicture}[scale=2]
\begin{scope}[xshift=-6cm]
\fill[fill=black!10]  (0,0) coordinate (p1)  -- ++(1,0)  coordinate[pos=.5] (q1)  coordinate (p2) -- ++(1,0) coordinate[pos=.5] (q2) coordinate (p3) -- ++(120:1) coordinate[pos=.5] (q3) coordinate (p4)   -- ++(150:0.577)  coordinate[pos=.6] (q4)  coordinate (p5) -- ++(90:.577) coordinate[pos=.5] (q5)coordinate (p6) -- (60:1) coordinate[pos=.5] (q6)coordinate (p7) --  (p1) coordinate[pos=.5] (q7);
\fill[fill=black!10] (p1) -- ++(-.1,-.3) -- ++ (1,0) -- (p2) -- cycle; 
\fill[fill=black!10] (p2) -- ++(.1,-.3) -- ++ (1,0) -- (p3) -- cycle; 
\fill[fill=black!10] (p6)  -- ++(-.3,.1) -- ++ (-120:1) -- (p7) -- cycle; 
\draw (p3) --(p4) -- (p5) -- (p6);
\draw (p7) -- (p1);
\draw (p1) -- ++(-.1,-.3);
\draw (p2) -- ++(-.1,-.3);
\draw (p2) -- ++(.1,-.3);
\draw (p3) -- ++(.1,-.3);
\draw (p6) -- ++(-.3,.1);
\draw (p7) -- ++(-.3,.1);

 \foreach \i in {1,2,3}
     \fill (p\i)  circle (1pt); 
  \foreach \i in {4,6,7}
    \filldraw[fill=white] (p\i)  circle (1pt); 
    \foreach \i in {5}
    \filldraw[fill=red] (p\i)  circle (1pt); 
    
    \node[above,rotate=-60] at (q3) {$1$};
   \node[below,rotate=-120] at (q7) {$1$};
   \node[above,rotate=-20] at (q4) {$2$};
     \node[below,rotate=100] at (q5) {$2$};
\end{scope}


\begin{scope}[xshift=-2cm,yshift=.5cm]
\fill[fill=black!10]  (0,0) coordinate (p1)  -- ++(1,0)  coordinate[pos=.5] (q1)  coordinate (p2) --  ++(-60:0.5) coordinate[pos=.5] (q2) coordinate (p3) -- ++(120:1) coordinate[pos=.5] (q3) coordinate (p4)   -- ++(-150:0.288)  coordinate[pos=.6] (q4)  coordinate (p5) -- ++(150:.288) coordinate[pos=.5] (q5)coordinate (p6) -- ++(-120:0.5) coordinate[pos=.5] (q6)coordinate (p7) -- ++(60:0.5) coordinate[pos=.5] (q7) coordinate (p7) -- ++(-120:1) coordinate[pos=.5] (q7) coordinate (p8)  --  (p1)  coordinate[pos=.5] (q8);

 \fill[fill=black!10] (p1) -- ++(0,-.5) -- ++ (1,0) -- (p2) -- cycle; 
 \fill[fill=black!10] (p3) -- ++(.5,.1) -- ++ (120:1) -- (p4) -- cycle; 
 \fill[fill=black!10] (p7)  -- ++(-.5,.1) -- ++ (-120:1) -- (p8) -- cycle; 
 \draw (p2) --(p3);
 \draw (p4) -- (p5) -- (p6);
 \draw (p8) -- (p1);
 \draw (p1) -- ++(0,-.5);
 \draw (p2) -- ++(0,-.5);
 \draw (p3) -- ++(.5,.1);
 \draw (p4) -- ++(.5,.1);
 \draw (p7) -- ++(-.5,.1);
 \draw (p8) -- ++(-.5,.1);

 \foreach \i in {1,2}
     \fill (p\i)  circle (1pt); 
  \foreach \i in {3,4,6,7,8}
    \filldraw[fill=white] (p\i)  circle (1pt); 
    \foreach \i in {5}
    \filldraw[fill=red] (p\i)  circle (1pt); 
    
    \node[above,rotate=-60] at (q2) {$1$};
   \node[below,rotate=-120] at (q8) {$1$};
   \node[right,rotate=-60] at (q4) {$2$};
     \node[left,rotate=60] at (q5) {$2$};
\end{scope}

\end{tikzpicture}
\caption{Les $3$-différentielles dans les strates $\Omega^{3}\moduli[0](-1,1,3;(-3^{3}))$ et $\Omega^{3}\moduli[0](-1,2,2;(-3^{3}))$ dont les $3$-résidus sont égaux à $(1,1,1)$.}\label{fig:k3troiszeros}
\end{figure} 
\par
Pour obtenir, les autres $3$-différentielles on fait les constructions suivantes. Pour la $3$-différentielle $\Omega^{3}\moduli[0](2,3,7;(-3^{6}))$ on utilise la construction suivante. On part du polygone donné en concaténant  $2$ fois le vecteur $1$, puis le vecteur $\exp(2i\pi/3)$, puis le vecteur $v_{1} = \exp(i\pi/3)$, une fois le vecteur $\exp(2i\pi/3)$, le vecteur $w_{1} = \tfrac{-2i\sqrt{3}}{3} $, le vecteur $\exp(2i\pi/3)$, le vecteur $w_{2}= \tfrac{-2\sqrt{3}}{3} \exp(-5i\pi/6)$,  le vecteur $\exp(-2i\pi/3)$ et enfin le vecteur $v_{2}= \exp(-i\pi/3)$. On identifie $v_{1}$ avec $v_{1}$ et $w_{1}$ avec $w_{2}$ par rotation et translation. Puis on ajoute des demi-cylindres infinis aux autres segments.
Enfin dans le cas de la strate $\Omega^{3}\moduli[0](2,5,5;(-3^{6}))$, on part de la $3$-différentielle de $\Omega^{3}\moduli[0](-1,2,2;(-3^{3}))$ de la figure~\ref{fig:k3troiszeros}. On ajoute un cylindre de circonférence $1$ à la singularité d'ordre $2$ en noir et un  cylindre de circonférence $\exp(i\pi3)$ à celle en blanc  et enfin  on ajoute un cylindre de circonférence $\exp(2i\pi3)$ à la singularité d'ordre~$-1$ (en rouge).
\end{proof}

\section{Pluridifférentielles en genre supérieur ou égal à un.}
\label{sec:ggeq1}

Dans cette section nous montrons que, à l’exception de quatre familles exceptionnelles, pour chaque strate de genre $g\geq1$, l'application $k$-résiduelle est surjective. 
Plus précisément, nous prouvons les théorèmes~\ref{thm:ggeq2}, \ref{thm:geq1} et \ref{thm:strateshol} pour les strates de $k$-différentielles avec $k\geq2$. 
Nous considérons dans la section~\ref{sec:pluri1} les strates ayant au moins un pôle d'ordre strictement inférieur à~$-k$. Dans la section~\ref{sec:pluri2} nous traitons les cas où tous les pôles sont d'ordre~$-k$. Enfin, nous traitons le cas des strates sans pôles dans la section~\ref{sec:plurifini}.

\subsection{Pluridifférentielles avec un pôle d'ordre strictement inférieur à~$-k$}
\label{sec:pluri1}

Dans ce paragraphe, nous considérons les strates avec au moins un pôle d'ordre strictement inférieur à~$-k$. Nous montrons tout d'abord que l'application $k$-résiduelle est surjective pour les strates de genre un possédant un unique zéro distinctes de $\Omega^{2}\moduli[1](4a;(-4^{a}))$. Dans ce cas, l'application $2$-résiduelle contient $\CC^{a}\setminus\left\{0\right\}$ (lemmes~\ref{lem:geq1} et \ref{lem:geq1bis}). Puis nous montrons la surjectivité des applications $2$-résiduelles de certaines strates quadratiques (lemme~\ref{lem:geq1ter}). Nous en déduisons la surjectivité de l'application $k$-résiduelle dans le cas général par éclatement de zéro et couture d'anse. Enfin nous montrons que $0\in\CC^{a}$ n'appartient pas à l'image de l'application $2$-résiduelle des strates $\Omega^{2}\moduli[1](4a;(-4^{a}))$ et $\Omega^{2}\moduli[1](2a-1,2a+1;(-4^{a}))$ pour tout $a\geq1$ (lemmes~\ref{lem:geq1cin} et \ref{lem:geq1six}).

\begin{lem}\label{lem:geq1}
L'application $k$-résiduelle $\appresk[1](m;-m)$ est surjective pour les couples $(k,m)$ tels que $m>k\geq2$ distincts de  $(2,4)$ et l'image de $\appresk[1][2](4;-4)$ contient $\CC^{\ast}$.
\end{lem}

\begin{proof}
Si $k\nmid m$, la surjectivité de  $\appresk[1](m;-m)$ est équivalent au fait que la strate $\komoduli[1](m;-m)$ soit non vide. Ceci est une conséquence élémentaire du théorème d'Abel. \`A partir de maintenant, nous supposons que $k\mid m$. 

Une $k$-différentielle dans $\komoduli[1](\ell k,-\ell k)$ avec $\ell\geq2$ et au moins un résidu non nul est donnée par le recollement de la $k$-partie polaire non triviale de type $k\ell$ associée aux vecteurs $(\emptyset;v_{1},\dots,v_{4})$  représentés sur la figure~\ref{fig:a-a}.

 \begin{figure}[htb]
 \centering
\begin{tikzpicture}[scale=1]
\fill[fill=black!10] (0,0) --(4,2) -- (8,0)  -- ++(0,2.4) -- ++(-8,0) -- cycle;

      \draw (0,0) coordinate (a1) -- node [below,sloped] {$1$} (2,1) coordinate (a2) -- node [below,sloped] {$2$} (4,2) coordinate (a3) -- node [above,sloped,rotate=180] {$1$} (6,1) coordinate (a4) -- node [above,sloped,rotate=180] {$2$} (8,0)
coordinate (a5);
  \foreach \i in {1,2,...,5}
  \fill (a\i) circle (2pt);

  \draw (a1)-- ++(0,2.5)coordinate[pos=.6](b);
    \draw (a5)-- ++(0,2.5)coordinate[pos=.6](c);
    \node[left] at (b) {$3$};
     \node[right] at (c) {$3$};
  \draw[->] (3.6,1.8) arc  (200:340:.4); \node at (4,1.3) {$\frac{2\pi}{k}$};
\end{tikzpicture}
\caption{Une $k$-différentielle de $\Omega^{k}\moduli[1](2k;-2k)$ (en blanc)  et  $\Omega^{k}\moduli[1](k;-k)$ (en gris).}
\label{fig:a-a}
\end{figure}

Nous traitons maintenant le cas où le $k$-résidu au pôle est nul.
Si $k\geq3$, alors la $k$-différentielle à gauche de la figure~\ref{fig:a-a,pasderes} appartient à la strate $\komoduli[1](2k;-2k)$ et possède un $k$-résidu nul. Plus généralement, une $k$-différentielle de $\komoduli[1](\ell k,-\ell k)$ dont le $k$-résidu est nul est donnée par le recollement d'une $k$-partie polaire triviale de type $k\ell$ associée à $(v_{1},v_{2};v_{3},v_{4})$ où les $v_{i}$ sont représentés à gauche de la figure~\ref{fig:a-a,pasderes}.
\begin{figure}[htb]
 \centering
\begin{tikzpicture}[scale=1,decoration={
    markings,
    mark=at position 0.5 with {\arrow[very thick]{>}}}]

    \fill[fill=black!10] (2,.2) ellipse (2.3cm and 1.5cm);
      \draw (0,0) coordinate (a1) -- node [below,sloped] {$1$} node [above,sloped] {$2$}  (2,1) coordinate (a2) --  node [below,sloped] {$2$}node [above,sloped] {$1$} (4,0)
coordinate (a3);
  \foreach \i in {1,2,...,3}
  \fill (a\i) circle (2pt);

  \draw[->] (1.6,.8) arc  (200:340:.4); \node at (2,.3) {$\frac{(k-2)\pi}{k}$};
  

  \begin{scope}[xshift=9cm]

\begin{scope}[yshift=1.5cm]
      \fill[fill=black!10] (0,0) ellipse (2cm and .9cm);
   \draw (-1,0) coordinate (a) -- node [below] {$2$} node [above] {$1$}  (0,0) coordinate (b);
 \draw (0,0) -- (1.5,0) coordinate[pos=.5] (c);
  \draw[dotted] (1.5,0) -- (2,0);
 \fill (a)  circle (2pt);
\fill[] (b) circle (2pt);
\node[above] at (c) {$4$};
\node[below] at (c) {$3$};

    \end{scope}

\begin{scope}[yshift=-.5cm]
      \fill[fill=black!10] (0,0) ellipse (2cm and .9cm);
   \draw (-1,0) coordinate (a) -- node [above,rotate=180] {$2$} node [below,rotate=180] {$1$}  (0,0) coordinate (b);
 \draw (0,0) -- (1.5,0) coordinate[pos=.5] (c);
  \draw[dotted] (1.5,0) -- (2,0);
 \fill (a)  circle (2pt);
\fill[] (b) circle (2pt);
\node[above] at (c) {$3$};
\node[below] at (c) {$4$};

    \end{scope}
\end{scope}
\end{tikzpicture}
\caption{Une $k$-différentielle sans $k$-résidu dans $\Omega^{k}\moduli[1](2k;-2k)$ avec $k\neq2$ à gauche et dans $\Omega^{2}\moduli[1](6;-6)$ à droite.}
\label{fig:a-a,pasderes}
\end{figure}
\par
Nous traitons maintenant le cas des strates quadratiques $\Omega^{2}\moduli[1](2\ell ;-2\ell)$ avec $\ell\geq3$. Une $2$-différentielle avec ces invariants locaux est obtenue à partir de $\ell-1$ domaines basiques positifs $D^{+}_{i}$ et $\ell-1$ négatifs $D^{-}_{j}$. Les domaines $D_{1}^{\pm}$ et $D_{2}^{\pm}$ sont donnés par les domaines basiques positifs (resp. négatifs) associés au vecteur~$v=1$. Les $2(\ell-2)$ autres domaines~$D_{i}^{\pm}$ sont les domaines polaires associés à l'ensemble vide. Pour tout~$i$, on colle la demi-droite~$\RR^{-}$ du domaine $D_{i}^{+}$ à celle de $D_{i}^{-}$ et la demi-droite $\RR_{+}$ de $D_{i}^{+}$ à celle de $D_{i+1}^{-}$ modulo~$\ell-1$. Les  segments au bord des domaines positifs (resp. négatifs) $D_{1}^{+}$ et $D_{2}^{+}$ (resp. $D_{1}^{-}$ et $D_{2}^{-}$) sont collés entre eux par une rotation d'angle~$\pi$. Cette construction est illustrée à droite de la figure~\ref{fig:a-a,pasderes}. On vérifiera facilement que la différentielle quadratique ainsi construite possède les invariants locaux souhaités.
\end{proof}

\begin{lem}\label{lem:geq1bis}
Soient $k\geq2$ et $\mu=(a;-b_{1},\dots,-b_{p};-c_{1},\dots,-c_{r};(-k^{s}))$ une partition de~$0$ telle que $(p,r)\neq(0,0)$. Si $\mu\neq(4p;(-4^{p}))$ ou $k\neq2$, alors l'application résiduelle $\appresk[1](\mu)$ est surjective. De plus, l'image de $\appresk[1][2](4p,(-4^{p}))$ contient $\CC^{p}\setminus\lbrace 0\rbrace$.
\end{lem}

\begin{proof}
Nous commençons par le cas où $r\geq2$. L'application résiduelle de la strate $\omoduli[0](a-2k;-b_{1},\dots,-b_{p};-c_{1},\dots,-c_{r};(-k^{s}))$ est surjective (c.f. le lemme~\ref{lem:g=0gen1}). On obtient donc la surjectivité  de  $\appresk[1](\mu)$ par couture d'anse.
\par
Supposons maintenant qu'il existe un unique pôle d'ordre non divisible par $k$. Le lemme~\ref{lem:g=0gen1} dans le cas des strates $\omoduli[0](a-2k;-b_{1},\dots,-b_{p};-c;(-k^{s}))$  implique que $\CC^{p}\setminus\lbrace 0\rbrace$ est contenu dans  l'image de $\appresk[1](\mu)$  par couture d'anse. Il suffit donc de prouver que l'origine est dans l'image de l'application résiduelle. En particulier, ces strates ne possèdent pas de pôles d'ordre $-k$.  Nous associons aux pôles d'ordre $-k\ell_{i}$ les $k$-parties polaires triviales d'ordres $k\ell_{i}$ associées à $(1;1)$. Pour le pôle d'ordre $-c$, nous prenons une $k$-différentielle de $\komoduli[1](c;-c)$ avec   un lien selle de période $1$. Nous obtenons une surface plate à bord en la coupant le long de ce lien selle.
La pluridifférentielle désirée est obtenue en collant les bords des fentes de manière cyclique. 
\par
Il nous reste à traiter le cas où tous les pôles sont d'ordres divisibles par $k$. Nous commençons par construire une $k$-différentielle dans $\komoduli[1](a;-k\ell_{1},\dots,-k\ell_{p})$ dont tous les $k$-résidus sont nuls (sauf dans le cas où $k=2$ et $\mu=(4p;(-4^{p}))$). Partons de la $k$-différentielle de $\komoduli[1](k\ell_{1};-k\ell_{1})$ sans résidu donnée par le lemme~\ref{lem:geq1} (on suppose que $\ell_{1}\geq3$ si $k=2$). On peut alors couper cette surface le long d'un lien selle dont la période sera notée $v$. Par exemple, on coupe les $k$-différentielles de la figure~\ref{fig:a-a,pasderes} le long du lien selle dénoté par~$1$. Pour tous les autres pôles on prend une $k$-partie polaire triviale d'ordre $k\ell_{i}$ associée à $(v;v)$. La surface obtenue en recollant les segments de manière cyclique (voir la figure~\ref{fig:nonresg1bis}) a les propriétés désirées.

\begin{figure}[htb]
\begin{tikzpicture}
\begin{scope}[xshift=-2cm,yshift=-.4cm]
      
    \fill[fill=black!10] (0,-.2) ellipse (1.8cm and .9cm);

   \draw (-1,-.5) coordinate (a) -- node [below] {$1$} node [above] {$2$}  (0,0) coordinate (b);
 \draw (0,0) -- (1.5,0) coordinate[pos=.5] (c);
 \draw[dotted] (1.5,0) -- (2,0);
 \fill (a)  circle (2pt);
\fill[] (b) circle (2pt);
\node[above] at (c) {$b$};
\node[below] at (c) {$a$};

    \end{scope}

\begin{scope}[xshift=1.5cm,yshift=-.4cm]
        \fill[fill=black!10] (.8,0) ellipse (1.2cm and .8cm);  
 \draw (0,0) -- (1.5,0) coordinate[pos=.5] (c);
 \fill (0,0)  circle (2pt);
 \draw[dotted] (1.5,0) -- (2,0);
\node[above] at (c) {$a$};
\node[below] at (c) {$b$};
\end{scope}


\begin{scope}[yshift=-2.5cm,xshift=-.5cm]
    \fill[fill=black!10] (1,.2) ellipse (1.8cm and 1.1cm);
   \draw (0,0) coordinate (a1) -- node [below] {$2$} node [above] {$3$}  (1,.5) coordinate (a2) --  node [below] {$3$}node [above] {$1$} (2,0)
coordinate (a3);
  \foreach \i in {1,2,...,3}
  \fill (a\i) circle (2pt);

  \draw[->] (.8,.4) arc  (200:340:.2); \node at (1,.9) {$\frac{(k-2)\pi}{k}$};

\end{scope}

\begin{scope}[xshift=7cm]

\begin{scope}[]
      
    \fill[fill=black!10] (0,0) ellipse (2cm and .7cm);

   \draw (-1,0) coordinate (a) -- node [below] {$2$} node [above] {$1$}  (0,0) coordinate (b);
 \draw (0,0) -- (1.5,0) coordinate[pos=.5] (c);
  \draw[dotted] (1.5,0) -- (2,0);
 \fill (a)  circle (2pt);
\fill[] (b) circle (2pt);
\node[above] at (c) {$b$};
\node[below] at (c) {$a$};

    \end{scope}

\begin{scope}[yshift=-1.6cm]
    \fill[fill=black!10] (0,0) ellipse (2cm and .7cm);
   \draw (-1,0) coordinate (a) -- node [above,rotate=180] {$2$} node [above] {$3$}  (0,0) coordinate (b);
 \draw (0,0) -- (1.5,0) coordinate[pos=.5] (c);
  \draw[dotted] (1.5,0) -- (2,0);
 \fill (a)  circle (2pt);
\fill[] (b) circle (2pt);
\node[above] at (c) {$a$};
\node[below] at (c) {$b$};

\end{scope}


\begin{scope}[yshift=-3cm,xshift=0cm]
    \fill[fill=black!10] (-.5,0) ellipse (1cm and .5cm);
   \draw (0,0) coordinate (a1) -- node [below] {$1$} node [above,xscale=-1] {$3$}  (-1,0) coordinate (a2);
  \foreach \i in {1,2}
  \fill (a\i) circle (2pt);

\end{scope}
\end{scope}
\end{tikzpicture}
\caption{$k$-différentielle de $\Omega^{k}\moduli[1](5k;-2k,-3k)$ avec un résidu nul pour $k\geq3$ à gauche et $k=2$ à droite} 
\label{fig:nonresg1bis}
\end{figure}
\par
Nous traitons maintenant les strates $\Omega^{k}\moduli[1](a;-k\ell_{1},\dots,-k\ell_{p};(-k^{s}))$, telles que $s\neq0$ ou $s=0$ et avec au moins un $k$-résidu non nul. Nous supposons que s'il existe un pôle avec un $k$-résidu nul, alors le $k$-résidu de $P_{1}$ est nul.
Considérons les pôles $P_{i}$ avec $2\leq i\leq p'$ qui possèdent un $k$-résidu $R_{i}$ non nul. Nous associons à $P_{i}$ la $k$-partie polaire non triviale d'ordre $k\ell_{i}$ associée à $(r_{i};\emptyset)$ où $r_{i}$ est une racine $k$-ième de $R_{i}$ de partie réelle positive (ou de partie imaginaire positive si la partie réelle est nulle). Considérons maintenant les pôles $P_{j}$ avec $j > p'$ ayant un $k$-résidu nul. Nous associons la $k$-partie polaire triviale d'ordre $k\ell_{j}$ associée à $(r_{i_{j}};r_{i_{j}})$ pour un $k$-résidu $R_{i_{j}}\neq0$. Puis nous collons le segment~$r_{i}$ du domaine positif de~$P_{i}$ au segment~$r_{i_{j}}$ du domaine basique négatif de~$P_{j}$. 
\par
Enfin, pour le pôle $P_{1}$ nous procédons à la construction suivante. Notons que la somme des~$r_{i}$ est non nulle.  Nous supposerons que les $r_{i}$ sont ordonnés par argument croissant. Nous prenons pour $P_{1}$ la $k$-partie polaire de type $k\ell_{1}$ associée à $(v_{1},v_{1},v_{2},v_{2};r_{2},\dots,r_{p'})$ où les $v_{i}$ sont donnés comme suit. Les $v_{i}$ sont de même longueur, vérifient l'égalité $r_{1}=2v_{1}+2v_{2}-\sum_{i\geq2} r_{i}$ et l'angle (au dessus) du point d'incidence de~$v_{1}$ et $v_{2}$ est $\frac{2(k-1)\pi}{k}$. 
\par
La $k$-différentielle est obtenue en identifiant par translation les bords $r_{i}$ des domaines polaires positifs aux segments $r_{i}$ de la $k$-partie polaire négative de $P_{1}$. Enfin, nous identifions par rotation d'angle $\frac{2\pi}{k}$ et translation, le premier $v_{1}$ au premier~$v_{2}$ et le second $v_{1}$ au second~$v_{2}$. Cela donne une pluridifférentielle primitive avec les invariants souhaités. Un exemple est donné dans la figure~\ref{fig:nonresg1ter}.
    \begin{figure}[htb]
\begin{tikzpicture}

\begin{scope}[xshift=-3cm]
\fill[fill=black!10] (0.5,0)coordinate (Q)  circle (1.2cm);
    \coordinate (a) at (0,0);
    \coordinate (b) at (1,0);

     \fill (a)  circle (2pt);
\fill[] (b) circle (2pt);
    \fill[white] (a) -- (b) -- ++(0,-1.2) --++(-1,0) -- cycle;
 \draw  (a) -- (b);
 \draw (a) -- ++(0,-1.1);
 \draw (b) -- ++(0,-1.1);

\node[above] at (Q) {$1$};
    \end{scope}

\begin{scope}[xshift=3cm]
\fill[fill=black!10] (0.5,0)coordinate (Q)  circle (1.2cm);
    \coordinate (a) at (0,0);
    \coordinate (b) at (1,0);

     \fill (a)  circle (2pt);
\fill[] (b) circle (2pt);
    \fill[white] (a) -- (b) -- ++(0,-1.2) --++(-1,0) -- cycle;
 \draw  (a) -- (b);
 \draw (a) -- ++(0,-1.1);
 \draw (b) -- ++(0,-1.1);

\node[above] at (Q) {$2$};
    \end{scope}


\begin{scope}[xshift=.5cm,yshift=-.3cm]
\fill[fill=black!10] (0,0)coordinate (Q)  circle (1.5cm);
    \coordinate (a) at (-1,-.5);
    \coordinate (b) at (0,-.5);
    \coordinate (c) at (1,-.5);
    \coordinate (d) at (0,1);
    \coordinate (e) at (-.5,.25);
    \coordinate (f) at (.5,.25);

\fill (a)  circle (2pt);
\fill[] (b) circle (2pt);
\fill (c)  circle (2pt);
\fill[] (d) circle (2pt);
\fill (e)  circle (2pt);
\fill[] (f) circle (2pt);

\fill[white] (a) -- (c) -- (d) -- cycle;
\draw (a) -- (b)coordinate[pos=.5] (g1) --(c)coordinate[pos=.5] (g2) -- (f)coordinate[pos=.5] (g3) -- (d)coordinate[pos=.5] (g4) -- (e)coordinate[pos=.5] (g5)-- (a)coordinate[pos=.5] (g6);

\node[below] at (g1) {$1$};
\node[below] at (g2) {$2$};
\node[above,rotate=-50] at (g3) {$3$};
\node[above,rotate=-50] at (g4) {$4$};
\node[below,rotate=-130] at (g5) {$3$};
\node[below,rotate=-130] at (g6) {$4$};

  \draw[->] (d)++(.1,-.1) arc  (-60:240:.2); \node at (0,1.6) {$\frac{2(k-1)\pi}{k}$};

\end{scope}

\end{tikzpicture}
\caption{Une  $k$-différentielle dans $\Omega^{k}\moduli[1](6k;(-2k^{3}))$ dont les $k$-résidus sont $(0,1,1)$.} \label{fig:nonresg1ter}
\end{figure}
\end{proof}

\begin{lem}\label{lem:geq1ter}
Pour tout $p\geq1$, l'image de l'application $2$-résiduelle $\appresk[1][2](a_{1},a_{2};(-4^{p}))$ avec $(a_{1},a_{2})\neq(2p+1,2p-1)$ et de $\appresk[2][2](4(p+1);(-4^{p}))$ contient l'origine.
\end{lem}

\begin{proof}
Pour commencer la figure~\ref{fig:exceptquadra} exhibe une différentielle quadratique primitive ayant un $2$-résidu nul dans chacune des strates $\Omega^{2}\moduli[1](5,-1;-4)$ et $\Omega^{2}\moduli[1](2,2;-4)$ (en gris).

 \begin{figure}[htb]
\begin{tikzpicture}
\begin{scope}[xshift=-4cm]

     \foreach \i in {1,2,...,4}
  \coordinate (a\i) at (0,2*\i/3); 
   \fill[black!10] (a1) -- (a4)-- ++(0,.6) arc (90:270:1.6) -- cycle;
 
       \foreach \i in {1,2,4}
   \fill (a\i)  circle (2pt);

        \foreach \i in {1,2,...,4}
  \coordinate (b\i) at (1.5,2*\i/3); 
     \fill[black!10] (b1) -- (b4)-- ++(0,.6) arc (90:-90:1.6) -- cycle;
       \foreach \i in {1,2,...,4}
   \fill (b\i)  circle (2pt);

\draw (a1)--+(0,-.6);
\draw (b1)--+(0,-.6);
\draw (a4)--+(0,.6);
\draw (b4)--+(0,.6);

\draw (a1)-- (a2) coordinate[pos=.5] (c1) -- (a3) coordinate[pos=.5] (c2) -- (a4) coordinate[pos=.5] (c3);
\draw (b1)-- (b2) coordinate[pos=.5] (d1) -- (b3) coordinate[pos=.5] (d2) -- (b4) coordinate[pos=.5] (d3);

      \fill[white] (a3)  circle (2pt);
      \draw (a3)  circle (2pt);
      
\node[left] at (c1) {$1$};
\node[left] at (c2) {$2$};
\node[right,rotate=180] at (c3) {$2$};
\node[right] at (d1) {$3$};
\node[right] at (d2) {$1$};
\node[left,rotate=180] at (d3) {$3$};

\end{scope}

\begin{scope}[xshift=4cm]
     \foreach \i in {1,2,...,4}
  \coordinate (a\i) at (0,2*\i/3); 
     \fill[black!10] (a1) -- (a4)-- ++(0,.6) arc (90:270:1.6) -- cycle;

       \foreach \i in {2,3}
   \fill (a\i)  circle (2pt);

        \foreach \i in {1,2,...,4}
  \coordinate (b\i) at (1.5,2*\i/3); 
       \fill[black!10] (b1) -- (b4)-- ++(0,.6) arc (90:-90:1.6) -- cycle;

       \foreach \i in {2,3}
   \fill (b\i)  circle (2pt);

\draw (a1)--+(0,-.6);
\draw (b1)--+(0,-.6);
\draw (a4)--+(0,.6);
\draw (b4)--+(0,.6);

\draw (a1)-- (a2) coordinate[pos=.5] (c1) -- (a3) coordinate[pos=.5] (c2) -- (a4) coordinate[pos=.5] (c3);
\draw (b1)-- (b2) coordinate[pos=.5] (d1) -- (b3) coordinate[pos=.5] (d2) -- (b4) coordinate[pos=.5] (d3);

       \foreach \i in {1,4}
      \fill[white] (a\i)  circle (2pt);
             \foreach \i in {1,4}
      \draw (a\i)  circle (2pt);
             \foreach \i in {1,4}
      \fill[white] (b\i)  circle (2pt);
             \foreach \i in {1,4}
      \draw (b\i)  circle (2pt);
      
\node[left] at (c2) {$1$};
\node[left] at (c1) {$2$};
\node[right,rotate=180] at (c3) {$2$};
\node[right] at (d1) {$3$};
\node[right] at (d2) {$1$};
\node[left,rotate=180] at (d3) {$3$};
\end{scope}
\end{tikzpicture}
\caption{Différentielles quadratiques dans les strates $\Omega^{2}\moduli[1](5,-1;-4)$ et $\Omega^{2}\moduli[1](2,2;-4)$ avec $2$-résidus nuls en gris. Différentielles quadratiques dans $\Omega^{2}\moduli[1](5,-1;(-2^{2}))$ et $\Omega^{2}\moduli[1](2,2;(-2^{2}))$ avec $2$-résidus $(1,1)$ en blanc.}
\label{fig:exceptquadra}
\end{figure}
\par
Nous prouvons maintenant le résultat pour les strates $\Omega^{2}\moduli[1](a_{1},a_{2};(-4^{p}))$ satisfaisant $(a_{1},a_{2})\neq(2p+1,2p-1)$ par récurrence sur le nombre $p$ de pôles.  L'hypothèse de récurrence est la suivante. Il existe une $2$-différentielle primitive dont les $2$-résidus sont nuls dans toutes les strates de cette forme avec $p-1$ pôles telles que si $a_{2}=-1$, alors il existe un lien selle fermé reliant le zéro d'ordre $a_{1}$ à lui-même et dans tous les cas il existe un lien selle entre les deux zéros. L'hypothèse de récurrence est satisfaite pour $p=1$ par les $2$-différentielles représentées dans la figure~\ref{fig:exceptquadra}. Nous allons construire des $2$-différentielles satisfaisant à l'hypothèse de récurrence avec $p$ pôles.
\par
Si $(a_{1},a_{2})=(4p+1,-1)$, coupons la différentielle quadratique de $\Omega^{2}\moduli[1](4p-3,-1;(-4^{p-1}))$ donnée par l'hypothèse le long du lien selle entre le zéro d'ordre~$4p-3$. Prenons une $2$-partie polaire d'ordre~$4$ associée à $(v;v)$ où $v$ est la période de ce lien selle. La surface formée en collant les bords de ces surfaces par translation est dans $\Omega^{2}\moduli[1](4p+1,-1;(-4^{p}))$, ces $2$-résidus sont nuls et possède un lien selle fermé reliant le zéro d'ordre $4p+1$ à lui même et entre les deux zéros.
\par
Si $(a_{1},a_{2})\neq(4p+1,-1),(2p+1,2p-1)$, la construction du paragraphe précédent en partant d'une $2$-différentielle de  $\Omega^{2}\moduli[1](a_{1}-2,a_{2}-2;(-4^{p-1}))$ et du lien selle entre les deux zéros donne une différentielle quadratique ayant les propriétés souhaités.
\par
Une différentielle quadratique dans la strate $\Omega^{2}\moduli[2](4(p+1);(-4^{p}))$ sans résidus aux pôles est donnée de la façon suivante. Le  lemme~\ref{lem:geq1bis} fourni une différentielle primitive de la strate $\Omega^{2}\moduli[1](4(p+1);(-4^{p});-2,-2)$ telle que les résidus aux pôles d'ordre $-4$ sont nuls et les résidus quadratiques aux pôles doubles sont égaux entre eux. Nous formons une $2$-différentielle entrelacée en collant les deux pôles d'ordre $-2$ ensemble. Cette $2$-différentielle entrelacée peut être lissée sans changer les $2$-résidus aux pôles d'ordre $-4$ (voir proposition~\ref{lem:lisspolessimples}) pour donner une $2$-différentielle ayant les invariants souhaités. 
\end{proof}

\begin{lem}\label{lem:geq1qua}
L'application $k$-résiduelle des strates $\komoduli(\mu)$ avec $g\geq1$, $\mu\neq(4p;(-4^{p}))$ et $\mu\neq(2p-1,2p+1;(-4^{p}))$  est surjective.
\end{lem}

\begin{proof}
Si $\mu$ possède un unique zéro on obtient les $k$-différentielles dans les strates souhaitées par couture d'anses à partir des $k$-différentielles de genre un ayant un unique zéro (voir la proposition~\ref{prop:attachanse}). La surjectivité de l'application $k$-résiduelle est une conséquence de la surjectivité des applications $k$-résiduelles de ces strates (voir lemme~\ref{lem:geq1} et \ref{lem:geq1bis}), à l’exception des strates de la forme $\Omega^{2}\moduli[g](a;(-4^{p}))$. Dans ces cas le $k$-résidu $(0,\dots,0)$ est obtenu en partant des strates $\Omega^{2}\moduli[2](4(p+1);(-4^{p}))$ et en utilisant le lemme~\ref{lem:geq1ter}.
\par
Considérons les strates $\komoduli(a_{1},\dots,a_{n};-c_{i};-b_{j};(-k^{s}))$ avec $n\geq2$ zéros. La surjectivité de l'application $k$-résiduelle est obtenue en éclatant l'unique zéro des $k$-différentielles de la strate $\komoduli(\sum a_{i};-c_{i};-b_{j};(-k^{s}))$ (voir proposition~\ref{prop:eclatZero}), sauf pour $\Omega^{2}\moduli[1](a_{1},\dots,a_{n};(-4^{p}))$ avec $n\geq2$. Dans ces cas la surjectivité a été démontrée dans le lemme~\ref{lem:geq1ter}  pour $n=2$. Pour $n\geq3$, il suffit d'éclater l'un des deux zéros de ces $2$-différentielles pour obtenir la surjectivité. 
\end{proof}

\begin{lem}\label{lem:geq1cin}
Il n'existe pas de différentielle quadratique primitive dans $\Omega^{2}\moduli[1](4p;(-4^{p}))$ pour $p\geq1$ dont tous les résidus quadratiques sont nuls.
\end{lem}

\begin{proof}
Considérons tout d'abord le cas de la strate $\Omega^{2}\moduli[1](4;-4)$. Si elle contenait une différentielle quadratique (primitive) dont le $2$-résidu est nul, alors on pourrait former une $2$-différentielle entrelacée lissable en collant au pôle le carré d'une différentielle abélienne de la strate  $\omoduli[1](\emptyset)$. D'après le lemme~\ref{lem:lissdeuxcomp}, le lissage produirait une différentielle quadratique primitive dans  $\Omega^{2}\moduli[2](4)$, ce qui n'existe pas (voir~\cite{masm}).  
\par
Supposons par l'absurde qu'il existe une différentielle quadratique (primitive) $\xi$ dans $\Omega^{2}\moduli[1](4p;(-4^{p}))$ pour $p\geq2$ dont tous les $2$-résidus soient nuls. Dans la suite, nous faisons référence aux notions de la section~\ref{sec:coeur}. On peut supposer que le cœur de cette surface est dégénéré et que tous les liens selles sont horizontaux. Cette surface possède alors $p+1$ liens selles et chaque domaine polaire est bordé par au moins deux liens de selles (sinon son résidu serait non nul).
On en déduit que le graphe d'incidence associé peut être de l'une des deux formes suivantes.
\begin{enumerate}[i)]
 \item Il y a un sommet de valence $4$ et les autres sont de valence $2$.
 \item Il y a deux sommets de valence $3$, et les autres sont de valence $2$.
\end{enumerate}
\par
Nous pouvons simplifier $\xi$ de la manière suivante. Prenons un domaine polaire bordé par deux liens selles. Coupons $\xi$ le long de ces liens selles, enlevons ce pôle et recollons le deux segments au bord que nous venons de créer. La différentielle que nous venons de créer est encore primitive, sans résidus et dans la strate $\Omega^{2}\moduli[1](4(p-1);(-4^{p-1}))$.
\par
Considérons une différentielle $\xi$ dont le graphe d'incidence possède un pôle de valence~$4$. En répétant l'opération du paragraphe précédent jusqu'à avoir éliminer tous les pôles de valence~$2$, on obtient une différentielle quadratique primitive sans résidu dans $\Omega^{2}\moduli[1](4;-4)$. Cela montre que $\xi$ ne peut pas avoir de graphe d'incidence avec un sommet de valence~$4$.
\par
Supposons que le graphe d'incidence de~$\xi$ possède deux pôles de valence~$3$. En répétant l'opération du paragraphe précédent nous pouvons obtenir deux graphes. Dans le premier cas, deux sommets sont reliés entre eux par trois arêtes. Dans le second cas, les deux sommets sont reliés entre eux par une unique arête et ont chacun une boucle. Remarquons que dans les deux cas, le fait que le résidu quadratique est nul et les liens selles sont horizontaux implique que l'un des liens selles est strictement plus long que les deux autres. Les deux liens selles les plus courts se trouvent du même côté du segment formé par le bord d'un domaine polaire. Ces deux segments ne peuvent pas être identifiés ensemble car la surface obtenue aurait un zéro d'ordre~$-1$. Cela exclut le graphe avec des boucles. Dans l'autre cas, la surface est obtenu par identification par des translations uniquement (pas de rotations) et n'est donc pas primitive.
\end{proof}

\begin{lem}\label{lem:geq1six}
L'application résiduelle $\appresk[1][2](2p-1,2p+1;(-4^{p}))$ ne contient pas l'origine. 
\end{lem}

\begin{proof}
Considérons tout d'abord le cas de la strate  $\Omega^{2}\moduli[1](3,1;-4)$. S'il y avait une différentielle quadratique dans cette strate dont le $2$-résidu est nul,  alors on pourrait former une $2$-différentielle entrelacée lissable en collant au pôle le carré d'une différentielle abélienne de~$\omoduli[1](\emptyset)$. Par le lemme~\ref{lem:lissdeuxcomp}, le lissage donnerait une différentielle quadratique dans $\Omega^{2}\moduli[2](3,1)$, qui est vide ( voir~\cite{masm}).   
\par
Nous supposons par l'absurde qu'il existe une différentielle quadratique $\xi$ dans la strate $\Omega^{2}\moduli[1](2p-1,2p+1;(-4^{p}))$ dont tous les résidus sont nuls. Nous supposerons (voir section~\ref{sec:coeur}) que le cœur de~$\xi$ est dégénéré et que tous les  $p+2$ liens selles sont horizontaux.
\par
Rappelons que le graphe d'incidence et le graphe d'incidence simplifié de $\xi$ ont été introduit dans la section~\ref{sec:coeur}. Les sommets du graphe d'incidence de valence supérieure ou égale à trois sont dits {\em spéciaux}. Remarquons qu'un domaine polaire ne peut pas être bordé par un unique lien selle, sinon le résidu de ce pôle serait non nul (ie, le graphe d'incidence n'a pas de sommet de valence~$1$). 
Les graphes d'incidence simplifiés peuvent être de l'une des formes suivantes. Soit le graphe possède un sommet de valence~$(6)$, soit deux sommets de valences respectives $(5,3)$ ou $(4,4)$, soit trois sommets de valences respectives $(4,3,3)$, ou quatre sommets de valences $(3,3,3,3)$. 
\par
Le reste de la preuve se présente en deux étapes. Tout d'abord nous simplifions la surface~$\xi$. Nous obtiendrons une {\em surface réduite}, c'est à dire une différentielle quadratique de  $\Omega^{2}\moduli[1](2p-1,2p+1;(-4^{p}))$ dont tous les résidus sont nuls, le cœur est dégénéré et les deux extrémités du bord des domaines polaires de valence~$2$ et~$3$ correspondent au zéro d'ordre~$2p-1$. Ensuite nous montrons que les surfaces réduites n'existent pas en considérant les graphes d'incidence possibles.
\par
Nous décrivons maintenant la procédure qui associe à~$\xi$ une surface réduite. Prenons un domaine polaire de valence~$2$  bordé par deux liens selles joignant le zéro d'ordre maximal ou les deux zéros. Coupons~$\xi$ le long de ces deux liens selles, supprimons ce pôle et collons les deux segments que nous avions créés. Cette opération fait diminuer l'ordre du zéro d'ordre maximal de $4$ ou ceux des deux zéros de $2$ suivant le cas. Dans les deux cas nous obtenons un élément de la strate $\Omega^{2}\moduli[1](2p-1,2p-3;(-4^{p-1}))$. Remarquons que l'opération qui fait diminuer l'ordre du zéro maximal par $4$ rend ce zéro d'ordre strictement inférieur à l'autre.
\par 
Prenons un sommet de valence $3$. Il n'y a pas d'arête reliant ce sommet à lui-même, sans quoi l'une des extrémité du segment en question est une singularité conique de degré $-1$ (si les deux bords identifiés sont du même côté de la fente) ou de degré $0$ (s'ils sont de part et d'autre de la fente). De plus, on suppose que les extrémités de la fente du domaine polaire correspondent aux deux zéros ou au zéro d'ordre maximal. Nous coupons $\xi$ le long de ces trois liens selles et supprimons ce pole. La surface que nous obtenons a un bord constitué de trois segments $v_{1}$, $v_{2}$ et $v_{3}$ avec $v_{1}=v_{2}+v_{3}$. Nous collons $v_{2}$ et $v_{3}$ sur $v_{1}$ en préservant les extrémités de la fente (cela va créer une singularité sur $v_{1}$). Cette opération diminue l'ordre des deux zéros de~$2$ ou le zéro d'ordre maximal par~$4$ (et ce zéro devient le zéro d'ordre minimal).
\par 
Nous appliquons les opérations décrites dans les deux paragraphes précédent sur~$\xi$ jusqu'à ce que aucune ne soit possible. La surface que nous obtenons est soit réduite, soit possède un unique pôle. Dans ce second cas, on obtiendrait une différentielle quadratique sans résidu dans $\Omega^{2}\moduli[1](3,1;-4)$. Ceci étant absurde, on supposera que la surface est réduite avec un nombre de pôles supérieur ou égal à deux.
\par
Nous aurons besoin du résultat suivant sur les liens selles contiguës dans les surfaces réduites. 
Deux liens selles adjacents ne peuvent pas être les extrémités d'une boucle formée de pôles de valence $2$ du graphe d'incidence.
En effet, supposons que les liens selles adjacents sont du même côté de la fente. Dans ce cas, le zéro correspondant à leur point d'intersection est différent du zéro correspondant aux autres extrémités de ces segments car le recollement des segments identifie ce sommet avec exactement une des extrémités dans la cicatrice dans chacun des domaines polaires correspondant aux sommets de valence $2$ dans la boucle et c'est tout. Comme la surface est réduite, la boucle ne contient pas de sommets de valence deux. Ainsi le zéro au point d'intersection est d'ordre~$-1$. Mais aucune des strates que nous considérons n'a un zéro d'ordre~$-1$. Supposons maintenant que les deux liens selles se rencontrent de part et d'autre de l'extrémité de la cicatrice. Le fait que la surface soit réduite implique que la boucle ne contient pas de sommets de valence $2$. Cela implique que le bout de la cicatrice est un point régulier. Cela conclut la preuve de ce résultat élémentaire.
\smallskip
\par
Nous traitons maintenant tous les cas possibles en fonction des valences des sommets spéciaux des graphes d'incidence des surfaces réduites. Remarquons que ceux-ci sont analogues aux graphes simplifiés avec la différence qu'il peut rester quelques sommets de valence deux qui ne contribuent qu'au zéro d'ordre minimal.
\par
Nous commençons par le cas où tous les sommets spéciaux sont de valence~$3$.  Comme la surface est réduite, le zéro correspondant au bout des fentes des domaines polaires est d'ordre minimal. Toutefois, ce zéro est clairement d'ordre supérieur ou égal à l'autre zéro, ce qui est absurde.
\par
Dans le cas où un pôle est de valence $4$ et deux autres sont de valence $3$, chaque sommet de valence $3$ contribue à $2\pi+2\pi=4\pi$ pour l'angle du zéro minimal, qui possède donc un angle minimal de $8\pi$. D'un autre côté, l'angle du zéro maximal est d'angle au plus $
7\pi$. En effet, chaque pôle de valence trois contribue d'un angle $\pi$ et celui de valence $4$ d'au plus $4\pi$. Cela donne une contradiction.

Nous traitons maintenant le cas des graphes d'incidence avec deux pôles spéciaux de valence $4$. Le premier graphe possède une boucle à chaque pôle de valence $4$ et deux arêtes les joignant. Le second graphe possède simplement quatre arêtes entre ces deux sommets. Remarquons qu'a priori, il peut y avoir des sommets de valence $2$ sur ces arêtes.

Considérons le graphe avec deux boucles. Aux sommets spéciaux, soit il y a deux segments de part et d'autre de la fente, soit un côté contient trois liens selles et l'autre un. Dans un sommet spécial où il y a deux liens selles de part et d'autre du segment, les liens selles correspondant aux boucles ne peuvent pas être du même côté. On en déduit facilement que la seule identification possible donne une surface à holonomie triviale. De plus, le fait que les résidus quadratiques soient nuls implique que les deux sommets spéciaux sont simultanément de la même forme. Supposons maintenant qu'il y ait trois liens selles d'un côté de chaque fente. Comme les liens selles des boucles ne peuvent pas être adjacents, les extrémités des fentes correspondent au même zéro. Ce pôle est donc d'ordre pair (quel que soit le nombre de pôles de valence deux sur les arêtes), ce qui est absurde.

Nous considérons enfin le graphe où les deux sommets spéciaux sont reliés par quatre arêtes contenant éventuellement des sommets de valence $2$. Dans ce cas,  les deux sommets spéciaux sont simultanément de la même forme. S'il y a trois segments d'un côté et un de l'autre, alors la surface est non primitive. Ainsi nous considérons le cas avec deux liens selles de part et d'autre de la fente.
Remarquons qu'il n'y a pas de sommets de valence $2$ sur les arêtes. Sinon il existerait des liens selles fermés attachés au zéro d'ordre minimal. Cela implique que $6$ des $12\pi$ des domaines spéciaux contribuent à ce zéro. Cela est clairement impossible. Donc on obtient une différentielle quadratique dans la strate $\Omega^{2}\moduli[1](5,3;(-4^{2}))$. Le zéro d'ordre maximal est d'angle $7\pi$ et le minimal de~$5\pi$. Ceci implique que chaque zéro occupe précisément deux angles dans chaque domaine spécial. Si ces deux angles sont aux extrémités d'un même lien selle, les deux zéros sont d'angle $6\pi$. S'ils sont opposés dans les deux domaines polaires, alors leur angle est un multiple pair de $\pi$. On a obtenu la contradiction souhaitée.

Nous regardons maintenant les graphes avec un sommet de valence $5$ et un de valence~$3$. Remarquons qu'il n'y a pas de boucle au sommet de valence $3$. Sinon les liens selles formant les extrémités de cette boucle seraient adjacents. Donc le graphe est formé de trois arêtes reliant les deux sommets spéciaux et d'une boucle au sommet de valence $5$ (et éventuellement d'autres sommets de valence $2$). Dans le domaine de valence $5$, il n'est pas possible qu'un unique segment se trouve d'une part de la fente. Sinon ce lien selle ne pourrait border ni la boucle ni l'une des trois arêtes connectant au pôle de valence $3$. En effet, ce lien selle ne peut être homologue à aucun autre segment et à aucune somme de deux des quatre autres segments. 

Donc il y a trois liens selles d'un côté et deux de l'autre. De plus, les liens selles bordant la boucle se trouvent du même côté de la fente (pour des raisons de primitivité). Elle est donc formée en identifiant les deux segments extérieurs du côté de la fente qui contient trois segments. Mais il est alors aisé de vérifier que toutes les différentielles quadratiques dont les $2$-résidus sont nuls que nous pouvons obtenir possèdent un unique zéro.

Pour terminer, nous considérons le cas d'un graphe d'incidence avec un unique pôle de valence $6$. Le bord de ce domaine polaire est composé de six segments rangés en trois paires de longueurs identiques. Comme deux segments de la même paire ne peuvent pas être adjacents, il y a trois segments de part et d'autre de la fente.  Si il y avait un segment de chaque paire de part et d'autre de la fente, alors la surface serait non primitive. Donc les liens selles doivent être de la forme $A_{1}B_{1}A_{2}$ d'un côté et $C_{1}B_{2}C_{2}$ de l'autre. Dans ce cas, il est facile de constater que les zéros sont d'ordres pairs. Cette dernière contradiction conclut la preuve.
\end{proof}

\subsection{Pluridifférentielles dont tous les pôles sont d'ordre~$-k$}
\label{sec:pluri2}
Dans ce paragraphe, nous considérons les strates de $k$-différentielles dans $\komoduli(a_{1},\dots,a_{n};(-k^{s}))$. Nous commençons par traiter le cas des strates de genre un avec un unique zéro. Nous montrons dans le lemme~\ref{lem:g=1residugen} que l'application $k$-résiduelle contient tous les $k$-résidus sauf éventuellement dans le cas quadratique si  tous les résidus sont colinéaires. Puis nous montrons dans le lemme~\ref{lem:g=1quadgen} que la seule exception en genre un peut être les résidus quadratiques proportionnels à $(1,\dots,1)$  dans $\Omega^{2}\moduli[1](2s;(-2^{s}))$ avec $s$ pair. Puis nous montrons dans le lemme~\ref{lem:g=1quadspe} que $(1,\dots,1)$ est dans l'image de l'application résiduelle des strates $\Omega^{2}\moduli[1](a_{1},a_{2};(-2^{s}))$ avec $(a_{1},a_{2})\neq(s-1,s+1)$ et $s$ pair. Les strates générales sont traitées dans le lemme~\ref{lem:ggentousres}. Enfin nous prouvons dans les lemmes~\ref{lem:nonsurjunzero} et \ref{lem:nonsurjdeuxzero} que les résidus quadratiques $(1,\dots,1)$ ne sont pas dans l'image de l'application résiduelle des strates $\Omega^{2}\moduli[1](2s;(-2^{s}))$ et $\Omega^{2}\moduli[1](s-1,s+1;(-2^{s}))$ avec $s$ pair.

\begin{lem}\label{lem:g=1residugen}
L'application $k$-résiduelle de $\komoduli[1](ks;(-k^{s}))$  est surjective pour $k\geq3$. L'application $2$-résiduelle de $\Omega^{2}\moduli[1](2s;(-2^{s}))$ est surjective pour $s=1$ et  contient les $2$-résidus $(R_{1},\dots,R_{s})$  tels que il existe $R_{i}$ et $R_{j}$ avec $\frac{R_{i}}{R_{j}}\notin\RR_{+}$ pour $s\geq2$. 
\end{lem} 

\begin{proof}
Pour tout $k\geq2$, la figure~\ref{fig:a-a} montre que les strates $\komoduli[1](k;-k)$ sont non vides. Nous supposerons donc que les strates ont $s\geq2$ pôles d'ordre~$-k$. Par hypothèse, il existe des racines $k$-ième $r_{i}$ des $R_{i}$ qui génèrent $\CC$ comme $\RR$-espace vectoriel. Sans perte de généralité, on peut supposer que l'argument de chaque $r_{i}$ est dans $\left]-\pi,0\right]$.
Nous concaténons alors les $r_{i}$ par argument croissant. L'intérieur de ce segment brisé se trouve dans le demi-plan supérieur ouvert déterminé par la droite $(DE)$ où $D$ et $E$ sont respectivement les points initial et final de la concaténation.
Nous joignons $D$ à $E$ par quatre segments $v_{i}$ d'égale longueur de la façon suivante.    Le segment brisé est formé en concaténant les~$v_{i}$ par $i$ croissant avec $v_{1}=v_{2}$, $v_{3}=v_{4}$ et l'angle en au dessus de l'intersection entre $v_{2}$ et $v_{3}$ est égal à~$\tfrac{2\pi}{k}$.  Notons que ce segment brisé  se trouve dans le demi plan inférieur à la droite~$(DE)$. 

Nous collons maintenant des demi-cylindres infinis aux segments $r_{i}$ de ce polygone. La surface plate est obtenue en collant $v_{1}$ à $v_{3}$ et $v_{2}$ à $v_{4}$ par translation et rotation. On vérifie facilement que cette $k$-différentielle possède un unique zéro. De plus, la différentielle associée est primitive  de genre $1$ car on la déconnecte en coupant les courbes correspondant aux $v_{i}$ et une autre courbe fermée quelconque.
\end{proof}

\begin{lem}\label{lem:g=1quadgen}
L'application résiduelle de  $\Omega^{2}\moduli[1](2s;(-2^{s}))$  est surjective si  $s$ est impair et contient $\espresk[1][2](2s;(-2^{s}))\setminus \CC^{\ast}\cdot (1,\dots,1)$ si $s$ est pair.  
\end{lem}

\begin{proof}
D'après le lemme~\ref{lem:g=1residugen}, il suffit de considérer le cas où les $2$-résidus $(R_{1},\dots,R_{s})$ sont des nombres réels strictement positifs. On dénote par $r_{i}$ la racine positive de~$R_{i}$.

Dans le cas où $s$ est impair, nous classons les $r_{i}$ par ordre croissant. Coupons le segment de longueur $r_{1}$ en deux segments de longueur $\epsilon$ et $r_{1}-\epsilon$. Ensuite, nous coupons le segment de longueur $r_{2}$ en deux segments de longueur $r_{1}-\epsilon$ et  $r_{2}-r_{1}+\epsilon$ en identifiant le premier avec le second du découpage précédent. Nous procédons ainsi pour les $s-1$ premiers cylindres et obtenons en bout de chaîne un segment de longueur $(r_{s-1}-r_{s-2})+\dots+(r_{2}-r_{1})+\epsilon$. Le dernier cylindre est découpé en quatre segments. Le premier d'entre eux est identifié à celui que nous venons d'obtenir. Les trois autres segments se répartissent une longueur $(r_{s}-r_{s-1})+\dots+(r_{3}-r_{2})+r_{1}-\epsilon$ qui est strictement positive. Comme $\epsilon$ peut être choisi arbitrairement petit, nous pouvons répartir cette longueur en trois segments $l+\epsilon+l$ et identifier les deux paires de segments de longueur égale de façon à fermer la surface.

Dans le cas où $s$ est pair, nous découpons le segment de longueur $r_{1}$ en deux segments de longueur $\epsilon$ et $r_{1}-\epsilon$. Ensuite, nous coupons le segment de longueur $r_{2}$ en deux segments de longueur $\alpha$ et $r_{2}-\alpha$. Nous procédons ainsi avec tous les cylindres de largeur $r_{2}$ à $r_{s-1}$ de façon à obtenir une chaîne de $s-2$ cylindres dont les deux segments de bord sont de longueur $\alpha$ et $(r_{s-1}-r_{s-2})+\dots+(r_{3}-r_{2})+\alpha$. Enfin, nous découpons dans le dernier cylindre de largeur $r_{s}$ quatre segments correspondant aux bords restants. La seule condition que nous devons vérifier est que $(r_{s-1}-r_{s-2})+\dots+(r_{3}-r_{2})+r_{1}<r_{s}$, de façon à ce qu'il puisse exister un segment de longueur $\alpha$ strictement positive. L'hypothèse selon laquelle les résidus ne sont pas tous égaux est suffisante.
\end{proof}

\begin{lem}\label{lem:g=1quadspe}
  L'application $2$-résiduelle $\appresk[1][2](a_{1},a_{2};(-2^{s}))$  avec $s$ pair et $(a_{1},a_{2})\neq(s-1,s+1)$ est surjective. 
\end{lem}

\begin{proof}
Par éclatement de zéro, le lemme~\ref{lem:g=1quadgen} implique qu'il suffit de prouver que $(1,\dots,1)$ est dans l'image de l'application résiduelle de ces strates. Nous prouvons ce résultat par récurrence sur le nombre (pair) de pôles d'ordre~$-2$. L'hypothèse de récurrence est que dans chaque strate avec  $s$ pair et $(a_{1},a_{2})\neq(s-1,s+1)$ il existe une $2$-différentielle dont les $2$-résidus sont $(1,\dots,1)$ qui possède un lien selle entre les deux zéros et un lien selle fermé reliant le zéro d'ordre maximal à lui même. Ces liens selles sont choisis non-horizontaux. 
\par
Pour $s=2$, il y les strates $\Omega^{2}\moduli[1](5,-1;-2,-2)$ et $\Omega^{2}\moduli[1](2,2;-2,-2)$ à considérer. Des différentielles quadratiques dans ces strates avec $2$-résidus $(1,1)$ sont représentées dans la figure~\ref{fig:exceptquadra} en blanc.  L'existence des liens selles comme ci-dessus est claire.
\par
Supposons par récurrence qu'il existe une $2$-différentielle de  $\Omega^{2}\moduli[1](a_{1},a_{2};(-2^{s}))$ avec $2$-résidus $(1,\dots,1)$  et $(a_{1},a_{2})\neq(s-1,s+1)$. Si $a_{2}\geq a_{1}$, nous construisons des différentielles quadratiques dans les strates $\Omega^{2}\moduli[1](a_{1}+2,a_{2}+2;(-2^{s+2}))$ et $\Omega^{2}\moduli[1](a_{1},a_{2}+4;(-2^{s+2}))$ satisfaisant à l'hypothèse de récurrence.
\par
Pour ajouter $4$ à l'ordre du zéro d'ordre maximal, nous coupons la différentielle le long du lien selle fermé $\gamma$ connectant ce zéro à lui même. Nous prenons un parallélogramme dont l'un des segment est donné par la période de $\gamma$ et les autres arêtes sont les vecteurs~$1$. Puis nous collons des demi-cylindres infinis de circonférence $1$ aux segments de longueur~$1$. Enfin nous collons les autres segments du parallélogramme aux bord de la surface.  La différentielle obtenue à partir de la différentielles à gauche de la figure~\ref{fig:exceptquadra} est représentée à gauche de  la figure~\ref{fig:exceptquadrabis}. Pour ajouter~$2$ à l'ordre des deux zéros, il suffit de faire la même construction en coupant un lien selle entre les deux zéros. Cette construction est représentée à droite de la figure~\ref{fig:exceptquadrabis}. 

 \begin{figure}[htb]
\begin{tikzpicture}
\begin{scope}[xshift=-3cm]

     \foreach \i in {1,2,...,4}
  \coordinate (a\i) at (0,2*\i/3); 
        \foreach \i in {1,2,...,4}
  \coordinate (b\i) at (1,2*\i/3); 
  \coordinate (e1) at (2,8/3);
\coordinate (e2) at (2,2);
   
    \fill[fill=black!10] (a1)   -- (a4) -- ++(0,.6) --++(1,0) -- (b4) -- ++(80:.6)-- ++(1,0) -- (e1) -- (e2)  -- ++(-80:1.5) -- ++(-1,0) -- (b3) -- (b1)-- ++(0,-.6) --++(-1,0) --cycle;

       \foreach \i in {1,2,4}
   \fill (a\i)  circle (2pt);
       \foreach \i in {1,2,...,4}
   \fill (b\i)  circle (2pt);

\draw (a1)--+(0,-.6);
\draw (b1)--+(0,-.6);
\draw (a4)--+(0,.6);
\draw (b4)--+(0,.6);

   \fill (e1)  circle (2pt);
      \fill (e2)  circle (2pt);

\draw (a1)-- (a2) coordinate[pos=.5] (c1) -- (a3) coordinate[pos=.5] (c2) -- (a4) coordinate[pos=.5] (c3);
\draw (b1)-- (b2) coordinate[pos=.5] (d1) -- (b3) coordinate[pos=.5] (d2);
\draw (e2) -- (e1)  coordinate[pos=.5] (d3);

 \fill[white] (a3)  circle (2pt);
\draw (a3)  circle (2pt);

\draw (b3)--+(-80:1.5);
\draw (e2)--+(-80:1.5);
\draw (e1)--+(80:.6);
\draw (b4)--+(80:.6);
      
\node[left] at (c1) {$1$};
\node[left] at (c2) {$2$};
\node[right,rotate=180] at (c3) {$2$};
\node[left] at (d1) {$3$};
\node[left] at (d2) {$1$};
\node[left,rotate=180] at (d3) {$3$};

\end{scope}

\begin{scope}[xshift=3cm]

     \foreach \i in {1,2,...,4}
  \coordinate (a\i) at (0,2*\i/3); 
        \foreach \i in {1,2,...,4}
  \coordinate (b\i) at (1,2*\i/3); 
   \coordinate (e1) at (2,8/3);
\coordinate (e2) at (2,2);
   
       \fill[fill=black!10] (a1)   -- (a4) -- ++(0,.6) --++(1,0) -- (b4) -- ++(80:.6)-- ++(1,0) -- (e1) -- (e2)  -- ++(-80:1.5) -- ++(-1,0) -- (b3) -- (b1)-- ++(0,-.6) --++(-1,0) --cycle;
       
       \foreach \i in {2,3}
   \fill (a\i)  circle (2pt);
       \foreach \i in {2,3}
   \fill (b\i)  circle (2pt);

\draw (a1)--+(0,-.6);
\draw (b1)--+(0,-.6);
\draw (a4)--+(0,.6);
\draw (b4)--+(0,.6);

      \fill (e2)  circle (2pt);

\draw (a1)-- (a2) coordinate[pos=.5] (c1) -- (a3) coordinate[pos=.5] (c2) -- (a4) coordinate[pos=.5] (c3);
\draw (b1)-- (b2) coordinate[pos=.5] (d1) -- (b3) coordinate[pos=.5] (d2);
\draw (e2) -- (e1)  coordinate[pos=.5] (d3);

\draw (b3)--+(-80:1.5);
\draw (e2)--+(-80:1.5);
\draw (e1)--+(80:.6);
\draw (b4)--+(80:.6);

       \foreach \i in {1,4}
      \fill[white] (a\i)  circle (2pt);
             \foreach \i in {1,4}
      \draw (a\i)  circle (2pt);
             \foreach \i in {1,4}
      \fill[white] (b\i)  circle (2pt);
             \foreach \i in {1,4}
      \draw (b\i)  circle (2pt);

   \fill[white] (e1)  circle (2pt);
   \draw (e1) circle (2pt);
      
\node[left] at (c2) {$1$};
\node[left] at (c1) {$2$};
\node[right,rotate=180] at (c3) {$2$};
\node[left] at (d1) {$3$};
\node[left] at (d2) {$1$};
\node[left,rotate=180] at (d3) {$3$};

\end{scope}
\end{tikzpicture}
\caption{Différentielle quadratique dans la strate $\Omega^{2}\moduli[1](9,-1;(-2^{4}))$ (à gauche) et $\Omega^{2}\moduli[1](4,4;(-2^{4}))$ (à droite) dont les $2$-résidus sont $(1,1,1,1)$.}
\label{fig:exceptquadrabis}
\end{figure}
\par
Pour conclure, il suffit de remarquer que toutes les strates  $\Omega^{2}\moduli[1](a_{1},a_{2};(-2^{s}))$ satisfaisant $(a_{1},a_{2})\neq(s-1,s+1)$ peuvent s'obtenir en ajoutant $2$ à l'ordre des deux zéros ou $4$ à l'ordre du zéro d'ordre maximal. 
\end{proof}

\begin{lem}\label{lem:ggentousres}
L'application $k$-résiduelle des strates $\komoduli(a_{1},\dots,a_{n};(-k^{s}))$ distinctes des strates $\Omega^{2}\moduli[1](2s;(-2^{s}))$ et $\Omega^{2}\moduli[1](s+1,s-1;(-2^{s}))$ avec $s$ pair  est surjective.
\end{lem}

\begin{proof}
Si $k\geq3$,  l'éclatement des zéros et la couture d'anse (cf propositions~\ref{prop:eclatZero} et~\ref{prop:attachanse}) à partir des $k$-différentielles données par le lemme~\ref{lem:g=1residugen} donne la surjection pour toutes les strates avec $g\geq1$.
\par
Considérons le cas $k=2$. Si $g=1$, alors nous avons montré le résultat pour au plus deux zéros dans les lemmes précédent. Pour les strates avec $n\geq3$ zéros, nous utilisons l'éclatement des zéros à partir des strates ayant deux zéros. Si $g\geq2$, il suffit de montrer que l'image de l'application résiduelle des strates de la forme $\Omega^{2}\moduli[2](2s+4;(-2^{s}))$ avec $s$ pair contient $(1,\dots,1)$. En effet, si c'est le cas l'éclatement des zéros et la couture d'anse impliqueront le résultat pour toutes les strates.
Pour cela, prenons une différentielle quadratique de la strate $\Omega^{2}\moduli[1](2s+4;(-2^{s+2}))$ dont les résidus quadratiques sont $(1,\dots,1,-1,-1)$ (cette différentielle existe en vertu du lemme~\ref{lem:g=1residugen}). Nous formons une différentielle quadratique entrelacée en collant les deux pôles avec les résidus quadratiques~$-1$. D'après la proposition~\ref{lem:lisspolessimples}, cette $2$-différentielle entrelacée est lissable et la différentielle obtenue par lissage possède les propriétés souhaitées.  
\end{proof}

\begin{lem}\label{lem:nonsurjunzero}
L'image de l'application résiduelle de $\Omega^{2}\moduli[1](2s;(-2^{s}))$, avec $s$ pair, est égale à $(\CC^{\ast})^{s}\setminus \CC^{\ast}\cdot(1,\dots,1)$.
\end{lem}

Rappelons que nous ne considérons que les différentielles quadratiques primitives. Sinon on pourrait simplement prendre le carré d'une différentielle abélienne dans $\omoduli[1](s,(-1^{s}))$ qui a pour résidus $(1,\dots,1,-1,\dots,-1)$.
\begin{proof}
D'après le lemme~\ref{lem:g=1quadgen}, il suffit de vérifier que les $2$-résidus $(1,\dots,1)$ ne sont pas dans l'image de l'application résiduelle de ces strates. 
\par
L'image de $\appresk[2][1](4;(-2^{2}))$ ne contient pas $(1,1)$ car sinon  on obtiendrait une différentielle quadratique entrelacée lissable en collant ses deux pôles (voir le lemme~\ref{lem:lisspolessimples}). La $2$-différentielle obtenue par lissage serait primitive dans la strate $\Omega^{2}\moduli[2](4)$, ce qui est absurde (voir \cite{masm}).
\par
Supposons qu'il existe une $2$-différentielle $\xi$ dans  $\Omega^{2}\moduli[1](2s;(-2^{s}))$ avec $s\geq2$ dont tous les résidus quadratiques sont $(1,\dots,1)$. Nous notons $S$ la surface plate associée à $\xi$. Sans perte de généralité, on peut supposer que le cœur de $S$ est dégénéré et contient $s+1$ liens selles distincts (voir la section~\ref{sec:coeur}). En coupant $S$ le long de ces liens selles, nous obtenons l'union disjointe de $s$ parties polaires quadratiques d'ordre $2$ bordées par $2s+2$ segments.

Remarquons qu'aucun domaine polaire ne possède un bord formé d'un unique segment. Sinon pour que  $S$ ne soit pas singulière il faudrait que l'autre pôle bordé par ce lien selle ait un résidu quadratique strictement supérieur à $1$. Donc les graphes d'incidence de $S$ (voir section~\ref{sec:coeur}) sont de l'une des deux formes suivante.
\begin{enumerate}[i)]
 \item Il existe un sommet de valence $4$ et les autres sont de valence $2$.
 \item Il existe deux sommets de valence $3$ et les autres sont de valence $2$.
\end{enumerate}
\par
Dans le graphe d'incidence de $S$, il y a des chaînes de sommets de valence~$2$. Pour tous ces domaines polaires, la somme des longueurs des liens selles est égale à~$1$. Donc dans ces chaînes, la période des liens selles est alternativement $\theta\in\left]0,1\right[$ et $1-\theta$. Nous pouvons enlever deux domaines polaires consécutifs de cette surface de la façon suivante. Nous coupons la surface $S$ le long des deux liens selles de période~$\theta$. Puis nous oublions les deux $2$-parties polaires et recollons les bords de la surface ainsi obtenue. La surface ainsi obtenue est encore primitive. Nous enlevons de $S$ toutes les paires adjacentes avec deux liens selles jusqu'à obtenir une surface dite {\em réduite}.
\par
Il suffit maintenant de montrer qu'il n'existe pas de surfaces réduites. Considérons tout d'abord le cas où l'un des sommets du graphe d'incidence est de valence~$4$. Les seuls graphes d'incidence possibles comportent un sommet d'où part deux boucles qui contiennent~$0$ ou~$1$ sommet de valence~$2$. Comme $s$ est pair, le seul graphe d'incidence possible correspond à une différentielle quadratique de $\Omega^{2}\moduli[1](4,(-2^{2}))$ avec résidus $(1,1)$. Cela montre qu'aucune surface réduite ne possède ce graphe d'incidence.
\par
Maintenant, nous considérons les surfaces réduites dont le graphe d'incidence possède deux sommets de valence~$3$. Dans ce cas deux types de graphes d'incidence sont possibles. L'un avec trois arêtes entre deux sommets et sur chaque arête  $0$ ou $1$ sommet de valence~$2$. L'autre avec une arête entre deux sommets et une boucle à chaque sommet. 
La seule possibilité est que ces graphes aient~$4$ sommets, sinon le nombre de pôles serait impaire ou égal à~$2$. Donc il y a un sommet de valence~$2$ sur deux arêtes. 
\par
Considérons le premier graphe. Du fait de la présence des deux sommets de valence $2$, les longueurs des liens selles au bord des domaines polaires de valence $3$ sont $a,b,c$ et $a,1-b,1-c$ respectivement avec $a,b,c \in \left] 0;1 \right[ $. De plus, par hypothèse on a les égalités $a+b+c=2+a-b-c =1$, ce qui implique que $a=0$. Ce donne la contradiction.
\par
Dans le second cas, il existe au moins une boucle qui contient un sommet. Comme la somme des longueurs des liens selles au bord de ce domaine polaire est~$1$, le troisième lien selle de ce domaine polaire est nul. Cela donne la dernière contradiction de cette preuve.
\end{proof}

\begin{lem}\label{lem:nonsurjdeuxzero}
 L'image de $\appresk[2][1](s+1,s-1;(-2^{s}))$, avec $s$ pair, est égale à $(\CC^{\ast})^{s}\setminus \CC^{\ast}\cdot(1,\dots,1) $.
\end{lem}

\begin{proof}
Nous commençons par remarquer que l'image de l'application résiduelle de la strate $\Omega^{2}\moduli[1](3,1;-2,-2)$ ne contient pas $(1,1)$. Sinon on pourrait former une différentielle entrelacée lissable en recollant les deux pôles doubles. La différentielle quadratique obtenue en la lissant serait dans $\Omega^{2}\moduli[2](3,1)$, qui est vide par \cite{masm}.

Nous traitons le cas $s\geq4$ par l'absurde. Nous supposons qu'il existe une différentielle quadratique dans la strate $\Omega^{2}(s-1,s+1;(-2^{s}))$ avec pour résidus quadratiques $(1,\dots,1)$. Dans un premier temps nous simplifions la différentielle quadratique puis nous montrons que les différentielles quadratiques simplifiées n'existent pas.

Nous supposerons que le cœur de la différentielle est dégénéré et qu'il y a donc $s+2$ liens selles horizontaux (voir section~\ref{sec:coeur}). En coupant le long des liens selles la surface nous obtenons~$s$ parties polaires quadratiques bordées par au moins~$2$ segments. Les domaines polaires dont le bord est composé par strictement plus de deux segments sont appelés {\em spéciaux}. Les valences des sommets correspondants dans le graphe d'incidence peuvent être de la forme $(6)$, $(5,3)$, $(4,4)$, $(4,3,3)$ ou $(3,3,3,3)$.

Nous considérerons tout d'abord le graphe d'incidence simplifié défini dans la section~\ref{sec:coeur}. On utilisera le vocabulaire des graphes et des surfaces de manière interchangeable. Par exemple, la {\em valence} d'un pôle est la valence du sommet correspondant dans le graphe d'incidence. Les graphes d'incidence simplifiés possibles sont représentés dans la figure~\ref{fig:graphesg=1} (d'autres graphes à priori possibles seront écartés par des considérations générales).

\begin{figure}[hbt]
\begin{tikzpicture}

\begin{scope}[xshift=-3.5cm]
\coordinate (z) at (0,0);\fill (z) circle (2pt);

\draw (z) .. controls ++(-30:1) and ++(30:1) .. (z);
\draw (z) .. controls ++(90:1) and ++(150:1) .. (z);
\draw (z) .. controls ++(-90:1) and ++(-150:1) .. (z);
\end{scope}

\begin{scope}[xshift=-1.3cm,yshift=-1.5]
\coordinate (z1) at (0,0);\fill (z1) circle (2pt);
\coordinate (z2) at (1,0);\fill (z2) circle (2pt);

\draw (z1) -- (z2);
\draw (z2) .. controls ++(-30:1) and ++(30:1) .. (z2);
\draw (z1) .. controls ++(90:1) and ++(150:1) .. (z1);
\draw (z1) .. controls ++(-90:1) and ++(-150:1) .. (z1);
\end{scope}

\begin{scope}[xshift=2cm]
\coordinate (z1) at (0,0);\fill (z1) circle (2pt);
\coordinate (z2) at (1,0);\fill (z2) circle (2pt);

\draw (z1) -- (z2);
\draw (z1) .. controls ++(30:.5) and ++(150:.5) .. (z2);
\draw (z1) .. controls ++(-30:.5) and ++(-150:.5) .. (z2);
\draw (z1) .. controls ++(-210:1) and ++(-150:1) .. (z1);
\end{scope}

\begin{scope}[xshift=4cm]
\coordinate (z1) at (0,0);\fill (z1) circle (2pt);
\coordinate (z2) at (1,0);\fill (z2) circle (2pt);

\draw (z1) .. controls ++(20:.5) and ++(160:.5) .. (z2);
\draw (z1) .. controls ++(-20:.5) and ++(-160:.5) .. (z2);
\draw (z1) .. controls ++(60:.5) and ++(120:.5) .. (z2);
\draw (z1) .. controls ++(-60:.5) and ++(-120:.5) .. (z2);

\end{scope}

\begin{scope}[xshift=6.3cm]
\coordinate (z1) at (0,0);\fill (z1) circle (2pt);
\coordinate (z2) at (1,0);\fill (z2) circle (2pt);

\draw (z1) .. controls ++(30:.5) and ++(150:.5) .. (z2);
\draw (z1) .. controls ++(-30:.5) and ++(-150:.5) .. (z2);
\draw (z1) .. controls ++(-210:1) and ++(-150:1) .. (z1);
\draw (z2) .. controls ++(30:1) and ++(-30:1) .. (z2);

\end{scope}

\begin{scope}[yshift=-1.5cm,xshift=-2cm]
\coordinate (z1) at (0,0);\fill (z1) circle (2pt);
\coordinate (z2) at (1,-.5);\fill (z2) circle (2pt);
\coordinate (z3) at (1,.5);\fill (z3) circle (2pt);

\draw (z1) .. controls ++(-210:1) and ++(-150:1) .. (z1);
\draw (z1) --(z2);
\draw (z1) --(z3);
\draw (z2) .. controls ++(60:.5) and ++(-60:.5) .. (z3);
\draw (z2) .. controls ++(120:.5) and ++(-120:.5) .. (z3);

\end{scope}

\begin{scope}[yshift=-2cm,xshift=1cm,rotate=90]
\coordinate (z1) at (0,0);\fill (z1) circle (2pt);
\coordinate (z2) at (1,-.5);\fill (z2) circle (2pt);
\coordinate (z3) at (1,.5);\fill (z3) circle (2pt);

\draw (z1) .. controls ++(0:.5) and ++(130:.5) .. (z2);
\draw (z1) .. controls ++(0:.5) and ++(-130:.5) .. (z3);
\draw (z2) --(z3);
\draw (z1) .. controls ++(-60:.5) and ++(180:.5) .. (z2);
\draw (z1) .. controls ++(60:.5) and ++(180:.5) .. (z3);

\end{scope}

\begin{scope}[yshift=-1.5cm,xshift=3.5cm,scale=.7]
\coordinate (z0) at (0,0);\fill (z0) circle (2pt);
\coordinate (z1) at (1,0);\fill (z1) circle (2pt);
\coordinate (z2) at (-.5,1.71/2);\fill (z2) circle (2pt);
\coordinate (z3) at (-.5,-1.71/2);\fill (z3) circle (2pt);

\draw (z0) circle (1cm);
\draw (z0) -- (z1);
\draw (z0) -- (z2);
\draw (z0) -- (z3);
\end{scope}

\begin{scope}[yshift=-1.5cm,xshift=6cm,scale=.7]
\coordinate (z0) at (1,0);\fill (z0) circle (2pt);
\coordinate (z1) at (0,1);\fill (z1) circle (2pt);
\coordinate (z2) at (-1,0);\fill (z2) circle (2pt);
\coordinate (z3) at (0,-1);\fill (z3) circle (2pt);

\draw (0,0) circle (1cm);
\draw (z0) .. controls ++(180:.5) and ++(-90:.5) .. (z1);
\draw (z2) .. controls ++(-10:.5) and ++(100:.5) .. (z3);
\end{scope}
\end{tikzpicture}
\caption{Tous les graphes d'incidence simplifiés des surfaces réduites} \label{fig:graphesg=1}
\end{figure}
\smallskip
\par
Commençons par donner une propriété importante des graphes d'incidence. Ces graphes ne peuvent pas contenir de boucle formée d'un nombre impair de sommets de valence~$2$. Sinon la somme des longueurs des liens selles initial et final serait $1$. Cela est impossible car au moins un autre lien selle borde ce domaine polaire spécial.

Cela nous permet d'exclure la possibilité d'un unique pôle spécial de valence $6$. En effet, le seule graphe d'incidence simplifié possible est la fleur à trois pétales (voir la figure~\ref{fig:graphesg=1}). Mais comme le nombre de pôles de la différentielle est par hypothèse pair, il y a au moins un pétale qui contient un nombre impair de sommets de valence $2$.
\smallskip
\par
Une première façon de simplifier la surface est d'{\em enlever} une paire de pôles voisins de valence $2$. Cette opération est décrite dans la preuve du lemme~\ref{lem:nonsurjunzero}. Remarquons qu'elle diminue l'ordre des deux zéros de $2$ ou seulement l'un des deux de $4$. Afin de rester dans une strate du type $\Omega^{2}\moduli[1](s+1,s-1;(-2^{s}))$, nous autorisons cette opération si et seulement si les liens selles connectent les deux zéros entre eux ou le zéro d'ordre maximal à lui même. Remarquons qu'après avoir enlevé une paire avec des liens selles connectant au zéro d'ordre maximal, celui-ci devient le zéro d'ordre minimal. Ainsi en jouant au ping-pong, on pourra enlever de nombreuses paires de pôles. En particulier, après avoir enlevé toutes les paires de sommets possibles, les paires restantes sont bordées par des liens selles entre le zéro d'ordre minimal. \`A partir de maintenant, nous considérons uniquement des surfaces où il n'est pas possible d'enlever une paire de pôles.
\smallskip
\par
Nous montrons maintenant qu'il n'est pas possible d'avoir une boucle dans le graphe d'incidence simplifié à un sommet de valence $3$. Comme nous l'avons remarqué, le nombre de sommets de valence deux sur cette boucle doit être pair. Si ce nombre est zéro, alors nous avons un zéro d'ordre $-1$ sur cette surface. Mais ce cas ne se présente dans aucune des strates considérées.
Si le nombre de sommets sur cette boucle était strictement positif, alors le pôle de valence $3$ serait bordé uniquement par les liens selles entre le zéro d'ordre minimum et lui même. Le lien selle se situe entre ce domaine polaire et un domaine polaire distinct. Cela implique qu'au moins deux autres segments au bord d'un autre domaine polaire spécial possèdent ce zéro à une extrémité. Nous faisons maintenant une analyse au cas par cas.

Dans le cas $(5,3)$ cela implique que la contribution angulaire des pôles spéciaux au zéro d'ordre inférieur est d'au moins $5\pi$. Ainsi leurs contributions angulaires au zéro d'ordre maximal est d'au plus $3\pi$. Comme l’extrémité de trois des cinq segments du bord du domaine de valence~$5$ contiennent le zéro d'ordre minimal, au plus une boucle au sommet de valence~$5$ peut être bordée par des liens selles connectant le zéro maximal à lui même. Comme il existe au maximum un pôle de valence deux sur cette boucle, l'ordre du zéro minimal est supérieur ou égal à celui du zéro d'ordre maximal. Cela donne la contradiction souhaitée. Dans les deux autres cas $(4,3,3)$ et $(3,3,3,3)$ une analyse similaire  aboutit à la même contradiction.
\smallskip
\par
Nous poursuivons en montrant qu'il existe au plus un pôle de valence $2$ sur une arête du graphe simplifié. Supposons qu'il existe une arête avec plus de $2$ sommets. Les liens selles correspondants relient le zéro d'ordre minimal à lui même. Soit cette arête connecte un pôle de valence $3$ ou $4$ à un autre sommet spécial. Soit c'est une boucle sur un sommet de valence $5$ ou $4$. On vérifie facilement que dans les deux cas cela entraîne que le zéro d'ordre minimal est l'extrémité de liens selles de deux sommets du graphe d'incidence simplifié. Cela entraîne, de manière analogue au paragraphe précédent, que la contribution angulaire au zéro d'ordre minimal est plus grande que celle au zéro d'ordre maximal. On en déduit la même contradiction. A partir de maintenant, nous considérerons des graphes d'incidence avec $0$ ou $1$ sommet de valence $2$ sur les arêtes du graphe simplifié.
\smallskip
\par
Considérons le cas où les pôles spéciaux sont de valences  $(5,3)$.  Dans ce cas, deux graphes d'incidence simplifiés sont a priori possibles (voir la figure~\ref{fig:graphesg=1}). Le premier a une boucle à un pôle de valence $3$ et est donc impossible. Le second graphe a trois arêtes entre les deux sommets spéciaux et une boucle au sommet de valence $5$. Comme nous considérons des strates avec un nombre pair supérieur ou égal à $4$ pôles, il y a  $2$ ou $4$ sommets de valence~$2$ sur les arêtes. Si il y en avait $4$, alors il y aurait un sommet sur la boucle, ce qui est impossible. Donc il y a forcément deux sommets sur les arêtes joignant les deux sommets spéciaux. Dénotons par $0<a,b,c<1$ les longueurs des liens selles bordant le pôle de valence $3$. Les longueurs des liens selles au pôle de valence $5$ seront donc $a,1-b,1-c$ et~$2d$ avec $d>0$. Comme par hypothèse $a+b+c=1$ on a $a+2-b-c+2d=1+2a+2d>1$. Ce qui contredit le fait que le résidu quadratique à ce pôle est $1$. 
\smallskip
\par
Dans le cas $(4,4)$, il y a deux graphes simplifiés possibles. Considérons tout d'abord le graphe avec~$4$ arêtes entre les deux sommets spéciaux. Il y a~$2$ ou~$4$ sommets de valence~$2$ sur ces arêtes. Les longueurs des liens selles à un pôle spécial  sont $0<a,b,c,d<1$ avec $a+b+c+d=1$. \`A l'autre pôle spécial, il y a au moins deux longueurs de la forme $1-a,1-b$. Donc le résidu à ce pôle est strictement supérieur à~$1$.

Le second graphe simplifié possède deux boucles et deux arêtes connectant les pôles spéciaux (voir la figure~\ref{fig:graphesg=1}). La seule possibilité est qu'il y ai deux sommets non spéciaux sur les arêtes connectant les pôles spéciaux. On obtient comme précédemment une contradiction en considérant les longueurs des liens selles.
\smallskip
\par
Nous définissons maintenant une nouvelle opération sur des pôles voisins appelée {\em fusion}. Prenons deux pôles d'ordres $-2$ connectés par au moins un lien selle. Nous pouvons alors remplacer ces deux pôles par un pôle d'ordre $4$, voir la figure~\ref{fig:fusion}. Remarquons que cette opération ne modifie ni l'ordre, ni le résidu des autres singularités. De plus, la différentielle ainsi modifiée reste primitive. Enfin, le résidu du pôle d'ordre $-4$ ainsi créé est nul.
\begin{figure}[htb]
\begin{tikzpicture}[scale=1]

\coordinate (a) at (-1,0);
\coordinate (b) at (0,0);
\coordinate (c) at (1,0);

    \fill[fill=black!10] (a)  -- ++(0,1.2)coordinate[pos=.5](f) -- ++(2,0) --++(0,-2.4) --++(-2,0) -- cycle;
    \fill (a)  circle (2pt);
\fill[] (c) circle (2pt);
 \draw  (a) -- (b) coordinate[pos=.5](h);
   \fill[white] (b)  circle (2pt);
      \draw[] (b)  circle (2pt);
 \draw (a) -- ++(0,1.1) coordinate (d);
 \draw (c) -- ++(0,1.1) coordinate (e);
  \draw (a) -- ++(0,-1.1) coordinate (f);
 \draw (c) -- ++(0,-1.1) coordinate (g);
 \draw[dotted] (d) -- ++(0,.2);
 \draw[dotted] (e) -- ++(0,.2);
  \draw[dotted] (f) -- ++(0,-.2);
 \draw[dotted] (g) -- ++(0,-.2);
\node[below] at (h) {$1$};
\node[above] at (h) {$2$};

 \draw[->] (1.8,0) --  ++(1.8,0) coordinate[pos=.5](t);
\node[above] at (t) {fusion};

\begin{scope}[xshift=4cm]
\fill[fill=black!10] (1.5,0)  circle (1.3cm);

    \draw (1,0) -- (2,0) coordinate[pos=.5] (c);

  \fill[white] (2,0)  circle (2pt);
      \draw[] (2,0)  circle (2pt);
         \fill (1,0)  circle (2pt);

\node[below] at (c) {$1$};
\node[above] at (c) {$2$};

\end{scope}
\end{tikzpicture}
\caption{La fusion de deux pôles d'ordre $-2$ en un pôle d'ordre $-4$.}
\label{fig:fusion}
\end{figure}
\smallskip
\par
Maintenant, nous traitons le cas où les pôles spéciaux sont de valences $(4,3,3)$. Il n'y a que deux graphes d'incidence simplifiés a priori possibles (voir la figure~\ref{fig:graphesg=1}).  Pour le graphe qui à une boucle au pôle de valence $4$, il peut y avoir $1$, $3$ ou $5$ pôles supplémentaires de valence $2$. Ce dernier cas est impossible car il y aurait un sommet de valence $2$ sur la boucle.
Si on avait un sommet supplémentaire, alors on pourrait fusionner les pôles par paires afin d'obtenir une différentielle dans $\Omega^{2}\moduli[1](5,3;-4,-4)$ dont les résidus soient nuls. Nous avons montré qu'une telle différentielle n'existe pas dans le lemme~\ref{lem:geq1six}. Donc la différentielle originale ne peut pas exister. Les cas avec trois pôles supplémentaires se traite de manière similaire par fusion des pôles.

Dans le cas du graphe d'incidence simplifié qui ne possède pas de boucle, nous pouvons avoir $1$, $3$ ou $5$ pôles de valence $2$ supplémentaires. Le cas avec $5$ pôles est impossible à cause des longueurs des liens selles. En effet, les longueurs des liens selles entre le pôle de valence $4$ et l'un des pôles de valence $3$ sont $a$ et $b$ avec $a+b<1$ à un pôle, et $1-a$ et $1-b$ avec $2-a-b>1$ à l'autre. Cela est clairement impossible. Enfin les cas avec $1$ et $3$ sommets de valence $2$ peuvent se traiter avec la fusion en se ramenant au lemme~\ref{lem:geq1six}.

\smallskip
\par
Pour conclure, nous considérons le cas des graphes dont les sommets spéciaux sont de valence $(3,3,3,3)$. Les graphes possibles sont représentés dans la figure~\ref{fig:graphesg=1}. Nous considérons tout d'abord le graphe complet $K_{4}$. Pour $0$, $2$ or $4$ pôles de valence deux sur les arêtes, l'impossibilité peut se montrer en se ramenant au lemme~\ref{lem:geq1six} par fusion. S'il y a six pôles de valence $2$, nous regardons les longueurs des liens selles aux pôles spéciaux. Les longueurs de celles bordant le sommet central sont de longueur $a$, $b$ et $1-a-b$. Les longueurs des liens selles correspondants  bordant les autres sommets spéciaux sont $1-a$, $1-b$ et $a+b$. Considérons le pôle avec le lien selle de longueur $a+b$. Il a un autre lien de selle de longueur~$c$. La longueur du lien selle correspondant sur l'autre pôle spécial (disons celui avec le lien selle de longueur $1-a$) est $1-c$.  Le troisième lien selle de ce pôle est de longueur   $a+c-1$. En continuant notre voyage autour du cercle, nous trouvons que la longueur des liens selles du pôle d'où nous sommes parti sont  $a+b$, $c$ et $3-a-b-c$ (c'est-à-dire $1+2a+2c$). Comme la somme est $3$ cela donne la contradiction souhaitée.
 
Enfin nous considérons le graphe en forme de balle de tennis. Comme précédemment, les cas où il y a $0$, $2$ ou $4$ pôles de valence $2$ supplémentaires peut être traité par fusion. Pour $6$ pôles supplémentaires, une analyse de la longueur des liens selles similaire au cas du graphe $K_{4}$, permet d'obtenir la dernière contradiction de cette preuve.
\end{proof}

\subsection{Pluridifférentielles d'aires finies}
\label{sec:plurifini}

Dans ce dernier paragraphe nous prouvons le théorème~\ref{thm:strateshol}. Ce résultat dit que les strates primitives $\komoduli(a_{1},\dots,a_{n})$ en genre $g\geq1$ sont vides si et seulement si $\mu=\emptyset$ ou $\mu=(1,-1)$ en genre $1$ ou $\mu=(4)$ et $\mu=(3,1)$ pour $k=2$.  
Toutes les $k$-différentielles de type $(\emptyset)$ sont la puissance $k$-ième d'une différentielle de $\omoduli[1](\emptyset)$. De plus, le théorème d'Abel implique que les strates  $\komoduli[1](1,-1)$ sont vides. Pour les deux strates quadratiques vides de genre $2$, nous renvoyons à \cite{masm}. Notre but est donc de montrer que toutes les autres strates sont non vides. 

En genre un, il existe des $k$-différentielles de type $\mu$ pour tout $\mu\neq(1,-1)$. Supposons que $\mu:=(a_{1},\dots,a_{n})\neq\emptyset$, il reste \`a montrer que certaines d'entre elles sont primitives. Étant donn\'e un tore $X$, il existe un diviseur $D:=\sum a_{i}z_{i}$ sur $X$ lin\'eairement équivalent \`a z\'ero tel que pour tout diviseur $d$ des $a_{i}$, le diviseur $\sum\frac{a_{i}}{d}z_{i}$ n'est pas lin\'eairement équivalent \`a z\'ero (voir par exemple \cite{Bo}). La $k$-différentielle ayant pour diviseur $D$ est clairement primitive.

Soit $\komoduli[g](a_{1},\dots,a_{n})$ une strate de genre $g\geq 2$ avec $a_{i}>-k$. Si $(a_{1},\dots,a_{n})$ est différent de $(4g-4)$ et $(2g-1,2g-3)$ dans le cas quadratique, alors, par le théorème~\ref{thm:geq1}, la strate  $\komoduli[1](a_{1},\dots a_{n};((-2k)^{g-1}))$ contient une différentielle primitive $(X_{0},\omega_{0})$ dont tous les $k$-résidus sont nuls. On obtient une $k$-différentielle entrelacée  en attachant la puissance $k$i\`eme d'une différentielle holomorphe sur un tore aux pôles de $X_{0}$.
Cette pluridifférentielle entrelacée est lissable par le lemme~\ref{lem:lissdeuxcomp}. Les lissages sont clairement des $k$-différentielles primitives. Donc ces strates sont non vides.

Enfin il reste le cas des strates $\Omega^{2}\moduli[g](4g-4)$ et $\Omega^{2}\moduli[g](2g-1,2g-3)$. En utilisant l'éclatement de zéros, il suffit de montrer que la strate $\Omega^{2}\moduli[g](4g-4)$ est non vide pour tout $g\geq3$. Nous savons que $\Omega^{2}\moduli[2](4(g-1);(-4^{g-2}))$ contient une différentielle primitive sans résidus quadratiques aux pôles (voir Théorème 1.1). On obtient une différentielle quadratique entrelacée en attachant à chaque pôle le carré d'une différentielle de $\omoduli[1](\emptyset)$. Les différentielles quadratiques obtenues par lissage sont dans $\Omega^{2}\moduli[g](4g-4)$ comme souhaité.

\section{Applications.}
\label{sec:appli}

Nous donnons quelques applications élémentaires de nos résultats. Une connaissance des notions introduites dans \cite{BCGGM} et \cite{BCGGM3} est recommandée.

\smallskip
\par
\subsection{Limites des points de Weierstra\ss: preuve de la proposition~\ref{prop:limWei}}

L'étude des limites des points de \Weierstrass~dans la compactification de Deligne-Mumford a pris une grande ampleur grâce aux travaux d'Eisenbud et Harris sur les séries linéaires limites (voir e.g. \cite{eiha}). L'une des limites de leur méthode est de se restreindre aux courbes de type compact. Esteves et Medeiros l'ont étendue aux courbes ayant deux composantes dans \cite{esme}. Nos résultats et la description de la compactification de la variété d'incidence de \cite{BCGGM} permettent une description complète (en théorie) des limites de points de \Weierstrass~dans~$\barmoduli[g,1]$.  

Nous montrons tout d'abord que l'adhérence du lieu de \Weierstrass~ne rencontre pas celui des courbes stables $X$ où $g$ courbes elliptiques sont attachées à un $\PP^{1}$ contenant le point marqué (voir le théorème 3.1 de \cite{eiha}). Ces courbes sont figurées à gauche de la~figure~\ref{fig:courbesWei}.

 \begin{figure}[htb]
\begin{tikzpicture}[scale=1]
\begin{scope}[xshift=-2.5cm]
\draw (-4,0.1) coordinate (x1) .. controls (-3.8,-.3) and (-4.2,-.6) .. (-4,-1);
\draw (-3.4,0.1) coordinate (x4) .. controls (-3.2,-.3) and (-3.6,-.6) .. (-3.4,-1);
\draw (-2,0.1)  coordinate (x2).. controls (-1.8,-.3) and (-2.2,-.6) .. (-2,-1);

\draw (-4.2,-.8) .. controls (-3.3,-1) and (-2.6,-.6) .. (-1.8,-.8) coordinate (x3)
coordinate [pos=.6] (z);
\fill (z) circle (2pt);\node [below] at (z) {$z$};
\node [above] at (x1) {$X_{1}$};\node [above] at (x2) {$X_{g}$};\node [right] at (x3) {$\PP^{1}$};\node [above] at (x4) {$X_{2}$};
\node at (-2.7,-.3) {$\dots$};
\end{scope}

\begin{scope}[xshift=2.5cm]
\draw (-4.8,-1) .. controls (-4.6,.5) and (-4.2,.5) .. (-4,-1) coordinate[pos=.5] (x1);
\draw (-3.4,0.1) coordinate (x4) .. controls (-3.2,-.3) and (-3.6,-.6) .. (-3.4,-1);
\draw (-2,0.1)  coordinate (x2).. controls (-1.8,-.3) and (-2.2,-.6) .. (-2,-1);

\draw (-5,-.8) .. controls (-3.7,-1) and (-2.9,-.6) .. (-1.8,-.8) coordinate (x3)
coordinate [pos=.7] (z);
\fill (z) circle (2pt);\node [below] at (z) {$z$};
\node [above] at (x1) {$X_{0}$};\node [above] at (x2) {$X_{g-2}$};\node [right] at (x3) {$\PP^{1}$};\node [above] at (x4) {$X_{1}$};
\node at (-2.7,-.3) {$\dots$};
\end{scope}
\end{tikzpicture}
\caption{Les courbes pointées considérées dans la  proposition~\ref{prop:limWei}.}
\label{fig:courbesWei}
\end{figure}

L'adhérence du lieu de \Weierstrass~dans $\barmoduli[g,1]$ coïncide avec la projection de la compactification de la variété d'incidence de $\omoduli[g](g,1,\dots,1)$ (voir \cite[Section~3.6]{BCGGM}). D'après le théorème~1.3 de l'article loc. cit., il suffit de montrer qu'il n'existe pas de différentielle entrelacée lissable sur une courbe semi-stablement équivalente à~$X$.

Soit~$\xi$ une différentielle entrelacée sur (une courbe semi-stablement équivalente à) ~$X$. La restriction~$\xi_{0}$ de $\xi$ à $\PP^{1}$ possède un zéro d'ordre supérieur ou égal à $g$. De plus,  $\xi_{0}$ possède une singularité d'ordre supérieure ou égale à $-2$ (et distinct de $-1$) aux nœuds entre $\PP^{1}$ et une courbe elliptique. On vérifie facilement que l'inégalité~\eqref{eq:genrezeroresiduzerofr} est satisfaite par~$\xi_{0}$. Le théorème~\ref{thm:geq0keq1} implique qu'au moins deux pôles possèdent des résidus non nuls. La {\em condition résiduelle globale} de la définition 1.2 de \cite{BCGGM} n'est donc pas satisfaite et~$\xi$ n'est pas lissable.

Maintenant, nous montrons que l'adhérence du lieu de \Weierstrass~coupe le lieu de $\barmoduli[g,1]$ donné de la façon suivante. Ces courbes sont formées d'un $\PP^{1}$ contenant le point marqué attaché à $g-2$ courbes elliptiques $X_{1},\dots,X_{g-2}$ par un unique point et à une courbe elliptique $X_{0}$ par deux points. Elles sont représentées à droite de la figure~\ref{fig:courbesWei}.

Il suffit construire une différentielle entrelacée lissable de type $(g,(1^{g-2}))$ sur une de ces courbes. Sur toutes les courbes elliptiques, on prend la différentielle holomorphe. Sur la courbe projective, on prend une différentielle dans $\omoduli[0](g,(1^{g-2});(-2^{g}))$ avec $g-2$ résidus nuls. Une telle différentielle existe par le théorème~\ref{thm:geq0keq1}. On colle alors la courbe elliptique~$X_{0}$ aux deux pôles dont les résidus ne sont pas nuls. Les autres courbes elliptiques sont collées aux pôles dont les résidus sont nuls. On vérifie que cette différentielle entrelacée est lissable grâce au théorème 1.3 de \cite{BCGGM}.

\smallskip
\par
\subsection{\'Eclatement de zéros: preuve de la proposition~\ref{prop:eclatintro}}

L'idée de scinder les singularités coniques d'une métrique plate est déjà connue. Dans le cas des métriques induites par une différentielle abélienne, une belle construction à été proposée par Eskin, Kontsevich, Masur et Zorich (voir  \cite[\S~8.1]{EMZ} et \cite{kozo}). En revanche, pour les différentielles quadratiques Boissy, Lanneau, Masur et Zorich ont remarqué qu'une telle construction peut ne pas être réalisable de manière locale (voir \cite{MZ}).
La compréhension de la variété d'incidence des strates de pluridifférentielles permet de décrire ce phénomène dans le cadre des pluridifférentielles (voir \cite[Example~7.1]{BCGGM3}). Nous pouvons maintenant caractériser les cas où l'éclatement d'une singularité conique n'est pas possible localement.

L'éclatement d'un zéro d'ordre $a>-k$ d'une $k$-différentielle~$\xi$ en $n$ zéros d'ordres $(a_{1},\dots,a_{n})$ correspond au lissage d'une $k$-différentielle entrelacée. Cette $k$-différentielle entrelacée est constituée d'une $k$-différentielle $\xi_{0}$ sur $\PP^{1}$ avec  des zéros d'ordres $(a_{1},\dots,a_{n})$ et un pôle d'ordre $-a-2k$ attachée au zéro d'ordre $a$ de~$\xi$ (voir la proposition~\ref{prop:eclatZero}). Cette $k$-différentielle entrelacée est lissable localement si et seulement tous les $k$-résidus de $\xi_{0}$ sont nuls. L'éclatement est possible localement si et seulement s'il existe une $k$-différentielle de genre zéro de type $(a_{1},\dots,a_{n};-2k-a)$ dont tous les $k$-résidus sont nuls. La proposition~\ref{prop:eclatintro} est alors une conséquence directe du théorème~\ref{thm:r=0s=0} et de l'équation~\eqref{eq:multiplires}.
\smallskip
\par
\subsection{Cylindres dans une surface plate: preuve de la proposition~\ref{prop:cylindres}}

Naveh (\cite{Na}) nous apprend que le nombre maximal de cylindres disjoints dans une surface de $S:=\omoduli(m_{1},\dots,m_{n})$ est $g+n-1$.
Nous décrivons les périodes possibles des circonférences de ces cylindres. On fixe $(\lambda_{1},\dots,\lambda_{t})\in (\mathbb{C}^{\ast})^{t}$ pour le reste de cette section.
\par
Supposons qu'il existe une différentielle $\omega$ de $S$ qui possède une famille de~$t$ cylindres disjoints de circonférences respectives $\lambda_{1},\dots,\lambda_{t}$. Nous montrons l'existence d'une différentielle stable dont les zéros sont d'ordres $(a_{1},\dots,a_{n})$ avec des pôles simples aux nœuds dont les résidus sont $\pm \lambda_{i}$.  Coupons $\omega$ le long d'une d'une géodésique périodique dans chaque cylindre. Nous pouvons alors remplacer les deux demi-cylindres obtenus par deux demi-cylindres infinis de même circonférence. On obtient donc une différentielle entrelacée avec des pôles simples aux nœuds. De plus, les résidus de ces pôles sont égaux à plus ou moins la circonférence des cylindres. On en déduit le sens direct de la proposition~\ref{prop:cylindres}.
\par
La direction réciproque est une application directe du lissage des nœuds des différentielles stables avec pôles simples aux nœuds (voir le lemme~\ref{lem:lisspolessimples}).
\par
Dans le cas des strates minimales, ces différentielles entrelacées sont irréductibles. On en déduit le résultat suivant.
\begin{cor}
Le $t$-uplet $(\lambda_{1},\dots,\lambda_{t})\in (\mathbb{C}^{\ast})^{t}$ est un vecteur constitué des périodes des circonférences de cylindres disjoints d'une différentielle de $\omoduli(2g-2)$ si et seulement si le $2t$-uplet $(\lambda_{1},\dots,\lambda_{t},-\lambda_{1},\dots,-\lambda_{t})$ est dans l'image de  $\appres[g-t](2g-2;(-1^{2t}))$.
\end{cor}

Nous énonçons maintenant la proposition~\ref{prop:cylindres} dans le langage des {\em représentations en graphe} développé dans \cite{tahar}. L'idée est de considérer le graphe dual de la différentielle entrelacée de cette proposition. Le théorème 2.3  de \cite{tahar} permet de caractériser de façon combinatoire les graphes qui apparaissent pour une strate donnée. Nous allons les enrichir de façon à prendre également en compte les circonférences des cylindres.

Soit $\omoduli(a_{1},\dots,a_{n})$ une strate de différentielles abéliennes. Une \textit{représentation en graphe enrichie de niveau $u$} $(G,f_0,\dots,f_u,g_0,\dots,g_u)$ est définie comme suit. Le graphe $G$ est connexe avec $u+1$ sommets de valences $v_0,\dots,v_u$ et $t$ arêtes (il peut exister plusieurs arêtes entre deux sommets). Les entiers naturels $g_0,\dots,g_u$ sont des poids associés à chaque sommet. Les nombres $a_1,\dots,a_n$ sont répartis en $u+1$ familles $f_i$ telles que :
 \begin{itemize}
 \item[i)]  il y a au moins un nombre parmi $(a_1,\dots,a_n)$ dans chaque famille $f_i$; 
 \item[ii)] chaque somme $\sigma(i)$ des éléments de  $f_i$ vérifie $\sigma(i) +v_{i}=2g_i-2$;
 \item[iii)] aucune arête ne déconnecte $G$.
 \end{itemize}
En déroulant les définitions on montre que la proposition~\ref{prop:cylindres} s'énonce de la façon suivante.
\begin{prop}\label{prop:cylindresgraphe}
Soient  $S:=\omoduli(a_{1},\dots,a_{n})$ une strate de différentielles abéliennes et $\lambda:=(\lambda_{1},\dots,\lambda_{t})\in (\mathbb{C}^{\ast})^{t}$.
Il existe  une différentielle dans $S$ qui possède une famille de $t$ cylindres disjoints dont les circonférences sont $\lambda$ si et seulement s'il existe une représentation en graphe $G$ avec $t$ arêtes et un marquage de chaque demi-arête par les éléments de $(\lambda,-\lambda)$ de telle façon que les propriétés suivantes sont vérifiées.
\begin{itemize}
\item[i)] Les demi-arêtes constituant une arêtes sont marqués par $\lambda_{i}$ et $-\lambda_{i}$;
\item[ii)] les poids $\lambda_{j_{i}}$ des demi-arêtes adjacentes au sommet $i$ doivent être dans l'image de  l'application $\appres[g_{i}](f_{i},(-1^{v_{i}}))$.%
\end{itemize}
\end{prop}


\printbibliography

\end{document}